\documentclass[11pt]{article}
\usepackage{amsmath}[1996/11/01]
\usepackage{amssymb,amsthm,amsxtra}
\usepackage{amscd}
\usepackage[margin=1in]{geometry}
\usepackage[title,titletoc]{appendix}

\def\[#1\]{\begin{equation}#1\end{equation}}
\makeatletter
\def\beq{%
   \relax\ifmmode
      \@badmath
   \else
      \ifvmode
         \nointerlineskip
         \makebox[.6\linewidth]%
      \fi
      $$
   \fi
}
\def\eeq{%
   \relax\ifmmode
      \ifinner
         \@badmath
      \else
         $$
      \fi
   \else
      \@badmath
   \fi
   \ignorespaces
}

\def\enddisplaymath{\eeq\global\@ignoretrue}
\makeatother

\newtheorem{thm}{Theorem}
\newtheorem{cor}[thm]{Corollary}
\newtheorem{lem}[thm]{Lemma}
\newtheorem{prop}[thm]{Proposition}

\theoremstyle{remark}
\newtheorem*{rem}{Remark}
\newtheorem{rems}{Remark}[thm]

\theoremstyle{definition}
\newtheorem{defn}{Definition}

\numberwithin{equation}{section}
\numberwithin{thm}{section}
\numberwithin{eg}{section}
\numberwithin{defn}{section}

\newcommand{\F}{\mathbb F}
\renewcommand{\P}{\mathbb P}
\newcommand{\Q}{\mathbb Q}
\newcommand{\Z}{\mathbb Z}
\newcommand{\N}{\mathbb N}
\newcommand{\C}{\mathbb C}

\newcommand{\A}{\mathbb A}
\newcommand{\G}{\mathbb G}

\DeclareMathOperator{\Aut}{Aut}
\DeclareMathOperator{\GL}{GL}
\DeclareMathOperator{\PGL}{PGL}
\DeclareMathOperator{\SL}{SL}
\DeclareMathOperator{\Pic}{Pic}
\DeclareMathOperator{\Alb}{Alb}
\DeclareMathOperator{\NS}{NS}

\DeclareMathOperator{\Ext}{Ext}
\DeclareMathOperator{\Sym}{Sym}
\DeclareMathOperator{\End}{End}
\DeclareMathOperator{\Hom}{Hom}
\DeclareMathOperator{\Mat}{Mat}
\DeclareMathOperator{\Lie}{Lie}

\DeclareMathOperator{\rank}{rank}
\DeclareMathOperator{\coker}{coker}
\DeclareMathOperator{\supp}{supp}
\DeclareMathOperator{\tr}{tr}
\DeclareMathOperator{\ad}{ad}
\DeclareMathOperator{\Ad}{Ad}
\DeclareMathOperator{\Tr}{Tr}
\DeclareMathOperator{\Spec}{Spec}
\DeclareMathOperator{\Proj}{Proj}
\DeclareMathOperator{\Quot}{Quot}
\DeclareMathOperator{\Tor}{Tor}
\DeclareMathOperator{\im}{im}
\DeclareMathOperator{\ch}{char}
\DeclareMathOperator{\coh}{coh}
\DeclareMathOperator{\qcoh}{qcoh}
\DeclareMathOperator{\Br}{Br}

\DeclareMathOperator{\rad}{rad}
\DeclareMathOperator{\gr}{gr}
\DeclareMathOperator{\Grass}{Grass}
\newcommand{\sO}{\mathcal O}

\DeclareMathOperator{\sExt}{{\mathcal E}\!{\it xt}}
\DeclareMathOperator{\sHom}{{\mathcal H}\!{\it om}}
\DeclareMathOperator{\sEnd}{{\mathcal E}\!{\it nd}}
\DeclareMathOperator{\Spl}{{\mathcal S}\!{\it pl}}
\DeclareMathOperator{\Tors}{{\mathcal T}\!{\it ors}}
\DeclareMathOperator{\Refl}{{\mathcal R}\!{\it efl}}

\DeclareMathOperator{\Hilb}{Hilb}

\newcommand{\dL}{{\bf L}}
\newcommand{\ratto}{\dashrightarrow}

\DeclareMathOperator{\ord}{ord}

\DeclareMathOperator{\cone}{cone}

\DeclareMathOperator{\gl}{{\mathfrak{gl}}}
\DeclareMathOperator{\so}{{\mathfrak{so}}}

\DeclareMathOperator{\num}{num}
\DeclareMathOperator{\perf}{perf}
\DeclareMathOperator{\pt}{pt}
\DeclareMathOperator{\red}{red}

\DeclareMathOperator{\nr}{nr}
\DeclareMathOperator{\id}{id}
\DeclareMathOperator{\Nm}{Nm}
\renewcommand{\_}{{\underline{\ \ }}}
\newcommand{\et}{\text{\'et}}
\newcommand{\whQ}{{\widehat{Q}}}
\newcommand{\frakm}{{\mathfrak{m}}}

\begin{document}

\title{The birational geometry of noncommutative surfaces}
  \author{
Eric M. Rains\\Department of Mathematics, California
  Institute of Technology}

\date{July 25, 2019}
\maketitle

\begin{abstract}
We show that any commutative rationally ruled surface with a choice of
anticanonical curve admits a 1-parameter family of noncommutative
deformations parametrized by the Jacobian of the anticanonical curve, and
show that many standard facts from commutative geometry (blowups commute,
Quot schemes are projective, etc.) carry over.  The key new tool in
studying these deformations is a relatively simple description of their
derived categories and the relevant $t$-structures; this also allows us to
establish nontrivial derived equivalences for deformations of elliptic
surfaces.  We also establish that the category of line bundles (suitably
defined) on such a surface has a faithful representation in which the
morphisms are difference or differential operators, and thus find that
difference/differential equations can be viewed as sheaves on such
surfaces.  In particular, we find that many moduli spaces of sheaves on
such surfaces have natural interpretations as moduli spaces of equations
with (partially) specified singularities, and in particular find that the
``isomonodromy'' interpretation of discrete Painlev\'e equations and their
generalizations has a natural geometric interpretation (twisting sheaves by
line bundles).

\end{abstract}

\tableofcontents

\section{Introduction}

The present work had its origin in the discovery of a connection between
noncommutative geometry and discrete Painlev\'e equations: the latter can
be interpreted as isomorphisms between moduli spaces coming from the former
(see \cite{P2Painleve} for the case of sheaves on elliptic noncommutative
planes).  Since the continuous Painlev\'e equations and many of their
discrete generalizations also have interpretations in terms of moduli
spaces of differential (or difference) equations, this suggested both that
there should be a connection between equations and noncommutative geometry
and, since moduli spaces of sheaves on commutative surfaces are
well-understood, that the latter could be a useful tool for understanding
the former.  This was carried out in the least degenerate level (relating
noncommutative rational surfaces with elliptic anticanonical curve to
symmetric elliptic difference equations) in \cite{generic}, where it was
shown that one could construct noncommutative Hirzebruch surfaces and their
blowups via difference operators on the (assumed smooth) anticanonical
curve.  (More precisely, the category of line bundles on the generic
commutative rational surface with chosen anticanonical curve has a flat
extension to any rational surface with smooth anticanonical curve, and that
category has a faithful representation in difference operators.)  A major
purpose of the present note is to extend this construction to cases in
which the anticanonical curve becomes singular.

One side effect of developing the elliptic theory is that, in addition to
the desired applications to difference equations and their moduli, the
construction also led to new results in noncommutative geometry.  For
instance, it is a standard fact of commutative geometry that blowups in
distinct points commute.  The analogous fact was not known in the
noncommutative setting, but in the elliptic case follows from the results
of \cite{generic}.  Thus another major purpose of the present note is to
establish similar isomorphisms in general; not just commutativity of
blowups, but other standard isomorphisms from the birational geometry of
commutative surfaces (e.g., that a one-point blowup of $\P^2$ is a
Hirzebruch surface, and that $\P^1\times \P^1$ admits two different
rulings).

One difficulty in generalizing the results from the elliptic case is that
it relies on a very strong version of flatness: for any two line bundles on
the surface, there is a natural subspace (generically everything) of the
$\Hom$ space with dimension independent of the surface.  Unfortunately,
already in the commutative case, this construction fails (per
\cite{me:hitchin}) if the surface contains a smooth rational curve of
self-intersection $<-2$.  Although this may not seem like a very
significant restriction, this means the construction fails for any
anticanonical surface with non-integral anticanonical curve of negative
self-intersection.  Although the surviving surfaces are quite interesting,
the excluded cases include most of the generalizations of Painlev\'e
equations in the literature.  (Indeed, any continuous such generalization
relates to a surface with nonreduced anticanonical curve!)

One could attempt to get around this by only establishing flatness for
suitably ample divisor classes.  This leads to other difficulties, however:
there are already quite a few cases that needed to be dealt with in the
elliptic case, and these would only get more complicated in the presence of
$-3$-curves or worse.  Moreover, the arguments there required a number of
statements about global sections of line bundles on the anticanonical
curve, which naturally become much more delicate when the anticanonical
curve becomes singular (or, worse, nonreduced).  This could presumably be
dealt with using the fact that a singular anticanonical curve is built out
of smooth rational curves, but this would again risk combinatorial
explosion, as well as introducing additional difficulties when dealing with
families.

With this in mind, we take a rather different approach in the present work.
The key insight is that in the process of obtaining certain derived
equivalences, \cite{generic} showed that the derived categories of the
surfaces considered there had particularly nice descriptions: starting with
a {\em commutative} projective rational surface, take the semiorthogonal
decomposition arising from the fact that the structure sheaf is
exceptional, and then deform the way in which the two subcategories are
glued.  It follows from this description that any isomorphism of
commutative surfaces extends to a family of derived equivalences between
corresponding noncommutative surfaces.  Thus to show that the
noncommutative surfaces are isomorphic, it suffices to show that these
derived equivalences respect the $t$-structures.  As stated, this is
somewhat difficult to carry out, since the $t$-structure is not easily
described when the derived category is viewed in this form.  However, it
turns out that each of the main constructions we need to consider
(noncommutative planes, noncommutative ruled surfaces, and noncommutative
blowups) comes with an inductive description of both the derived category
and the $t$-structure that is simple enough to enable one to show
exactness.  This also has a number of nice consequences; e.g., one can read
off the Grothendieck group, the Mukai pairing, and the existence of a Serre
functor from the description of the derived category, and it is easily seen
from the description of the $t$-structure that the Serre functor is exact
up to a shift (the analogue of being Gorenstein).  Moreover, having a nice
description of the derived category makes it particularly simple to
establish interesting derived equivalences (extending the derived
autoequivalences of elliptic surfaces), and to deal with fairly general
families.

We should note here that what we mean by a noncommutative scheme in general
is actually an abelian category with marked object, corresponding to the
category of quasicoherent sheaves with choice of structure sheaf.  In
particular, our description via $t$-structures on triangulated categories
actually gives a new (and simpler!) construction of the noncommutative
surfaces we are considering.  Unfortunately, we as yet have been unable to
show that the result is a $t$-structure, or that the triangulated category
is the derived category of its heart, without reference to some more direct
construction of the abelian category.

For planes and blowups, the requisite facts about the derived category have
already been established in the literature, but we will need to establish
them for ruled surfaces below.  Moreover, in order to further the
connection to special functions, we need to establish the connection
between noncommutative ruled surfaces and difference equations.  Another
desideratum is to understand the ``semicommutative'' case; i.e., when the
noncommutative surface can be represented as a finite sheaf of algebras on
a commutative projective surface.  It is somewhat difficult to work with
the original construction of \cite{VandenBerghM:2012} for these purposes,
and thus we introduce a new approach inspired by the construction of
\cite{elldaha} for the elliptic case.  The idea is to construct a Morita
equivalent description of the underlying ``algebra'' as a variation of a
twisted group algebra over the infinite dihedral group; this not only
enables us to compute the center when the surface is semicommutative, but
to show that in that case, the algebra is a maximal order in a central
simple algebra on a related commutative ruled surface.  (We also establish
that this property survives birational transformations \`a la van den Bergh
\cite{VandenBerghM:1998}.)

One complication is that van den Bergh's original construction of
noncommutative $\P^1$-bundles involves two base schemes rather than only
one.  This makes it quite complicated to classify such surfaces in general,
but it turns out that there is a very useful fact: if the two curves are
not isomorphic (more precisely, if the corresponding curve on their product
is not algebraically equivalent to twice the graph of an isomorphism), then
the center is large!  Thus the most complicated components of the moduli
stack of noncommutative ruled surfaces as constructed in
\cite{VandenBerghM:2012} consist entirely of maximal orders on commutative
ruled surfaces, rendering most of the more basic questions trivial.  The
remaining components necessarily contain {\em commutative} ruled surfaces,
and (with one mild exception in characteristic 2) are in fact naturally
fibered over the corresponding stacks of commutative ruled surfaces.  This
leads us to define the result of van den Bergh's construction (over curves)
as a ``quasi-ruled'' surface, reserving the word ``ruled'' to refer to
those surfaces that can be deformed to truly noncommutative surfaces.

Another complication is that, although van den Bergh gave a construction of
blowups of points on noncommutative surfaces (subject to some technical
assumptions that need to be established in the cases of interest), there is
not as yet a satisfying discussion of blowdowns.  In particular, there is
no analogue of Castelnuovo's criterion for when a surface can be blown
down.  However, it turns out that when the surface in question is an
iterated blowup of a quasi-ruled surface (or noncommutative plane), then
the isomorphisms mentioned above can make any ``-1-curve'' (suitably
defined) into the exceptional curve of the last blowup in one of its
(possibly many) descriptions as such an iterated blowup.  One thus finds
that if some iterated blowup of a noncommutative surface is ``rationally
quasi-ruled'', then the surface is itself either rationally quasi-ruled or
a noncommutative plane.  (Note that this means that when considering how
birational transformations affect the semicommutative case, we need only
consider blowups\dots)

In \cite{ChanD/NymanA:2013}, Chan and Nyman proposed a definition of
``proper'' noncommutative surfaces and established that noncommutative
ruled surfaces satisfied their definition.  It should not be too
surprising, therefore, that we can show that their axioms are satisfied by
any rationally quasi-ruled surface.  In fact, some of the axioms hold in
stronger forms.  The most significant strengthening involves Quot schemes:
we show that given any coherent sheaf on a Noetherian family of such
surfaces, the corresponding Quot scheme (classifying quotients with
specified Hilbert polynomial) is projective.  This is somewhat tricky to
establish, as the usual construction of Quot schemes uses the flattening
stratification, which is far more delicate in noncommutative geometry.  In
fact, although we establish below that the flattening stratification is
well-behaved for our surfaces, the proof in general actually requires first
showing that the Quot schemes are projective! Luckily, there are
approaches (\cite{ArtinM/ZhangJJ:2001,ArtinM/SmallLW/ZhangJJ:1999}) to
controlling the flattening stratification over sufficiently nice
(``admissible'') base, and this is good enough (given our understanding of
the moduli of noncommutative surfaces) to deal with Quot schemes of powers
of the structure sheaf, to which the general case can be reduced.

Since part of the strengthening of the Quot scheme result involves the
notion of a Hilbert polynomial, we need an understanding of ample divisors.
Although it would suffice for such purposes to restrict to any reasonable
set of such divisors, we in fact establish below that ample divisors behave
in much the same way as in the commutative case.  In particular, we show
that for a suitable notion of an ``effective'' divisor, the ample divisors
are precisely those that lie in the interior of the corresponding ``nef''
cone.  Here we make significant use of the above dichotomy: for strictly
(rationally) quasi-ruled surfaces, this follows immediately from the
corresponding statement on the center, while for (rationally) ruled
surfaces, we need to argue directly.  Luckily, it turns out that the latter
case is actually {\em simpler} in a number of respects, in that we can give
a relatively explicit combinatorial description of the effective and nef
cones, which is not available in the quasi-ruled case.  In fact, in the
rational case, we can give an essentially combinatorial algorithm for
computing the Betti numbers of $R\Hom(L,L')$ for any pair of ``line
bundles'' (i.e., the unique sheaves that arise by deforming line bundles
from the commutative case).

Since the motivating objective was to understand moduli spaces of
equations by relating them to moduli spaces of sheaves, we still
need to understand the latter.  This could, of course, be asked in
several different ways.  If one only asks for an algebraic space,
then the moduli spaces are well-behaved for any ``simple'' sheaf
(i.e., with no non-scalar endomorphisms), and in fact for simple
objects in the derived category in general.  Moreover, at least for
sheaves, we establish that the moduli space has a natural Poisson
structure with well-behaved (and modular) symplectic leaves.  (More
generally, there is always a natural biderivation with nice weakly
symplectic (i.e., with nondegenerate 2-forms) leaves; what remains
to be established is the Jacobi identity.)  The argument here
involves reducing to the case of ``reflexive'' (relative to a
suitable contravariant derived equivalence) sheaves on a semicommutative
surface, where it follows from a more general fact (of independent
interest) about Poisson structures on moduli spaces of $G$-torsors.

If one wishes the moduli space to be a scheme, the
situation is less satisfying: there is an obvious analogue of the
semistable moduli space, but we can only show that it is projective when
the rank is either 0 or 1.  (The latter is particularly interesting, as
those moduli spaces are deformations of the Hilbert scheme of points on a
commutative surface; we in fact had to control these in order to establish
sufficient boundedness to make the Quot scheme construction work.)
Luckily, the rank 0 case turns out to be precisely the case of greatest
interest in the application to equations, and allows us to establish that
there is indeed a quasi-projective Poisson moduli space classifying
difference/differential equations.  Moreover, we show that the symplectic
leaves are not only naturally described in terms of noncommutative
geometry, but in terms of equations: they control the singularities of the
equation.

Note that although the applications to special functions are via
noncommutative surfaces over $\C$, the theory is developed in much greater
generality: most results apply either over algebraically closed fields of
arbitrary characteristic or over general (usually Noetherian) base schemes,
with non-closed fields being viewed as a special case of the latter.  Note
that several results (in particular, the Poisson structure on the moduli
space of sheaves and the existence of irreducible 1-dimensional sheaves
with specified invariants) are proved below by reduction to the case of
semicommutative surfaces.  In particular, it turns out that even if one
were only interested in the characteristic 0 instances of those results,
one would still need to consider the finite characteristic case, since a
higher genus noncommutative ruled surface is semicommutative precisely when
its field of definition has finite characteristic.

The plan of the paper is as follows.  In section 2, we give the ``double
affine Hecke algebra'' construction of noncommutative quasi-ruled surfaces,
and show that the corresponding ``spherical'' algebra essentially agrees
with van den Bergh's construction.  In section 3, we consider the
semicommutative case in detail, showing that when a certain automorphism
has finite order, then the noncommutative quasi-ruled surface is a maximal
order in a central simple algebra on a commutative ruled surface, as well
as showing that this property survives blowing up.

Section 4 begins the discussion of derived categories, showing that the
well-known semiorthogonal decomposition of commutative ruled surfaces
survives to the noncommutative case; it also recalls corresponding
semiorthogonal decompositions for the other two constructions, and
describes in each case how to obtain the $t$-structure, and how the various
isomorphisms from birational geometry extend as derived equivalences.
(This section also sketches the proofs that these equivalences are exact,
with details in some cases left to section 8.)  Section 5 discusses some
additional derived equivalences, both deformations of autoequivalences of
elliptic surfaces and an analogue of the usual Cohen-Macaulay duality of
commutative smooth projective schemes.  Section 6 considers families of
surfaces, both in a relatively na\"{i}ve sense (families described as
explicit iterated blowups) and in a somewhat more sophisticated version
(those which are \'etale-locally expressible as iterated blowups).

Section 7 then turns to the underlying abelian category, with an aim
towards establishing most of the Chan-Nyman axioms, most notably that there
is a well-behaved notion of the ``dimension'' of a sheaf; in the process,
it also studies the Mukai pairing on the Grothendieck group and uses it to
establish a reasonable notion of (first) Chern class.  Section 8 uses this
control of Chern classes to establish that the derived equivalences
constructed in section 4 preserve the $t$-structure (by giving an alternate
description of the $t$-structure which is clearly preserved and showing
that it agrees with the original description).  This section in addition
shows that an exceptional sheaf with no cohomology can be blown down iff
its Chern class looks like the class of a $-1$-curve, and if so, the
blowdown is again a noncommutative surface of the type we are considering.
Section 9 defines effective divisors (Chern classes of rank 0 sheaves) and
establishes sufficient conditions for one ``line bundle'' to be acyclically
globally generated by another; this is then used to show that any divisor
in the interior of the nef cone satisfies all the properties one might wish
of an ample divisor.

Section 10 turns to moduli problems, showing that coherent sheaves on
Noetherian families are globally bounded, and using this to show that Quot
schemes are projective.  Part of the argument requires controlling the
invariants of globally generated sheaves, which in turn requires
controlling subsheaves of the structure sheaf; as a result, the proof for
Quot schemes includes most of a proof that moduli spaces of torsion-free
sheaves of rank 1 are projective, also included in this section.

Section 11 considers more general moduli problems, not just of sheaves, but
of objects in the derived category, showing that the moduli space of
``simple'' objects is an algebraic space with a natural bivector, that the
symplectic leaves are not only algebraic but (modulo one case which may
be empty!) smooth, and that the Jacobi identity holds on the open subspace
classifying sheaves.  This section also discusses the relatively minimal
changes needed to the arguments of \cite{generic} to extend the results of
that paper on semistable moduli spaces to more general cases.

Section 12 discusses applications of the above results to moduli spaces of
difference and differential equations.  In particular, the discussion from
\cite{me:hitchin} of how the structure of the anticanonical curve relates
to singularities is shown to hold in the noncommutative case, so that in
combination with the results on moduli spaces of sheaves, one finds that
any moduli space of difference or differential equations with specified
singularities is a smooth quasiprojective scheme.  Moreover, the
noncommutative analogue of twisting by a line bundle extends to
isomorphisms between such moduli spaces that act as isomonodromy (or gauge)
transformations on the equations.  In particular, one recovers all of the
discrete Painlev\'e equations in Sakai's hierarchy \cite{SakaiH:2001} in
this way.  (For continuous isomonodromy deformations, the situation is not
entirely satisfactory; it is shown that one has the expected number of such
deformations, and that they survive the various isomorphisms coming from
birational geometry, but a geometric interpretation is still lacking.)  In
addition, it is shown that one can use one of the derived equivalences of
section 5 to construct many other such interpretations of discrete
Painlev\'e equations: for any point $d/r\in \P^1(\Q)$ and any discrete
Painlev\'e equation, there is an isomonodromy interpretation in terms of
(discrete) connections on a vector bundle of rank $2r$ and degree $d$.
This section also discusses how certain constructions on equations (e.g.,
viewing a $q$-difference equation as a $q^r$-difference equation) suggest
that there should be related morphisms of noncommutative surfaces, as well
as a theory of difference equations for structure groups other than
$\GL_n$.

Finally, in an appendix, we consider, both due to direct interest and as an
example of the sort of combinatorial explosion that justifies the derived
category approach, the explicit isomorphisms between algebras of
difference/differential operators that arise from the fact that $\P^1\times
\P^1$ has two distinct rulings.  In the most degenerate case, this is the
quite familiar (Fourier!) automorphism of the Weyl algebra that swaps
multiplication and differentiation, and the more general versions are
qualitatively similar (interchanging certain multiplication and
difference/differential operators).  We thus in this way obtain 16
different generalizations of the Fourier transform, which in addition to
several known cases (Fourier, Mellin and its inverse, middle
convolution) include a number of less familiar instances (e.g., a
transformation between nonsymmetric $q$-difference equations and symmetric
$q$-difference equations).  

{\bf Acknowledgements}.  The author would like to thank D. Chan for some
helpful discussion about orders on surfaces, B. Pym for helpful discussions
about derived stacks and Poisson structures, and A. Nakamura for pointing
out some interesting examples of equations with non-$\GL$ structure group.
This work was partially supported by a grant from the National Science
Foundation, DMS-1500806.

\section{Noncommutative quasi-ruled surfaces}

Among the noncommutative projective surfaces we consider, the analogues of
ruled surfaces are particularly fundamental; every surface we consider will
be obtainable from such a surface via a sequence of blowups and blowdowns.
The basic construction is due to Van den Bergh \cite{VandenBerghM:2012},
but it will be helpful to give an alternate construction.  In particular,
we expect from \cite{me:hitchin,generic} that there should be a
relation between sheaves on such surfaces and difference or differential
equations, and thus would like to have a construction in which that
relation is visible.

Since the fundamental object associated to a noncommutative projective
surface is its category of coherent sheaves, there is a great deal of
freedom in how we can obtain such a surface.  For instance, if $X$ is a
commutative projective scheme with ample bundle $\sO_X(1)$, then for any
vector bundle $V$ on $X$, the (noncommutative) graded algebra $\bigoplus_d
\Hom(V,V(d))$ has the same category of graded modules as $\bigoplus_d
\Gamma(\sO_X(d))$, and the same subcategory of ``torsion'' modules, and
thus gives rise to the same category of coherent sheaves.  Our first
construction of noncommutative surfaces will be similarly Morita equivalent
to Van den Bergh's construction.

With this in mind, we begin by considering a special case of this
construction.  Let $Y/S$ be a scheme and $\pi:X\to Y$ a finite flat
morphism of degree 2 such that $\pi_*\sO_X$ is locally free.  Then for any
invertible sheaf ${\cal L}$ on $X$, $\pi_*{\cal L}$ is a vector bundle of
rank 2 on $Y$, and thus we obtain an $\sO_Y$-algebra
\[
{\cal A}_{{\cal L},\pi}:=\sEnd_Y(\pi_*{\cal L}).
\]
This is Morita equivalent to $\sO_Y$, with an explicit equivalence
$\sO_Y\text{-mod}\to {\cal A}_{{\cal L},\pi}\text{-mod}$ given by
\[
M\mapsto M\otimes_{\sO_Y} \pi_*{\cal L},
\]
and the inverse given by
\[
M\mapsto \Hom_{{\cal A}_{{\cal L},\pi}}(\pi_*{\cal L},M).
\]
Moreover, since this algebra contains $\pi_*\sO_X$, we may interpret it as
a $\pi_*\sO_X$-bimodule, which we abbreviate to ``$\sO_X$-bimodule'' (as
justified below).

\begin{lem}
  There are canonical isomorphisms
  ${\cal A}_{{\cal L}\otimes_{\sO_X} \pi^*{\cal L}',\pi}\cong {\cal
    A}_{{\cal L},\pi}$
  for any line bundle ${\cal L}'$ on $Y$, as well as a canonical involution
  ${\cal A}_{{\cal L},\pi}^{\text{op}}\cong {\cal A}_{{\cal L},\pi}$.
\end{lem}

\begin{proof}
  For the first claim, we have
  \[
  \sEnd_Y(\pi_*({\cal L}\otimes_{\sO_X} \pi^*{\cal L}'))
  \cong
  \sEnd_Y(\pi_*({\cal L})\otimes_{\sO_Y}{\cal L}')
  \cong
  \sEnd_Y(\pi_*({\cal L})),
  \]
  while for the second claim, we observe that for any rank 2 vector bundle
  $V$ on $Y$, we have isomorphisms
  \[
  \sEnd_Y(V)^{\text{op}}
  \cong
  \sEnd_Y(\sHom(V,\sO_Y))
  \cong
  \sEnd_Y(V\otimes \det(V)^{-1})
  \cong
  \sEnd_Y(V),
  \]
  which in the case of the trivial bundle is the (adjoint) involution
  \[
  \begin{pmatrix} a&b\\c&d\end{pmatrix}
  \mapsto
  \begin{pmatrix} d&-b\\-c&a\end{pmatrix},
  \]
  so is an involution in general.  
\end{proof}

\begin{rem}
  It is worth noting that although the {\em algebra} ${\cal A}_{{\cal
      L},\pi}$ only depends on ${\cal L}$ modulo $\Pic(Y)$, the explicit
  Morita equivalence depends on ${\cal L}$.
\end{rem}

\begin{prop}
  There is a natural isomorphism $\pi^!\cong
  \pi^*(\_\otimes_{\sO_Y}\det(\pi_*\sO_X)^{-1})$
\end{prop}

\begin{proof}
  This reduces to showing that there is a natural isomorphism
  \[
  \sHom_Y(\pi_*\sO_X,N)\cong \pi_*\sHom_X(\sO_X,\pi^*(N\otimes_{\sO_Y}
  \det(\pi_*\sO_X)^{-1})),
  \]
  which in turn reduces to the case $N=\sO_Y$, where it becomes
  \[
  \sHom_Y(\pi_*\sO_X,\sO_Y)\cong \pi_*\sO_X\otimes
  \det(\pi_*\sO_X)^{-1}.
  \]
\end{proof}

\begin{rem}
  When $Y$ is smooth and $X$ is integral, this is essentially
  \cite[Prop.~0.1.3]{CossecFR/DolgachevIV:1989}.
\end{rem}

There is an alternate form for the inverse Morita equivalence.

\begin{prop}
  The functors $M\mapsto \pi_*{\cal L}\otimes_{\sO_Y} M$ and $M\mapsto
  \pi_*({\cal L}^{-1}\otimes \pi^!\sO_Y)\otimes_{{\cal A}_{{\cal L},\pi}}
  M$ are inverse equivalences between $\sO_Y\text{-mod}$ and ${\cal
    A}_{{\cal L},\pi}\text{-mod}$.
\end{prop}
 
\begin{proof}
  This reduces to showing that
  \[
  \pi_*{\cal L}
  \otimes_{\sO_Y}
  \pi_*({\cal L}^{-1}\otimes \pi^!\sO_Y)
  \cong
  {\cal A}_{{\cal L},\pi},
  \]
  which follows from the (bimodule!) isomorphism
  \[
  \pi_*({\cal L}^{-1}\otimes \pi^!\sO_Y) \cong \sHom_Y(\pi_*{\cal L},\sO_Y).
  \]
\end{proof}

The adjoint involution can be written the form $x\mapsto \Tr(x)-x$, and
thus in particular restricts to an involution $s:\pi_*\sO_X\to \pi_*\sO_X$
of the same form.  (This involution is nontrivial except when $\ch(k)=2$
and either $X$ is nonreduced or $\pi$ is inseparable.)  Given a coherent
sheaf $M$ on $X$, let $M s$ denote the $\sO_X$-bimodule with multiplication
$r_1 m r_2=r_1 s(r_2) m$; we will also conflate $\sO_X$-modules on $X$ with
the corresponding bimodules.  (Recall that for the moment, an
$\sO_X$-bimodule means a $\pi_*\sO_X$-bimodule on $Y$.)  And as we have
already mentioned, ${\cal A}_{{\cal L},\pi}$ has a natural $\sO_X$-bimodule
structure.

\begin{prop}\label{prop:aha_filtration}
  There is a natural short exact sequence of $\sO_X$-bimodules
  \[
  0\to \sO_X\to {\cal A}_{{\cal L},\pi}\to {\cal L}\otimes_{\sO_X} (\pi^!\sO_Y
  s)\otimes_{\sO_X} {\cal L}^{-1}\to 0
  \]
\end{prop}

\begin{proof}
  Since ${\cal A}_{{\cal L},\pi}\cong {\cal L}\otimes_{\sO_X} {\cal
    A}_{\sO_X,\pi}\otimes_{\sO_X} {\cal L}^{-1}$, it suffices to prove this
  in the case ${\cal L}=\sO_X$.  If we forget about the right module
  structure, then not only do we have such an exact sequence, but it
  splits.  Indeed, the natural inclusion $\sO_X\to {\cal A}_{\sO_X,\pi}$ is
  naturally split by the evaluate-at-1 map, so it remains to identify the
  other direct summand (the endomorphisms that annihilate 1) and show that
  the right action is triangular.

  There is a natural short exact sequence
  \[
  0\to \sO_Y\to \pi_*\sO_X\to \det(\pi_*\sO_X)\to 0;
  \]
  that the quotient is invertible follows as in
  \cite[\S0.1]{CossecFR/DolgachevIV:1989}, and can then be identified by
  comparing determinants.  An endomorphism that annihilates 1 annihilates
  $\sO_Y$ and is thus determined by its action on $\det(\pi_*\sO_X)$, so
  that we have the left $\sO_X$-module decomposition
  \[
  {\cal A}_{\sO_X,\pi}\cong \sO_X\oplus \sHom_Y(\det(\pi_*\sO_X),\pi_*\sO_X)
  \cong \sO_X\oplus \sHom_Y(\pi_*\sO_X,\sO_Y)
  \cong \sO_X\oplus \pi^!\sO_Y.
  \]
  Since the adjoint involution acts as $x\mapsto \Tr(x)-x$, it is clearly
  triangular with respect to this decomposition, and thus takes the
  diagonal left action of $\sO_X$ to a triangular right action.  Moreover,
  the induced action on the quotient is itself induced by the adjoint, or
  in other words by $s$.
\end{proof}

Note that the other (left) direct summand of ${\cal A}_{\sO_X,\pi}$ consists
(locally) of $s$-derivations, i.e. $\sO_Y$-linear endomorphisms that
satisfy
\[
\phi(r_1r_2) = s(r_2)\phi(r_1)+r_1\phi(r_2).
\]
If $\pi_*\sO_X$ is actually free over $\sO_Y$, with generator $\xi$, then
the morphism $\nu:a+b\xi\mapsto b$ is such an $s$-derivation, and
any other $s$-derivation has the form $f\mapsto g \nu(f)$.
If $\xi-s(\xi)$ is not a zero divisor, then we can write this in the form
\[
\nu(f)=\frac{f-s(f)}{\xi-s(\xi)}.
\]
The $\sO_Y$-module it generates is self-adjoint (the adjoint has eigenvalue
$-1$) and may be characterized as the space of {\em nilpotent}
endomorphisms that annihilate $\sO_Y$, which makes sense even when $\sO_X$
is only locally free over $\sO_Y$.  This module is locally free of rank 1,
and is easily seen to be isomorphic to $\pi^!\sO_Y\cong
\pi^*\det(\pi_*\sO_X)^{-1}$.

\medskip

Although it is most natural to view ${\cal A}_{{\cal L},\pi}$ as a sheaf of
$\sO_Y$-modules, we will need a slightly different perspective coming from
the bimodule structure.  Indeed, the $\pi_*\sO_X$-bimodule structure on
${\cal A}_{{\cal L},\pi}$ induces a $\pi_*\sO_X\otimes_{\sO_Y}
\pi_*\sO_X$-module structure, which in turn allows us to interpret it as a
coherent sheaf on $X\times_Y X=\Spec(\pi_*\sO_X\otimes_{\sO_Y}\pi_*\sO_X)$.
The embedding $X\times_Y X\to X\times X$ then makes it a coherent sheaf on
$X\times X$ supported on the union of the diagonal and the graph of $s$, so
finite over either projection.  In other words, ${\cal A}_{{\cal L},\pi}$
is a {\em sheaf bimodule} over $X\times X$ (in the sense of
\cite{VandenBerghM:1996,ArtinM/VandenBerghM:1990}).  Moreover, the algebra
structure on ${\cal A}_{{\cal L},\pi}$ is compatible with this bimodule
structure, in that the multiplication induces a morphism
\[
{\cal A}_{{\cal L},\pi}\otimes_X {\cal A}_{{\cal L},\pi}
:=
\pi_{2*}(\pi_3^*{\cal A}_{{\cal L},\pi}\otimes_{\sO_{X\times X\times X}}
\pi_1^*{\cal A}_{{\cal L},\pi})
\to
{\cal A}_{{\cal L},\pi},
\]
and thus ${\cal A}_{{\cal L},\pi}$ is a {\em sheaf algebra}.  (Indeed, this
is immediate from the corresponding fact for the fiber product.)

The Morita equivalences are also naturally expressed in terms of sheaf
bimodules; the sheaf $(1\times \pi)_*{\cal L}$ is a sheaf bimodule on
$X\times Y$, and the action of ${\cal A}_{{\cal L},\pi}$ induces a morphism
\[
{\cal A}_{{\cal L},\pi}\otimes_X (1\times \pi)_*{\cal L}
=
\pi_{2*}(\pi_3^*{\cal A}_{{\cal L},\pi}\otimes_{\sO_{X\times X\times Y}}
\pi_1^*(1\times\pi)_*{\cal L})
\to
(1\times \pi)_*{\cal L}
\]
making $(1\times\pi)_*{\cal L}$ a $({\cal A}_{{\cal L},\pi},\sO_Y)$-sheaf
bimodule.  Similarly, there is a $(\sO_Y,{\cal A}_{{\cal L},\pi})$-sheaf
bimodule given by
\[
(\pi\times 1)_*({\cal L}\otimes \pi^*\det\pi_*{\cal L}^{-1}),
\]
and these induce Morita equivalences.

One caution is that the natural adjoint involution does not in general
respect the sheaf bimodule structure, since it acts as $s$ on $\sO_X$.
Luckily, there is a variant that works more generally, at the cost of
changing ${\cal L}$.

\begin{prop}
  There is a natural sheaf algebra isomorphism ${\cal A}_{{\cal
      L},\pi}^{\text{op}}\cong {\cal A}_{{\cal L}^{-1},\pi}$.
\end{prop}

\begin{proof}
  Consider the composition
  \[
  \pi_*{\cal L}\otimes_{\sO_Y} \pi_*({\cal L}^{-1})
  \to
  \pi_*\sO_X
  \to
  \det(\pi_*\sO_X),
  \]
  where the first map is multiplication and the second is the quotient by
  the subbundle $\sO_Y$.  The corresponding pairing respects multiplication
  by $\sO_X$, and thus will induce a sheaf algebra isomorphism as required
  so long as it is perfect, i.e., the induced map
  \[
  \pi_*{\cal L}\to \sHom_{\sO_Y}(\pi_*({\cal L}^{-1}),\det(\pi_*\sO_X))
  \]
  is an isomorphism.  It suffices to check this locally, so that we may
  assume ${\cal L}=\sO_X$.  In that case, the (now symmetric) pairing
  vanishes on $\sO_Y\otimes \sO_Y$, and thus induces a pairing
  $\sO_Y\otimes \det(\pi_*\sO_X)\to \det(\pi_*\sO_X)$; since this is an
  isomorphism, the original pairing is perfect.
\end{proof}

\begin{rem}
  If ${\cal L}=\sO_X$, then $r$ and $s(r)$ always pair to $0$, and thus
  this involution is the composition of the adjoint with conjugation by
  $s$.  If also $\pi_*\sO_X$ is free, with basis $(1,\xi)$, then the
  pairing is given by $\begin{pmatrix} 0 & 1\\1 & \Tr(\xi)\end{pmatrix}$.
\end{rem}

Since this isomorphism is also an adjoint (albeit with respect to a
different pairing), and much more important for our purposes, we will refer
to the original involution on ${\cal A}_{{\cal L},\pi}$ as the
``intrinsic'' adjoint, reserving the unadorned term for this involution.

\medskip

Now, suppose we are given two finite flat degree 2 morphisms $\pi_i:X\to
Y_i$, $i\in \{0,1\}$, as well as two invertible sheaves ${\cal L}_i$ on
$X$.  This gives rise to a pair ${\cal A}_{{\cal L}_0,\pi_0}$, ${\cal
  A}_{{\cal L}_1,\pi_1}$ of sheaf algebras on $X$, and we define the
(quasicoherent) sheaf algebra ${\cal H}_{{\cal L}_0,{\cal
    L}_1,\pi_0,\pi_1}$ on $X$ to be their pushforward over $\sO_X$.  That
is, ${\cal H}_{{\cal L}_0,{\cal L}_1,\pi_0,\pi_1}$ is generated by ${\cal
  A}_{{\cal L}_0,\pi_0}$ and ${\cal A}_{{\cal L}_1,\pi_1}$ subject only to
the relation that the common subalgebras $\sO_X$ should be identified.
Since the respective adjoints act trivially on $\sO_X$, they combine to
form an automorphism
\[
  {\cal H}_{{\cal L}_0,{\cal L}_1,\pi_0,\pi_1}^{\text{op}}
  \cong
  {\cal H}_{{\cal L}_0^{-1},{\cal L}_1^{-1},\pi_0,\pi_1}.
\]  
Note that when $Y_0\cong Y_1\cong \P^1$ and $X$ is a smooth genus 1 curve,
the resulting sheaf algebra is an instance of the ``type $C_1$ elliptic
double affine Hecke algebra'' construction of \cite{elldaha} (with the
Morita-equivalent algebra ${\cal S}$ being the spherical algebra in those
terms).  This suggests that there should be a multivariate version of this
more general construction, possibly for general root systems, but at the
very least for type $C$.

To work with this sheaf algebra, it will be convenient to work locally.
Localization is somewhat trickier with sheaf algebras than with sheaves of
algebras, since the restriction to an open subset is usually not a sheaf
algebra.  In this case, however, there is not too much difficulty.  If an
open subset is invariant under the action of the two involutions $s_0$,
$s_1$, then the restriction to that open subset will again be a sheaf
algebra.  This is still too restrictive (if $s_0s_1$ has infinite order, we
cannot expect to have a covering by affine opens of that form), but a
slight generalization will suffice in most cases.  Define a {\em
  localization} of $X$ to be a nonempty intersection of a nonempty (and
possibly infinite) collection of affine opens.  This is a fiber product of
affine morphisms, so inherits a scheme structure, and is affine over each
open set in the original collection, so affine, with coordinate ring given
by the limit of $\Gamma(U;\sO_X)$ over all open subsets $U$ containing the
intersection.  The significance of this notion is that although there are
not in general any nontrivial $\langle s_0,s_1\rangle$-invariant affine
opens in $X$, there are typically many invariant localizations.  Moreover,
${\cal H}_{{\cal L}_0,{\cal L}_1,\pi_0,\pi_1}$ induces a sheaf algebra on any
invariant localization, which will be an honest algebra containing the
structure sheaf of the localization.

If we not only have nontrivial invariant localizations but have a locally
finite covering by such localizations with consistent algebras on each
localization, then the theory of fpqc descent gives rise to an induced
sheaf algebra.  We assume that not only does such a covering exist, but
that there is such a covering such that each bundle ${\cal L}_0$, ${\cal
  L}_1$, $\pi_{0*}\sO_X$, $\pi_{1*}\sO_X$ becomes trivial.  Note that if
$Y_0$ and $Y_1$ are smooth curves over a field (which is the case we care
about for the present application), the existence of such a configuration
is straightforward.  Indeed, any actual configuration of curves and bundles
is the pullback of a configuration over a field which is not algebraically
closed, and any bundle can be trivialized by removing finitely many points
which are not defined over that field (and thus such that no point of the
orbit is in that field).  It follows that any point of $X$ is contained in
an invariant localization that trivializes the bundles.  Moreover, each
such localization only omits finitely many orbits, and thus there is a
finite subcovering.

Thus let $R$ be a ring which is free of rank 2 over its subalgebras $S_0$
and $S_1$, and consider the algebra $H$ generated over $R$ by
$\End_{S_0}(R)$ and $\End_{S_1}(R)$.  We can give a more explicit
presentation of this algebra as follows.  The subspace of $\End_{S_0}(R)$
consisting of nilpotent elements annihilating $S_0$ is a free $S_0$-module,
so that we may choose a generator $\nu_0$ of this module, and similarly for
$\nu_1$.  Then $H$ has the presentation
\[
H = R\langle N_0,N_1\rangle/(N_0^2,N_0 r-s_0(r) N_0-\nu_0(r),N_1^2,N_1
r-s_1(r) N_1-\nu_1(r)).
\]
Note that since ${\cal L}_0$ and ${\cal L}_1$ are trivial, the adjoint
becomes an involution of this algebra, acting trivially on $R$, $N_0$, and
$N_1$; one has
\[
N_0 r-s_0(r) N_0
=
N_0 r+r N_0-\Tr_0(r)N_0
=
N_0 r+r N_0-N_0\Tr_0(r)
=
r N_0 - N_0 s_0(r),
\]
so that this is indeed a contravariant automorphism, and is clearly of
order 2.

Let $D_\infty$ denote the infinite dihedral group, with generating
involutions denoted by $s_0$, $s_1$, acting in the obvious (not necessarily
faithful) way on on $R$.  Each element of $D_\infty$ is represented by a
unique reduced word $s_{i_1}s_{i_2}\cdots$ in which $i_j\ne i_{j+1}$ for
all $j$.  (Thus there are two such words of each positive length,
corresponding to the two possible values of $i_1$.)  Given any such element
$w$, let $N_w$ denote the corresponding product $N_{i_1}N_{i_2}\cdots$.
There is a natural partial ordering on $D_\infty$ given by $w<w'$ iff
$\ell(w)<\ell(w')$, where $\ell(w)$ denotes the length of the reduced
word.\footnote{This is the Bruhat ordering on the infinite dihedral group,
  viewed as the free Coxeter group on two generators.}

\begin{prop}
  As a left $R$-module, one has
  \[
  H\cong \bigoplus_{w\in D_\infty} R N_w,
  \]
  and thus the partial ordering on $D_\infty$ induces a left $R$-module
  filtration of $H$.  This is in fact a bimodule filtration, and the
  subquotient corresponding to $w$ is the bimodule $R w$.
\end{prop}

\begin{proof}
  Define a left action of $H$ on the free module $\bigoplus_{w\in
    D_\infty}R e_w$ by extending the $R$ action by
  \[
  N_i\cdot (r e_w) =
  \begin{cases}
    s_i(r) e_{s_iw} + \nu_0(r)e_w & \ell(s_i w)>\ell(w),\\
    \nu_0(r)e_w & \ell(s_i w)<\ell(w).
  \end{cases}
  \]
  This is easily seen to satisfy the relations, so indeed gives a module
  structure, and since $r N_w\cdot e_1 = r e_w$, may be identified with the
  regular representation, establishing the desired isomorphism.  It remains
  to show that right multiplication by $R$ respects the filtration and acts
  correctly on the associated graded: i.e., that for $r\in R$, $N_w r-w(r)
  N_w\in \bigoplus_{w'<w} R N_{w'}$.  Moving $r$ to the left uses the
  relations $N_0 r=s_0(r) N_0+\nu_0(r)$ and $N_1 r = s_1(r)N_1+\nu_1(r)$,
  from which the result follows by induction in the length.
\end{proof}

\begin{rem}
  The adjoint immediately implies that $H$ is also free as a right
  $R$-module.  Moreover, the adjoint respects the filtration by $D_\infty$,
  modulo the (order-preserving) map $w\mapsto w^{-1}$.
\end{rem}

For the global version, we need to understand how the gluing interacts with
the filtration; in other words, we need to show that the corresponding
automorphisms of $H$ are triangular and understand their diagonal
coefficients.  The automorphisms are given by $N_0\mapsto u_0^{-1} N_0
u_0$, $N_1\mapsto u_1^{-1} N_1 u_1$ for $u_0,u_1\in R^*$, which clearly
respect the relations.  (And, of course, the adjoint involution inverts
both $u_0$ and $u_1$ as expected.)  Since
\[
u_0^{-1} N_0 u_0 = s_0(u_0)u_0^{-1} N_0 + u_0^{-1}\nu_0(u_0),
\]
we conclude that the action of this automorphism on $H$ is indeed
triangular with respect to the filtration, and on the subquotient
corresponding to $w=s_0s_1s_0\cdots$ acts as left multiplication by
\[
s_0(u_0)u_0^{-1}\times s_0(s_1(u_1)) s_0(u_1)^{-1}\times s_0(s_1(s_0(u_0)))
s_0(s_1(u_0))^{-1}\times \cdots
\]
(Note that for gluing purposes, we also need to take into account the fact
that $N_0$ is only determined up to multiplication by $S_0^*$; this gives
rise to factors $\pi^!\sO_{Y_0}$ in the resulting line bundles, and
similarly for $N_1$.)

With this in mind, let ${\cal N}_w$ denote the line bundle defined
inductively by ${\cal N}_1 = \sO_X$ and
\[
  {\cal N}_{s_i w} = s_i^*{\cal L}_i^{-1}\otimes {\cal L}_i\otimes
  \pi^!_i\sO_{Y_i}\otimes s_i^* {\cal N}_w
\]
whenever $\ell(s_iw)>\ell(w)$; note that more generally one has
\[
  {\cal N}_{ww'}\cong {\cal N}_w\otimes (w^{-1})^*{\cal N}_{w'}.
\]  
Similarly, let ${\cal H}_w$ be the sheaf subbimodule defined inductively by
${\cal H}_1 = \sO_X$ and
\[
{\cal H}_{s_i w} = {\cal A}_{{\cal L}_i,\pi_i}{\cal H}_w
\]
for $\ell(s_i w)>\ell(w)$.  In the affine case, this is precisely the
submodule corresponding to the interval under $w$ under the Bruhat
filtration.

\begin{thm}
  For any $w\in D_\infty$, ${\cal H}_w$ is locally free as a left
  $\sO_X$-module, of rank $|\{w':w'\le w\}|$, and in the corresponding
  bimodule filtration of ${\cal H}$, the subquotient corresponding to $w$
  is ${\cal N}_w w$.  Moreover, if $\ell(ww')=\ell(w)+\ell(w')$, then
  ${\cal H}_{ww'}={\cal H}_w{\cal H}_{w'}$.
\end{thm}

\begin{proof}
  Everything except the precise identification of the line bundle in the
  subquotient follows immediately from the corresponding statement in the
  affine case, while the identification of the line bundle follows by an
  easy induction from Proposition \ref{prop:aha_filtration}.
\end{proof}

Although this filtration is particularly convenient for calculations, we
will need some slightly coarser filtrations for the final construction.
Define subbimodules $\overline{\cal H}_{ij}$ for $i,j\in \Z$ with $j\ge i$ as
follows:
\[
  \overline{\cal H}_{i,i+2l} = {\cal H}_{s_{i}(s_{i+1}s_{i})^l},\qquad
  \overline{\cal H}_{i,i+2l+1} = {\cal H}_{(s_{i+1}s_{i})^{l+1}}
\]
where by convention $s_i:=s_{i\bmod 2}$.
(We add 1 to the length here to reflect the fact that ${\cal H}_{s_0}$ and
${\cal H}_{s_1}$ are subalgebras.)  We similarly denote the affine version
by $\overline{H}_{ij}$, and note that the adjoint identifies $\overline{H}_{ij}$ with
$\overline{H}_{-j,-i}$.

\begin{prop}
  For $i\le j\le k$ one has $\overline{\cal H}_{jk}\overline{\cal H}_{ij}\subset
  \overline{\cal H}_{ik}$.
\end{prop}

In other words, the sheaf bimodules $\overline{\cal H}_{j,k}$ fit together
to form a (positively graded) sheaf $\Z$-algebra $\overline{\cal H}$.
(I.e., an enhanced category with objects $\Z$ and Hom spaces given by sheaf
bimodules which are 0 unless the degree is nonnegative.)  This sheaf
$\Z$-algebra is manifestly invariant under shifting the degrees by 2, and
thus one could also consider the corresponding graded sheaf algebra, the
Rees algebra of ${\cal H}$ with respect to the filtration $\{\overline{\cal
  H}_{0,2l}:l\in \N\}$, without changing the category of coherent sheaves.

In the $\Z$-algebra form, each object has an endomorphism ring of the form
${\cal A}_{{\cal L},\pi}$, and thus there is a Morita-equivalent sheaf
$\Z$-algebra in which $\overline{\cal S}_{ij}$ is a sheaf bimodule on
$Y_i\times Y_j$.  (Here and below, we extend the indices on $Y_i$, $N_i$,
etc.~to be periodic of period 2.)  Locally, this has the following
description.  For $i\in \Z$, consider the (cyclic) $H$-module $H/\langle
N_i\rangle$, with generator $e_i$.  Then $\overline S_{ij}$ is the subspace
of $\Hom(H e_j,H e_i)^\text{op}$ such that the image of $e_j$ is in the
image of $\overline H_{ij}$; equivalently (since it is determined by the
image of $e_j$), it is the subspace of $\overline H_{ij} e_i$ which is
annihilated by $N_j$, and the composition $\overline S_{ij}\times \overline
S_{jk}\to \overline S_{ik}$ may be computed by taking $x * y$ to be the
image of $y$ after replacing $e_j$ by $x$.  For gluing purposes, in
addition to $N_i\mapsto u_i^{-1} N_i u_i$, we must take $e_i\mapsto
u_i^{-1} e_i$.  (Note that in the description of $\overline S_{ij}$ as a
subspace of $\overline H_{ij} e_i$, we must left multiply by $u_j$ after
applying the automorphism to $\overline H_{ij}e_i$.)

The following slightly modified local description is useful.

\begin{lem}
  We have $\overline S_{ij}=N_j\overline{H}_{ij}e_i$.
\end{lem}

\begin{proof}
  Since $\overline{H}_{ij}e_i$ is a projective $\End_{S_j}(R)$-module, this
  reduces to the fact that $\ker(N_j)=\im(N_j)$ inside $R$ (the unique
  indecomposable projective module).
\end{proof}

The relations in $\overline{H}_{ij}$ make it straightforward to give a direct
sum decomposition of $\overline{S}_{ij}$.  Indeed, $\overline{H}_{ij}e_i$ is clearly
the span of
\[
R N_j N_{j-1}\cdots N_{i+1} e_i,
R N_{j-1}\cdots N_{i+1} e_i,
\cdots
\]
since every other element of the standard basis annihilates $e_i$.
Since $N_j \overline{H}_{jj}=N_j R$, we can write any element of
$N_j \overline{H}_{ij} e_i=N_j\overline{H}_{i(j-1)} e_i$ as a sum
\[
N_j s_j(c_j) N_{j-1}\cdots N_{i+1} e_i
+N_{j-2} s_j(c_{j-2}) N_{j-1}\cdots N_{i+1} e_i
+\cdots
\]
with $c_j\in R$, except that when $j-i$ is even, the last term has the form
$c_i e_i$ with $c_i\in S_i$.  This, of course, glues to give an analogue
for $\overline{\cal S}_{ij}$ of the Bruhat filtration of $\overline{\cal H}_{ij}$.

\begin{lem}\label{lem:bruhat_for_S}
  One has $\overline{\cal S}_{ii}\cong \sO_{Y_i}$, while for $j>i$ one has a
  short exact sequence
  \[
  0\to \overline{\cal S}_{i(j-2)}\to \overline{\cal S}_{ij} \to (\pi_i\times
  \pi_j)_*( {\cal L}_j^{-1}\otimes{\cal N}_{s_j\cdot s_{i+1}}s_j\cdots
  s_{i+1}\otimes {\cal L}_i ) \to 0
  \]
  of $(\sO_{Y_i},\sO_{Y_j})$-bimodules.
\end{lem}

\begin{rem}
  Note that when taking the direct image under $\pi_i\times \pi_j$ that one
  must take into account the twisting by $s_j\cdots s_{i+1}$.  In addition,
  one can move the twist by $s_j$ to the left and absorb it into $\pi_j$
  to get an alternate description as the bimodule induced by
  \[
  {\cal L}_j^{-1}\otimes \pi^!_j\sO_{Y_j}\otimes {\cal N}_{s_{j-1}\cdots
    s_{i+1}}s_{j-1}\cdots s_{i+1}\otimes {\cal L}_i.
  \]
\end{rem}

\begin{lem}
  The sheaf $\Z$-algebra $\overline{\cal S}$ is a quadratic algebra; that is, it
  is generated by the $\Hom$ bimodules of degree 1 and the relations are
  generated by relations in degree 2.
\end{lem}

\begin{proof}
  It suffices to show this locally, i.e., for the $\Z$-algebra
  $\overline{S}$.  For generation in degree 1, we need to show that
  \[
  N_{j+1} \overline{H}_{j(j+1)} N_j \overline{H}_{0j} e_0
  =
  N_{j+1} \overline{H}_{0(j+1)} e_0,
  \]
  but this is an immediate consequence of the fact that $\overline{H}_{jj} N_j
  \overline{H}_{jj} = \overline{H}_{jj}$.

  It thus remains only to show that the only relations are in degree 2.
  For $i\in \{0,1\}$, let $\xi_i\in R$ be such that $\nu_i(\xi_i)=1$.  Then
  as a (bimodule) $\Z$-algebra, $\overline S$ is generated by elements $N_1
  e_0, N_1 \xi_1 e_0\in \overline S_{2i(2i+1)}$ and $N_0 e_1$, $N_0 \xi_0
  e_1\in \overline S_{(2i-1)2i}$ subject to the bimodule relations (i.e.,
  that right-multiplying a generator by an element of the relevant $S_i$
  equals the appropriate left-$S_{i+1}$-linear combination of generators)
  and the quadratic relations
  \begin{align}
  (N_0 e_1)(N_1 \xi_1 e_0)-(N_0 s_1(\xi_1) e_1)(N_1 e_0) &= 0\notag\\
  (N_1 e_0)(N_0 \xi_0 e_1)-(N_1 s_0(\xi_0) e_0)(N_0 e_1) &= 0.\notag
  \end{align}
  (This follows from $N_1 \xi_1-s_1(\xi_1)N_1=\nu_1(\xi_1)=1$ and $N_0
  e_0=0$.)  Using these quadratic relations, any monomial in the generators
  that has $(N_0 e_1)$ followed by $(N_1 \xi_1 e_0)$ or $(N_1 e_0)$
  followed by $(N_0 \xi_0 e_1)$ can be expressed in terms of monomials
  which are strictly smaller in an appropriate ordering: take $(N_1
  e_0)<(N_1 \xi_1 e_0)$ and $(N_0 e_1)<(N_0 \xi_0 e_1)$ and order monomials
  reverse lexicographically (i.e., compare the last factor, then the second
  to last, etc.).  We thus conclude that in the $\Z$-algebra presented in
  this way, each $\Hom$ bimodule has a left spanning set consisting of
  those monomials such that all appearances of $\xi_i$ are as far to the
  left as possible.  In $\Hom(i,j)$ for $j\ge i$, this spanning set
  contains $j-i+1$ elements; since the Bruhat filtration shows that
  $\overline S_{ij}$ is free of rank $j-i+1$, there can be no further
  relations.
\end{proof}

\begin{rem}
  In particular, to compute the adjoint, it suffices to compute it in
  degree 1, where we find that it swaps $(N_0 r e_1)$ and $(N_1 r e_0)$.
  Globally, we need to twist slightly to make the $\Hom$ bimodules agree,
  and thus either take ${\cal L}_i\mapsto {\cal L}_i^{-1}\otimes
  \pi_i^*\det(\pi_{i*}\sO_{Y_i})$ or ${\cal L}_i^{-1}\otimes
  \pi_i^*\omega_{Y_i}$.  (These are not the same but differ only by a
  common factor of $\omega_X$, which has no effect on the sheaf
  $\Z$-algebra.)
\end{rem}

Locally, the quadratic relation has the following less basis-dependent
description: we have a natural (surjective!) map
\[
\overline{H}_{11}\to \overline{S}_{02}
\]
given by $x\mapsto N_0 x e_0$, and thus an induced map
\[
R\to \overline{H}_{11}\to \overline{S}_{02}.
\]
The quadratic relation starting at $0$ is then the composition of this map
with the inclusion $S_0\to R$.  Much the same holds globally, except that
$\overline{H}_{11}$ gets twisted so that we instead have a surjection
\[
\sHom_{Y_1}(\pi_{1*}{\cal L}_1,\pi_{1*}({\cal L}_1\otimes \pi^!_0\sO_{Y_0}))
\to
\overline{\cal S}_{02}
\]   
and thus an induced map
\[
\pi_{0*}\pi^!_0\sO_{Y_0}\to \overline{\cal S}_{02},
\]
and the relation is the composition with the natural map
\[
\det(\pi_{0*}\sO_X)^{-1}\to \pi_{0*}\pi^!_0\sO_{Y_0}.
\]
Note that since the composition is a relation in $\overline{\cal S}_{02}$, the
cokernel maps to $\overline{\cal S}_{02}$, giving a natural global section of
$\overline{\cal S}_{02}$, locally given by the element $e_0$.

Since the $\Z$-algebra is generated in degree 1, we record the
corresponding special case of Lemma \ref{lem:bruhat_for_S}.

\begin{cor}
  One has the bimodule isomorphism
  \[
  \overline{\cal S}_{01}
  \cong
  (\pi_0\times \pi_1)_*({\cal L}_1^{-1}\otimes \pi^!_1\sO_{Y_1}\otimes {\cal L}_0)
  \]
\end{cor}

Let $\qcoh \overline{\cal S}$ denote the category of quasicoherent sheaves over
$\overline{\cal S}$; that is, the quotient of the category of $\overline{\cal
  S}$-modules by the subcategory generated by right-bounded modules.

\begin{thm}
  There is a natural equivalence between $\qcoh \overline{\cal S}$ and the
  noncommutative $\P^1$-bundle \cite{VandenBerghM:2012} corresponding to
  the sheaf bimodule $(\pi_0\times\pi_1)_*({\cal L}_0\otimes {\cal
    L}_1^{-1})$.
\end{thm}

\begin{proof}
  Both categories are constructed from quadratic sheaf $\Z$-algebras, so it
  will suffice to show that the algebras are twists of each other; that is,
  that we can associate a line bundle ${\cal L}'_i$ on $Y_i$ in each degree
  such that ${\cal L}'_i\otimes_{Y_i}\overline{\cal S}_{ij}\otimes_{Y_j}
  {\cal L}'_j$ agrees with the algebra constructed in
  \cite{VandenBerghM:2012}.  Since
  \[
  (\pi_0\times \pi_1)_*({\cal L}_1^{-1}\otimes \pi^!_1\sO_{Y_1}\otimes {\cal L}_0)
  \cong
  (\pi_0\times \pi_1)_*({\cal L}_1^{-1}\otimes {\cal L}_0)
  \otimes_{Y_1} \det(\pi_{1*}\sO_{Y_1})^{-1},
  \]
  this holds for $\overline{\cal S}_{01}$, and $\overline{\cal S}_{-10}$ is the
  appropriate adjoint bimodule, so is also correct.  Both constructions are
  invariant under shifting the object group, and thus it remains only to
  verify that the quadratic relations agree, and this is straightforward.
\end{proof}

This is most useful in the case that $Y_0$, $Y_1$ are smooth
quasiprojective curves over a field $k$, as in that case {\em every}
noncommutative $\P^1$-bundle arises in this way.

\begin{prop}
  Let $k$ be a field and $C_0/k$, $C_1/k$ smooth quasiprojective curves.
  For any sheaf bimodule ${\cal E}$ on $C_0\times C_1$ such that
  $\pi_{0*}{\cal E}$ and $\pi_{1*}{\cal E}$ are both locally free of rank
  2, there is a curve $\whQ$ with finite flat morphisms
  $\phi_i:\hat{Q}\to C_i$ of degree 2 and an invertible sheaf ${\cal L}$ on
  $\whQ$ such that ${\cal E}\cong (\phi_0\times \phi_1)_*{\cal L}$.
\end{prop}

\begin{proof}
  Suppose first that ${\cal E}$ is not a rank 2 vector bundle supported on
  the graph of an isomorphism, and let ${\cal Q}:=\sHom_{C_0\times
    C_1}({\cal E},{\cal E})$.  Then $\pi_{0*}{\cal Q}$ may be interpreted
  as the subalgebra of $\sEnd_{C_0}(\pi_{0*}{\cal E})$ consisting locally
  of elements that commute with the action of $\sO_{C_1}$.  The generic
  fiber has rank 2, and the cokernel is torsion-free, so flat, and thus
  $\pi_{0*}{\cal Q}$ is itself flat.  In particular $\whQ:=\Spec{\cal
    Q}$ is a double cover of $C_0$ (and thus of $C_1$, by symmetry) as
  required.  Since $\pi_{0*}{\cal Q}$ is saturated in
  $\sEnd_{C_0}(\pi_{0*}{\cal E})$, it follows that $\pi_{0*}{\cal E}$ is
  locally a cyclic module over $\pi_{0*}{\cal Q}$, and thus ${\cal E}$ is
  the image of an invertible sheaf as required.

  Now suppose ${\cal E}$ is a rank 2 vector bundle supported on the graph
  of an isomorphism, or WLOG that $C_0=C_1=C$ and ${\cal E}$ is a vector
  bundle on the diagonal.  Let $\whQ$ be any curve in the corresponding
  projective bundle over $C$ that does not contain any fiber and meets the
  generic fiber twice.  Then there is a line bundle ${\cal L}_0$ on $C$
  such that $\pi_*\sO_{\whQ}\cong {\cal E}\otimes {\cal L}_0$, and thus
  ${\cal E}\cong \pi_*(\pi^*{\cal L}_0^{-1})$.
\end{proof}

\begin{rem}
  The notation $\whQ$ here reflects the fact that this curve is only the
  ``horizontal'' part of a more natural curve $Q$.
\end{rem}

Note that although both constructions give the same results geometrically,
they behave quite differently in families.  This appears not just in the
commutative case (where the $\whQ$-based construction depends on a
suitable choice of curve in the surface), but in the noncommutative case as
well.  The problem is that all we can say about the algebra ${\cal Q}$ in
general is that its cokernel in the relevant endomorphism ring is
torsion-free.  Over a curve, this is not a problem, but once we are dealing
with a family, it is quite possible for the quotient to fail to be flat.
Indeed, if $\overline{Q}$ is a nodal biquadratic curve in $\P^1\times
\P^1$, then there is a natural flat family of bimodules in $\P^1\times
\P^1$ parametrized by $\overline{Q}$, such that the fiber over each point
is the image of the corresponding ideal sheaf in $\sO_{\overline{Q}}$.
Over the smooth locus of $\overline{Q}$, the bimodule is the image of an
invertible sheaf on $\overline{Q}$, so has associated algebra
$\sO_{\overline{Q}}$, but over the singular point, the associated algebra
is instead the structure sheaf of the normalization.  But this cannot be
flat: $\sO_{\overline{Q}}$ has Euler characteristic 0, while the structure
sheaf of the normalization has Euler characteristic 1.  This also gives
rise to a sheaf bimodule on $(\P^1\times\overline{Q})\times (\P^1\times
\overline{Q})$ (or on the corresponding normalization!) such that the
corresponding noncommutative $\P^1$-bundle does not arise from the above
construction.  Luckily, our present interest is only in the curve case, and
many of the results we wish to show need only be checked on geometric
fibers.  (For instance, if we did not already know flatness from
\cite{VandenBerghM:2012}, we could easily prove it when $C_0$, $C_1$ are
projective by using the Bruhat filtration to check that the Euler
characteristic is the same on every geometric fiber.)

One nice aspect of the curve case is that there is a natural representation
of $\overline{\cal S}_{ij}$ in terms of (twisted) difference or differential
operators.  It suffices to give such an interpretation in degree 1 (the
quadratic relation will turn out to be automatically satisfied).  We may
write $\overline{\cal S}_{ij}$ as the tensor product
\[
(1\times \pi_1)_*({\cal L}_1^{-1}\otimes \pi^!_1\sO_{Y_1})
\times
(\pi_0\times 1)_*{\cal L}_0\otimes_{X}
.
\]
Here the second factor may be interpreted as the sheaf bimodule of
$\sO_{Y_0}$-linear maps $\sO_{Y_0}\to {\cal L}_0$, while the first factor
may similarly be viewed as the sheaf bimodule of $\sO_{Y_1}$-linear maps
${\cal L}_1\to \sO_{Y_1}$.  So if ${\cal L}_0={\cal L}_1$, then we may
interpret $\overline{\cal S}_{01}$ as the sheaf bimodule of compositions
\[
\sO_{Y_0}\to {\cal L}_0\to \sO_{Y_1}
\]
with the first factor $\sO_{Y_0}$-linear and the second factor
$\sO_{Y_1}$-linear.  In particular, if we take ${\cal L}_0=\sO_{X}$, then
the first factor can clearly be taken to be the natural inclusion, and thus
the map $\sO_{Y_0}\to \sO_{Y_1}$ is the restriction to $\sO_{Y_0}$ of a
general $\sO_{Y_1}$-linear map $\sO_{X}\to \sO_{Y_1}$.

Thus suppose $C_0/k$, $C_1/k$ are projective curves over a field $k$, or
localizations thereof, and let $\pi_i:\whQ\to C_i$ be a pair of double
covers.  There are three main cases to consider: both maps $\pi_i$ are
separable, both maps are inseparable, or precisely one of the two maps is
inseparable (and thus $k$ is a field of characteristic 2).  (In the first
two cases, we assume that the two maps cannot be identified by an
isomorphism $C_0\cong C_1$, and thus when both maps are inseparable,
$\whQ$ must be nonreduced.)

If $\whQ$ is reduced and both maps are separable, then the typical
$\sO_{C_1}$-linear map $\sO_{\whQ}\to \sO_{C_1}$ locally has the form
$f\mapsto c_0 f + s_1(c_0) s_1(f)$ with $c_0$ in the inverse different,
i.e., the sheaf of elements of $\sO_{\whQ}\otimes k(C_1)$ such that any
element of the corresponding fractional ideal has integral trace.  (As a
line bundle, the inverse different is isomorphic to $\pi_1^!\sO_{C_1}$, and
gives an expression of the latter in terms of an effective Cartier
divisor.)  The restriction to $\sO_{C_0}$ can be expressed in the same
form, but we can also write it as $f\mapsto \alpha f + s_1(\alpha)
s_1(s_0(f))$, and similarly for $\overline{\cal S}_{12}$.  In each case, the
operator involves either $s_1\circ s_0$ or its inverse, so that $\overline{\cal
  S}_{0d}$ maps to an operator of the form
\[
f\mapsto 
\sum_{-\lfloor d/2\rfloor\le i\le \lceil d/2\rceil} c_i (s_1\circ s_0)^i(f)
\]
with coefficients satisfying $c_{(d\bmod 2)-i}=s_d(c_i)$ (as well as
conditions along the ramification divisor which can be fairly complicated
in general, especially where $\whQ$ is singular).  In the case $\whQ$
is elliptic and $C_0,C_1$ rational, this recovers the untwisted version of
the symmetric elliptic difference operators of \cite{generic}.

If $\whQ$ is reduced but not integral, then each component induces an
isomorphism between $C_0$ and $C_1$, so we may as well use one component to
identify them with a single curve $C$, and then the other component induces
a nontrivial automorphism $\alpha\in \Aut(C)$.  In that case, the
description of $\overline{\cal S}_{01}$ simplifies to operators of the form
\[
f\mapsto c_0 f + c_1 \alpha(f)
\]
where $c_0$, $c_1\in \sO_C$ must agree to appropriate order wherever the
corresponding branches of $\whQ$ meet (necessarily at fixed points of
$\alpha$).  More generally, $\overline{\cal S}_{0d}$ maps to operators
\[
f
\mapsto 
\sum_{-\lfloor d/2\rfloor\le i\le \lceil d/2\rceil} c_i \alpha^i(f).
\]

In the nonreduced case, the reduced subscheme of $\whQ$ induces an
isomorphism $C_0\cong C_1$, so that again we may as well identify both
curves with some fixed curve $C$.  Then the coordinate ring of $\whQ$
over $C_0$ is
\[
\sO_{\whQ} = \sO_C\oplus \epsilon {\cal L}
\]
for some line bundle ${\cal L}$ on $C$.  The embedding $\sO_{C_1}\to
\sO_{\whQ}$ then has the form $f\mapsto f+\epsilon \partial f$ where
$\partial$ is a derivation $\sO_{C_1}\to {\cal L}$.  It follows that ${\cal
  L}$ contains $\omega_C$, and thus there is some divisor $D_c$
(essentially a conductor) such that ${\cal L}\cong \omega_C(D_c)$ and
$\partial f=df$.  The general $\sO_{C_1}$-linear map $\sO_{\whQ}\to
\sO_{C_1}$ then takes the form
\[
f+g \epsilon\mapsto \alpha f + \beta(g-df)
\]
with $\alpha\in \sO_{\whQ}$ and $\beta\in \omega_{C_0}(D_c)^{-1}$.
Restricting this to $\sO_{C_0}$ gives
\[
f\mapsto \alpha f - \beta df,
\]
or in other words the general first-order differential operator such that
the coefficient of differentiation vanishes along $D_c$.  Similarly, the
restriction to $\sO_{C_1}$ of the general $\sO_{C_0}$-linear map
$\sO_{\whQ}\to \sO_{C_0}$ takes the form $f\mapsto \alpha f + \beta
df$.  We thus see more generally that $\overline{\cal S}$ has a
representation inside the sheaf of differential operators on $C$ such that
the image of ${\cal S}_{ij}$ locally consists of operators of order $j-i$
such that the coefficient of $(d/dz)^l$ vanishes along $lD_c$.

Finally, if only one of the two maps is separable, then we end up with a
hybrid operator representation.  Supposing $\pi_1$ is the separable
morphism, then $\overline{\cal S}_{01}$ has the same interpretation as in the
purely separable case ($f\mapsto \alpha f + s_1(\alpha f)$ for suitable
$\alpha$).  On the other hand, $\overline{\cal S}_{12}$ consists of the
restriction to $\sO_{C_1}$ of linear maps $\sO_{\whQ}\to \sO_{C_0}$.
Since $\pi_0$ is inseparable, for any local section $f$ of $\sO_{\whQ}$,
$df$ may be interpreted as a local section on $C_0$ of $\omega_{C_0}\otimes
\det(\pi_{0*}\sO_{\whQ})^{-1}$, so that if $g$ is any local section of
$\pi_0^*(\omega_{C_0}^{-1}\otimes \det(\pi_{0*}\sO_{\whQ}))$, then
$f\mapsto d(fg)$ is a linear map $\sO_{\whQ}\to \sO_{C_0}$, and any such
map arises in this way.  We thus see that $\overline{\cal S}_{12}$ has a
natural interpretation as differential operators.  It is worth noting that
the resulting hybrid operator representation is far from faithful; indeed,
it already has a nontrivial kernel in degree 2.  (We will see below that in
this case $\overline{\cal S}$ has the same category of coherent sheaves as a
maximal order in a quaternion algebra on a commutative ruled surface.)

When ${\cal L}_0\otimes {\cal L}_1^{-1}$ is nontrivial, the above
interpretation must of course be twisted accordingly.  To understand the
nature of such twisting, it suffices to consider the corresponding
automorphisms in the affine case.  Although this in principle involves a
pair of units, only their ratio actually appears, and one thus obtains the
automorphism $N_1 r e_0\mapsto N_1 u r e_0$ and $N_0 r e_1\mapsto N_0
u^{-1} r e_1$ of $\overline S$ for some unit $u\in R^*$.  To interpret this,
note that $N_1 e_0$ is an operator from $S_0$ to $S_1$ that annihilates
$1$, and similarly for $N_0 e_1$, while $N_1 \xi_1 e_0$ and $N_0 \xi_0 e_1$
similarly map $1$ to $1$.  Let $F_{0,u}$ be a formal solution (in $S_0$) of
the equation $N_1 u e_0\cdot F_{0,u}=0$, and let $F_{1,u}$ be the formal
image $N_1 u \xi_1 e_0 \cdot F_{0,u}$.  We then find that
\[
(N_0 u^{-1} e_1)\cdot F_{1,u}
=
(N_0 u^{-1} N_1 u \xi_1 e_0)\cdot F_{0,u}
=
(N_0 u^{-1} s_1(\xi_1) N_1 u e_0)\cdot F_{0,u}
=
0,
\]
so that $F_{1,u}$ is a formal solution to $N_0 u^{-1} e_1\cdot F_{1,u}=0$.
Similarly,
\[
(N_0 u^{-1} \xi_0 e_1)\cdot F_{1,u} = e_0\cdot F_{0,u},
\]
so that one may reasonably identify $F_{2,u}$ with $F_{0,u}$.  We also find
that for $a\in S_0$,
\begin{align}
(N_1 u a e_0)\cdot F_{0,u} &= \nu_1(a) F_{1,u},\notag\\
(N_1 u \xi_1 a e_0)\cdot F_{0,u} &= \nu_1(a\xi_1) F_{1,u}
\end{align}
corresponding to the formal identities
\begin{align}
F_{1,u}^{-1}(N_1 u e_0)F_{0,u} &= N_1 e_0.\notag\\
F_{1,u}^{-1}(N_1 u \xi_1 e_0)F_{0,u} &= N_1 \xi_1 e_0.
\end{align}
In other words, the automorphism corresponding to $u$ may be interpreted as
gauging by the system of formal symbols $F_{i,u}$.

In the differential case, the formal equation satisfied by $F_{0,u}$ has
the form $(D-v)F_{0,u}=0$ for suitable $v$; that is, $F_{0,u}$ is a formal
symbol with logarithmic derivative $v$, and one further finds that
$F_{1,u}=(u\bmod \epsilon)F_{0,u}$.  Similarly, in the nonsymmetric
difference case, $F_{0,u}$ is a formal solution to $\alpha(F_{0,u}) = v
F_{0,u}$, while in the symmetric difference case, it is a formal solution
to the equations $s_0(F_{0,u})=F_{0,u}$ and $s_1(uF_{0,u})=uF_{0,u}$.  Note
that in suitably analytic settings, we can represent $F_{0,u}$ and
$F_{1,u}$ by honest functions, e.g., as the exponential of the integral of
the appropriate meromorphic differential, or (when $s_0s_1$ has infinite
order) as a suitable infinite product.

Any local trivialization of ${\cal L}_0\otimes {\cal L}_1^{-1}$ leads to an
representation of $\overline{\cal S}$ as sheaves of {\em meromorphic} operators
of the appropriate kind, and the above calculation shows that any other
such representation will be related by a suitable scalar gauge
transformation (by the solution of a first-order equation).  This is
particularly useful in the rational case; when $C_0\cong C_1\cong \P^1$ and
the image of $\whQ$ in $\P^1\times \P^1$ is singular, we can always
localize to the complement of a singular point, which will make all curves
affine and all line bundles trivial.  This covers every case for which
$C_0\cong C_1\cong \P^1$ except the case when $\whQ$ is smooth genus 1,
for which see \cite{generic}.

\medskip

One useful consequence of the operator interpretation is that it makes it
straightforward to show that $\overline{\cal S}_{ij}$ is a domain.  We in fact
have the following.

\begin{prop}
  Let $Y_0$, $Y_1$ be integral schemes, and let ${\cal E}$ be a sheaf
  bimodule of birank $(2,2)$ on $Y_0\times Y_1$.  Then the corresponding
  sheaf $\Z$-algebra is a domain.
\end{prop}

\begin{proof}
  Enlarging the sheaf $\Z$-algebra can only introduce zero divisors, so we
  may as well restrict to the generic points of $Y_0$ and $Y_1$.  This
  makes ${\cal E}$ invertible on its support, and thus the resulting
  $\Z$-algebra is of the form $\overline{S}_{ij}$ with
  $S_0=\overline{S}_{00}$, $S_1=\overline{S}_{11}$ fields.

  If $R$ is a field, then given any pair of morphisms composing to 0, one
  of their leading coefficients with respect to the Bruhat filtration must
  vanish, since the leading coefficient of the product is the (automorphism
  twisted) product of leading coefficients.  But in that case we can factor
  out the degree 2 element $e_0$ or $e_1$ on the appropriate side to obtain
  a smaller such pair, eventually yielding a contradiction.

  If $R$ is a sum of two fields, then the derivation of the difference
  operator interpretation gives a faithful homomorphism from
  $\overline{S}_{ij}$ to the twisted group algebra $S_0[\langle
    \alpha\rangle]$.  But this is a domain (by the same leading coefficient
  argument).

  Finally, if $R$ is nonreduced, then $\overline{S}_{ij}$ consists of
  polynomials in some fixed derivation (or, rather, an element acting as
  such a derivation) of $S_0$, and again leading coefficients produce the
  domain property.
\end{proof}

\medskip

As mentioned, the curve $\whQ$ is only a piece of a somewhat larger
curve contained in the noncommutative surface.  That $\whQ$ itself
embeds in the surface follows by observing that the system of elements
$(e_i)\in \overline{S}_{i(i+2)}$ is ``central'' in a suitable sense;
indeed, if $D\in \overline{S}_{ij}$ then $(e_j)D = D(e_i)$, where we use
the natural identification between $\overline{S}_{ij}$ and
$\overline{S}_{(i+2)(j+2)}$.  The quotient by the corresponding ideal takes
$\overline{S}_{ij}$ to the quotient
$\overline{S}_{ij}/\overline{S}_{i(j-2)}$, which as computed above
globalizes to the image of an invertible sheaf on $\whQ$.  We thus
find that the quotient $\Z$-algebra is a twisted version (\`a la
\cite{ArtinM/VandenBerghM:1990}) of the homogeneous coordinate ring of
$\whQ$, and thus the corresponding category of sheaves is just
$\qcoh \whQ$.  Moreover, this not only embeds $\whQ$ in the
noncommutative surface, but it does so {\em as a divisor} in the sense of
\cite{VandenBerghM:1998}, which in particular means that any point of
$\whQ$ is a suitable candidate for the blowing up construction given
therein.  (Note that everything stated here applies more generally to embed
$X$ as a divisor in the corresponding noncommutative $\P^1$-bundle.)  To be
precise, there is an endofunctor $\_(-\whQ)$ of $\qcoh
\overline{S}_{ij}$, which simply shifts indices by $2$, and a natural
transformation $\_(-\whQ)\to \text{id}$ which is 0 on a subcategory
isomorphic to $\qcoh\whQ$, and has cokernel mapping to that
subcategory.

The larger curve is also embedded as a divisor, and arises by asking for
the largest quotient $\Z$-algebra that satisfies the same twisted
commutativity relation as $\whQ$ (or $X$).  Passing to the local case,
we find that we should consider the two-sided ideal generated by
quasi-commutators of the form
\[
(N_0 x e_1)(N_1 s_1(y) e_0) - (N_0 y e_1)(N_1 s_1(x) e_0)
\quad\text{or}\quad
(N_1 x e_0)(N_0 s_0(y) e_1) - (N_1 y e_0)(N_0 s_0(x) e_1).
\]
We may simplify
\[
(N_0 x e_1)(N_1 s_1(y) e_0) - (N_0 y e_1)(N_1 s_1(x) e_0)
=
N_0 (x N_1 s_1(y)-y N_1 s_1(x)) e_0,
\]
and observe that since $N_1$ is nilpotent, so is negated by the intrinsic
adjoint in $\End_{S_1}(R)$, the terms $x N_1 s_1(y)$ and $-y N_1 s_1(x)$
are intrinsic adjoints, and thus their sum is in $S_1$.  (Moreover, taking
$x=1$ gives $\nu_1(y)$, and thus any element of $S_1$ arises in this way.)
We thus see that the two-sided ideal of quasi-commutators is the same as
that generated by $N_i S_{i+1}e_i = \nu_i(S_{i+1}) e_i\subset
\overline{S}_{i(i+2)}$.  In particular, it is indeed contained in the ideal
corresponding to $\whQ$.  Moreover, it is essentially given by a conductor.

\begin{lem}
  We have $\nu_0(S_1)R=\nu_1(S_0)R=\{x:x\in R\mid xR\subset S_0S_1\}$.
\end{lem}

\begin{proof}
  For $f\in S_1$, $g\in R$, we have $\nu_0(f)g = \nu_0(f
  s_0(g))+f\nu_0(g)\in S_0S_1$, so that
  \[
  \nu_0(S_1)R\subset \{x:x\in R\mid xR\subset S_0S_1\}.
  \]
  Conversely, if $cR\subset S_0S_1$, and $\xi_0$ is such that
  $\nu_0(\xi_0)=1$, then both $c$ and $c s_0(\xi_0)$ are in $S_0S_1$,
  so that $\nu_0(c)$ and $\nu_0(c s_0(\xi_0))$ are in
  $\nu_0(S_0S_1)=S_0\nu_0(S_1)$, and thus
  \[
  c = \nu_0(c s_0(\xi_0))+\xi_0\nu_0(c)\in \nu_0(S_1)R
  \]
  as required.
\end{proof}

\begin{rem}
  Note that if $S_0$, $S_1$ are Dedekind domains with distinct images
  inside $R$, then both $R$ and $S_0S_1$ are orders inside the
  normalization, and $\{x:x\in R\mid xR\subset S_0S_1\}$ is the conductor
  of $S_0S_1$ as a suborder of $R$.  (In the commutative case, this ideal
  is of course $0$.)  This can be described in more global terms by noting
  that, as an ideal in $S_0$, the conductor is $\nu_0(S_0S_1)$, or in other
  words the image of the suborder under the natural map $R\mapsto R/S_0$.
  We thus see that globally (for curves) the corresponding ideal sheaf is
  the image of the natural map $\det(\pi_{0*}\sO_{\overline{Q}})\otimes
  \det(\pi_{0*}\sO_{\whQ})^{-1}\to \sO_{C_0}$, where $\overline{Q}$
  is the image of $\whQ$ in $C_0\times C_1$ and the map is induced
  by the natural inclusion $\sO_{\overline{Q}}\to \sO_{\whQ}$, which
  is the identity on $\sO_{C_0}$.  This measures the failure of ${\cal E}$
  to be invertible along $\overline{Q}$.
\end{rem}

In the curve case (and assuming $\pi_0$ and $\pi_1$ are not related by an
isomorphism $C_0\cong C_1$), the ideals $S_0\nu_0(S_1)$ and $S_1\nu_1(S_0)$
globalize to line bundles contained in $\sO_{C_0}$, $\sO_{C_1}$
respectively, with the property that both bundles pull back to isomorphic
bundles on $\whQ$.  Normally, twisting ${\cal L}_0$ or ${\cal L}_1$ by
pulled back line bundles gives an equivalence between different surfaces,
but in this case the twisting has no effect on $\overline{\cal S}_{i(i+1)}$ and
thus gives rise to an autoequivalence.  Composing this with the shift-by-2
autoequivalence gives an autoequivalence equipped with a natural
transformation to the identity such that the cokernel is quasicommutative.
In particular, we find that the result is a twisted version of the
homogeneous coordinate ring of a commutative scheme.  Of course, if $\pi_0$
and $\pi_1$ {\em are} related by an isomorphism, then the ideal is trivial,
and $\overline{\cal S}$ is itself the twisted homogeneous coordinate ring of a
commutative scheme (in this case a ruled surface).  We denote the
corresponding autoequivalence of $\qcoh \overline{\cal S}$ by $\_(-Q)$, and
note that the natural transformation $\_(-Q)\to \text{id}$ factors through
$\_(-\whQ)\to\text{id}$.

To identify this scheme, suppose $k\subset S_0\cap S_1$ is some subring,
and let $M\subset R$ be an $S_0S_1$-submodule such that $M/R$ is locally
free of rank 1 over $k$.  Base changing to make this free, we then see that
there are ideals $I_i\subset S_i$ such that $M\supset I_i R$, and then
for $y\in I_1$,
\[
(N_1 R e_0)(N_0 M e_1)
\ni
(N_1 e_0)(N_0 \xi y e_1)-(N_1 s_0(\xi) e_0)(N_0 y e_1)
=
N_1 y e_1
=
\nu_1(y) e_1,
\]
so that $(N_1 R e_0)(N_0 M e_1)\supset \nu_1(I_1)e_1$.  Any element of
$S_1$ can be written as $a+y$ for $a\in k$, $y\in I_1$, and thus
$\nu_1(I_1)=\nu_1(S_1)$, so that the left ideal generated by $N_0 M e_1$
contains the left ideal generated by $\nu_1(S_1)$, and thus the two-sided
ideal so generated.  In particular, for $r\in R$, $m\in M$, we have
\[
(N_1 r e_0)(N_0 m e_1)\equiv (N_1 s_0(m) e_0)(N_0 s_0(r) e_1)
\]
modulo that ideal, and thus
\[
(N_1 R e_0)(N_0 M e_1)
=
(N_1 s_1(M) e_0)(N_0 R e_1).
\]
We thus obtain two conclusions: first that any such $M$ induces a
representation of $\overline{S}_{ij}$ which is locally free of rank 1 over $k$
in every degree, and second the degree 2 elements of the form
$\nu_1(S_0)e_1$ and $\nu_0(S_1)e_0$ act as 0 on any such representation.
In other words (and globalizing), the corresponding commutative scheme is
precisely the moduli space of such representations, which in turn may be
identified with $\Quot(\overline{\cal S}_{01},1)$ (the moduli space globalizing the moduli space of
submodules $M$).

These identifications were already known for general noncommutative
$\P^1$-bundles \cite{VandenBerghM:2012,NymanA:2004}; the one new
observation is that in the surface case (i.e., when $Y_0$, $Y_1$ are
curves) the resulting curve (denoted $Q$) is embedded as a divisor in the
noncommutative surface.  In other words, {\em any} point of a
noncommutative surface obtained from this construction is eligible to be
blown up \`a la \cite{VandenBerghM:1998}.  To see that $Q$ is a curve, note
that since a $S_0S_1$-module is also an $S_0$-module, there is a natural
morphism $Q=\Quot(\overline{\cal S}_{01},1)\to \Quot(\pi_{0*}\overline{\cal
  S}_{01},1)$, embedding $Q$ as a subscheme of a $\P^1$ bundle over $C_0$.
Moreover, that subscheme is locally cut out by the equation $v\wedge
\phi(v)=0$, where $\phi\in \End_{\sO_{C_0}}(\pi_{0*}\overline{\cal
  S}_{01})$ is the image of a local generator of $\sO_{\overline{Q}}$ over
$\sO_{C_0}$.  Since
\[
v\wedge \phi(v) = \gamma (v\wedge \xi v)
\]
where $\xi$ is a local generator of $\sO_{\whQ}$ over $\sO_{C_0}$ and
$\gamma$ is a local generator of the conductor, we see that $Q$ splits as a
divisor isomorphic to $\whQ$ and transverse to every fiber, plus a
vertical divisor given by the pullback of the conductor.  Note in
particular that this implies that the arithmetic genera of the three curves
are related by $g(Q)=g(\whQ)+c=g(\overline{Q})$, where $c$ is the degree of
the conductor.  Indeed, the first equation follows from intersection
theory, while the second follows by comparing differents.

We should also note that in the 1-dimensional case, the pullback of any
$S_0$-ideal containing the conductor descends to an $S_1$-ideal containing
the conductor, and thus gives rise to a corresponding commutative curve
embedded as a divisor.  This is mainly relevant in the commutative case
(when the conductor is 0), as it means that we still have natural models of
$\whQ$ and various curves $Q$ inside the surface.

\bigskip

We would like to have a classification of sheaf bimodules associated to
noncommutative surfaces, but this in general appears to be quite difficult
to control; indeed, any smooth curve with a pair of order 2 automorphisms
gives rise to at least one such bimodule for every invertible sheaf (and
possibly more, if the composition of the automorphisms has fixed points).
Although this is largely intractable, it turns out that nearly every such
configuration is uninteresting from a noncommutative geometry perspective,
as the corresponding noncommutative $\P^1$-bundle can be described as a
finite algebra over a corresponding commutative ruled surface.

With this in mind, we define a noncommutative {\em quasi-ruled} surface to
be any surface obtained from the above construction in the case that $C_0$,
$C_1$ are smooth projective curves, so that we may single out a subclass of
noncommutative {\em ruled} surfaces, which (by Proposition
\ref{prop:qr_is_semicomm} and Theorem \ref{thm:semicomm} below) include all
of the interesting cases and are much easier to classify.  We say that the
quasi-ruled surface corresponding to $\whQ$, $\pi_i$, ${\cal L}_i$ is a
(geometrically) ruled surface if (possibly after a field extension) there
is an isomorphism $\phi:C_0\cong C_1$ such that the divisor $(\pi_0\times
\pi_1)(\whQ)$ is algebraically equivalent to twice the graph of $\phi$.
This of course includes the commutative case $\pi_1=\pi_0\circ\phi$, as
well as the case that $\whQ$ is nonreduced (since then the reduced
subscheme is the graph of such a $\phi$).  Moreover, any component of the
moduli stack of quasi-ruled surfaces that contains a ruled surface contains
either a commutative surface or a surface on which $\whQ$ is nonreduced,
and it is easy to see that the latter can be degenerated to the commutative
case.  So roughly speaking the noncommutative ruled surfaces are those
cases which are deformations of commutative ruled surfaces (though there is
one exotic type of such a deformation which is not ruled, see below).

The ruled surfaces are also nearly characterized by the following fact.

\begin{prop}\label{prop:genus_1_almost_implies_ruled}
  The curve $Q$ associated to a quasi-ruled surface has arithmetic genus
  $\ge 1$, with equality precisely when either the surface is ruled or $Q$
  is smooth of genus 1 and $\pi_0$ and $\pi_1$ are distinct $2$-isogenies.
\end{prop}

\begin{proof}
  We have already shown that $g(Q)=g(\overline{Q})$, so it suffices to consider
  the latter.  In particular, since the double diagonal in $C\times C$ has
  arithmetic genus 1, we certainly have equality in the ruled surface case.
  Similarly, in the $2$-isogeny case, $Q$ is embedded in the product of
  quotients, and thus $\overline{Q}\cong Q$ has arithmetic genus 1.

  For the other direction, we may assume $\overline{Q}$ reduced.  Let
  $\widetilde{Q}$ be the normalization of $\overline{Q}$, and observe that
  $g(\widetilde{Q})\le g(\overline{Q})$.  Since $\widetilde{Q}$ may not be connected,
  its arithmetic genus could be negative, but since it has at most two
  components, the arithmetic genus is at least $-1$.  In any event, if
  $g(\widetilde{Q})\le 0$, then some component of $\widetilde{Q}$ has genus 0, and
  $C_0$ and $C_1$ are images of that component, so must themselves have
  genus 0.  But in that case, we (geometrically) have $C_0\times C_1\cong
  \P^1\times \P^1$ with $\overline{Q}$ of bidegree $(2,2)$, so are in the ruled
  case.

  We thus reduce to the case that $g(\widetilde{Q})=1$ and no component has
  genus 0.  Moreover, $\widetilde{Q}$ must embed in $C_0\times C_1$, since
  otherwise $\overline{Q}$ has strictly larger arithmetic genus than
  $\widetilde{Q}$.  If $\widetilde{Q}$ is reducible, then both components
  must be smooth genus 1 curves isomorphic to both $C_0$ and $C_1$.
  Without loss of generality, we may use one component to identify $C_0$
  and $C_1$, so that $\widetilde{Q}$ is the union of the diagonal and the
  graph of some automorphism of $C$.  If that automorphism had fixed
  points, then the two graphs would intersect, and thus the automorphism
  must be a translation; since the group of translations is connected,
  $\overline{Q}$ is algebraically equivalent to the double diagonal, so the
  surface is ruled.  If $\widetilde{Q}$ is integral, then any degree 2
  morphism is either a $2$-isogeny or a map to $\P^1$.  If both maps are to
  $\P^1$, then we are again in the bidegree $(2,2)$ case, while if both
  maps are $2$-isogenies then they are distinct (since otherwise
  $\overline{Q}$ would be nonreduced).  In the remaining case, $\bar{Q}$ is
  a bidegree $(2,2)$ curve on a surface $E\times \P^1$, so has arithmetic
  genus 3 by Hirzebruch-Riemann-Roch.
\end{proof}

\begin{rems}
  Unlike $\whQ$, the curve $Q$ {\em is} flat for any flat family of rank
  $(2,2)$ sheaf bimodules, and thus in particular its arithmetic genus is
  constant on any component of the moduli stack.  As a result, we find that
  any component on which the genus is $>1$ cannot possibly contain any
  commutative fibers.  Since the components corresponding to ruled surfaces
  certainly do contain such fibers, the only question involves the
  $2$-isogeny case.  When $\Pic^0(C)[2]$ is \'etale, the requirement that
  the $2$-isogenies are distinct is preserved under arbitrary deformation,
  and thus any commutative fiber must have characteristic 2.  On the other
  hand, the modular curve parametrizing pairs of $2$-isogenies
  (equivalently the modular curve classifying ``cyclic'' $4$-isogenies)
  splits into three components in characteristic 2, and on one of those
  components, the two $2$-isogenies agree.  In other words, there exists a
  curve over a $2$-adic dvr such that the generic fiber admits a pair of
  distinct $2$-isogenies which become the same on the special fiber.
  Taking that curve to be $\whQ$ gives a $2$-adic point of the given
  component of the moduli stack such that the special fiber is indeed
  commutative.  In the untwisted case, the special fiber is
  $\P(\pi_*\sO_{\whQ})$, and $\pi_*\sO_{\whQ}$ is the unique indecomposable
  bundle of rank 2 and trivial determinant.  The deformation induces a
  Poisson structure on that ruled surface which vanishes precisely on the
  image of $\whQ$.  This explains the exotic characteristic 2 Poisson
  surfaces arising in the classification of \cite{poisson}.
\end{rems}

\begin{rems}
  In the ruled and $2$-isogeny cases, since every connected component of
  $\overline{Q}$ has arithmetic genus 1, there is a noncanonical isomorphism
  $\omega_{\overline{Q}}\cong \sO_{\overline{Q}}$.  Since $\overline{Q}$ is a double cover
  of $C_0$, we can also write $\omega_{\overline{Q}}\cong \pi_0^!\omega_{C_0}
  \cong \pi_0^*(\omega_{C_0}\otimes \det(\pi_{0*}\sO_{\overline{Q}})^{-1})$, and
  similarly for $C_1$.  In the ruled cases, one in fact has
  $\omega_{C_i}\cong \det(\pi_{i*}\sO_{\overline{Q}})$ for $i\in \{0,1\}$;
  indeed one can check in each case that $\Pic(C_i)\to \Pic(\overline{Q})$ is
  injective.  In the $2$-isogeny case, this fails for at least one of the
  $C_i$.  Indeed, in that case, $\det(\pi_{i*}\sO_{\overline{Q}})$ corresponds
  to the nontrivial point of the kernel of the dual $2$-isogeny; since
  $\overline{Q}$ has at most one $2$-isogeny with inseparable dual, we have
  $\det(\pi_{i*}\sO_{\overline{Q}})\not\cong \sO_{C_i}\cong \omega_{C_i}$ for at
  least one $i$.  We will see below that the isomorphism $\omega_{C_i}\cong
  \det(\pi_{0*}\sO_{\overline{Q}})^{-1}$ is equivalent to the divisor $Q$ of
  points being anticanonical (i.e., such that $\_(-Q)[2]$ is a Serre
  functor), giving us another characterization of ruled surfaces.
\end{rems}

The key to showing that non-ruled quasi-ruled surfaces are nearly
commutative is the following result.

\begin{prop}\label{prop:qr_is_semicomm}
  If $\whQ$, $\pi_0$, $\pi_1$ corresponds to a noncommutative
  quasi-ruled surface which is not ruled, then the automorphism $s_0s_1$
  has finite order.
\end{prop}

\begin{proof}
  Suppose otherwise.  Since we have excluded the ruled surface case, either
  $g(\overline{Q})>1$ or $\overline{Q}$ is a smooth genus 1 curve and
  $s_0$, $s_1$ are translations by distinct $2$-torsion points (one of
  which may be the identity in characteristic 2).  In the latter case,
  $s_0s_1$ is again translation by a $2$-torsion point, so has order 2.  We
  may thus assume $g(\overline{Q})>1$.  Since $s_0s_1$ acts on the
  normalization $\widetilde{Q}$ preserving the components, and smooth
  projective curves of genus $>1$ have finite automorphism groups, the
  claim is automatic unless the (isomorphic via $s_0$) components have
  genus $\le 1$; since a genus 0 component would force the configuration to
  correspond to a ruled surface, the components in fact have genus $1$.
  But $s_0s_1$ not only acts on $\widetilde{Q}$, but preserves the finite
  set of preimages of singular points of $\overline{Q}$, so that some power
  of $s_0s_1$ fixes all preimages and a further power fixes $\overline{Q}$.
\end{proof}

In general, if $\pi_0$ and $\pi_1$ are separable, and $s_0s_1$ has finite
order, then the \'etale algebra $k(\whQ)^{\langle s_0s_1\rangle}$ is a
quadratic extension of the field $k(\whQ)^{\langle s_0,s_1\rangle}$.  We
let $C'$ be the smooth curve with the latter function field, and define
$\whQ'$ to be the unique curve with the former generic \'etale algebra such
that the conductor of $\whQ'$ in its normalization is the norm of the
corresponding conductor for $\whQ$.  We will see in Section
\ref{sec:semicomm} that the center of the untwisted quasi-ruled surface
corresponding to $\pi_i:\whQ\to C_i$ is the (commutative) untwisted
quasi-ruled surface corresponding to $\pi:\whQ'\to C'$, in the sense that
the category of coherent sheaves on the noncommutative surface is the
category of coherent modules of a sheaf of algebras on the commutative
surface.  The twisted situation is slightly trickier to describe, but we
note that any line bundle on $\whQ$ can be represented by a Cartier divisor
supported on the smooth locus of $\overline{Q}$, and thus by taking images
induces a Cartier divisor on the corresponding locus of $\whQ'$.  This
induces a well-defined map from the possible twists of the noncommutative
surface to the possible twists of the commutative surface, compatible with
the sheaf of algebras structure.

The hybrid case is similar; if $\pi_0$ is inseparable, then we take $C'$ to
be the Frobenius image of $C_1$, so that $C_0$ is a separable double
cover of $C'$, and we take $\whQ'$ to be the corresponding singular
model of $C_0$, such that the conductor is the norm of the conductor of
$\whQ$.

Finally, if $\whQ$ is nonreduced of characteristic $p$ (so that WLOG
$C_0=C_1=C$ and $\overline{Q}$ is the double diagonal), we take $C'$ to be
the image of $C$ under Frobenius.  If $\whQ=\overline{Q}$, then we take
$\whQ'$ to be the double diagonal in $C'$, and in general construct $\whQ'$
so that its conductor relative to the double diagonal is the norm of that
of $\whQ$.  Nontrivial twists (i.e., not in $\Pic(C)$) only exist when the
conductor is trivial, in which case they are given (modulo $\Pic(C)$) by a
class in $H^1(\omega_C)$.  To describe the resulting map $H^1(\omega_C)\to
H^1(\omega_{C'})$, it suffices to give a local ($\F_p$-linear) map
$\Omega_C\to \Omega_{C'}$ inducing it.

The key idea is the following observation.

\begin{prop}\label{prop:twist_differential}
  Let $u\in k(C)$ be a function which is a uniformizer at some point.  Then
  in the algebra of differential operators on $k(C)$, one has the relation
  \[
  \Bigl(\frac{d}{du} + f\Bigr)^p
  =
  \Bigl(\frac{d}{du}\Bigr)^p + f^p + \frac{d^{p-1}f}{du^{p-1}}
  \]
  for any function $f$.
\end{prop}

\begin{proof}
  Since $\frac{d}{du}\mapsto \frac{d}{du}+f$ induces an automorphism of the
  algebra of differential operators and $(d/du)^p$ is central, the operator
  $(\frac{d}{du}+f)^p$ must also be central, and thus a (monic)
  $k(C')$-linear combination of $(d/du)^p$ and $1$.  Applying this to a
  composition of such automorphisms, we find that the map
  \[
  f\mapsto 
  \Bigl(\frac{d}{du} + f\Bigr)^p
  -
  \Bigl(\frac{d}{du}\Bigr)^p
  \]
  from $k(C)$ to $k(C')$ is $\F_p$-linear.

  This map is given by a polynomial in $f$ and its iterated derivatives
  which is homogeneous of degree $p$ with respect to the grading in which
  $\deg(d^lf/du^l)=l+1$, corresponding to rescaling both $u$ and $f$.
  Moreover, $\F_p$-linearity implies that every monomial that appears must
  have degree a power of $p$ with respect to the grading by rescaling $f$
  alone.  It follows that
  \[
  \Bigl(\frac{d}{du}+f\Bigr)^p = \Bigl(\frac{d}{du}\Bigr)^p+a f^p + b
  \frac{d^{p-1}f}{du^{p-1}}
  \]
  for suitable constants $a,b\in \F_p$.  Each of those monomials arises
  from exactly one of the $2^p$ terms on the left, and thus one easily
  verifies that $a=b=1$ as required.
\end{proof}

\begin{lem}\label{lem:quasi-norm-differential}
  For any meromorphic differential $\omega$ on $C$, the differential
  \[
  \tau(\omega)
  :=
  \Bigl[\Bigl(\frac{d+\omega}{du}\Bigr)^p-\Bigl(\frac{d}{du}\Bigr)^p\Bigr]
  d(u^p)
  =
  \Bigl[\Bigl(\frac{\omega}{du}\Bigr)^p + \frac{d^{p-1}}{du^{p-1}}\frac{\omega}{du}\Bigr]d(u^p)
  \]
  on $C'$ is independent of $u$, so induces a canonical $\F_p$-linear sheaf
  map $\tau:\Omega_C\to \Omega_{C'}$.
\end{lem}

\begin{proof}
  We first claim that for $f\in k(C)$, we have the invariant description
  $\tau(df) = d(f^p)$.  Indeed, plugging into the definition of $\tau$ and
  using the fact that $p$-th derivatives vanish gives us the expression
  $\tau(df) = \Bigl(\frac{df}{du}\Bigr)^p d(u^p)$, which agrees with
  $d(f^p)$ when $f$ is a power of $u$ and thus by $\F_p$-linearity for all
  $f\in k((u))$, and in particular on the subspace $k(C)$.

  Similarly, if $g\in k(C')$, so that $dg=0$, we claim that $\tau(g f^{-1}
  df) = (g^p-g) f^{-p} d(f^p)$ for any $f\in k(C)^*$.  It suffices to
  verify this in the completion, and thus by continuity when $f$ is a
  polynomial, and thus, since both sides are multiplicative homomorphisms
  in $f$, for $f=u-u_0$.  But in that case we explicitly compute
  \[
  \frac{\tau(g (u-u_0)^{-1}d(u-u_0))}{d((u-u_0)^p)}
  =
  g^p (u-u_0)^{-p} + g \frac{d^{p-1}(u-u_0)^{-1}}{du^{p-1}}
  =
  (g^p-g)(u-u_0)^{-p}.
  \]

  The topological span of the two $\F_p$-subspaces on which these invariant
  descriptions are defined is $k((u))du$ (indeed, each differential $\alpha
  u^k du$ can be expressed in one of the two forms), and thus the two
  formulas uniquely determine $\tau$, making it invariant.  Moreover, if
  $\omega$ is holomorphic at $x$, then we may use a uniformizer at $x$ to
  see that that $\tau(\omega)$ is also holomorphic, so that it gives a
  sheaf map as required.
\end{proof}

\begin{cor}\label{cor:definition_of_Nmprime}
  There is a multiplicative monoid map $\Nm':k(C)\oplus
  \epsilon\Omega_{k(C)}\to k(C')\oplus \epsilon \Omega_{k(C')}$ given by
  \[
  \Nm'(f+\epsilon \omega)\mapsto f^p + f^p\tau(\omega/f) \epsilon
  \]
  when $f\ne 0$ and with $\Nm'(\epsilon g dh) = g^p d(h^p)\epsilon$.
  Moreover, $\Nm'(f)=\Nm'(f+\epsilon df)=f^p$.
\end{cor}

\begin{proof}
  The definition of $\tau$ via a uniformizer $u$ gives an expression for
  $\Nm'(f+\epsilon g du)$ which is polynomial in the derivatives of $f$ and
  $g$, and agrees with the given description when $f=0$.  Furthermore, the
  values of $\Nm'(f)$ and $\Nm'(f+\epsilon df)$ follow from the definition of
  $\tau$.  To show that $\Nm'$ is multiplicative, it will suffice to show it
  for $f$ nonzero, i.e., that it gives a homomorphism on
  $(\sO_{\overline{Q}}\otimes k(C))^*$.  Since
  \[
  \Nm'(h(f+\epsilon \omega)) = h^p \Nm'(f+\epsilon \omega) =
  \Nm'(h)\Nm'(f+\epsilon\omega),
  \]
  we may further reduce to the case $f=1$, where it follows from the fact
  that $\tau$ is an additive homomorphism.
\end{proof}

By considering how $\tau$ acts when applied to a differential with poles,
we find that the restriction of $\Nm'$ to $\sO_{\whQ}$ takes values in
$\sO_{\whQ'}$, and thus in particular gives a well-defined homomorphism
$H^1(\sO_{\whQ}^*)\to H^1(\sO_{\whQ'}^*)$ between the respective Picard
groups.

\medskip

In each case, we have associated a double cover of a smooth curve and a
line bundle on that double cover, and thus a commutative ruled surface
corresponding to the direct image of that line bundle.  We will see below
that the original quasi-ruled (possibly ruled) surface is a maximal order
in a certain division ring over the function field of this commutative
ruled surface.

\bigskip

As we mentioned above, although quasi-ruled surfaces are difficult to fully
classify, the ruled case is fairly straightforward.  We should first note
that the isomorphism $C_0\cong C_1$ making the sheaf bimodule algebraically
equivalent to the double diagonal is not in general unique: we can compose
it with anything in the identity component of $\Aut(C_1)$ without actually
changing the noncommutative surface (which depends only on the category of
representations and the choice of structure sheaf, i.e., the pullback of
$\sO_{C_0}$).  We can similarly twist by the pullback of any line bundle on
$C_1$ without any effect.  (Twisting by a line bundle on $C_0$ gives a
Morita-equivalent noncommutative surface.)  Thus when classifying ruled
surfaces, we should feel free to identify $C_0\cong C_1$ with $C$, but must
take these symmetries into account.

We assume we are over an algebraically closed field; although the
classification could be done over a general field with only slightly more
work, the result would be misleading, since not every noncommutative
surface over $k$ which is ruled over $\overline{k}$ is ruled over $k$.  (This is
true even in the commutative case: not only must one contend with conic
bundles in general, but the fact that $\P^1\times \P^1$ has two geometric
rulings, neither of which need be defined over the ground field.  These
issues can be dealt with, but the simplest approach is via yet another
construction to be discussed below.)

In the nonreduced case (e.g., if $g(C_0)\ge 2$), we have already seen that
$\whQ$ has the form $\Spec(\sO_{C_0}\oplus \omega_{C_0}(D)\epsilon)$ for some
effective divisor $D$, with $\epsilon^2=0$, and it remains only to
understand the possible twists, i.e., $\Pic(\whQ)/\Pic(C_1)$.  (The
automorphism freedom is fixed by taking $\overline{Q}$ to be the double
diagonal.)  We claim that in fact $\Pic(\whQ)/\Pic(C_1)=1$ unless $D=0$,
when it is isomorphic to $\G_a(k)$.  Restricting to the underlying reduced
curve gives a splitting of $\Pic(C_1)\to \Pic(\whQ)$, so it suffices to
understand the case that ${\cal L}$ is trivial on the reduced curve.  In
that case, the corresponding cocycle takes values in the group $1+\epsilon
\omega_{C_0}(-D)$, or equivalently in the sheaf $\omega_{C_0}(-D)$, and
thus $\Pic(\whQ)/\Pic(C_1)$ is canonically isomorphic to
$H^1(\omega_{C_0}(D))\cong H^0(\sO_C(-D))$.  In other words, if $D>0$,
there is no scope for twisting, while for $D=0$ there is a one-parameter
family of twists (corresponding to the usual notion of twisting for
differential operators).

In the reduced but nonintegral case, we split into three cases.  The first
is that $g(C_0)=g(C_1)=1$.  In that case, we note that any divisor which is
algebraically equivalent to the double diagonal is a union of two
translates of the diagonal, and thus we may use the $\Aut(C_1)^0$ freedom
to make one of the two components be the diagonal.  Then the other
component is disjoint from the diagonal, so that $\whQ=\overline{Q}$ is
smooth.  We can further use the $\Pic(C_1)$ freedom to make the line bundle
trivial on the diagonal.  We thus find that such cases are classified by
$\Pic^0(C)\times \Pic(C)$, with the first coordinate determining the
translation and the second determining the line bundle on the corresponding
component, and the corresponding algebra is a compactification of the
relevant algebra of twisted elliptic difference operators.  (To be precise,
these surfaces are classified by the quotient of $\Pic^0(C)\times \Pic(C)$
by the involution corresponding to swapping the two components.)

In the reduced but nonintegral cases with $g(C_0)=g(C_1)=0$, we similarly
find that the curve is determined by a choice of nontrivial automorphism.
Geometrically, an element of $\PGL_2(k)$ is either diagonalizable,
corresponding to a linear fractional transformation $z\mapsto qz$ (the
multiplicative case), or unipotent, corresponding to $z\mapsto z+\hbar$
(the additive case).  There are a few possible choices for $\whQ$: in
the multiplicative case, there are four, depending on whether $\whQ$ is
singular over $0$ or $\infty$, while in the additive case there are three:
$\whQ=\overline{Q}$, $\whQ=\widetilde{Q}$, and an intermediate case for which
the conductor has degree 1.  There is a nontrivial map
$\Pic(\whQ)/\Pic(C_0)\to \Z$ given by taking the degree on the
nonidentity component, which is an isomorphism unless $\whQ=\overline{Q}$,
when the kernel is $\G_m$ or $\G_a$ as appropriate.

Finally, in the reduced integral case, where necessarily $g(C_0)=g(C_1)=0$,
$\overline{Q}$ is an integral biquadratic curve, which is either smooth, nodal,
or cuspidal.  The smooth case is classified by a smooth genus 1 curve $Q$
along with two classes $\eta_0,\eta_1\in \Pic^2(Q)$ (i.e., the two maps to
$\P^1$) and one in $\Pic(Q)/\langle \eta_1\rangle$ (the line bundle for
twisting), and corresponds to the construction in \cite{generic}.

In the nodal and cuspidal cases, $\widetilde{Q}\cong \P^1$ is (assuming
both maps are separable) equipped with a pair of involutions, and $\whQ\in
\{\widetilde{Q},\overline{Q}\}$.  If the product of involutions has two
fixed points, then the involutions swap the fixed points and can be made to
look like $z\mapsto 1/z$, $z\mapsto q/z$, giving symmetric $q$-difference
operators (with $\overline{Q}$ a nodal curve).  If the product of
involutions has a single fixed point, then the involutions have the form
$z\mapsto a-z$, $z\mapsto a+\hbar-z$, where we may take $a=0$ except in
characteristic 2.  In characteristic 2, we also have the (hybrid)
inseparable case in which one of the maps is $z\mapsto z^2$ and the other
can be taken to be $z\mapsto z^2+z$.  In each of these cases, there are two
possible twists when $\whQ\ne \overline{Q}$ (since $\Pic(\P^1)$ injects as
an index 2 subgroup of $\Pic(\whQ)$), and otherwise the group of twists is
an extension of $\Z/2\Z$ by $\G_m$ or $\G_a$ as appropriate.

For Hirzebruch surfaces (i.e., ruled surfaces over $\P^1$), there is
another approach to the classification which is particularly convenient for
some purposes.  Recall that $Q$ naturally embeds in a ruled surface over
$C_1=\P^1$; furthermore, it intersects the generic fiber transversely with
multiplicity 2, and has arithmetic genus 1, thus must be anticanonical.
There is thus a natural map from the moduli stack of noncommutative
Hirzebruch surfaces to the moduli stack of (commutative) anticanonical
Hirzebruch surfaces.  Given a noncommutative Hirzebruch surface, there is
also an associated line bundle $q$ of degree 0 on $Q$ given by
$\pi_0^*\sO_{\P^1}(-1)\otimes \pi_1^*\sO_{\P^1}(1)$, giving a map from the
moduli stack of noncommutative Hirzebruch surfaces to the relative $\Pic^0$
(that is, the stack classifying invertible sheaves which are degree 0 on
every component of every fiber) of the universal anticanonical curve of the
moduli stack of anticanonical Hirzebruch surfaces.

We claim that this map to the relative $\Pic^0$ is an isomorphism.  Indeed,
we can recover line bundles $\eta_0=\pi_0^*\sO_{\P^1}(1)$ and
$\eta_1=\eta_0\otimes q$ from the given data.  Both bundles are acyclic of
degree 2: their inverses have degree $-2$ and have (the same) nonpositive
degree on every component, so are ineffective.  Thus both bundles induce
(isomorphism classes of) maps from $Q$ to $\P^1$, and since they have
degree 0 on every vertical component, this restricts to an embedding
$\overline{Q}\to \P^1\times \P^1$.  The sheaf bimodule can then be
recovered as the direct image of the relative $\sO(1)$ of the commutative
Hirzebruch surface.  (Since $\whQ$ meets every fiber twice, there is
a natural map $\sO_Q(1)\to \sO_{\whQ}(1)$ with kernel an extension
of sheaves $\sO_f(-1)$ supported on fibers of the ruling, and thus both
sheaves have the same direct image.)

Note here that $\Pic^0(Q)=\Pic^0(\overline{Q})$ is an elliptic curve,
$\G_m$, or $\G_a$, depending on whether $Q$ is smooth, nodal, or
cuspidal/nonreduced.  Moreover, when $\overline{Q}$ is reduced, $q$
corresponds directly to the shift in the difference operator
interpretation, while for $\overline{Q}$ nonreduced, it corresponds to the
traditional $\hbar$ used to view differential operators as deformations of
commutative polynomials.

There is a minor technical issue ignored above: to recover the data from
the surface, we must in general include the embedding of $Q$ (i.e.,
consider the moduli stack of {\em anticanonical} noncommutative Hirzebruch
surfaces), as when $q$ is the identity, the corresponding surface is
commutative.  We also caution the reader that although every
(anticanonical) commutative Hirzebruch surface over the generic point of a
dvr extends (nonuniquely), this fails in the noncommutative case, since
$\Pic^0(Q)$ need not be proper.  (This is related to the fact that although
$\overline{Q}$ has a limit inside $C_0\times C_1$, that limit may contain
preimages of points of $C_0$ or $C_1$, making it impossible for a limiting
sheaf on $\overline{Q}$ to be a sheaf bimodule.)

Something similar holds for other ruled surfaces; in general, the moduli
stack of (anticanonical) noncommutative ruled surfaces over $C$ is the
relative $\Pic^0(Q)/\Pic^0(C)$ over the moduli stack of anticanonical ruled
surfaces over $C$.  (If $C$ is a genus 1 curve of characteristic 2, we must
also exclude those anticanonical ruled surfaces for which the anticanonical
curve is connected and smooth, as these correspond to the $2$-isogeny
case.)  The map from noncommutative surfaces is the same as for genus 0
(where we note that $\Pic^0(C_0)$ and $\Pic^0(C_1)$ have the same embedding
in $\Pic^0(Q)$), but the inverse is trickier to construct.  However, we can
verify in both the differential and elliptic difference cases that the map
is surjective and that we can reconstruct the sheaf bimodule from this
data.

\section{The semicommutative case}
\label{sec:semicomm}

As mentioned above, many examples of noncommutative quasi-ruled surfaces
are not ``truly'' noncommutative, in that the corresponding sheaf
$\Z$-algebra has a large center.  Our objective in the present section is
to make this more precise, by constructing a (coherent) sheaf of algebras
on a corresponding commutative ruled surface $Z$ having the same category
of coherent modules.  In fact, we will see that this $\sO_Z$-algebra is a
{\em maximal order} (see \cite{ReinerI:2003} for an expository treatment).
(Recall that an order on an integral scheme $Z$ is a coherent,
torsion-free, $\sO_Z$-algebra such that the generic fiber is a central
simple algebra, and an order is maximal if it is not a proper subalgebra of
an order with the same generic fiber.)  This not only completes our
justification for frequently restricting our attention to ruled surfaces
(as maximal orders have been extensively studied via commutative means),
but gives useful additional information about semicommutative ruled
surfaces (the differential case in characteristic $p$, as well as
difference cases in which the translation is torsion).

Our first goal is to establish the following result.

\begin{thm}\label{thm:semicomm}
  Suppose that $X$ is the quasi-ruled surface associated to maps
  $\pi_i:\hat{Q}\to C_i$ (with $C_i$ smooth projective curves) and a line
  bundle ${\cal L}$ on $\hat{Q}$.  Suppose that $k(C_0)$ is algebraic of
  degree $r$ over the intersection $k(C_0)\cap k(C_1)$ inside $k(\hat{Q})$,
  and let $C'$ be the curve with function field $k(C_0)\cap k(C_1)$.  Then
  there is a commutative ruled surface $Z$ over $C'$ and a maximal order
  ${\cal A}$ on $Z$ with generic fiber a division ring of index $r$ such
  that there is a natural equivalence $\coh {\cal A}\cong \coh X$ taking
  the regular representation to the structure sheaf.
\end{thm}

\begin{rem}
  By Proposition \ref{prop:qr_is_semicomm}, the hypothesis is automatically
  satisfied (for some $r$) for any quasi-ruled surface which is not ruled.
  We will further be able to identify the rank 2 vector bundle on $Z$,
  although this identification is surprisingly subtle when ${\cal L}$ is
  nontrivial.
\end{rem}

A key observation is that being a maximal order is a local condition (up to
some mild subtleties in how one identifies the center), so to a large
extent we may restrict our attention to the analogous affine question.

Let us first consider the differential case in characteristic $p$, where as
usual we may assume that $R=S_0[\epsilon]/\epsilon^2$, with $S_1\cong S_0$
embedded as $f\mapsto f+(df/\omega)\epsilon$ for some nonvanishing
meromorphic differential $\omega$.  Let $u$ be a uniformizer at some point
of the curve $\Spec(S_0)$, and observe that the differential $du$ is
holomorphic and nonvanishing in a neighborhood of that point.  We may thus
reduce to the case that $du$ is nowhere vanishing, and thus the conductor
is principal, generated by $c:=du/\omega\in S_0$.  We then find that
$\overline{S}_{01}=\overline{S}_{12}=S_0+S_0 c D$, where $D$ satisfies the
commutation relation $D f - f D = df/du$.  In other words, $\overline{S}$ is the
$\Z$-algebra associated to the subalgebra $A_c:=S_0\langle t, cDt\rangle$
of the graded algebra $A=S_0\langle t,Dt\rangle$, where $t$ is central.

\begin{lem}
  If $S_0$ has characteristic $p$, then the $p$-th Veronese of the center
  of $A_c$ is $S_0^p[t^p,c^p D^p t^p]$, and the corresponding sheaf of
  algebras over $\Proj(S_0^p[t^p,c^p D^p t^p])$ is locally free.
\end{lem}

\begin{proof}
  We first need to show that the central element $c^p D^p t^p\in A$ is in
  $A_c$.  But this follows easily from the identity
  \[
  c^a (c D t) = ((c D t) - a c' t) c^a
  \]
  valid over $S_0\langle t,Dt\rangle$, where $c'=dc/du$; in
  particular, we find 
  \[
  c^k D^k t^k =
  c^{k-1} (cDt)D^{k-1}t^{k-1}
  = (c D t-(k-1)c't) c^{k-1}D^{k-1} t^{k-1},
  \]
  so that by induction, $c^kD^kt^k\in A_c$ for all $k$.

  As an operator, $c^pD^pt^p$ annihilates $S_0$, and any homogeneous
  element of degree $d$ in $A_c$ can be expressed as $t^d$ times a
  polynomial of degree $<p$ in $cD$ plus a multiple of $c^p D^p$.  Since
  differential operators of degree $<p$ act faithfully, we conclude that
  $c^p D^p t^p$ generates the kernel of the natural map $A_c\to
  \End_{S_0^p} (S_0)[t]$.  A central element $x\in A_c$ of degree a
  multiple of $p$ maps to a central element of $\End_{S_0^p}(S_0)[t]$, and
  thus has the form $f_0 t^{pd} + (c^p D^p t^p)y$ where $f_0\in S_0^p$ and
  $x\in A_c$.  Since $A_c$ is a domain, $y$ must also be central, and since
  it has lower degree than $x$, the first claim follows by induction.

  For local freeness, it suffices to show that the submodule of $A_c$
  consisting of elements of degree congruent to $-1$ modulo $p$ is free
  over the center of $A_c$.  A degree $dp+(p-1)$ element of $A_c$ has a
  unique expression of the form
  \[
  \sum_{0\le i\le d} x_i t^{p(d-i)}(c^p D^p t^p)^i
  \]
  with each $x_i$ homogeneous of degree $p-1$.  Since the space of degree
  $p-1$ elements is free as a $S_0^p$-module, the claim follows.
\end{proof}

In particular, we have established that $\coh \overline{S}$ is the category of
coherent sheaves on an order ${\cal A}$ over the surface
$Z:=\Proj(S_0^p[t^p,c^p D^p t^p])$, which is a ruled surface over
$\Spec(S_0^p)$.  Note that since $\overline{S}$ is a domain, the generic fiber
of ${\cal A}$ is a division ring.  Over most of $Z$, maximality is
straightforward, and in fact we have the following stronger statement.

\begin{lem}
  On the complement of the divisor $c^p t^p=0$, ${\cal A}$ is an Azumaya
  algebra.
\end{lem}

\begin{proof}
  Without loss of generality, we may assume that $c$ is a unit, and thus
  this reduces to showing that $S_0\langle D\rangle$ is an Azumaya algebra
  over $\Spec(S_0^p[D^p])$.  The fiber over any point (with field $\ell$)
  of $\Spec(S_0^p[D^p])$ has a presentation of the form $\ell\langle
  u,D\rangle$ with relations $u^p=U$, $D^p=\delta$, $Du-uD=1$, with
  $U,\delta\in \ell$ and both $u$ and $D$ commuting with $k$.  Passing to
  the algebraic closure of $\ell$ lets us subtract $p$th roots of $U$ and
  $\delta$ from $D$ and $u$ respectively, and thus reduce to the case
  $U=\delta=0$, in which case the algebra is readily identified with
  $\Mat_p(\overline\ell)$ as required.
\end{proof}

Since ${\cal A}$ is locally free, and thus reflexive, it is maximal iff its
codimension 1 localizations are maximal, so that it remains only to check
maximality along components of $c^p t^p=0$.  Here we may use the criterion
that an order over a dvr is maximal if its radical is principal (as a left
ideal) and its semisimple quotient is central simple.  (If the radical is
principal, it is invertible, and thus the order is hereditary, and the
result follows by \cite[Thm.~2.3]{AuslanderM/GoldmanO:1960}).  There are
two cases to consider: the fibers where $c$ vanishes, and the curve
$t^p=0$.

In the first case, we may assume $S_0$ is a dvr and the conductor has the
form $c=z^c$ for some uniformizer $z$ and some positive integer $c$, and
consider the completion $\hat{S}$ of $S$ along the divisor $z=0$, or
equivalently the base change of $S|_{t=1}$ to the extension
$k(z^{cp}D^p)[[z^p]]$ of its center.  Consider the two-sided ideal $\hat{S}
z \hat{S}$.  Since $c$ is positive, $z$ is normalizing:
\[
z (z^c D) = (z^c D-z^{c-1})z,
\]
and thus $\hat{S}z \hat{S}=z \hat{S}$ is principal as a left ideal.
Moreover, the $p$-th power ideal $(z \hat{S})^p = z^p \hat{S}$ is contained
in the radical, since $z^p$ generates the maximal ideal of the center.  We
conclude that $\hat{S} z\hat{S}$ is contained in the radical of $\hat{S}$.
The quotient by this two-sided ideal is the field $k(z^c D)$, and thus
$\hat{S} z \hat{S}$ actually is the radical, and the quotient is a field,
so is central simple as required.

In the second case, we have a completion of the form $K\langle\langle
D^{-1}\rangle\rangle$ where $K$ is the field of fractions of $S_0$ and
$D^{-1}$ satisfies the commutation relation
\[
D^{-1} f = \sum_{0\le j} (-1)^j \frac{d^j f}{du^j} D^{-1-j}
         = \sum_{0\le j<p} (-1)^j \frac{d^j f}{du^j} D^{-1-j}.
\]
Here we find that the radical is the principal one-sided ideal generated by
$D^{-1}$, and the quotient is again a field.

The above calculation carries over to the untwisted case, showing that
$\coh {\cal S}$ is a maximal order on a ruled surface over the image
$C^{(p)}$ of $C_0=C_1=C$ under Frobenius.  In fact, our computation allows
us to identify the corresponding rank 2 vector bundle as
$\sO_{C^{(p)}}\oplus \omega_{C^{(p)}}(c')$ where $c'$ is the norm of the
conductor, corresponding to a curve $\hat{Q}'$ embedded as the appropriate
double section.  In this case, twisting is relatively easy to deal with,
since the center is preserved by the automorphisms used under gluing, and
we can compute the effect of nontrivial twists using Proposition
\ref{prop:twist_differential}.  (This will be trickier to deal with in
general, when we need a more subtle notion of center.)

Note that when the conductor is nontrivial, we should really consider the
curve $Q$ rather than $\hat{Q}$, which leads us to consider the curve $Q'$
obtained by adding fibers with the appropriate multiplicities to
$\hat{Q}'$.  We can then verify that $Q'$ is an anticanonical curve on the
center.

\medskip

Another useful observation for the non-differential cases is that since
$\overline{H}$ and $\overline{S}$ are Morita equivalent, they have the same centers
(appropriately defined), and one is maximal iff the other is maximal.  This
means that we can feel free to work with whichever of the two is most
convenient at any given point in the argument.  In particular, in the
difference case, it will be more convenient to work with $\overline{H}$ for the
bulk of the argument.  As in the differential case, a key step is
identifying the kernel of the action on $R$.

We first suppose that $S_0$ and $S_1$ are complete dvrs (and we are in the
difference case, so $R$ is separable over both $S_0$ and $S_1$).  Then $R$
is an order in either a dvr or a sum of two copies of $S_0$.  In the first
case, $R$ inherits a valuation from $S_0$, and we take $\xi$ to be an
element of smallest absolute value that generates $R$ over $S_0=k[[x_0]]$.
Valuation considerations then imply that $\xi$ also generates $R$ over
$S_1=k[[x_1]]$.  In the disconnected case, we similarly suppose that $\xi$
vanishes in one summand, and among such elements has maximal absolute value
in the other summand; this again forces $\xi$ to generate over both $S_0$
and $S_1$.  In either case, we use this element to define the requisite
twisted derivations: for $f,g\in S_0$, $\nu_0(f+g\xi)=g$, and similarly for
$\nu_1$.

In the differential case, the fact that the conductor is normalizing was
crucial.  In the differential case, it turns out that we get a nontrivial
normalizer not only when the conductor is nontrivial but also when there is
suitable ramification.  Associate a positive integer $m$ to the
configuration $(R,S_0,S_1)$ as follows.  If $R\ne S_0S_1$ (which in
particular happens if $S_0=S_1$), let $m=1$; otherwise, if $r>1$ is the
order of $s_0s_1$, we take $m$ to be the $p'$-part of $r$, unless $r$ is a
power of $p$, in which case $m=p$.

\begin{lem}\label{lem:normalizing_element_of_order_m}
  The element $(x_1/x_0)^m\in R^*$.
\end{lem}

\begin{proof}
  It will suffice to show that $\nu_0(x_1^m)\in x_0^m S_0$, since then
  valuation considerations force $x_1^m-\nu_0(x_1^m)\epsilon\in x_0^m S_0$
  as well, so that $(x_1/x_0)^m\in R$, with the unit property following by
  symmetry.  If $R\ne S_0S_1$, so $m=1$, then $x_1$ cannot generate $R$
  over $S_0$, and thus neither can $\nu_0(x_1)\epsilon =
  x_1-(x_1-\nu_0(x_1)\epsilon)\in x_1+S_0$.  But this implies that
  $\nu_0(x_1)\in S_0$ is not a unit, and thus $\nu_0(x_1)\in x_0 S_0$ as
  required.

  In the remaining cases, since $\nu_0$ is a twisted derivation, we may
  write
  \[
  \nu_0(x_1^m) = \nu_0(x_1)\sum_{0\le i<m} x_1^i s_0(x_1)^{m-1-i}.
  \]
  Since $s_0$ preserves the valuation(s) of $x_1$ (it either preserves the
  unique valuation of the normalization of $R$ or swaps the two valuations;
  since $x_1$ is $s_1$-invariant and $s_1$ has the same effect on the
  valuations, its valuation is $s_0$-invariant), we find that
  $s_0(x_1)/x_1\in \widetilde{R}$, and thus we may compute
  \[
  \nu_0(x_1^m) = \nu_0(x_1) x_1^{m-1} \sum_{0\le i<m}
  (s_0(x_1)/x_1)^{m-1-i}.
  \]
  In particular, if
  \[
  \sum_{0\le i<m} (s_0(x_1)/x_1)^{m-1-i}
  \]
  is in the radical of $\widetilde{R}$, then $x_1^{1-m}\nu_0(x_1^m)$ has
  positive valuation, so that $x_0^{1-m}\nu_0(x_1^m)$ has positive
  valuation as required.

  If $\Spec(R)$ is disconnected, then $x_1$ is a uniformizer in each direct
  summand, and thus modulo the radical, $s_0(x_1)/x_1=s_0(s_1(x_1))/x_1$ is
  given by the action of $s_0s_1$ on the two tangent vectors.  Since
  $S_0\ne S_1$, we have $s_0s_1\ne 1$; if it has order a power of $p$, then
  it fixes the tangent vectors, and otherwise it acts as inverse $m$-th
  roots of unity, and in either case the sum vanishes.

  If $\Spec(R)$ is connected and $\pi$ is a uniformizer of $\widetilde{R}$,
  then $s_0s_1(\pi)/\pi$ is congruent to $1$ since both $s_0$ and $s_1$ act
  as $-1$ on the tangent vector.  It follows that the order of $s_0s_1$ is
  a power of $p$, and again the claim follows.
\end{proof}

\begin{cor}
  The elements $x_0^m$ and $x_1^m$ normalize $H$.
\end{cor}

\begin{proof}
  Since $x_1^m$ commutes with $N_1$, it suffices to show that
  $x_1^m N_0\in N_0 H$.  But this follows by writing
  \[
  x_1^m N_0 = (x_1/x_0)^m x_0^m N_0 = (x_1/x_0)^m N_0 x_0^m = (x_1/x_0)^m
  N_0 (x_0/x_1)^m x_1^m.
  \]
\end{proof}

Define a sequence of elements $x_0^{(l)}$, $1\le l$, by $x_0^{(1)}:=x_0$,
$x_0^{(l)}:=s_{(l-1)\bmod 2}(x_0^{(l-1)})$, and similarly for $\xi^{(l)}$.
Note that all of these elements have the same valuation(s) as $x_0$.

\begin{cor}\label{cor:nux0l_normalizes}
  For any integer $1\le l$, we have
  \[
  \nu_{l\bmod 2}(x_0^{(l)}) H = H \nu_{l\bmod 2}(x_0^{(l)}).
  \]
\end{cor}

\begin{proof}
  It will suffice to show that $\nu_{l\bmod 2}(x_0^{(l)})$, when nonzero,
  has valuation(s) a multiple of that of $x_{l\bmod 2}^m$.  This is
  immediate if $m=1$, so we may assume we have trivial conductor; i.e.,
  $R=S_0S_1$.  Since we are in the difference case, we can write
  \[
  \nu_{l\bmod 2}(x_0^{(l)}) = \frac{x_0^{(l)}-s_{l\bmod
      2}(x_0^{(l)})}{\xi-s_{l\bmod 2}(\xi)}
  \]
  Now, $x_0^{(l)}$ is invariant under the involution $\cdots
  s_1s_0s_1\cdots $ of length $2l-1$, and thus we can rewrite the second
  term as $(s_0s_1)^{\pm l}(x_0^{(l)})$ where the sign depends on the
  parity of $l$.  That the ratio has valuation as required then follows
  from the Hasse-Arf theorem.
\end{proof}

To understand the kernel of the action, it will be helpful to consider a
larger algebra of ($k$-linear) operators.  For each integer $l\ge 0$, the
subspace of $H$ with length bounded by $l$ has rank $2l+1$, and thus we
would naively expect that there should be a unique (up to scale) such
operator that annihilates the $2l$ elements $1,\dots,x_0^{l-1}$ and
$\xi,\dots,x_0^{l-1}\xi$.  The nonuniqueness could then be eliminated (at
the cost of making the section of $H$ meromorphic) by insisting that the
image of $x_0^l$ be 1.  This of course fails once $l$ is large enough, as
the image of $H$ in $\End_k(R)$ has finite rank, but either via
experimentation for small $l$ or by computing appropriate determinants, one
finds that the resulting action can be computed explicitly, and extends to
arbitrary $l$.  Thus for each integer $l\ge 0$, let $L_l$ denote the unique
continuous $k$-linear operator on $R$ such that
\begin{align}
  L_l \frac{1}{1-tx_0} &= \prod_{1\le i\le l+1} \frac{t}{1-t x_0^{(i)}}\\
  L_l \frac{\xi}{1-tx_0} &= \xi^{(l)}\prod_{1\le i\le l+1} \frac{t}{1-t x_0^{(i)}}.
\end{align}
This should be interpreted as saying that if we expand both sides as formal
power series in $t$, then each coefficient transforms as stated.  This is
clearly well-defined, since it specifies $L_l$ on a (topological) basis of
$R$.  Moreover, it follows by triangularity that the operators
$L_0,N_0L_0,L_1,N_0L_1,\dots$ form a topological basis of the algebra of
continuous $k$-linear endomorphisms of $R$.

By our original motivation, we expect the operators $L_l$ to be
representable as meromorphic sections of $H$, at least for $l$ sufficiently
small.  Such an expression is an immediate consequence of the following
recurrence.

\begin{lem}
  These operators satisfy the recurrence
  \[
  \nu_l(x_0^{(l)}) L_l
  =
  (N_l-\frac{\nu_l(\xi^{(l)})}{\nu_{l-1}(\xi^{(l-1)})}N_{l-1})
  L_{l-1}.
  \]
\end{lem}

\begin{proof}
  We need to show that applying the right-hand side to $(1-t x_0)^{-1}$ and
  $\xi(1-t x_0)^{-1}$ gives the correct result.  The key observation is
  that
  \[
  \prod_{1\le i\le l} \frac{t}{1-t x_0^{(i)}}
  \]
  is $s_{l-1}$-invariant and thus annihilated by $\nu_{l-1}$, and similarly
  \[
  \prod_{1\le i\le l-1} \frac{t}{1-t x_0^{(i)}}
  \]
  is $s_l$-invariant.  Since $\nu_l$ is an $s_l$-twisted derivation, we
  compute
  \begin{align}
  \nu_l\biggl(\prod_{1\le i\le l} \frac{t}{1-t x_0^{(i)}}\biggr)
  &=
  \nu_l\biggl(\frac{t}{1-t x_0^{(l)}}\biggr)
  \prod_{1\le i\le l-1} \frac{t}{1-t x_0^{(i)}}\notag\\
  &=
  \frac{1}{\xi-s_l(\xi)}
  \biggl(\frac{t}{1-t x_0^{(l)}}-\frac{t}{1-t x_0^{(l+1)}}\biggr)
  \prod_{1\le i\le l-1} \frac{t}{1-t x_0^{(i)}}
  \notag\\
  &=
  \nu_l(x_0^{(l)})
  \prod_{1\le i\le l+1} \frac{t}{1-t x_0^{(i)}},
  \end{align}
  so that both sides indeed have the same action on $(1-tx_0)^{-1}$.
  Since
  \[
  (N_l-\frac{\nu_l(\xi^{(l)})}{\nu_{l-1}(\xi^{(l-1)})}N_{l-1})\xi^{(l-1)} =
  \xi^{(l)}N_l -
  s_{l-1}(\xi^{(l-1)})\frac{\nu_l(\xi^{(l)})}{\nu_{l-1}(\xi^{(l-1)})}N_{l-1},
  \]
  the actions on $\xi (1-tx_0)^{-1}$ also agree.
\end{proof}

\begin{rem}
  Since $\xi^{(l)}$ satisfies the same valuation conditions as $\xi$ (apart
  from possibly swapping the two components when $R$ is not a domain), we
  find that $\nu_l(\xi^{(l)})$ is a unit.
\end{rem}

Although this only gives a meromorphic expression for $L_l$, we can control
the poles.

\begin{cor}
  For any integer $l\ge 0$, there is an element of
  \[
  H_{\le s_{l-1}\cdots s_0}+H_{\le s_l\cdots s_1}
  \]
  with unit leading coefficients that acts as the operator
  \[
  \Bigl(\prod_{1\le i\le l} \nu_i(x_0^{(i)})\Bigr) L_l.
  \]
\end{cor}

\begin{proof}
  We have the operator identity
  \[
  \Bigl(\prod_{1\le i\le l} \nu_i(x_0^{(i)})\Bigr) L_l
  =
  \Bigl(\prod_{1\le i\le l-1} \nu_i(x_0^{(i)})\Bigr)
  (N_l-\frac{\nu_l(\xi^{(l)})}{\nu_{l-1}(\xi^{(l-1)})}N_{l-1})
  L_{l-1}.
  \]
  The coefficient normalizes $H$ by Corollary \ref{cor:nux0l_normalizes},
  and thus we can move it to the right to obtain an expression of the form
  \[
  \Bigl(\prod_{1\le i\le l} \nu_i(x_0^{(i)})\Bigr) L_l
  =
  (\alpha_l N_l + \beta_l N_{l-1} + \gamma_l)
  \Bigl(\prod_{1\le i\le l-1} \nu_i(x_0^{(i)})\Bigr) L_{l-1},
  \]
  with $\alpha_l,\beta_l\in R^*$, $\gamma_l\in R$.  The claim follows
  by induction.
\end{proof}
  
\begin{rem}
  This expression is uniquely determined unless $\nu_i(x_0^{(i)})=0$ for
  some $1\le i<l$.
\end{rem}  

\begin{cor}
  Suppose that $S_0$ has degree $r<\infty$ over $S_0\cap S_1$.  Then
  the image of $H$ in $\End_{S_0\cap S_1}(R)$ is the same as the images of
  the Bruhat intervals corresponding to the two elements of length $r$,
  both of which inject in $\End_{S_0\cap S_1}(R)$.
\end{cor}

\begin{proof}
  Since $r$ is the order of $s_0s_1$, we find that $\nu_i(x_0^{(i)})=0$ iff
  $i$ is a multiple of $r$, and thus there is an element in the union of
  the Bruhat intervals with unit leading coefficients that acts as
  \[
  \Bigl(\prod_{1\le i\le r} \nu_i(x_0^{(i)})\Bigr) L_r = 0.
  \]
  We can use the resulting element of the kernel to reduce any
  word of length $>r$ to a linear combination of shorter words with the
  same image, and thus the Bruhat intervals of length $r$ both span the
  image of $H$.

  For injectivity, we observe that the elements
  \[
  (\prod_{1\le i\le l} \nu_i(x_0^{(i)})) L_l,
  (\prod_{1\le i\le l} \nu_i(x_0^{(i)})) (N_{l+1}-\nu_{l+1}(\xi^{(l+1)})) L_l
  \]
  for $0\le l<r$ form a basis of the relevant Bruhat interval, so it
  suffices to show that their images are linearly independent.  But this
  follows by considering how they act on $1,\xi,x_0,x_0\xi,\dots$.
\end{proof}

\begin{rem}
Although we did the calculations assuming $S_0$ and $S_1$ are complete
local rings, the corresponding result before completion follows
immediately.
\end{rem}

Now, let us consider the nonintegral case; that is, $S_0$ and $S_1$ are
regular with $S_0\cap S_1$ a dvr.  Assuming again that $S_0\cap S_1$ is
complete, the normalization of $R$ contains a number of idempotents, all of
which must in fact be contained in $R$; indeed, we can represent any
idempotent of the normalization as a product of $s_0$- and $s_1$-invariant
idempotents.  Since $S_0\cap S_1$ is local, the dihedral group acts
transitively on the idempotents of $R$, and if we multiply an idempotent
which is not $s_i$-invariant by general element of $S_i$, we obtain a
general element of that summand of the normalization.  In other words, if
$R$ is not integral, then it is normal.  Furthermore, the automorphism
$s_0s_1$ permutes the idempotents; if it acts as an $r'$-cycle, then the
action of $D_{2r}$ on $R$ can be obtained by induction from an action of
$D_{2r/r'}$ on either a complete dvr or a sum of two complete dvrs swapped
by the reflections.  In particular, the nonintegral case is always obtained
by such an induction from a normal instance $R_0$ of the integral case.

We can use this structure to construct a generator of the kernel in
general.  If $R$ has $2r'$ idempotents (so $R_0$ has two summands), then
$H$ is equal to the twisted group algebra (the coefficient of $N_i$ in
$s_i$ is a unit), and thus the kernel is generated by the difference of the
reduced words of length $r$.  Thus suppose that $R$ has $r'$ idempotents.
Then precisely two of those idempotents are invariant under some reflection
(which will be the same reflection if $r'$ is even, and different
otherwise), giving rise to a pair of operators $\iota_i N_{j(i)} \iota_i$
where $\iota_1$, $\iota_2$ are the two idempotents and $j(i)$ is the index
of the reflection fixing $\iota_i$.  Note that each operator acts on the
corresponding summand as an operator of the form $\nu$.  Using this, we
obtain a collection of $2r'$ operators from the two reduced words of length
$r'$ by replacing single reflections by the corresponding operators
$\iota_i N_{j(i)}\iota_i$.  (We can either replace any reflection in this
way, or can only replace half the reflections, but can do so in two
different ways.)  Each resulting operator annihilates all but one summand
of $R$ and maps into the image of that summand under either word of length
$r'$.  We have two operators for each such pair of summands, which
essentially act as the $N_0$ and $N_1$ on the corresponding copy of $R_0$.
We can thus construct a sum of words of length $\le r/r'$ in such elements
with unit (in $R_0$) leading coefficients acting trivially, giving rise to
a corresponding sum of words of length $\le r$ in $H$ inside the kernel.
The resulting expressions no longer have unit leading coefficients, but
their leading coefficients are units inside the appropriate summands, and
thus adding up $r'$ such expressions gives the desired generator of the
kernel.  (Uniqueness in either case follows by base changing to the field
of fractions of $S_0\cap S_1$.)

Thus in any difference case, we have constructed a generator of the kernel
of the action of $H$, which we denote by ${\bf k}$.  As an element of the
twisted group algebra, we have an expression
\[
{\bf k} = f ((\cdots s_1s_0)-(\cdots s_0s_1))
\]
for some element $f$ of the base change to the field of fractions.  Unlike
in the differential case, this element does not even commute with
multiplication by $R$.  This at least partly arises from the fact that
``center'' is a somewhat ill-defined concept in the case of a $\Z$-algebra
like $\overline{H}$.  The point is that in order to interpret a $\Z$-algebra as
coming from a graded algebra, we need to choose an isomorphism between the
algebra and a suitable shift, and such an isomorphism need not be unique!
In our case, although ${\bf k}$ does not {\em commute} with $R$, it does
act as an automorphism of $R$; indeed, if $\omega$ denotes the longest
element of $D_{2r}$, then we find that ${\bf k}x = \omega(x) {\bf k}$ for
all $x\in R$.  We also find
\[
\omega(\nu_i(\omega(f)))
=
\nu_{i+r}(\omega(\xi))^{-1} \nu_{i+r}(f)
\]
and thus this automorphism of $R$ induces an (order 2) automorphism of $H$
by
\[
\omega(N_i) = \nu_{i+r}(\omega(\xi))^{-1} N_{i+r}.
\]
We may thus define the ``quasi-center'' of $\overline{H}$ to consist of
those morphisms $\phi$ of degree $dr$ such that $\phi \psi = (\omega^d\psi)
\phi$.  This agrees with the naive notion of center in degrees a multiple
of $2r$, but is better-behaved for odd multiples of $r$, and globalizes
just as easily, at least in the untwisted case.  It is then straightforward
to determine which elements of the rank 2 module $R(\cdots s_1s_0)+R {\bf
  k}$ are quasi-central.

\begin{lem}
  An element $f(\cdots s_1s_0)+g(\cdots s_0s_1)$ of degree $r$
  is quasi-central iff $f=s_0(g)=s_0s_1(f)$.
\end{lem}

Any quasi-central element of degree $r$ must have this form, and thus to
determine the degree $r$ elements of the quasi-center, it suffices to
understand the corresponding conditions on $f$ and $g$.  It follows from
our above calculations that there is a fractional ideal $I$ over $R$ such
that $f(\cdots s_1s_0)+g(\cdots s_0s_1)$ is in $H$ iff $f,g\in I$ and
$f+g\in R$, and thus we need to understand the elements $f\in R^{\langle
  s_0s_1\rangle}$ such that $f\in I$ and $f+s_0(f)\in R$.

The fractional ideal $I$ can be expressed as a norm (of the inverse of the
conductor of $R$ inside its normalization) times a product of inverse
differents (of the normalization of $R$ relative to the $r$ subrings fixed
by involutions).  The product of inverse differents is itself an inverse
different, due to the following fact.

\begin{prop}
  Let $K$ be the field of fractions of a Dedekind ring, and let $L/K$ be a
  Galois \'etale $K$-algebra with Galois group $D_{2r}$, let
  $F_1,\dots,F_r$ be the subalgebras fixed by the $r$ reflections, and let
  $E$ be the subalgebra fixed by $C_r$.  Then the different of $E/K$ can be
  computed in terms of the differents of $L/F_i$:
  \[
  d_{E/K} = \prod_{1\le i\le r} d_{L/F_i}.
  \]
\end{prop}

\begin{proof}
  This is equivalent to the corresponding claim for discriminants, which in
  turn follows from the conductor-discriminant formula and basic properties
  of the Artin conductor.
\end{proof}

\begin{cor}
  The quasi-central elements of degree $r$ in $H$ are naturally identified
  with the elements of the inverse different of a quadratic extension of
  $S_0\cap S_1$.
\end{cor}

\begin{proof}
  Letting $K,L,\dots$ be the \'etale algebras corresponding to the action
  of $D_{2r}$ on $R$, we find that the coefficient ideal $I$ is the product
  of the inverse different $d_{E/K}^{-1}$ by the inverse of the norm $C'$ of
  the conductor of $R$ inside $O_K$.  We thus find that the quasi-central
  elements may be identified with the space
  \[
  \{x:x\in C^{\prime{-}1} d_{E/K}^{-1}|x+s_0(x)\in O_E\}.
  \]
  But this is the inverse different of the suborder of conductor $C'$ in
  $O_E$, which can be expressed as
  \[
  \{x:x\in O_E|x-s_0(x)\in C' d_{E/K}\}.
  \]
\end{proof}

\begin{cor}
  The rank 2 free module $R(\cdots s_0s_1)+R{\bf k}$ has an $R$-basis consisting
  of quasi-central elements.
\end{cor}

\begin{proof}
  Let $R'$ be the above quadratic order over $O_K$, so that the
  quasi-center may be described as the space of elements
  \[
  s_0(f)(\cdots s_0s_1)+f(\cdots s_1s_0)
  \]
  with $f$ in the inverse different of $R'$.  If $\xi'$ generates $R'$ over
  $O_K$, then this has an $O_K$-basis of the form
  \[
  s_0(f_0)(\cdots s_0s_1)+f_0(\cdots s_1s_0),
  s_0(\xi' f_0)(\cdots s_0s_1)+\xi' f_0(\cdots s_1s_0),
  \]
  where $f_0$ generates the inverse different.  It suffices to show that
  the corresponding $R'$-module contains $(\cdots s_0s_1)$, since then it
  also contains
  \[
  f_0((\cdots s_0s_1)-(\cdots s_1s_0))\in R^*{\bf k}.
  \]
  The $R'$-module clearly contains
  \[
  (s_0(\xi'f_0)-\xi's_0(f_0)) (\cdots s_0s_1);
  \]
  since
  \[
  s_0(\xi'f_0)-\xi's_0(f_0)
  =
  s_0((\xi'-s_0(\xi'))f_0)
  \]
  and $(\xi'-s_0(\xi'))f_0$ is a unit, the result follows.
\end{proof}

\begin{cor}
  The quasi-center of $\overline{H}$ is generated in degree $r$, and
  $\overline{H}$ is locally free over its quasi-center.
\end{cor}

\begin{proof}
  The space of degree $r$ quasi-central elements has an $S_0\cap S_1$-basis
  $Z_0$, $Z_1$ such that $Z_1$ is a unit multiple of ${\bf k}$ and $Z_0$
  acts (as an operator) on $R$ as $\omega$.  (Indeed, we may take $Z_1 = f
  ((\cdots s_0s_1)-(\cdots s_1s_0))$ where $f$ is an anti-invariant element
  of $I$, and then the other basis element in any extension to a basis will
  act as a unit multiple of $\omega$).  If $\phi$ is quasi-central of
  degree $dr$, then $\phi$ acts on $R$ as an $S_0\cap S_1$-multiple of
  $\omega^d$, and thus has the same operator action as some $S_0\cap
  S_1$-multiple of $Z_0^d$.  Subtracting this multiple to make the operator
  action trivial makes $\phi$ a left multiple of $Z_1$, and right dividing
  by $Z_1$ gives a quasi-central element of degree $(d-1)r$, so that the
  quasi-center of $\overline{H}$ is generated by $Z_0$, $Z_1$.

  Similarly, if $\phi$ is a morphism of degree $dr+(r-1)$, then there is a
  unique element $\mu_0$ of degree $r-1$ acting on $R$ as $\omega^d \phi$,
  and subtracting $Z_0^d$ times this element gives $Z_1$ times a morphism
  of degree $(d-1)r+(r-1)$, inducing a unique expression of the form
  \[
  \phi = \sum_{0\le i\le d} Z_1^i Z_0^{d-i} \mu_i.
  \]
  (Here we have used $\omega$ to identify the different $\Hom$ spaces of
  degree $r$ so that it makes sense to take a product of elements $Z_i$,
  and such products commute as expected.)  Local freeness follows
  immediately.
\end{proof}

In particular, ${\overline H}$ indeed corresponds to a (locally free) order
over the appropriate ruled surface $Z$, and to show maximality it remains
only to check maximality in codimension 1.  This is again an Azumaya
algebra over a large open subset of $Z$; indeed, if $R$ is regular and
unramified over $S_0\cap S_1$, then the fiber over any point not ``at
infinity'' (i.e., not on the image of $\hat{Q}$) is a ``dihedral'' algebra,
i.e., the analogue of a cyclic algebra in which the cyclic group is
replaced by a dihedral group.  The same argument as in the cyclic case
tells us that the fiber is central simple of rank $4r^2$.  (This is $4$
times what we had in the differential case since we are working with
${\overline H}$ rather than ${\overline S}$.)

It remains to consider fibers over ramified points and components of
$\hat{Q}$.  The first case reduces to the case that $S_0\cap S_1$ is a
complete dvr.  If both $S_0$ and $S_1$ are integral with generators $x_0$,
$x_1$, then Lemma \ref{lem:normalizing_element_of_order_m} above gives us a
minimal positive integer $m$ such that $(x_1/x_0)^m\in R^*$, and we find as
in the differential case that $H x_0^m H = x_0^m H$ is contained in the
radical of the localization of $H$.  The quotient by this ideal is central
simple (a dihedral algebra of rank $4m^2$), and thus we have maximality as
required.  (Note that if $m=r$, then we still obtain an Azumaya algebra
away from the intersection of $\hat{Q}$ with the given fiber.)  If at least
one of $S_0$ or $S_1$ fails to be integral, then $R$ must be integrally
closed, and one finds that the radical of the localization of $H$ is
generated by the (principal) radical of $S_0$, with quotient central simple
of degree $4c^2$ where $c$ is the number of irreducible components of
$\Spec(R)$.  (So again, if $c=r$, we obtain an Azumaya algebra away from
$\hat{Q}$.)

For the completion along a component of $\hat{Q}$, we may first base change
to the function field $K$ of the base curve.  Then $S_0$ and $S_1$ are
fields, and $R$ is either a field or a sum of two fields, and in each case
the completion along $\hat{Q}$ is the completion along the two-sided
homogeneous ideal $\overline{H}t^{2r}\overline{H} = t^{2r}\overline{H}$ (which may not be
prime).  (The factor of 2 comes from the fact that $t^r$ is not
quasi-central in general!)  We find as before that $\overline{H} t^2 \overline{H} =
t^2\overline{H}$ is contained in the radical of this completion.  The $2r$-th
Veronese of the quotient is precisely $R t^{2r}$, so that the quotient is a
field whenever $R$ is a field, giving maximality in that case.  When $R$ is
a sum of two fields, we note that it suffices to prove the corresponding
maximality for the spherical algebra, which is a complete Ore ring of the
form $S_0\langle\langle \rho\rangle\rangle$ where $\rho$ is an order $r$
automorphism depending on the choice of component of $\hat{Q}$.  The
radical is then the principal ideal generated by $\rho$, and the quotient
is the field $S_0$ as required.

We thus conclude that for a quasi-ruled surface of difference type such
that $s_0s_1$ has order $r$, $\overline{H}$ is (in large degree) isomorphic to
the algebra of sections of a maximal order of rank $4r^2$ over the
appropriate commutative ruled surface, and this continues to hold globally
for an {\em untwisted} surface.  Twisting is somewhat tricky in this
instance, as the quasi-center is not actually invariant under twisting.
The existence of {\em some} global surface is easy enough, as we can simply
take the second Veronese of the quasi-center (i.e., the elements of degree
a multiple of $2r$), as then the notion of quasi-center agrees with the
naive notion of center and is thus invariant under twisting.  This
immediately tells us that there is at least a global {\em conic} bundle
over which the noncommutative surface is a maximal order.  We could in
principle pin this down by carefully looking at how the twisting occurs,
but the precise identification will in any event follow once we have a
sufficient understanding of noncommutative elementary transformations.
Indeed, we will see below that any twisted quasi-ruled surface of
difference type can be related to an untwisted such surface by a sequence
of (noncommutative) elementary transformations in smooth points of
$\hat{Q}$, and thus (by Theorem \ref{thm:blowup_of_order}) its center is
related to the untwisted center by a corresponding sequence of commutative
elementary transformations.

\medskip

In the remaining (hybrid) case, in which one of the two maps (say
$\Spec(R)\to \Spec(S_0)$) is inseparable, we have a similar result, in
which $\overline{S}$ corresponds to a quaternion algebra over an appropriate
ruled surface.  Note that the argument from the difference case still gives
an element ${\bf k}$ generating the kernel, so the only missing step is
showing that the span of $s_1$ and ${\bf k}$ is generated by quasi-central
elements.  This can in principle be done by a direct computation, but we
also note that there is an argument via semicontinuity.  This is slightly
tricky, as we do not actually have suitable semicontinuity locally, but we
can rescue this by embedding the desired local situation in a suitable
global configuration.  It is easy to see that any complete local hybrid
case appears inside a global case in which $C_0$ (and thus $C'$) is
isomorphic to $\P^1$, so that $\hat{Q}$ (and thus $C_1$) is hyperelliptic.
We can embed such cases in a larger family by allowing the degree 2 map
$C_0\to \P^1$ to vary (constructing $\overline{Q}$ via the compositum of field
extensions).  In the untwisted case with $\hat{Q}=\overline{Q}$, the global
version of the span of $s_0s_1$ and ${\bf k}$ is flat over this family,
while the $\sO_{\hat{Q}}$-submodule generated by quasi-central elements is
upper semicontinuous (as a $\sO_{C'}$-module) and generically agrees with
the larger module, so always agrees with the larger module.  It thus
follows that any {\em local} configuration with trivial conductor has the
desired basis of quasi-central elements.  To deal with conductors, we
observe that the general configuration ($C_0\cong \P^1$, $C_1$
hyperelliptic, $s_0s_1=s_1s_0$) makes sense in characteristic 0, and there
are characteristic 0 examples with nontrivial conductor.  A similar
semicontinuity argument gives some characteristic 2 examples with
nontrivial conductor, and we can then get more general characteristic 2
examples with nontrivial conductor by suitable limits (e.g., taking a limit
in which several simple cusps coalesce).

\medskip

One reason we did not spend too much effort determining the explicit center
in the twisted case is that there is an alternate approach to computing the
center once we know that we actually {\em have} a maximal order.  The key
idea is that, most easily on the Azumaya locus, we can identify the center
as a moduli space of $0$-dimensional sheaves of degree $r$ on the
noncommutative surface.  Indeed, any point on the center pulls back to a
degree $r^2$ representation of the maximal order corresponding to
$\overline{S}$, which if the point is in the Azumaya locus is the $r$-th power
of an irreducible representation of degree $r$.  Thus in particular the
fiber of the Azumaya locus over a point of $C'$ will correspond to such a
sheaf which is supported over that point.  Using the semiorthogonal
decomposition (Theorem \ref{thm:semiorth_for_quasiruled} below), such a
sheaf determines and is determined by a morphism of the form
\[
\pi_1^* Z_1\to \pi_0^* Z_0\otimes {\cal L}
\]
where $Z_0$, $Z_1$ are the respective pullbacks of the point sheaf on $C'$.
Thus the space of such morphisms may be naturally identified with the fiber
of $\pi_*{\cal L}$ over the given point of $C'$.  Two such morphisms
determine the same object of the derived category iff they are in the same
orbit under the actions of $\Aut(Z_0)$ and $\Aut(Z_1)$, which act as the
corresponding fibers of $\sO_{C_0}^*$ and $\sO_{C_1}^*$ respectively.

There are some technical issues here in general, since not every morphism
will correspond to an injective morphism of sheaves on the noncommutative
surface, and one must also consider the sheaves only up to $S$-equivalence
to obtain a well-behaved moduli space.  The first issue can be resolved by
considering only those morphisms which are isomorphisms of sheaves on
$\hat{Q}$; this gives an open subset which is still large enough to
determine the surface.  The second issue is more subtle, in general, though
only an issue away from the Azumaya locus.

This is, however, enough to let us understand twisting in the difference
setting.  Consider the locus of $C'$ over which $\hat{Q}$ is unramified and
equal to $\overline{Q}$.  Over this locus, the fibers of $\sO_{\hat{Q}}^*$ (over
which $\pi_*{\cal L}^*$ is a torsor) are algebraic tori, and it is
straightforward to see that two points are in the same coset of
$\sO_{C_0}^*\sO_{C_1}^*$ iff their norms down to $\sO_{\hat{Q}'}^*$ are in
the same coset over $\sO_{C'}^*$.  We thus find more generally that this
component of the moduli space of $0$-dimensional sheaves compactifies
(fiberwise) to a $\P^1$-bundle (which will, in fact, be the GIT quotient):
two nonzero elements of an unramified fiber of $\pi_*{\cal L}$ correspond
to the same point of this $\P^1$-bundle iff their norms in the
corresponding fiber of
$N_{\hat{Q}^{\text{unr}}/\hat{Q}^{\prime\text{unr}}}({\cal L})$ are in the
same $\sO_{C'}^*$-orbit.

This description makes sense even on the ramified fibers, with one
important caveat: we do not know in general that the norm on $k(\hat{Q})$
takes $\sO_{\hat{Q}}$ to $\sO_{\hat{Q}'}$.  This holds trivially whenever
$\hat{Q}$ is smooth, and is easy to verify at simple nodes of $\hat{Q}$,
and thus in particular holds as long as $k(\hat{Q})/k(\hat{Q}')$ is tamely
ramified wherever it is totally ramified.  As long as it holds, we can
define norms of line bundles on $\hat{Q}$ by taking norms of their
$1$-cocycles, and find that the norm indeed gives a polynomial function of
degree $r$ from ${\cal L}$ to its norm, compatible with the norms from
$\sO_{C_0}^*$ and $\sO_{C_1}^*$ to $\sO_{C'}^*$.  In other words, the norm
induces a morphism from the family of quotient stacks
$\sO_{\hat{Q}}/\sO_{C_0}^*\sO_{C_1}^*$ to the $\P^1$-bundle on $C'$
associated to $N_{\hat{Q}/\hat{Q}'}({\cal L})$.  Over the locus where
$k(\hat{Q})$ is unramified over $k(\hat{Q}')$, it is easy to see the
restriction to $\sO_{\hat{Q}}^*/\sO_{C_0}^*\sO_{C_1}^*$ makes the image a
coarse moduli space (i.e., that the ratio of the coefficients of the norm
generates the field of invariants); outside the Azumaya locus, this fails
since the resulting sheaves are no longer irreducible and the
$S$-equivalence classes contain multiple isomorphism classes.

In fact, since any local trivialization of ${\cal L}$ can be refined to one
in which the overlaps are contained in the smooth (or unramified) locus, we
can define $N_{\hat{Q}/\hat{Q}'}({\cal L})$ by taking the norm of the
associated cocycle; since this is consistent with the gluing of the center
over the unramified locus, this gives us the desired description of the
global center.  (This strongly suggests that the norm indeed always takes
$\sO_{\hat{Q}}$ to $\sO_{\hat{Q}'}$.)  We can rephrase this in terms of
divisors: if $D$ is a divisor supported on the smooth locus of $\hat{Q}$,
then
\[
N_{\hat{Q}/\hat{Q}'}(\sO_{\hat{Q}}(D))
\cong
\sO_{\hat{Q}'}(D')
\]
where $D'$ is the image of $D$ in $\hat{Q}'$.  (Note that any line bundle
has such a description, which corresponds to a sequence of (noncommutative)
elementary transformations that untwist ${\cal L}$.)

In the differential case, the situation is more subtle, as now the quotient
is no longer even generically the quotient by a torus.  It is still
relatively straightforward to identify two semi-invariants of degree $p$ on
the fibers of $\sO_{\hat{Q}}$, namely the coefficients of the restriction of
$\Nm'$ (as defined in Corollary \ref{cor:definition_of_Nmprime}) to the fiber.
Furthermore, at least over the locus with trivial conductor, these are the
{\em only} semi-invariants of degree $p$.  To see this, it suffices to
consider the restriction to fibers of $\sO_{\hat{Q}}^*$, or equivalently to
$k[z,\epsilon]/(z^p,\epsilon^2)^*$.  That $\sO_{C_0}^*$ acts as the norm
lets us further restrict to the subgroup $1+(k[z]/z^p)\epsilon$, and then
using $\sO_{C_1}^*$-invariance to the quotient of $(k[z]/z^p) dz$ by the
$\F_p$-span of the logarithmic differentials.  These include the elements
$\frac{ay}{1+byz} dz$ for $a,b\in \F_p^*$, $y\in k$, and thus (by taking
suitable $\F_p$-linear linear combinations, essentially a discrete Fourier
transform relative to the group $\F_p^*$) the elements $(y+y^p z^{p-1})dz$
and $y^l z^{l-1} dz$ for $1\le l<p$.  It follows immediately that the field
of invariants is generated by a unique function of degree $p$, giving the
desired result.  (This fails when the conductor is nontrivial just as in
the difference case.)  Again, it follows immediately that $\Nm'$ gives the
desired map on twisting data, and here there is no difficulty in dealing
with the case of nontrivial conductor, since then $\hat{Q}$ has trivial
Picard group.

In the hybrid case, we suppose that $\hat{Q}\to C_0$ is inseparable and
consider the locus of $C'$ over which $C_0$ is \'etale.  The corresponding
fibers of $\sO_{\hat{Q}}$ have the form $k[z]/z^2\oplus k[z]/z^2$, with
$\sO_{C_1}^*$ acting diagonally and $\sO_{C_0}^*$ acting as multiplication
by $k^*\times k^*$.  We again find that there are (only) two invariants of
degree 2, which on $(a_0+a_1z,b_0+b_1z)$ take the values $a_0b_0$ and
$a_0b_1+a_1b_0$.  In particular, just as in the difference case, this is
the norm down to $k(\hat{Q}')$, so that we can deal with twisting in the
same way.  (Again, we expect this norm to take $\sO_{\hat{Q}}$ to
$\sO_{\hat{Q}'}$ even at singular points of $\hat{Q}$.)

\medskip

One reason why twisting was tricky to deal with is that not every line
bundle on the noncommutative surface will be the pullback of a line bundle
on the center, and we cannot expect there to be a natural choice of such a
pullback of degree $r$ relative to the fibration.  In the untwisted case,
there was no difficulty: the above construction associates sheaves to any
point of a fiber of $\sO_{\hat{Q}}^*$, and thus in particular associates a
{\em canonical} family of sheaves to the identity section.  This gives rise
to a corresponding section of the $\P^1$-bundle over $C'$, which is why in
this case we not only can identify the relevant rank 2 vector bundle on
$C'$, but have a natural choice of map from that vector bundle to
$\sO_{C'}$, such that the corresponding section is disjoint from the image
of $\hat{Q}$.


\bigskip

Since we are not only interested in quasi-ruled surfaces themselves, but
also those surfaces which are merely birationally quasi-ruled, we also want
to understand how to translate a blowup \`a la van den Bergh into an
operation on maximal orders.  There is a folklore result (subject to some
implicit conditions that exclude some cases of interest in finite
characteristic) saying that such a blowup corresponds to blowing up the
corresponding point of the underlying commutative surface and choosing a
maximal order containing the pulled back algebra.  We give a fully general
proof of this below.  In principle, we would also need to understand
blowing down, but this will be dealt with (at least for blowdowns of
blowups of quasi-ruled surfaces) via Theorem
\ref{thm:birationally_qr_is_rationally_qr} below.

Our objective is to prove the following.

\begin{thm}\label{thm:blowup_of_order}
  Let $Z$ be a (commutative) smooth surface, let ${\cal A}$ be an order on
  $Z$, and let ${\cal I}\subset {\cal A}$ be an invertible two-sided ideal
  such that there is an equivalence $\phi:\coh Y\cong \coh({\cal A}/{\cal
    I})$ for a commutative curve $Y$.  Let $x\in Y$ be a point such that
  $\phi(\sO_X)$ has finite injective dimension as a ${\cal A}$-module and
  ${\cal A}$ is locally free at the point $z\in Z$ on which $\phi(\sO_x)$
  is supported.  Then the van den Bergh blowup \cite{VandenBerghM:1998} of
  $\coh {\cal A}$ at $x$ is equivalent to the category $\coh \widetilde{A}$
  where $\widetilde{A}$ is a maximal element of the poset of orders on the
  blowup $\widetilde{Z}$ of $Z$ at $z$ that agree with ${\cal A}$ on the
  complement of the exceptional divisor.
\end{thm}

\begin{rem}
  In particular, if ${\cal A}$ is a maximal order, then so is
  $\widetilde{\cal A}$.
\end{rem}

\begin{proof}
Since ${\cal A}$ is an honest sheaf of algebras, the construction of the
van den Bergh blowup simplifies considerably, and we find that the blowup
may be expressed as the relative $\Proj$ of a Rees algebra:
\[
\widetilde{X}=\Proj({\cal A}[{\frakm}_x\otimes {\cal I}^{-1} t]),
\]
where ${\frakm}_x$ is the maximal ideal corresponding to $x$ and
$t$ is an auxiliary central variable.  On the complement of $z$, we get
\[
\widetilde{X}\cong \Proj({\cal A}[{\cal I}^{-1} t])\cong \coh {\cal A},
\]
so it will suffice to understand what happens over the local ring at $z$.

For each integer $l$, define a maximal ideal ${\frakm}_l\subset {\cal
  A}$ by
\[
{\frakm}_l = {\cal I}^l \otimes_{\cal A} {\frakm}_x \otimes_{\cal A}
{\cal I}^{-l}.
\]
The quotient ${\cal A}/{\frakm}_l$ is supported on $z$ for all $l$, so we
may view these as maximal ideals of ${\cal A}_z$.  The significance of
these maximal ideals is that
\[
({\frakm}_x {\cal I}^{-1})^l
=
{\frakm}_0{\frakm}_{-1}\cdots {\frakm}_{1-l} {\cal I}^{-l}.
\]
Since ${\cal I}$ permutes the maximal ideals, the sequence ${\frakm}_i$ is
periodic of some period $n$, and this lets us twist the $rn$-th Veronese of
the graded algebra:
\[
\widetilde{X}
\cong \Proj({\cal A}[{\frakm}_0\cdots {\frakm}_{1-rn} {\cal I}^{-n} t]
\cong \Proj({\cal A}[{\frakm}_0\cdots {\frakm}_{1-rn} t]).
\]
The latter turns out to have the better behaved center; in particular, we
claim that for a suitable choice of $r$, its center is precisely
\[
\sO_Z[{\frakm}_z t],
\]
or in other words that
\[
({\frakm}_0\cdots {\frakm}_{1-n})^{lr}\cap \sO_Z
=
{\frakm}_z^l
\]
for all $l\ge 0$.  For this claim, it suffices to consider the completion
$\hat{\cal A}_z\cong {\cal A}\otimes_{\sO_Z} \hat\sO_{Z,z}$.

A somewhat different completion of ${\cal A}$ was studied in
\cite{vanGastelM/VandenBerghM:1997}, namely the completion with respect to
the intersection $\cap_{0\le i<n} {\frakm}_{-i}$.  (To be precise, they
studied a ring constructed from the Serre subcategory generated by the $n$
simple modules of ${\cal A}$, but this is Morita equivalent to the given
completion.)  Luckily, this completion is equal to $\hat{\cal A}_z$, due to
the following fact.

\begin{lem}
  Every maximal ideal of ${\cal A}_z$ is of the form ${\frakm}_i$ for some
  $i$.
\end{lem}

\begin{proof}
  Suppose otherwise, and let ${\frakm}'$ be some other maximal ideal lying
  over $z$.  Idempotents in ${\cal A}_z/\rad {\cal A}_z$ lift over
  $\hat{\cal A}_z$, and thus we may find an idempotent $e\in \hat{\cal
    A}_z$ such that $e\in {\frakm}_i$ for all $i$ but $e\notin {\frakm}'$.
  Since central simple algebras do not have nontrivial two-sided ideals,
  the ideal $\hat{\cal A}_z e \hat{\cal A}_z$ contains an ideal of the form
  ${\cal J}\hat{\cal A}_z$ where ${\cal J}$ is a nonzero ideal of
  $\hat\sO_{Z,z}$.  In particular, we find that
  \[
    {\cal J}\hat{\cal A}_z
    \subset
    \hat{\cal A}_z e \hat{\cal A}_z
    \subset
    {\frakm}_0\cdots {\frakm}_{1-l}
  \]
  for all $l\ge 0$, so that
  \[
    {\frakm}_0\cdots {\frakm}_{1-l}
    \supset
    ({\frakm}_z^l+{\cal J}) \hat{\cal A}_z.
  \]
  It follows in particular that
  \[
  \dim(\hat{\cal A}_z/{\frakm}_0\cdots {\frakm}_{1-l})
  \le
  \rank({\cal A}) \dim(\hat\sO_{Z,z}/{\frakm}_z^l+{\cal J})
  =
  O(l).
  \]
  On the other hand, it was shown in \cite[Cor.~5.2.4]{VandenBerghM:1998}
  that this quotient has length $l(l+1)/2$, giving a contradiction.
\end{proof}

It thus follows from \cite{vanGastelM/VandenBerghM:1997} that the
completion $\hat{\cal A}_z$ is Morita equivalent to an algebra $A$ of the
following form.  If $A=R$ is local (so $\hat{\cal A}_z$ is local), then $R$
has the form $k\langle x,y\rangle/\langle \phi\rangle$ where the relation
$\phi\in {\frakm}_R^2$ and its image in ${\frakm}_R^2/{\frakm}_R^3$ is a
rank 2 element of $({\frakm}_R/{\frakm}_R^2)^{\otimes 2}$, and $Y$ has the
form $\Spec(R/UR)$ where $U\in {\frakm}_R$ is a normalizing element such
that $[R,R]\subset UR$.  (Note that per
\cite[Prop.~7.1]{vanGastelM/VandenBerghM:1997}, the commutator $[x,y]$ is
itself a normalizing element.)  More generally, if $A$ (thus ${\cal A}_z$)
has $n$ maximal ideals, then there is a pair $(R,U)$ as above with $U\notin
{\frakm}_R^2$ such that $A$ is isomorphic to the subalgebra of $\Mat_n(R)$
in which the coefficients above the diagonal are multiples of $U$.  In this
case, the curve $Y$ is cut out by the normalizing element $N$ which has
$1$s just below the diagonal, $U$ in the upper right corner, and $0$
everywhere else.  We note that the center of $A$ consists precisely of the
matrices $z I$ with $z\in Z(R)$, so that that $A$ is free over its center
iff $R$ is free over its center.  The maximal ideals of $A$ are
${\frakm}_1$,\dots,${\frakm}_n$ where ${\frakm}_i$ is cut out by the
condition that $a_{ii}\in {\frakm}_R$, and these satisfy $N {\frakm}_i =
{\frakm}_{i+1} N$, with subscripts interpreted cyclically.

Note that in our case, the algebra $A$ satisfies the additional condition
that it is free as a module over its center.  Moreover, the claim about how
the center of $\hat{\cal A}_z$ meets the products of maximal ideals reduces
to the corresponding claim about the center of $A$, which we refine to
\[
({\frakm}_n\cdots {\frakm}_1)^{lr}
\cap
Z(A)
=
{\frakm}_{Z(A)}^l
\]
where $r$ is the square root of the rank of $R$ over its center.
(Equivalently, $\rank(A) = r^2n^2$.)  An easy induction shows that
\[
{\frakm}_n {\frakm}_{n-1}\cdots {\frakm}_1\subset {\frakm}_R\Mat_n(R),
\]
or in other words that given any sequence of elements of the appropriate
ideals, their images in $\Mat_n(k)$ multiply to 0.  The codimension of this
product is known (again by \cite[Cor.~5.2.4]{VandenBerghM:1998}), so that
one immediately concludes
\[
  {\frakm}_n {\frakm}_{n-1}\cdots {\frakm}_1
  =
  A\cap {\frakm}_R\Mat_n(R),
\]
and similarly
\[
  ({\frakm}_n {\frakm}_{n-1}\cdots {\frakm}_1)^l
  =
  A\cap {\frakm}_R^l\Mat_n(R).
\]
(These identities are implicit in \cite[Prop.~5.2.2]{VandenBerghM:1998}.)
Since the center of $A$ is the center of the diagonal copy of $R$, we
conclude that
\[
({\frakm}_n {\frakm}_{n-1}\cdots {\frakm}_1)^{lr}
\cap
Z(A)
=
{\frakm}_R^{lr}\cap Z(R),
\]
so that it remains only to show that
\[
  {\frakm}_R^{lr}\cap Z(R)
  =
  {\frakm}_{Z(R)}^l.
\]

To finish the argument, we will need a couple of lemmas about orders over
dvrs with sufficiently nice residue field.

\begin{lem}
  Let $R$ be a dvr with residue field $k$ and field of fractions $K$, and
  suppose that $cd(k)\le 1$.  Let $A$ be an $R$-order in a central simple
  $K$-algebra of degree $r$.  Then any simple $A$-module has dimension at
  most $r$ as a $k$-space.
\end{lem}

\begin{proof}
  We may as well pass to the completion, as this has no effect on the
  simple modules.  Suppose first that $A$ is a maximal order.  Then
  $A/\rad(A)\cong \Mat_d(l)$ where $l/k$ is a field extension of degree
  $r/d$.  If $A\otimes_R K$ is a division ring, then this follows from
  \cite[Thm.~14.3]{ReinerI:2003} with $d=\sqrt{r}$; this result was only
  stated for the case of finite residue field, but the proof only used
  triviality of $\Br(k)$.  More generally, $A$ is a matrix ring over a
  maximal order in a division ring with center $K$, and thus the result
  follows in general from \cite[Thm.~17.3]{ReinerI:2003}.  Thus in the
  maximal order case, $A$ has a unique simple module isomorphic to $l^d$,
  and thus of dimension $r$ over $k$.

  In general, there exists a maximal order $B$ containing $A$.  Then
  $\rad(B)\cap A\subset \rad(A)$, and $A/(\rad(B)\cap A)\subset
  B/\rad(B)\cong \Mat_d(l)$.  It follows that $A/\rad(A)$ is a sum of
  matrix algebras of the form $\Mat_{d'}(l')$ where $l'/l$ is a field
  extension such that $d'[l':l]\le d$.  Each summand corresponds to a
  simple module with $k$-dimension $d'[l':k]\le d[l:k]=r$ as required.
\end{proof}

\begin{rem}
  To see that $\rad(B)\cap A$ is contained in the radical of $A$, observe
  that this is a two-sided ideal, and thus it suffices to show that $1-z$
  is a unit for all $z\in \rad(B)\cap A$.  But this follows from the fact
  that $\lim_{n\to\infty} z^n=0$ relative to the natural topology in $B$.
\end{rem}

\begin{cor}
  With $R$ and $A$ as above, the left $A$-module $A\otimes_R k$ has length
  at least $r$ (with equality if $A$ is maximal).
\end{cor}

\begin{proof}
  Indeed, $A\otimes_R k$ has dimension $r^2$, while each simple constituent
  contributes at most $r$ to the dimension.
\end{proof}

\begin{prop}
  Let $k$ be an algebraically closed field, and let $R=k\langle\langle x
  ,y\rangle\rangle/\Phi$ with $\Phi\in {\frakm}_R^2$, $\Phi_2$
  nondegenerate.  If $Z(R)$ is a regular $2$-dimensional ring and $R$ is a
  free $Z(R)$-module of rank $r^2$, then for any $l\ge 0$,
  \[
    {\frakm}_R^{lr}\cap Z(R) = {\frakm}_{Z(R)}^l.
  \]
\end{prop}

\begin{proof}
  Since $Z(R)$ is regular, we may write it as $k[[u,v]]$ for suitable
  central elements $u$ and $v$, where WLOG $\deg(u)\le \deg(v)$; moreover,
  if $\deg(u)=\deg(v)$, we may assume the leading terms are linearly
  independent, as otherwise we can subtract a multiple of $u$ from $v$ to
  increase the latter's degree.  The claim will follow immediately if we
  can show that $\deg(u)=\deg(v)=r$, since then their images in $\gr R$ are
  algebraically independent.  (The center of $\gr R$ is easily seen to be
  of the form $k[u',v']$ with $\deg(u')=\deg(v')$ a divisor of $r$, and
  linearly independent homogeneous polynomials of the same degree are
  algebraically independent.)

  The freeness condition implies that $\dim_k(R/{\frakm}_{Z(R)}R)=r^2$,
  while the Hilbert series of ${\frakm}_{Z(R)} R=uR+vR$ (i.e., the Hilbert
  series of the associated graded) is upper bounded (in each coefficient)
  by
  \[
  \frac{z^{\deg(u)}}{(1-z)^2}+
  \frac{z^{\deg(v)}}{(1-z)^2}.
  \]
  Thus the Hilbert series of the quotient is lower bounded by
  \[
  \frac{1-z^{\deg(u)}-z^{\deg(v)}}{(1-z)^2}
  =
  \frac{(1-z^{\deg(u)})(1-z^{\deg(v)})}{(1-z)^2}
  +
  O(z^{\deg(u)+\deg(v)}).
  \]
  The first term has degree $\deg(u)+\deg(v)-2$, and thus the dimension is
  at least the value of the first term at $z=1$, i.e., $\deg(u)\deg(v)$, so
  that $\deg(u)\deg(v)\le r^2$.  

  Now, $u$ generates a prime ideal in $Z(R)$, and the residue field of the
  localization is isomorphic to $k((u))$, so has cohomological dimension
  $\le 1$ by \cite{LangS:1952}.  We may thus apply the Corollary to
  conclude that $R_{(u)}/uR_{(u)}$ has length $\ge r$, so that we may
  choose a chain
  \[
  R_{(u)}=I_0\supsetneq I_1\supsetneq \cdots\supsetneq I_r=u R_{(u)}
  \]
  of left ideals.  Since $R$ has global dimension $2$
  (\cite{vanGastelM/VandenBerghM:1997}), each module $R\cap I_i$ is free,
  and thus we may choose a sequence of elements $z_i$ such that $R\cap I_i
  = z_i R$, with $z_0=1$, $z_r=u$, so that we have represented our chain as
  a descending chain of principal left ideals.  Since $z_i R\supsetneq
  z_{i+1}R$, it follows that $z_{i+1}\in z_i{\frakm}$ for $1\le i\le n$,
  and thus $u=z_r\in {\frakm}^r$, so that $r\le \deg(u)\le \deg(v)$, which
  together with $\deg(u)\deg(v)\le r^2$ implies $\deg(u)=\deg(v)=r$.
\end{proof}

\begin{rem}
  It follows from the proof that $uR+vR$ has Hilbert
  series $(2z^r-z^{2r})/(1-z)^2$, and thus that $\gr u$ and $\gr v$ have no
  nontrivial syzygies.  In particular, they have trivial $\gcd$ as elements
  of the polynomial ring $Z(\gr R)$.
\end{rem}

We thus find that the van den Bergh blowup indeed corresponds to a sheaf
$\widetilde{\cal A}$ of algebras on $\widetilde{Z}$.  Moreover, that sheaf
is coherent, since if we quotient by the degree 1 elements of the center,
the resulting graded algebra has bounded degree; this reduces immediately
to the completion and thus to $A$ since it is really a question about modules.

It remains only to show that $\widetilde{\cal A}$ is maximal among orders
agreeing with ${\cal A}$ outside the exceptional locus $e$.  This reduces to
showing that it is locally free at any point of $e$ and that its completion
at the corresponding valuation is maximal.  Here we may use what we know
about sheaves on the blowup.

The local freeness condition reduces to showing that $\widetilde{\cal
  A}\otimes_{\sO_{\widetilde{Z}}}\sO_e$ is a torsion-free sheaf on $e$.  The
result is an extension of $\widetilde{\cal A}$-modules of the form
$\sO_e(d)$, so it suffices to prove the result for those modules.  Twisting
by $\sO_{\widetilde{Z}}(-ne)$ for $n\gg 0$ lets us reduce to the case $d<0$,
where it follows from the fact that $\sO_e(d)$ has no global sections.
(Here we meant $\sO_e(d)$ as a sheaf on the blowup, though it follows from
torsion-freeness and the Euler characteristic that it is also $\sO_e(d)$ as
a sheaf on $\widetilde{Z}$\dots)

It remains only to show maximality over the local ring at $e$.  Since
${\cal A}_z$ has global dimension 2 (its simple modules have injective
dimension 2, since this holds for $A$), the same holds for the blowup, and
thus the localization of the blowup has global dimension 1, so is
hereditary.  Moreover, it has a unique simple module (the now isomorphic
sheaves $\sO_e(d)$), and thus is maximal as required, finishing the proof
of Theorem \ref{thm:blowup_of_order}.
\end{proof}

\begin{cor}
  With hypotheses as above, if the orbit of $x\in Y$ has size $r$, then
  $\widetilde{\cal A}$ is an Azumaya algebra on every point of $e$ not
  meeting the blowup of $Y$.
\end{cor}

\begin{proof}
  In this case, $R=Z(R)$ and $A$ is the order in $\Mat_r(R)$ consisting of
  matrices in which the entries above the diagonal are multiples of the
  equation $U$ of $Y$ inside $\Spec(R)$.  The relevant affine patch of the
  blowup is then given by $A[{\frakm}_r\cdots {\frakm}_1 U^{-1}]\cong
  \Mat_r(R[{\frakm}_R/U])$, so that $\widetilde{\cal A}$ is a matrix
  algebra on the complement of $Y$ in the pullback of the local ring at
  $z$.
\end{proof}

\begin{rem}
  This result is particularly nice in the case of ruled surfaces, when
  every orbit has size $1$ or $r$, so that when an iterated blowup of a
  noncommutative ruled surface is a maximal order, it is an Azumaya algebra
  on the complement of the curve of points.
\end{rem}

In the cases of present interest, the invertible ideal corresponding to the
divisor $Y$ is $[{\cal A},{\cal A}]{\cal A}$ (i.e., the ideal sheaf of the
scheme classifying ${\cal A}$-module quotients of ${\cal A}$ with Hilbert
polynomial $1$ over $Z$), and we would like to understand the corresponding
ideal in $\widetilde{\cal A}$.  That is, we want to know how the curve of
points on $\widetilde{X}$ is related to the curve of points on $X$.

\begin{prop}
  Let $Z$ be a (commutative) smooth surface, let ${\cal A}$ be an order on
  $Z$ such that the commutator ideal is invertible, and let ${\frakm}$ be
  a maximal ideal of ${\cal A}$ containing the commutator ideal.  If
  $\widetilde{\cal A}$ is the order obtained by blowing up ${\frakm}$,
  then the commutator ideal of $\widetilde{\cal A}$ is invertible, and
  in the notation of \cite{VandenBerghM:1998} is isomorphic to $\sO(-1)$.
\end{prop}

\begin{proof}
  This is a local question, and thus reduces to showing that for $l\gg 0$,
  the degree $l$ component of the commutator ideal of the graded algebra
  $
  A[({\frakm}_n\cdots {\frakm}_1)^rt]
  $
  is equal to 
  \[
  N {\frakm}_{n-1}\cdots {\frakm}_1({\frakm}_n\cdots {\frakm}_1)^{rl-1}
  t^l,
  \]
  corresponding to the ideal generated by $t$ in the original graded
  algebra before taking the Veronese and twisting.

  For $n>1$, we first observe that
  \[
  A [{\frakm}_n\cdots {\frakm}_1,{\frakm}_n\cdots {\frakm}_1]A
  \subset
  N {\frakm}_{n-1}\cdots {\frakm}_1
  ({\frakm}_n\cdots {\frakm}_1)
  \]
  or in other words the subspace of $A$ in which the coefficients on or
  above the diagonal are in ${\frakm}_RU$ and the coefficients below the
  diagonal are in ${\frakm}_R^2$.  This is a two-sided ideal, so it suffices
  to consider commutators, and since the ideal differs from
  $
  ({\frakm}_n\cdots {\frakm}_1)^2
  $
  only in the diagonal, it suffices to consider the diagonal coefficients
  of commutators.  The only commutators that can possibly contribute are
  those of diagonal matrices, where the claim follows from $[R,R]\subset UR$.

  We can now compute by induction that for $a,b\ge 1$,
  \[
    A [({\frakm}_n\cdots {\frakm}_1)^a,({\frakm}_n\cdots {\frakm}_1)^b]A
    \subset
    N {\frakm}_{n-1}\cdots {\frakm}_1
    ({\frakm}_n\cdots {\frakm}_1)^{a+b-1},
  \]
  while a similar calculation gives
  \[
  A[A,({\frakm}_n\cdots {\frakm}_1)^a]A
  =
  N {\frakm}_{n-1}\cdots {\frakm}_1
  ({\frakm}_n\cdots {\frakm}_1)^{a-1}.
  \]
  It follows that the given ideal agrees with the commutator ideal in
  positive degree, so the saturations agree.

  It remains only to consider the case $n=1$.  Let $r'=p$ if $\gr R$ is
  abelian, and otherwise let $r'$ be the rank of $\gr R$ over its center.
  Then $r$ is a multiple of $r'$, and we may reduce to showing
  \[
  R[{\frakm}^{r'},{\frakm}^{r'}]R \subset {\frakm}^{2r'-1}U
  \]
  and
  \[
  R[R,{\frakm}^{r'}]R={\frakm}^{r'-1}U.
  \]

  If $\gr R$ is abelian, then $(g,h)\mapsto [g,h]U^{-1}$ induces a
  Poisson structure on $\gr R=k[x,y]$, namely the degree $-2$ bracket
  $\{x,y\} = 1$, and the claims become that the bracket vanishes
  when both arguments have degree $p$ and that any homogeneous
  polynomial of degree $p-1$ is in the span of the brackets with
  one argument of degree $p$.  For the first claim, we have
  \[
  \{x^a y^{p-a},x^b y^{p-b}\}
  =
  (a(p-b)-(p-a)b) x^{a+b-1} y^{2p-a-b-1} = 0,
  \]
  and the second claim follows by noting that the brackets with $x$ span
  the polynomials not of degree $p-1$ in $y$, while the brackets with $y$
  span the polynomials not of degree $p-1$ in $x$.

  When $\gr R$ is not abelian, then the calculation reduces to one
  in $\gr R$, so that we reduce to considering $k\langle
  x,y\rangle/(yx-qxy)$ and $k\langle x,y\rangle/(yx-xy-x^2)$.  The
  first ($q$-Weyl) case is straightforward:
  \[
    [x^a y^b,x^c y^d]
    =
    (q^{bc}-q^{ad}) x^{a+c} y^{b+d},
  \]
  which again is 0 if $q^{a+b}=q^{c+d}=1$, and the span
  $
  \langle [x,x^a y^b],[y,x^a y^b]\rangle
  $
  for $a+b=\ord(q)$ has the correct dimension.

  For $k\langle x,y\rangle/(yx-xy-x^2)$, we observe that this has a
  representation in differential operators: $x=t$, $y=t^2 D_t$.
  We in particular easily find
  $
    [y,x^a y^{p-a}] = a x^{a+1} y^{p-a}
  $
  which together with $[x,x^{p-1} y]=-x^{p+1}$ span a space of the correct
  dimension.  It thus remains only to show that
  $
    [x^a y^{p-a},x^b y^{p-b}]=0.
  $
  Doing the change of variable $t\mapsto 1/t$ in the corresponding
  commutator of differential operators and observing that $t^p$ is
  central reduces this to showing that
  $
    [t^a D_t^a,t^b D_t^b] = 0,
  $
  which is true in arbitrary characteristic since $t^a D_t^a$ is a
  polynomial in $t D_t$ for all $a\ge 0$.
\end{proof}

\medskip

In the proof that the blowup of a maximal order is a maximal order, we
skirted around the question of determining the precise structure of the
localization along the exceptional divisor.  Although this is somewhat
tricky to determine in general, it is fairly straightforward in the cases
corresponding to ruled surfaces.  

For instance, in the difference cases of elliptic or multiplicative type,
the division ring is a localization of the ring of (symmetric) difference
operators.  Over the local ring of a codimension 1 point, the residue field
of the maximal order is a separable extension of the residue field of the
center.  Choose an element $u$ whose image generates that extension (of
degree $n$ if the division ring has dimension $n^2$), and observe that if
take the corresponding unramified base change, then this splits the
division ring, and lets us see that for each element $\sigma$ of the Galois
group, there is an $n$-dimensional subspace of the division ring consisting
of elements $v$ satisfying $vu=\sigma(u)v$, which is then a 1-dimensional
space over $k(u)$.  In particular, there exists an element of this form
having valuation 1, so that $v^n$ is a uniformizer of the center.

For the remaining cases, a key observation is that given a discrete
valuation on the center of a division ring, there is a unique extension to
a maximal order over the valuation ring.  In particular, blowing up does
not change the division ring, merely the set of available valuations.  We
claim that the division ring is generated over its center by elements $u$,
$v$ such that $[u,v]=1$.  This is clear in the differential cases, while
the remaining cases (additive difference and the hybrid case) are
birational to the differential case, so have the same division ring.  We
then find the following.

\begin{prop}\label{prop:local_structure_of_order_single_component}
  Let $X$ be a rationally ruled surface over a field of characteristic $p$
  such that $Q$ has a non-nodal singular point.  Then $X=\Spec({\cal A})$
  for a maximal order ${\cal A}$ with generic fiber a division ring, and
  for any associated point $x\in Q$, the localization ${\cal A}_x$ is
  generated by elements $u$, $v$ with $u$ a unit, $v$ of valuation $1$, and
  $[u,v]=v^m$, where $m$ is the multiplicity of $Q$ along $x$.
\end{prop}

\begin{proof}
  We need to understand the extension of the generic fiber of ${\cal A}$ to
  the corresponding valuation ring, and have already observed that the
  generic fiber is generated by elements $u,v$ with $[u,v]=1$, $u^p=U$,
  $v^p=V$ for appropriate elements $U$, $V$ of the center.  If these
  elements are both integral, then $1\in [{\cal A}_x,{\cal A}_x]$, so that
  $x$ was not a component of $Q$.  If $U$ has valuation a multiple of $p$
  and its leading coefficient is a $p$-th power, then we can subtract a
  central element from $u$ to make $U$ smaller, and similarly for $V$.  If
  $U$ and $V$ still both have valuation a multiple of $p$, then we can
  rescale $(u,v)\mapsto (\pi^l u,\pi^{-l} v)$, where $\pi$ is a uniformizer
  of the center, to make $U$ a unit which is not a $p$-th power in the
  residue field.  In particular $U+{\frakm}$ generates the residue field
  over its $p$-th power subfield, so that we may add a suitable polynomial
  in $u$ to $v$ to reduce its valuation.

  In this way, we reduce to the case that at least one of $u$ or $v$ (say
  $v$) has valuation prime to $p$.  For any integers $a$, $l$ with $a$
  prime to $p$, there is an automorphism
  \[
  (u,v)\mapsto (a^{-1}\pi^{-l}v^{1-a}u,\pi^l v^a)
  \]
  of the division ring, and thus by suitable choice of $l$ and $a$, we may
  arrange to have $v$ (and thus $V$) of valuation $1$.  If $U$ has
  valuation $-l$, then $v^l u$ is a unit, and we have
  \[
  [v^l u,v] = v^l.
  \]

  These elements still generate the division ring, and the order they
  generate is easily verified to be maximal, so that they generate ${\cal
    A}_x$ as required.  Moreover, it follows that the 2-sided ideal
  generated by commutators is the principal ideal generated by $v^l$, so
  that the quotient has length $l=m$.
\end{proof}

\begin{rem}
   Note that when the multiplicity is $1$, this looks like $uv=v(u+1)$,
   which is an instance of the separable case, and in particular shows
   directly that the additive difference cases are birational to the
   differential cases.
\end{rem}

\begin{rem}
The situation for quasi-ruled surfaces is more complicated, as the residue
field extension may be an inseparable extension of degree larger than $p$.
\end{rem}

It would be nice to have a similar description for the localizations of
${\cal A}$ at closed points of $Q$, even if only for the corresponding
completion.  For the completion, there is no difficulty at smooth points,
as in the ruled case the relevant automorphism acts faithfully, and thus we
obtain the situation above with $A\subset \Mat_n(R)$ and $R$ abelian.  At
nodes, the situation is only slightly more complicated.  In that case, the
leading term of the relation of $R$ must be $vu-quv$ with $q$ a primitive
$n$-th root of unity, and in any such ring, one can perform changes of
variables to make the relation have the form $vu = quv + u\Phi v$ where
$\Phi$ is in the span of $u^{ni}v^{nj}$ with $i+j>0$.  (We certainly have a
relation of the form $vu=quv+\Phi'$ with $\Phi'$ in the span of $u^i v^j$
with $i+j>0$, and if $\Phi'$ is only of the form $u\Phi v$ to degree $d$,
then we can improve the agreement by adding suitable terms of degree $d$ to
$u$ and $v$.)  Then $(u,v)\mapsto (qu,v)$ and $(u,v)\mapsto (u,qv)$ are
automorphisms, so act on the center.  The center is generated by two
elements of degree $n$, the leading terms of which must be central in the
associated graded of $R$, so are in the span of $u^n$ and $v^n$.  It then
follows that the center is generated by invariant elements, and must
therefore (by Hilbert series considerations) be isomorphic to
$k[[u^n,v^n]]$.  In particular, $\Phi$ is central, and thus
\[
v^n u^n = (q+\Phi)^{n^2} u^n v^n,
\]
so that $\Phi=0$.  In other words, the completion at a node has the form
$k\langle\langle u,v\rangle\rangle/(vu-quv)$ where $q$ is a primitive
$n$-th root of unity.

Thus only the differential (or additive) case remains open.  It is unclear
whether simply having degree $p^2$ over the center is enough to pin down
the completion (possibly together with the known possibilities for the
structure of $R/R[R,R]R$) as it was in the above case of degree prime to
$p$.

\section{Noncommutative surfaces as $t$-structures}

Although the above construction of noncommutative ruled surfaces was well
suited to an interpretation via difference/differential operators as well
as to understanding the center, it is not ideal for a number of other
applications, especially since it does not behave well in families.  Van
den Bergh's original construction does not have this issue, but in some
respects turns out to be too explicit; constructing the various
isomorphisms we require (e.g., between the two interpretations of a
noncommutative $\P^1\times \P^1$ as a ruled surface) via that construction
either requires working with sheaves on highly singular curves or requires
that one split into a large number of separate cases that can be dealt with
explicitly.  Luckily, it turns out that there is another way to construct
noncommutative $\P^1$-bundles: there is a relatively simple description of
the corresponding derived category, as well as the associated
$t$-structure.  Although this construction cannot quite replace van den
Bergh's construction (our proof that the category is the derived category
of the heart of the $t$-structure uses van den Bergh's construction), it
gives an alternate approach to constructing isomorphisms, namely as derived
equivalences respecting the $t$-structures.  In each of the cases of
interest, it is easy to construct the derived equivalence corresponding to
the desired isomorphism, and not too difficult to show sufficient
exactness.

The overall approach works for each of the main constructions in the
literature (noncommutative (projective) planes
\cite{ArtinM/TateJ/VandenBerghM:1990,BondalAI/PolishchukAE:1993},
noncommutative $\P^1$-bundles \cite{VandenBerghM:2012}, and blowups
\cite{VandenBerghM:1998}), but we begin by considering the $\P^1$-bundle
case.  The basic idea is that there is a natural semiorthogonal
decomposition of the derived category of a commutative $\P^1$-bundle over a
smooth projective scheme.  Not only does this extend to the noncommutative
setting, but we shall see that one can recover the $t$-structure from the
decomposition.  Although the discussion below works for general
noncommutative $\P^1$-bundles over smooth projective schemes, we consider
only the surface case, and at least initially work over an algebraically
closed field.  (Otherwise, we would need to consider things like conic
bundles!)


Recall that the category $\qcoh \bar{\cal S}$ is defined as the quotient of
the category of $\bar{\cal S}$-modules by the subcategory of bounded such
modules (i.e., which become 0 in sufficiently large degree), with $\coh
\bar{\cal S}$ the subcategory of Noetherian modules.  For each $d$, there
is a natural functor $\rho_d^{-1}$ from $\coh (C_d:=C_{d\bmod 2})$ to the
category of $\bar{\cal S}$-modules taking a coherent sheaf $M$ on $C_d$ to
the representation which in degree $d'$ is $\bar{\cal S}_{dd'}\otimes_{C_d}
M$.  Composition with the quotient morphism gives a functor $\rho_d^*:\coh
C_d\to \coh \bar{\cal S}$.  Since the quotient morphism has a right
adjoint, as does $\rho_d^{-1}$, it follows that $\rho_d^*$ has a right
adjoint $\rho_{d*}$.

\begin{lem}\label{lem:exact_tri_quasi_ruled}
  For any vector bundle $V$ on $C_0$, the quadratic relation gives rise to
  an exact sequence
  \[
  0\to \rho_2^*(V\otimes \det(\pi_{0*}\sO_{\hat{Q}})^{-1})
  \to \rho_1^*(\pi_{1*}(\pi_0^*V\otimes \bar{\cal S}_{01}))
  \to \rho_0^*V
  \to 0.
  \]
\end{lem}

\begin{proof}
  It suffices to check exactness locally, where it reduces to showing
  exactness of
  \[
  0\to \bar{S}_{2n}\to \bar{S}_{1n}\otimes \bar{S}_{01}\to \bar{S}_{0n}\to 0.
  \]
  The second (multiplication) map is surjective since $\bar{S}$ is
  generated in degree 1, while the first map is
  \[
  x\mapsto (x N_0 e_1\otimes N_1 \xi_1 e_0)-(x N_0 s_1(\xi_1) e_1\otimes
  N_1 e_0),
  \]
  which maps into the kernel of the second map.  Since
  \[
  x = (x N_0 e_1)(N_1 \xi_1 \xi_0 e_0)-(x N_0 s_1(\xi_1) e_1)(N_1 \xi_0 e_0),
  \]
  the map from $\bar{S}_{2n}$ to the kernel of the second map splits, and
  is thus an isomorphism by rank considerations.
\end{proof}  

\begin{lem}
   For $M\in D_{\qcoh} \bar{\cal S}$, if $R\rho_{0*}M=R\rho_{1*}M=0$, then $M=0$.
\end{lem}

\begin{proof}
  Suppose $M$ is an object such that $R\rho_{0*}M=R\rho_{1*}M=0$.  From the
  previous Lemma, we deduce that for any vector bundle $V$ on $C_0$,
  $R\Hom(L\rho_2^*V,M)=0$, and thus $R\rho_{2*}M=0$.  It follows by
  induction that $R\rho_{d*}M=0$ for all $d\ge 0$, and thus that the
  cohomology modules of the corresponding complex in $\bar{\cal
    S}\text{-mod}$ are torsion, so that $M$ is 0 in $D_{\qcoh}\bar{\cal
    S}$.
\end{proof}

The following was essentially shown in \cite{MoriI:2007}, subject to
the (luckily unused) assumption that the two base schemes (i.e., $C_0$ and
$C_1$) are equal.  (To be precise, the reference showed the result when $M$
and $N$ are line bundles, but this implies it in general.)

\begin{lem}\cite[Lem.~4.4]{MoriI:2007}
  For $M,N\in D^b_{\coh} C_d$, one has
  $R\Hom(L\rho_d^*M,L\rho_d^*N)=R\Hom(M,N)$, while for $N'\in D^b_{\coh}
  C_{d+1}$ one has $R\Hom(L\rho_d^*M,L\rho_{d+1}^*N')=0$.
\end{lem}

Combining the above, we obtain the following.

\begin{thm}\label{thm:semiorth_for_quasiruled}
  The subcategories $L\rho_0^*D^b_{\coh} C_0$ and $L\rho_1^*D^b_{\coh} C_1$
  form a semiorthogonal decomposition of $D^b_{\coh} \bar{\cal S}$: for any
  object $M\in D^b_{\coh} \bar{\cal S}$, there is a unique distinguished
  triangle of the form
  \[
  L\rho_0^* N_0\to M\to L\rho_1^*N_1\to.
  \]
\end{thm}

\begin{proof}
  Adjunction gives a natural morphism $L\rho_0^* R\rho_{0*}M\to M$, which
  extends to a distinguished triangle
  \[
  L\rho_0^*R\rho_{0*}M\to M\to M'\to.
  \]
  Applying $R\rho_{0*}$ and using $R\rho_{0*}L\rho_0^*\cong \text{id}$ implies
  that $R\rho_{0*}M'=0$.  The natural distinguished triangle
  \[
  L\rho_1^*\rho_{1*}M'\to M'\to M''\to
  \]
  then gives an object $M''$ such that $R\rho_{1*}M''=0$ and
  $R\rho_{0*}M''\cong R\rho_{0*}M'=0$, and thus $M''=0$, so that $M'$ is in
  the image of $L\rho_1^*$.

  For uniqueness, observe that for any such triangle, applying $R\rho_{0*}$
  gives $N_0\cong R\rho_{0*}M$.
\end{proof}

\begin{rem}
  As we mentioned, the same proof shows that this holds more generally for
  the noncommutative $\P^1$-bundle associated to any sheaf bimodule on
  Noetherian schemes $Y_0$ and $Y_1$, except that if the base schemes are
  singular, one should instead work with the categories of compact objects
  (which on $Y_i$ are just the perfect complexes).
\end{rem}

The semiorthogonal decomposition yields two immediate corollaries.  The
first is that we can easily compute the Grothendieck group of $\coh
\bar{\cal S}$, since the Grothendieck of a derived category with a
semiorthogonal decomposition is just the sum of the Grothendieck groups of
the two subcategories.

\begin{cor}
  We have $K_0\coh \bar{\cal S}\cong K_0\coh C_0\oplus K_0\coh C_1$.
\end{cor}

Furthermore, the Mukai pairing on $K_0\coh \bar{\cal S}$ is easily
reconstructed from the pairings on $C_0$ and $C_1$ and the induced pairing
\[
(M,N)\mapsto \chi R\Hom(\rho_1^*M,\rho_0^*N)
\]
between the two curves.  The relevant Hom space was again computed in
\cite{MoriI:2007}.

\begin{lem}
  We have $R\Hom(\rho_1^*M,\rho_0^*N)\cong R\Hom(M,N\otimes_{C_0} \bar{\cal
    S}_{01})$.
\end{lem}

The other corollary is slightly more subtle.  Since we have not just the
one semiorthogonal decomposition, but one for each $d$, and $L\rho_d^*D^b
\coh C_d$ appears in two such decompositions, once on each side, we find
that each such subcategory is admissible.  Since both subcategories have
Serre functors (being derived categories of projective schemes), the same
is true for $D^b_{\coh} \bar{\cal S}$ (per
\cite[Prop.~3.8]{BondalAI/KapranovMM:1990}).  Moreover, we can compute the
Serre functor explicitly.

\begin{prop}\label{prop:Serre_for_quasiruled}
  The triangulated category $D^b_{\coh} \bar{\cal S}$ has a Serre functor $S$
  which is the composition of $M\mapsto M(-\hat{Q})[2]$ with twisting by
  the pair of line bundles $\omega_{C_i}\otimes
  \det(\pi_{i*}\sO_{\hat{Q}})^{-1}$ having the same pullback to $\hat{Q}$.
\end{prop}

\begin{proof}
  It suffices to compute how the Serre functor acts on the two pieces of
  the semiorthogonal decomposition, and thus (since shifting degrees gives
  another quasi-ruled surface) how it acts on an object $L\rho_0^*M$.  We
  thus need an object $SL\rho_0^*M$ such that
  \[
  \Hom(L\rho_0^*N,SL\rho_0^*M)\cong \Hom(L\rho_0^*M,L\rho_0^*N)^*
  \cong \Hom(M,N)^*\cong \Hom(N,SM),
  \]
  where we also denote the Serre functor on $D^b_{\coh} C_0$ by $S$, and
  \[
  \Hom(L\rho_1^*N,SL\rho_0^*M)\cong \Hom(L\rho_0^*M,L\rho_1^*N)^*=0.
  \]
  In other words, we want $R\rho_{0*}SL\rho_0^*M\cong SM$ and
  $R\rho_{1*}SL\rho_0^*M=0$.

  Now, the short exact sequence of Lemma \ref{lem:exact_tri_quasi_ruled}
  extends to a distinguished triangle for perfect complexes in general, of
  the form
  \[
  L\rho_2^*(N\otimes \det(\pi_{0*}\sO_{\hat{Q}})^{-1})
  \to
  L\rho_1^*(N\otimes_{C_0} \bar{\cal S}_{01})
  \to
  L\rho_0^*(N)
  \to.
  \]
  Applying $R\rho_{0*}$ gives
  \[
  R\rho_{0*}L\rho_2^*(N\otimes \det(\pi_{0*}\sO_{\hat{Q}})^{-1})
  \cong
  N[-1]
  \]
  while applying $R\rho_{1*}$ makes the second map an isomorphism, so that
  \[
  R\rho_{1*}L\rho_2^*(N\otimes \det(\pi_{0*}\sO_{\hat{Q}})^{-1})=0.
  \]
  We thus conclude that
  \[
  S\rho_0^*M\cong \rho_2^*(SM\otimes \det(\pi_{0*}\sO_{\hat{Q}})^{-1})[1].
  \]
  The claimed form for $SM$ then follows by writing $SM$ as $M\otimes
  \omega_{C_0}[1]$.
\end{proof}

\begin{rem}
  The analogous result was shown (again with the assumption that the base
  schemes are isomorphic) for general smooth projective base in
  \cite{NymanA:2005}.  The semiorthogonal decomposition yields a
  significant streamlining of the argument, however.
\end{rem}

Since $M(-Q)$ itself differs from $M(-\hat{Q})$ by a pair of twists by
line bundles, we can also write the Serre functor as the composition of
$\_(-Q)[2]$ with twisting by $\omega_{C_i}\otimes
\det(\pi_{i*}\sO_Q)^{-1}$.  As we remarked after Proposition
\ref{prop:genus_1_almost_implies_ruled}, these bundles are trivial iff the
surface is ruled, and thus we obtain the following.

\begin{cor}
  The surface $\bar{\cal S}$ is ruled iff the Serre functor is equivalent
  to $\_(-Q)[2]$.
\end{cor}

In other words, a quasi-ruled surface is ruled iff the curve of points is
anticanonical.
\medskip

There is one important notational issue to consider here, namely that since
$\hat{Q}$ does not behave well in flat families, neither do the functors
$L\rho_d^*$ in general (though for $d\notin \{0,1\}$, we can rephrase
everything in terms of the sheaf bimodule).  We could fix this by using the
labelling of \cite{VandenBerghM:2012}, but it turns out that there is an
even nicer choice.  If we leave $\rho_0^*$ and $\rho_1^*$ fixed, then we
may twist the remaining functors by line bundles and automorphisms so that
\[
L\rho_{d+2}^* = S L\rho_d^* S^{-1}[-1]
\]
With this relabelling, the distinguished triangle coming from the
semiorthogonal decomposition becomes
\[
L\rho_{d+1}^*R\rho_{(d-1)*}M\to L\rho_d^*R\rho_{d*}M\to M\to.
\]
Indeed, applying $R\rho_{(d-1)*}$ to the semiorthogonal decomposition
reduces this to showing
\[
R\rho_{(d-1)*}L\rho_{d+1}^*[1]\cong \text{id},
\]
which follows (for coherent sheaves, and thus in general) from the
computation
\begin{align}
R\Hom(M,R\rho_{(d-1)*}L\rho_{d+1}^*N[1])
&\cong
R\Hom(L\rho_{d-1}^*M,S L\rho_{d-1}^*S^{-1}N)\notag\\
&\cong
R\Hom(L\rho_{d-1}^*S^{-1}N,L\rho_{d-1}^*M)^*\notag\\
&\cong
R\Hom(S^{-1}N,M)^*\notag\\
&\cong
R\Hom(N,SM)^*\notag\\
&\cong
R\Hom(M,N).
\end{align}
We will use this alternate labelling below.

\medskip

The results of \cite{OrlovD:2016} tell us that we can reverse the above
construction: given a pair of (dg-enhanced, automatic for commutative
projective schemes) triangulated categories and a suitable functor from one
to the other, there is a corresponding ``glued'' category with a
semiorthogonal decomposition such that the $\Hom$s between the two
subcategories are induced by the functor.  This gives an immediate
construction of a triangulated category from a sheaf bimodule: simply glue
the two categories using the associated Fourier-Mukai functor.  In the case
of sheaf bimodules such that both direct images are locally free of rank 2,
the above discussion shows that the resulting category is precisely the
derived category of the corresponding noncommutative $\P^1$-bundle.

As mentioned above, we can actually extend this to give an alternate
construction of noncommutative $\P^1$-bundles.  To determine the
noncommutative $\P^1$-bundle from its derived category, we need to specify
the $t$-structure, i.e., those objects with only nonnegative or only
nonpositive cohomology.  We again consider the surface case, leaving the
modifications for more complicated base schemes to the reader.  In that
case, the Serre functor has the form $\theta[2]$ where $\theta$ is an
autoequivalence of the desired abelian category.  Moreover, since $\theta$
is just a twist of a degree shift of the sheaf algebra, we find that
$\theta^{-1}$ is {\em relatively} ample, in that any nonzero sheaf $M\in
\coh \bar{\cal S}$ has the property that $R\rho_{0*}\theta^{-d}M$ is a
nonzero sheaf for all sufficiently large $d$.  It then follows more
generally that $M\in D^b_{\coh} \bar{\cal S}^{\ge 0}$ iff
$R\rho_{0*}\theta^d M\in D^b_{\coh} C_0^{\ge 0}$ for all $d$.  Indeed,
$\rho_{0*}$ is right exact and $\theta$ is exact, so the condition
certainly holds for all $M\in D^b_{\coh} \bar{\cal S}^{\ge 0}$, while if
$M$ has cohomology in some negative degree, then some power of
$\theta^{-1}$ makes that cohomology have nonzero direct image.  We can then
reconstruct the other half $D^b_{\coh} \bar{\cal S}^{\le 0}$ of the
$t$-structure as the objects with no maps to objects in $D^b_{\coh} \bar{\cal
  S}^{\ge 1}$.

The significance of this is that the Serre functor is intrinsic to the
triangulated category, and thus so is $\theta$.  In other words, the
``geometric'' $t$-structure can be reconstructed given only the
triangulated category and the functor $\rho_{0*}$.  This construction works
fairly generally: given any functor $f:{\cal A}\to {\cal B}$ of
triangulated categories such that ${\cal A}$ has a Serre functor and ${\cal
  B}$ is equipped with a $t$-structure, and any integer $d$, there is a
putative $t$-structure defined by taking ${\cal A}^{\ge 0}$ to be those
objects such that $f(S^k M[-dk])\in {\cal B}^{\ge 0}$ for all $k\in \Z$.
If we already have a $t$-structure in mind for ${\cal A}$, then we say that
$f$ is ``pseudo-canonical (of dimension $d$)'' if the putative
$t$-structure agrees with the specified $t$-structure.  (Note that this
forces $S[-d]$ to be exact, so is probably not quite the right notion in
general, but is good enough for us since we only care about analogues of
smooth surfaces.  In general, one should probably conjugate $f$ by the
respective shifted Serre functors.) In the case of a smooth commutative
projective $d$-fold $X$, this condition holds if either $-K_X$ or $K_X$ is
relatively ample.

\begin{lem}
  Suppose that $X$ and $Y$ are Gorenstein quasi-schemes ($X$ of dimension
  $d$) with $f:\qcoh X\to \qcoh Y$ a left exact functor such that for all
  $M\in \coh X$, there exists $l\in \Z$ such that $f(S^l M[-dl])\ne 0$.
  Then $Rf$ is pseudo-canonical of dimension $d$.
\end{lem}

\begin{proof}
  Since $f$ is left exact, $Rf$ certainly takes $D^b_{\qcoh} X^{\ge 0}$ to
  $D^b_{\qcoh} Y^{\ge 0}$, and the same applies if we compose it with any
  power of $S[-d]$.  So it will suffice to show that if $\tau_{<0}Rf S^l
  M[-dl]=0$ for all $l$ then $\tau_{<0}M=0$.  If not, then let $p$ be the
  smallest integer such that $h^pM\ne 0$.  This has a coherent subsheaf,
  and thus there is an integer $l$ such that $f(S^l(h^p M)[-ld])\ne 0$.
  But the standard spectral sequence then tells us that $h^p(Rf S^l
  M[-ld])\ne 0$, giving a contradiction.
\end{proof}

\begin{rem}
  The converse is also true: if for some nonzero sheaf $M$,
  $h^0(Rf S^l M[-ld])=0$ for all $l\in \Z$, then $M[1]$ is not a
  nonnegative complex, but its image under $RfS^l[-ld]$ is for all $l$.
\end{rem}

It turns out that the other two constructions we need also come with
semiorthogonal decompositions such that one of the associated functors is
pseudo-canonical.  The first of these is blowing up.  If $C$ is a
(commutative) divisor on the noncommutative surface $X$, a ``weak line
bundle relative to $C$'' is a sheaf on $X$ such that $X|_C$ is a line
bundle\footnote{These were called ``line bundles'' in
  \cite{VandenBerghM:1998}, but we reserve this terminology for a
  particular class of weak line bundles to be considered in Definition
  \ref{defn:line_bundle} below.}, and we say that $X$ is ``generated by
weak line bundles relative to $C$'' if every element of $\qcoh X$ is a
quotient of a direct sum of weak line bundles.  (Presumably one could
weaken this to allow sheaves with locally free restriction.)  Note that
this is easily seen to be hold for a quasi-ruled surface with $C$ the curve
of points, as any sheaf of the form $\rho_d^*{\cal L}$ is a weak line
bundle.  A similar statement also holds for noncommutative planes, since
the sheaves $\sO_X(d)$ are weak line bundles relative to any curve in $X$.

\begin{thm}\cite[Thm.~8.4.1]{VandenBerghM:1998}
  Let $X$ be a noncommutative surface with $C$ a commutative divisor and
  $p\in C$ a point, and suppose that $X$ is generated by weak line bundles
  relative to $C$. Let $\widetilde{X}$ be the blowup at $p$, with associated
  functors $\alpha^*$, $\alpha_*$, and let $\sO_e(-1)$ denote the
  corresponding exceptional sheaf.  Then the subcategories $D^b_{\qcoh}
  k\otimes \sO_e(-1)$ and $L\alpha^* D^b_{\qcoh} X$ form a semiorthogonal
  decomposition of $D^b_{\qcoh} \widetilde{X}$.
\end{thm}

\begin{rem}
  To be precise, the associated distinguished triangle is
  \[
  L\alpha^*R\alpha_*M\to M\to R\Hom_k(R\Hom(M,\sO_e(-1)),\sO_e(-1))\to
  \]
  as long as $M$ is coherent.  It follows from the explicit form of Serre
  duality that one can instead take
  \[
  L\alpha^*R\alpha_*M\to M\to R\Hom(\sO_e,M[2])\otimes_k \sO_e(-1)\to,
  \]
  and this works even when $M$ is only quasicoherent, see the proof of
  Proposition \ref{prop:blowup_acyclic} below.
\end{rem}

\begin{cor}
  We have $K_0(\widetilde{X})\cong K_0(X)\oplus \Z$.
\end{cor}

To control the Serre functor, we will need to know that the two
subcategories are admissible.  Define a new functor $\alpha^!$ by
$\alpha^!M = \alpha^*(M(C))(-1)$.

\begin{lem}\label{lem:alpha_shriek}
  The derived functor $L\alpha^!$ is right adjoint to $R\alpha_*$.
\end{lem}

\begin{proof}
  It suffices to show that for any object $M\in D^b_{\qcoh} X$,
  \[
  R\alpha_*L\alpha^!M\cong M\qquad\text{and}\qquad
  R\Hom(\sO_e(-1),L\alpha^!M)=0.\notag
  \]
  Indeed, the second claim lets us reduce to checking the adjunction on the
  image of $L\alpha^*$, which in turn reduces by the adjunction between
  $\alpha^*$ and $\alpha_*$ to checking the first claim.
  
  By \cite[Prop.~8.3.1]{VandenBerghM:1998}, we have
  $R\alpha_*((\alpha^*E)(-1))\cong E(-C)$ for any sheaf $E$ which is a
  direct sum of weak line bundles, and since these generate, the same holds
  for any object in $D^b_{\qcoh} X$.  Applying this to $M(C)$ gives the first
  result.  It also follows that
  \[
  R\alpha_*((\alpha^*M)(-1))(C)\cong M\cong R\alpha_*L\alpha^*M
  \]
  and thus $R\Hom(\sO_e,L\alpha^*M)=0$ by \cite[Lem.~8.3.3(2)]{VandenBerghM:1998}.
  We then find
  \[
  R\Hom(\sO_e(-1),L\alpha^!M)\cong R\Hom(\sO_e,L\alpha^*M(C))=0
  \]
  as required.
\end{proof}

\begin{rem}
  It follows that there is another semiorthogonal decomposition, with
  distinguished triangle
  \[
  \sO_e(-1)\otimes_k R\Hom(\sO_e(-1),M)\to M\to L\alpha^!R\alpha_*M\to.
  \]
\end{rem}

This lets us compute the gluing functor, which again can be expressed in
terms of an adjoint pair of functors.

\begin{prop}
  We have
  \[
  R\Hom_{\widetilde{X}}(V\otimes \sO_e(-1),L\alpha^*M)
  \cong
  R\Hom_X(V\otimes \sO_{\tau p}[-1],M)
  \cong
  R\Hom_k(V,R\Hom_X(\sO_{\tau p}[-1],M)).
  \]
\end{prop}

\begin{proof}
  We first observe that
  \[
  R\Hom_{\widetilde{X}}(\sO_e(-1),L\alpha^*M)
  \cong
  R\Hom_{\widetilde{X}}(\sO_e(-2),L\alpha^!(M(-C)))
  \cong
  R\Hom_X(R\alpha_*\sO_e(-2),M(-C)),
  \]
  so that we need to compute $R\alpha_*\sO_e(-2)$.  By
  \cite[Prop.~8.3.2]{VandenBerghM:1998}, we have a distinguished triangle
  \[
  L\alpha^* \sO_{\tau p}\to \sO_e\to \sO_e(-1)[2]\to
  \label{eq:dtri_e_em1}
  \]
  which twists to give
  \[
  L\alpha^!(\sO_{\tau p}(-C))\to \sO_e(-1)\to \sO_e(-2)[2]\to.
  \]
  Applying $R\alpha_*$ then gives
  \[
  R\alpha_*\sO_e(-2)\cong \sO_{\tau p}(-C)[-1],
  \]
  so that
  \[
  R\Hom_{\widetilde{X}}(\sO_e(-1),L\alpha^*M)\cong R\Hom_X(\sO_{\tau p}[-1],M),
  \]
  from which the claim follows.
\end{proof}

In particular, $D_{\qcoh}(\widetilde{X})$ is compactly generated (so has a
Serre functor) iff $D_{\qcoh}(X)$ is compactly generated.  Indeed, if $G$
is a compact generator of $D_{\qcoh}(X)$, then $L\alpha^*G\oplus \sO_e(-1)$
certainly generates $D_{\qcoh}(X)$, and the Proposition tells us that it is
compact.  Moreover (\cite{BondalAI/KapranovMM:1990}), given such a
semiorthogonal decomposition, one can compute the Serre functor on the
ambient category from that on the subcategories (and vice versa, if
desired).  Since the Serre functor on $\perf(k)$ is just the identity, this
is not too hard in our case.

\begin{cor}
  If $X$ is compactly generated (i.e., $D_{\qcoh}(X)$ is compactly
  generated), then the Serre functor on $\widetilde{X}$ satisfies $S
  L\alpha^*M\cong L\alpha^! S M$ and $S \sO_e\cong \sO_e(-1)[2]$.
\end{cor}

\begin{proof}
  The first claim follows immediately from Lemma \ref{lem:alpha_shriek},
  while the second claim follows from \eqref{eq:dtri_e_em1}, since that
  distinguished triangle implies that $\_\otimes \sO_e(-1)[2]$ is
  similarly right adjoint to the projection to $\_\otimes \sO_e$ in the
  semiorthogonal decomposition coming from $R\Hom(\sO_e,L\alpha^* M)=0$.
\end{proof}

Note that since the Serre functor is intrinsic, it commutes with any
autoequivalence of $\perf \widetilde{X}$.

\smallskip

For the pseudo-canonical condition, we need a lemma about how the various
functors related to a divisor interact.

\begin{lem}
  Let ${\cal A}$ be an abelian category, let $G:{\cal A}\to {\cal A}$ be an
  autoequivalence, and let $\tau:G\to \text{id}$ be a natural
  transformation.  If ${\cal A}$ is generated by objects on which $\tau$ is
  surjective or cogenerated by objects on which $\tau$ is injective, then
  $\tau G = G\tau$.
\end{lem}

\begin{proof}
  The ``cogenerated'' case is essentially just
  \cite[Lem.~8.3]{vanGastelM/VandenBerghM:1997} (using the cogenerators in
  place of injective objects), and the ``generated'' case is dual.
\end{proof}

\begin{rem}
  In particular, if $X$ is generated by weak line bundles along $C$, then
  this applies to the autoequivalence $\_(-C)$ and the natural
  transformation $\_(-C)\to\text{id}$.  We may thus safely refer to ``the''
  natural map $M(aC)\to M(bC)$ for any $a<b$, since all ways of
  constructing such a map out of the natural transformation will agree.
\end{rem}

\begin{lem}
  Let $C$ be a divisor on the noncommutative surface $X$ and suppose that
  $X$ is generated by weak line bundles along $C$ and has a Serre functor.
  Let $i_*:\qcoh C\to \qcoh X$ be the inclusion, with adjoints $i^*$,
  $i^!$.  Then there is a natural isomorphism $Ri^!\cong Li^*(\_(C))[-1]$.
  Furthermore, for $M\in D_{\qcoh} C$, $S (i_* M)(C)\cong i_* SM[1]$.
\end{lem}

\begin{proof}
  By \cite[Lem.~8.1]{vanGastelM/VandenBerghM:1997}, there is a 4-term exact
  sequence
  \[
  0\to i_*i^!M\to M\to M(C)\to i_*i^*(M(C))\to 0.
  \]
  Since there are enough acyclic objects for both functors (injectives for
  $i^!$ and weak line bundles for $i^*$), we find that the functors fit
  into functorial distinguished triangles
  \begin{align}
  \_(-C)\to {}&\text{id}\to Li^*\to\\
  Ri^!\to {}&\text{id}\to \_(-C)\to,
  \end{align}
  from which the first claim follows.  

  Serre duality immediately gives
  \[
  Ri^! \cong S Li^* S^{-1},
  \]
  and thus comparing the two expressions of $Ri^!$ gives a natural isomorphism
  \[
  Li^*(S^{-1}M(-C))\cong S^{-1} Li^* M[-1].
  \]
  Taking right adjoints gives $S(i_*M)(C)\cong i_* SM[1]$ as required.
\end{proof}

\begin{rem}
  The fact that the twist by $C$ is an autoequivalence identifying the two
  adjoints of $i_*$ establishes that $i_*$ is a {\em spherical functor} in
  the sense of \cite{AnnoR/LogvinenkoT:2017}.
\end{rem}

\begin{prop}
  If $X$ and $C$ are Gorenstein, then so is $\widetilde{X}$, and the functor
  $R\alpha_*:D^b_{\qcoh} \widetilde{X}\to D^b_{\qcoh} X$ is pseudo-canonical.
\end{prop}

\begin{proof}
  Express the Serre functor on $X$ by $\theta[2]$ with $\theta$ an exact
  functor.  The autoequivalence $M\mapsto \theta M(C)$ takes the sheaf
  $\sO_{\tau p}=i_*\sO_{\tau p}$ to $i_*(\sO_{\tau p}\otimes \omega_C)\cong
  \sO_{\tau p}$, so respects the gluing, and thus induces an
  autoequivalence of $D^b_{\qcoh} \widetilde{X}$ acting trivially on
  $\sO_e(-1)$.  Composing this autoequivalence with $M\mapsto M(-1)$ gives
  an autoequivalence acting by
  \[
  L\alpha^* M \to (L\alpha^* \theta M(C))(-1)
  \cong S(L\alpha^! M(C))(-1)[-2]
  \cong S L\alpha^*M[-2]
  \]
  and
  \[
  \sO_e(-1)\to \sO_e(-2)\cong S\sO_e(-1)[-2],
  \]
  which must therefore actually agree with the shifted Serre functor on
  $\widetilde{X}$ (which we also denote $\theta$).  We then find
  \[
  R\alpha_* \theta^{-d} M
  \cong
  \theta^{-d} (R\alpha_* M(d))(-dC).
  \]
  Since $\theta$ and $M\mapsto M(C)$ are exact on $X$, we conclude that
  $\theta^{-1}$ is relatively ample for $R\alpha_*$ iff $M\mapsto M(1)$ is
  relatively ample for $R\alpha_*$, and the latter holds by construction.
\end{proof}

The construction of van den Bergh depends on a choice of divisor containing
$p$, but it follows immediately from the semiorthogonal decomposition and
the pseudo-canonical property that the result does not actually depend on
the divisor.  For quasi-ruled surfaces, there is a natural choice of such
divisor, since the moduli scheme of points is embedded as a divisor.  It is
thus worth noting that this is preserved under blowup.

\begin{prop}\label{prop:nice_divisor_on_blowup}
  Suppose $X$ is a noncommutative surface and $C$ is a curve embedded as a
  divisor, and let $\widetilde{X}$ be the blowup of $X$ in a point $p\in C$.
  Then there is a curve $C^+$ embedded in $\widetilde{X}$ as a divisor via a
  natural transformation $\_(-1)\to \text{id}$, which is Gorenstein
  whenever $C$ is Gorenstein.  If $C$ is anticanonical on $X$, so is $C^+$
  on $\widetilde{X}$.
\end{prop}

\begin{proof}
  The graded bimodule algebra associated to the blowup is such that the
  homogeneous component of degree $i$ is a sub-bimodule of the homogeneous
  component of degree $i+1$, which gives the desired natural
  transformation.  We need to show that the quotient is a commutative
  curve.  If $\tau(p)\ne p$, then the quotient is just the curve
  $\widetilde{C}$, and thus the claim holds, and $C^+\cong C$, so the
  Gorenstein claim is immediate.

  If $\tau(p)=p$, then we can still factor the natural transformation as
  \[
  \_(-1)\to \_(-\widetilde{C})\to \text{id},
  \]
  so that the (possibly noncommutative) scheme $C^+$ is isomorphic to $C$
  away from $p$.  We thus reduce to showing that if we quotient by any
  power of the bimodule ideal associated to $p$, then the result is
  commutative, and since every homogeneous component of the result is an
  extension of $\sO_p$, we may pass (via
  \cite{vanGastelM/VandenBerghM:1997}) to a local calculation in a ring of
  the form $R=k\langle\langle x,y\rangle\rangle/(uy-vx)$ where $u,v$
  generate the maximal ideal ${\frakm}$, in such a way that the point
  sheaves of $C$ correspond to maximal ideals of a commutative quotient
  $R/U$ where $U$ is normalizing.  The blowup is then locally given by the
  Proj of the Rees algebra $\bigoplus_i {\frakm}^i U^{-i} t^i$ where $t$
  is a central variable of degree 1.  (Note that since conjugation by $U$
  is an invertible automorphism of $R$, it makes sense to adjoin an inverse
  of $U$ to $R$.)  We need to show that the quotient by $t$ is a
  commutative curve.

  Since multiplying the degree $0$ element $U$ by any element of degree $1$
  gives a multiple of $t$ (in fact, an element of ${\frakm}t$), we find
  that $U$ is contained in the saturation of the ideal generated by $t$.
  Thus the objective is to show that the scheme of points is the $\Proj$ of
  \[
  R/\langle U\rangle \oplus \bigoplus_{i\ge 1} ({\frakm}^i/{\frakm}^{i-1} U)
  U^{-i}t^i.
  \]
  The corresponding $\Z$-algebra are unchanged if we remove the factors
  $U^{-1}$ from these algebras, so that we may as well consider instead the
  algebra
  \[
  R/\langle U\rangle \oplus \bigoplus_{i\ge 1} ({\frakm}^i/{\frakm}^{i-1}
  U)t^i.
  \]
  This is the $R/\langle U\rangle$-algebra generated by ${\frakm}/\langle
  U\rangle t$ subject to the single relation $(ut)(yt)=(vt)(xt)$, and thus
  to show that its $\Proj$ is the category of coherent sheaves of a
  commutative scheme, it will suffice to find an $R/\langle
  U\rangle$-linear automorphism of the degree 1 submodule that twists this
  into a commutator.

  In other words, we need an $R/\langle U\rangle$-module automorphism of
  ${\frakm}/\langle U\rangle$ that takes $x$ to $u$ and $y$ to $v$.  By
  \cite[Prop.~7.1]{vanGastelM/VandenBerghM:1997}, there is a short exact
  sequence
  \[
  \begin{CD}
    0@>>> R @> (v,-u) >> R^2 @> (x,y) >> {\frakm}\to 0
  \end{CD}
  \]
  of $R$-modules, and thus there is an $R$-linear morphism
  $\phi^+:{\frakm}\to {\frakm}/\langle U\rangle$ taking $x$ to $u$ and $y$
  to $v$ iff the corresponding map on $R^2$ annihilates $(v,-u)$.  Since
  this has image $vu-uv\in \langle U\rangle$, such a map indeed exists and
  is unique.  Moreover, the map on $R^2$ takes $(y,x)$ to $yu-xv\in
  uy-vx+\langle U\rangle=\langle U\rangle$, so that $\phi^+$ annihilates
  $\langle U\rangle$, so induces a unique $R$-linear morphism
  $\phi:{\frakm}/\langle U\rangle\to {\frakm}/\langle U\rangle$.  The same
  reasoning shows that there is a unique map of right $R$-modules taking
  $(u,v)$ to $(x,y)$, which since $R/\langle U\rangle$ is commutative must
  be the inverse of $\phi$.

  We may thus twist by $\phi$ to obtain a description of this scheme as the
  $\Proj$ of the free commutative graded $R/\langle U\rangle$-algebra
  generated by the maximal ideal of $R/\langle U\rangle$.  Since $R/\langle
  U\rangle\cong \hat\sO_C$, we may consider this in entirely commutative
  terms, and find that the local scheme of points is obtained by blowing up
  the origin of $\Spec k[[x,y]]$, taking the total transform of
  $\hat\sO_C$, and removing a single copy of the exceptional divisor.  In
  particular, the new scheme is indeed a Gorenstein curve when $C$ is
  Gorenstein.

  For the final claim, we note that the endofunctor $M\mapsto \theta
  M(C^+)$ of $D_{\qcoh}(X)$ is the endofunctor induced via the
  semiorthogonal decomposition by the pair $(M\mapsto \theta
  M(C),\text{id})$, and is thus trivial iff $C$ is anticanonical.
\end{proof}

\begin{rem}
  If one can embed $C$ as a Cartier divisor in a smooth projective surface
  $X'$, then it follows from the above local description that $C^+$ is
  naturally isomorphic to the effective Cartier divisor $\pi^*C-e$, where
  $\pi:\widetilde{X}'\to X'$ is the blowup at $\tau p$.  Note that $C$ being
  Gorenstein implies only that such an embedding exists locally, but we
  have already seen that it exists globally for quasi-ruled surfaces, while
  for noncommutative planes, the curve of points embeds in a commutative
  plane.
\end{rem}

\begin{rem}
  We would like to have a statement along the lines that if $C$ is the
  moduli space of point sheaves on $X$, then $C^+$ is the moduli space of
  point sheaves on $\widetilde{X}$.  Unfortunately, there is no intrinsic
  notion of a ``point sheaf'' on a noncommutative surface, so this is
  difficult to state in general.  For the surfaces of present interest
  (iterated blowups of quasi-ruled surfaces and projective planes), we do
  have such a notion, and it follows easily from the above calculation that
  $C^+$ is the moduli space of such sheaves on $\widetilde{X}$ iff $C$ is
  the moduli space of point sheaves on $X$.  (This involves only local
  calculations near the exceptional curve.)
\end{rem}


We also need to know that generation by weak line bundles is inherited, so
that further blowups still have the semiorthogonal decomposition.

\begin{lem}
  If $X$ is generated by weak line bundles relative to $C$,
  then $\widetilde{X}$ is generated by weak line bundles relative to $C^+$.
\end{lem}

\begin{proof}
  If $M$ is a weak line bundle on $X$ relative to $C$, then $\alpha^*M$ is
  a weak line bundle relative to the new curve $C^+$, and $\theta$ takes
  weak line bundles to weak line bundles.  Since $\theta^{-1}$ is
  relatively ample for $\alpha_*$, $\qcoh \widetilde{X}$ is generated by sheaves
  of the form $\theta^{-d}\alpha^*M$ where $M$ ranges over any set of
  generators of $\qcoh X$.
\end{proof}

For the next result, we assume not just that $C$ is Gorenstein, but that it
is a surface curve, i.e., that it embeds as a Cartier divisor in a
commutative smooth projective surface.

\begin{lem}\label{lem:blowup_generates_correct_curve}
  Let $X$ be a compactly generated noncommutative surface and $i:C\to X$
  an embedding of a surface curve as a divisor.  Let $\widetilde{X}$ be a
  blowup of $X$ in some point of $C$.  If the image of $\perf(X)$ generates
  $\perf(C)$, then the image of $\perf(\widetilde{X})$ generates $\perf(C^+)$.
\end{lem}

\begin{proof}
  Let $X'$ be a smooth projective surface containing $C$.  The line bundle
  $\sO_C(d)$ is certainly in the image of $\perf(X')$ for any $d$, and
  since $\sO_C(1)$ is ample, these bundles generate $\perf(C)$.  By the
  same argument, $\perf(C^+)$ is generated by the image of
  $\perf(\widetilde{X}')$.  By the semiorthogonal decomposition,
  $\perf(\widetilde{X}')$ is generated by $L\alpha^*\perf(X')$ and
  $\sO_e(-1)$, and thus $\perf(C^+)$ is generated by the image of
  $L\alpha^*\perf(X')$ and the image of $\sO_e(-1)$.  The result follows by
  noting that the Karoubian subcategory generated by the image of
  $L\alpha^*\perf(X')$ in $\perf(C^+)$ is the same as the image of
  $\perf(C)$ in $\perf(C^+)$, and the image of $\sO_e(-1)$ in $\perf(C^+)$
  is the same for either $\widetilde{X}$ and $\widetilde{X}'$.
\end{proof}

\medskip

The other construction is for noncommutative planes.

\begin{prop}\cite{BondalAI/PolishchukAE:1993}
  Any noncommutative plane has a strong exceptional collection consisting
  of three objects $\sO_X(-2)$, $\sO_X(-1)$, $\sO_X$ such that
  $\Hom(\sO_X(-2),\sO_X(-1))$ and $\Hom(\sO_X(-1),\sO_X)$ are
  three-dimensional and
  \[
  \Hom(\sO_X(-2),\sO_X(-1))\otimes \Hom(\sO_X(-1),\sO_X)
  \to
  \Hom(\sO_X(-2),\sO_X)
  \]
  is surjective with three-dimensional kernel.  Moreover, $D_{\qcoh} X$ is
  compactly generated, has Serre functor $M\mapsto M(-3)[2]$, and the
  functor $R\Hom(\sO_X,\_)$ is pseudo-canonical of dimension 2.
\end{prop}

\begin{rem}
  The condition that the resulting category has a well-behaved
  $t$-structure is open in the set of surjections $k^3\otimes k^3\to
  k^6$, and reduces to the condition that there be a cubic curve $C\subset
  \P^2$ and a degree $0$ line bundle $q$ on $C$ such that
  $\Hom(\sO_X(-1),\sO_X)\cong \Gamma(C;q(1))$,
  $\Hom(\sO_X(-2),\sO_X(-1))\cong \Gamma(C;\sO_C(1))$, and the composition
  is given by multiplication.  The scheme of points is isomorphic to $C$,
  which embeds as a divisor via a natural transformation $\_(-3)\to \text{id}$.
\end{rem}

\medskip

In particular, both noncommutative planes and noncommutative quasi-ruled
surfaces have the property that any point can be blown up, and this
property is inherited by such blowups.  This motivates the following two
definitions: a ``noncommutative rationally quasi-ruled surface'' is an
iterated blowup of a noncommutative quasi-ruled surface, while a
``noncommutative rational surface'' is an iterated blowup of either a
noncommutative plane or a noncommutative ruled surface over a curve of
genus 0.  (For technical reasons, we should include the anticanonical curve
as part of the data when the surface is commutative.)

\medskip

Thus in each case, we could have constructed the desired abelian category
by first gluing the appropriate commutative derived categories and using
the pseudo-canonical property to construct the $t$-structure.  Of course,
it is nontrivial to show that this is a $t$-structure and even less trivial
to show that the resulting triangulated category is the derived category of
its heart.  So although these constructions are relatively simple to state,
they are difficult enough to control to make this a less than ideal
definition.  It is still useful to think about the categories this way,
however, as it leads to rather different approaches to constructing
morphisms.

In particular, there are a number of isomorphisms we would like to
establish, analogous to standard constructions in commutative birational
geometry: blowups in (sufficiently) distinct points should commute, a
blowup of a quasi-ruled surface should be a blowup of a different
quasi-ruled surface (i.e., elementary transformations), a blowup of a
noncommutative plane should be a noncommutative Hirzebruch surface, and a
noncommutative $\P^1\times \P^1$ should be a noncommutative Hirzebruch
surface in two different ways.  In each case, we will show this by first
constructing a corresponding {\em derived} equivalence, and then showing
that the derived equivalence respects the $t$-structure.  The latter will
in turn reduce to showing that the derived equivalence respects some
pseudo-canonical functor.

For commuting blowups, we note that if $X$ is a noncommutative Gorenstein
surface with successive blowups $\alpha_{1*}:X_1\to X$ and
$\alpha_{2*}:X_2\to X_1$, then there are three natural subcategories
$\sO_{e_2}(-1)\otimes D^b_{\coh} k$, $\alpha_2^*\sO_{e_1}(-1)\otimes
D^b_{\coh} k$ and $\alpha_2^*\alpha_1^*D^b_{\coh} X$ of $D^b_{\coh}
X_2$, and these moreover form a three-step semiorthogonal
decomposition.  Note that
\[
R\Hom(\sO_{e_2}(-1),\alpha_2^*\sO_{e_1}(-1))
\cong
R\Hom(\sO_{e_1}(-1),\sO_{p_2})^*,
\]
and thus if $R\Hom(\sO_{e_1}(-1),\sO_{p_2})=0$, we can swap the first two
subcategories and still have a semiorthogonal decomposition.  We find that
the new semiorthogonal decomposition is again the decomposition associated
to a two-step blowup, but now with $p_1$ and $p_2$ swapped.  (Note that if
$p_1$ is a singular point of $Q$, then the condition implies that $p_2$ is
not on the component $e_1$ of $Q_1=Q^+$, and thus it makes sense to swap
$p_1$ and $p_2$.)  The resulting derived equivalence fixes the functor
$\alpha_{1*}\alpha_{2*}$, and thus when this functor is pseudo-canonical,
the derived equivalence respects the $t$-structure.

To see that this is pseudo-canonical (in fact relatively Fano in a suitable
sense) under reasonable conditions, it will be helpful to have a way of
checking in the case of a single blowup whether a sheaf is acyclic and
globally generated for $\alpha$.  We assume the original surface is
Gorenstein (which holds in all cases of interest in the present
work), but with care one should be able to replace $\theta$ by suitable
functors coming from the relevant divisors of points.

\begin{prop}\label{prop:blowup_acyclic}
  Let $X$ be a noncommutative surface and $\alpha:\widetilde{X}\to
  X$ the van den Bergh blowup of $X$ in the point $x$.  Then the object
  $M\in \qcoh \widetilde{X}$ is acyclic and globally generated for $\alpha$ iff
  $\Ext^2(\sO_e,M)=0$.
\end{prop}

\begin{proof}
  We first recall that the distinguished triangle associated to the
  semiorthogonal decomposition takes the form
  \[
  L\alpha^*R\alpha_*M\to M\to FM\otimes_k \sO_e(-1)\to
  \]
  for some functor $F:D_{\qcoh} \widetilde{X}\to D_{\qcoh} k$, and we claimed
  above that $FM=R\Hom(\sO_e,M[2])$.  To see this, note that since
  $R\Hom(\sO_e,L\alpha^*N)=0$, we have
  \[
  R\Hom(\sO_e,M[2])\cong FM\otimes_k R\Hom(\sO_e,\sO_e(-1)[2])
  \cong FM.
  \]
  Substituting into the distinguished triangle and taking the corresponding
  long exact sequence gives an exact sequence ending
  \[
      h^0(L\alpha^*R\alpha_*M)
  \to M
  \to \Ext^2(\sO_e,M)\otimes_k \sO_e(-1)
  \to \alpha^* R^1\alpha_*M
  \to 0.
  \]
  In particular, if $\Ext^2(\sO_e,M)=0$, then $\alpha^* R^1\alpha_*M=0$
  and thus $R^1\alpha_*M=0$ so $M$ is acyclic for $\alpha_*$.

  In general, if $M$ is acyclic for $\alpha_*$, then the tail of the exact
  sequence becomes
  \[
  \alpha^*\alpha_*M\to M \to \Ext^2(\sO_e,M)\otimes_k \sO_e(-1)\to 0
  \]
  and thus $\alpha^*\alpha_*M\to M$ is surjective iff
  $\Ext^2(\sO_e,M)=0$.
\end{proof}

\begin{rem}
  Of course, if $M$ is coherent and $X$ has a Serre functor, then we may
  use Serre duality to rephrase the condition as $\Hom(M,\sO_e(-1))=0$.
\end{rem}

\begin{rem}
  The distinguished triangle from the $\sO_e(-1),L\alpha^!$ decomposition
  similarly tells us that
  \[
  \alpha^!R^1\alpha_*M\cong \Ext^2(\sO_e(-1),M)\otimes\sO_e(-1),
  \]
  so that $M$ is acyclic for $\alpha_*$ iff $\Ext^2(\sO_e(-1),M)=0$.  This
  can be viewed as a relative Castelnuovo-Mumford regularity statement,
  which in this case is an equivalence: relative to $\alpha$, $M$ is
  acyclic and globally generated iff $\theta M$ is acyclic.
\end{rem}

\begin{prop}
  If $M\in D^b_{\qcoh} X$ is such that $R\Hom(\sO_{\tau p},M)=0$, then
  $L\alpha^*M\cong L\alpha^!M$.
\end{prop}

\begin{proof}
  We have a distinguished triangle
  \[
  L\alpha^* M\to L\alpha^!M\to R\Hom(\sO_e,L\alpha^!M[2])\otimes \sO_e(-1)\to,
  \]
  and thus it suffices to have $R\Hom(\sO_e,L\alpha^!M)=0$.  But then
  \[
  R\Hom(\sO_e,L\alpha^!M)
  \cong
  R\Hom(L\alpha^*\sO_{\tau p},L\alpha^!M)
  \cong
  R\Hom(\sO_{\tau p},M),
  \]
  so that the desired result follows.
\end{proof}

\begin{rem}
  Of course, if $X$ is Gorenstein and $M$ is coherent, then the hypothesis
  is equivalent to $R\Hom(M,\sO_p)=0$.
\end{rem}
  
\begin{prop}
  Let $X_0$ be a noncommutative Gorenstein surface, and let $\alpha:X_1\to
  X_0$, $\beta:X_2\to X_1$ be a sequence of van den Bergh blowups.
  If the images of the corresponding points $x_1$, $x_2$ in $X_0$ are not
  in the same orbit under $\tau^{\Z}$, then $\alpha_*\beta_*$ is
  pseudo-canonical.
\end{prop}

\begin{proof}
  Let $M\in \coh X_2$ be a nonzero sheaf.  We need to show that for some
  $l$, $\alpha_*\beta_* \theta^{-l} M\ne 0$.  It will of course suffice to
  find $l$ such that $\theta^{-l} M$ is globally generated relative to the
  composed morphism.  With this in mind, let $b(M)$ be the function that
  assigns to each $M$ the smallest integer such that
  $\Ext^2(\sO_{e_2},\theta^{-b}M)=0$, and if $b(M)\le 0$, let $a(M)$ be the
  smallest integer such that $\Ext^2(\sO_{e_1},\theta^{-a}\beta_*M)=0$.  If
  $b(M)\le 0$ and $a(M)\le 0$, then we find that $M$ is acyclic and
  globally generated for $\beta_*$, while $\beta_*M$ is acyclic and
  globally generated for $\alpha$, which implies that $M$ is acyclic and
  globally generated for $\alpha\circ\beta$ as required.  It will thus
  suffice to show that $b(\theta^{-1}M)=b(M)-1$ and if $b(M)\le 0$, then
  $a(\theta^{-1}M)=a(M)-1$.

  The claim for $b(M)$ is trivial, so it remains to consider the claim for
  $a(M)$.  We have
  \[
  R\Hom(\sO_{e_1},\theta^{-a}\beta_*\theta^{-1}M)
  \cong
  R\Hom(\sO_{e_1},\theta^{-a}R\beta_*\theta^{-1}M)
  \cong
  R\Hom(L\beta^!\theta^a \sO_{e_1}(-1),M)
  \]
  If $R\Hom(\theta^a \sO_{e_1}(-1),\sO_{x_2})=0$, then we have
  \begin{align}
  R\Hom(L\beta^!\theta^a \sO_{e_1}(-1),M)
  &\cong
  R\Hom(L\beta^*\theta^a \sO_{e_1}(-1),M)\notag\\
  &\cong
  R\Hom(\theta^a \sO_{e_1}(-1),R\beta_* M)\notag\\
  &\cong
  R\Hom(\sO_{e_1},\theta^{-a-1} R\beta_* M),
  \end{align}
  and thus the claim follows.

  It thus remains only to show that
  $R\Hom(\theta^a\sO_{e_1}(-1),\sO_{x_2})=0$, which we may rewrite using
  \[
  R\Hom(\theta^a \sO_{e_1}(-1),\sO_{x_2})
  \cong
  R\Hom(\sO_{e_1}(-1),\sO_{\tau^{-a}x_2}).
  \]
  By assumption, $e_1$ is either not a component of $C_1^+$ or does not
  contain $x_2$, and thus we may choose the divisor $C_2$ containing $x_2$
  so that $e_1$ is not a component of $C_2$.  We then have
  \[
  R\Hom(\sO_{e_1}(-1),\sO_{\tau^{-a}x_2})
  \cong
  R\Hom_{C_2}(\sO_{e_1}(-1)|_{C_2},\sO_{\tau^{-a}x_2}).
  \]
  Now, $\sO_{e_1}(-1)|_{C_2}$ either vanishes (if $x_1$ and $x_2$ came from
  different components of the original curve of points) or equals $x_1$,
  in which case
  \[
  R\Hom(\sO_{e_1}(-1),\sO_{\tau^{-a}x_2})
  \cong
  R\Hom_{C_2}(\sO_{x_1},\sO_{\tau^{-a}x_2})
  \]
  vanishes unless $x_2=\tau^a x_1$.
\end{proof}

We then have the following immediate consequence.

\begin{thm}
  Let $X$ be a Gorenstein surface containing a divisor $C$, and let $x_1$,
  $x_2$ be a pair of points of $C$ which are not in the same orbit under
  $\tau$.  Then the blowup of $X$ in $x_1$ then $x_2$ (as a point in
  $\widetilde{C}$) is isomorphic to the blowup of $X$ in $x_2$ then $x_1$.
\end{thm}

\medskip

It will be useful to have an understanding of what happens when $x_1$ and
the image of $x_2$ {\em are} in the same orbit.  If $x_1$ is a fixed point,
then $e_1$ will be an actual divisor in the first blowup $X_1$, and we can
use its strict transform to control things.  So suppose that $x_1$ is not a
fixed point, but rather that it has an orbit of size $r\in [2,\infty]$, and
thus that the curve of points on $X_1$ locally agrees with $C$, and
similarly (since we are assuming $x_2$ in the same orbit) for the two-fold
blowup.  Applying a suitable equivalence of the form $\_(de_2)$ then lets
us assume that $x_2$ is equal to $x_1=:x$.  Then there is a natural
morphism $\sO_{e_2}(-1)\to \sO_{e_1}(-1)$, and we define
$\sO_{e_1-e_2}(-1)$ to be its cokernel.  Note that since $\sO_{e_1}(-1)$
and $\sO_{e_2}(-1)$ both restrict to $\sO_x$ on $C$, the cone of this
morphism has trivial restriction to $C$, and thus is invariant under the
autoequivalence $\_(C)$.  Since this autoequivalence is relatively ample
for $\beta_*$ and the direct image of the cone is the sheaf
$\sO_{e_1}(-1)$, we conclude that the morphism is in fact injective.  Note
also that we have $\theta\sO_{e_1-e_2}(-1)\cong \sO_{e_1-e_2}(-1)(-C)\cong
\sO_{e_1-e_2}(-1)$.  This object is readily verified to be spherical, and
thus gives rise to an inverse pair of spherical twists.  (This of course
also makes sense when $r=1$, taking the appropriate line bundle on the
strict transform of $e_1$.)

The size of the orbit plays a role via the following fact.

\begin{lem}
  Let $\widetilde{X}$ be the blowup of $X$ in a point $x$ with orbit of size
  $r\in [1,\infty]$.  Then for all $0\le d<r$, one has
  \[
  \Ext^1_{\widetilde{X}}(\sO_{e_1}(-1),\sO_{e_1}(d-1))=
  \Ext^2_{\widetilde{X}}(\sO_{e_1}(-1),\sO_{e_1}(d-1))=0
  \]
  and $\dim\Hom_{\widetilde{X}}(\sO_{e_1}(-1),\sO_{e_1}(d-1))=1$.
\end{lem}

\begin{proof}
  This holds for $d=0$ since $\sO_{e_1}(-1)$ is exceptional.  We proceed by
  induction in $d$.  Using the short exact sequence
  \[
  0\to \sO_{e_1}(d-1)\to \sO_{e_1}(d)\to \sO_{x_d}\to 0
  \]
  (where $x_d$ is an appropriate point of $\widetilde{C}$), we find that
  \[
  R\Hom(\sO_{e_1}(-1),\sO_{e_1}(d-1))\cong
  R\Hom(\sO_{e_1}(-1),\sO_{e_1}(d))
  \]
  unless
  $R\Hom(\sO_{e_1}(-1),\sO_{x_d})\ne 0$, or equivalently unless
  $x_d=x_{-1}$. Since $x_d$ ranges over consecutive points in the orbit of
  $x$, this fails only when $d+1$ is a multiple of $r$.
\end{proof}

Define a sheaf $\sO_{X_2}(ae_1+be_2)$ by
\[
\sO_{X_2}(ae_1+be_2):=
\theta^b \beta^* \theta^{a-b} \alpha^* \theta^{-a} \sO_X.
\]
(This could of course be defined using $\_(-C)$ in place of $\theta$; the
use of $\theta$ is only for notational convenience.)  Our key result
for use below can be viewed as computing certain special cases of the
spherical twists of such sheaves.

\begin{prop}\label{prop:weak_reflection}
  If $0\le b-a\le r$, then there is a short exact sequence
  \[
  0\to \sO_{X_2}(-be_1-ae_2)\to \sO_{X_2}(-ae_1-be_2)\to
  \sO_{e_1-e_2}(-1)^{b-a}\to 0
  \]
\end{prop}

\begin{proof}
  We show this by induction in $b-a$, noting that the base case $a=b$ is
  trivially true.  Since the functor $\_(e_1+e_2)$ preserves
  $\sO_{e_1-e_2}(-1)$ (it agrees with $\_(-C)$ and $\theta$ on $D^b_{\coh}
  X^{\perp}$), we may as well assume that $a=0$.  We thus need to show that
  if the claim holds for some $b<r$, then it holds for $b+1$.

  There is certainly a natural injection $\sO_{X_2}(-(b+1)e_1)\to
  \sO_{X_2}(-be_1)$, and composing with the already constructed map
  $\sO_{X_2}(-be_1)\to \sO_{X_2}(-be_2)$ gives an injection such that the
  cokernel $M_b$ is an extension of $\sO_{e_1-e_2}(-1)^b$ by
  $\beta^*\sO_{e_1}(b)$.  But
  \begin{align}
  \Ext^1(\sO_{e_1-e_2}(-1),\beta^*\sO_{e_1}(b))
  &\cong
  \Ext^1(\theta^{-1}\sO_{e_1-e_2}(-1),\beta^*\sO_{e_1}(b))\notag\\
  &\cong
  \Ext^1(\sO_{e_1-e_2}(-1),\beta^!\theta\sO_{e_1}(b))\notag\\
  &\cong
  \Ext^1_{X_1}(\sO_{e_1}(-1),\sO_{e_1}(b-1))\notag\\
  &=0,
  \end{align}
  and thus this extension splits.  Since
  \[
  \Hom(\sO_{e_1-e_2}(-1),\sO_{e_2}(b))\cong
  \Hom(\sO_{e_1-e_2}(-1),\sO_{e_2}(-1))=0
  \]
  and
  \[
  \Hom(\beta^*\sO_{e_1}(b),\sO_{e_2}(b))
  \cong
  \Hom(\sO_{e_1}(b),\beta_*\sO_{e_2}(b))
  \cong
  k
  \]
  (using the fact that $b<r$, so $\beta_*\sO_{e_2}(b)$ is the direct sum of
  $x_0$,\dots,$x_b$, only one of which has a map from $\sO_{e_1}(b)$), we
  see that there is a unique morphism from $M_b$ to $\sO_{e_2}(b)$, and
  thus the map $\sO_{X_2}(-(b+1)e_1)\to \sO_{X_2}(-be_2)$ factors through
  $\sO_{X_2}(-(b+1)e_2)$.  The cokernel of the resulting map is the sum of
  $\sO_{e_1-e_2}(-1)^b$ and the cokernel of the map $\beta^*\sO_{e_1}(b)\to
  \sO_{e_2}(b)$.  Since $b<r$, we have
  \[
  \beta^*\sO_{e_1}(b)\cong \theta^{-1}\beta^!\theta \sO_{e_1}(b)
  \cong \theta^{-1}\beta^*\sO_{e_1}(b-1),
  \]
  and thus this cokernel is $\theta^{-1}$ of the cokernel for $b-1$, and
  thus by induction is $\theta^{-b-1}\sO_{e_1-e_2}(-1)\cong
  \sO_{e_1-e_2}(-1)$.
\end{proof}

\begin{rem}
  This construction of a morphism $\sO_{X_2}(-be_1-ae_2)\to
  \sO_{X_2}(-ae_1-be_2)$ is somewhat reminiscent of the construction of
  (multivariate) operators of degree $d(s-f)$ in \cite[Prop.~8.7]{elldaha}.
  This is not a coincidence, as if we blow up a point on a noncommutative
  $\P^1\times \P^1$, the result is isomorphic to a two-point blowup of a
  noncommutative plane, and the commutation of blowups symmetry on the
  latter is the same as the exchange of rulings symmetry on the former.
\end{rem}  

\begin{rem}
  When $X$ is rational or rationally quasi-ruled, we can replace $\sO_X$ by
  any line bundle (see Definition \ref{defn:line_bundle} below) by applying
  the corresponding twist autoequivalence before blowing up.  If
  $r=\infty$, we find that every line bundle on $X_2$ fits into such an
  exact sequence (and thus, modulo checking that the other $\Ext$ groups
  vanish, we can compute one of its two spherical twists).  If $r$ is
  finite, this fails, but every line bundle fits into a sequence along the
  above lines with $\sO_{e_1-e_2}(-1)$ replaced by some
  $\sO_{e_1-e_2}(-1-dr)$.
\end{rem}

\medskip

For the remaining cases, showing the pseudo-canonical property will require
some additional facts about sheaves, so we will postpone that to a future
section and consider only the construction of the derived equivalences.
The $\P^1\times \P^1$ case can be dealt with similarly, but the other two
cases (elementary transformations and $F_1$ as a blowup of $\P^2$) require
more complicated modifications of the semiorthogonal decomposition.  To
deal with these cases, it will be helpful to consider different
semiorthogonal decompositions.  In the rational case, the object $\sO_X$ is
exceptional, and thus induces a semiorthogonal decomposition
$(\sO_X^\perp,\langle \sO_X\rangle)$, while for blowups of quasi-ruled
surfaces, the image of $D^b_{\coh} C_0$ is admissible, so induces a
decomposition.  It turns out that in each case, the orthogonal subcategory
is also essentially commutative, in that it appears in the same way as a
subcategory of the derived category of a commutative surface.  Moreover,
the gluing data can also be expressed in commutative terms.

A key observation is that when $X$ has an anticanonical curve $C$, the gluing
data can be expressed as a $R\Hom$ in $D^b_{\coh} C$.  Let $i_*:D^b_{\coh}
C\to D^b_{\coh} X$ be the embedding, with left adjoint $Li^*$ and
associated twist $\_(-C)$.

\begin{lem}
  Let $X$ be a noncommutative surface with an anticanonical divisor $C$.
  If $M,N\in D^b_{\coh} X$ are such that $R\Hom(M,N)=0$, then $R\Hom(N,M)\cong
  R\Hom_C(Li^* N,Li^* M)$.
\end{lem}

\begin{proof}
  By Serre duality, we have $R\Hom_X(N,M(-C)[2])\cong R\Hom_X(M,N)=0$, so that
  \[
  R\Hom_X(N,M)\cong R\Hom_X(N,i_*Li^*M)\cong R\Hom_C(Li^*N,Li^*M)\notag
  \]
  as required.
\end{proof}

We need a slight variation of this in order to deal with iterated blowups
of quasi-ruled surfaces.

\begin{lem}
  Let $X$ be a noncommutative surface with a Serre functor and a curve $C$
  embedded as a divisor, and suppose that ${\cal A}$ is a full subcategory
  of $D^b_{\coh} X$ preserved by $M\mapsto S M(C)$.  Then for any $A\in {\cal
    A}$, $B\in {\cal A}^\perp$, $B'\in {}^\perp{\cal A}$
  \begin{align}
    R\Hom(B,A)&\cong R\Hom_C(Li^*B,Li^*A)\\
    R\Hom(A,B')&\cong R\Hom_C(Li^*A,Li^*B').
  \end{align}
\end{lem}

\begin{proof}
  For the first claim, it again suffices to show that $R\Hom(B,A(-C))=0$.
  By Serre duality, this is isomorphic to
  $R\Hom(A,SB(C))=R\Hom(S^{-1}A(-C),B)=0$ since $S^{-1}A(-C)\in {\cal A}$.
  Similarly, $R\Hom(A,B'(-C))\cong R\Hom(B',SA(C))=0$.
\end{proof}

\begin{rem}
  Note that if $M\mapsto SM(C)$ preserves ${\cal A}$, then it also
  preserves ${\cal A}^\perp$ and ${}^\perp{\cal A}$.
\end{rem}

For a noncommutative rational surface, the curve $Q$ of points is
anticanonical, and thus $M\mapsto SM(Q)$ preserves $\langle \sO_X\rangle$.
For an iterated blowup of a quasi-ruled surface, this can fail, but the
functor always preserves the category $L\rho_0^*D^b_{\coh} C_0$.

\begin{thm}
  Suppose the noncommutative surface $X$ is an iterated blowup of either a
  noncommutative plane or a noncommutative Hirzebruch surface, with
  anticanonical curve $Q$.  Then there is a commutative surface $X'$, an
  embedding $Q\subset X'$ as an anticanonical curve, and a point $q\in
  \Pic^0(Q)$ such that there is an equivalence $\kappa:\sO_X^\perp\cong
  \sO_{X'}^\perp$ satisfying
  \[
  R\Hom_X(M,\sO_X)\cong R\Hom_{X'}(\kappa(M),q)\cong R\Hom_Q(\kappa(M)|_Q,q).
  \]
\end{thm}

\begin{proof}
  The construction of $X$ as an iterated blowup gives rise to an
  exceptional collection
  \[
  \sO_{e_m}(-1),\alpha_m^*\sO_{e_{m-1}}(-1),\cdots,
  \]
  ending with an exceptional collection for the original surface, in turn
  ending with $\sO_X$.  The forward maps in this exceptional collection
  (from which we can reconstruct the full derived category via gluing) can
  be computed via the restrictions to $Q$.  What we need to prove is that
  there is a line bundle $q$ in the identity component of $\Pic(Q)$ such
  that if we tensor the restrictions with $q$ and then replace the last
  term by $\sO_Q$, then the result is the sequence corresponding to an
  exceptional collection in some commutative $X'$ with anticanonical curve $Q$.

  For $X$ a noncommutative plane, the exceptional collection is
  $\sO_X(-2)$, $\sO_X(-1)$, $\sO_X$, with restrictions ${\cal L}_{-6}$,
  ${\cal L}_{-3}$, $\sO_Q$ with ${\cal L}_{-6}$, ${\cal L}_{-3}$ line
  bundles of the given degree such that ${\cal L}_{-3}^{-1}$, ${\cal
    L}_{-6}^{-1}\otimes {\cal L}_{-3}$ are effective, acyclic, and
  numerically equivalent.  The commutative case is that ${\cal L}_{-6}\cong
  {\cal L}_{-3}^2$, and thus if we tensor everything with $q$ and replace
  the last term with $\sO_Q$, the result will corresponding to a
  commutative surface as long as $q\otimes {\cal L}_{-6}\cong q^2\otimes
  {\cal L}_{-3}^2$, or in other words when $q\cong {\cal L}_{-6}\otimes
  {\cal L}_{-3}^{-2}$.  The degree condition ensures that $q$ has degree 0
  on every component, so is in $\Pic^0(Q)$ as required.  Note that the
  commutative surface is the $\P^2$ in which $Q$ is embedded via the line
  bundle $q^{-1}\otimes {\cal L}_{-3}^{-1}$.

  For a noncommutative Hirzebruch surface, we note that $D^b_{\coh} \P^1$
  has an exceptional collection $\sO_{\P^1}(-1),\sO_{\P^1}$, and thus we
  can refine the semiorthogonal decomposition of Theorem
  \ref{thm:semiorth_for_quasiruled} to an exceptional collection.
  Restricting to $Q$ gives a sequence
  \[
    {\cal L}'_{-2},{\cal L}',{\cal L}_{-2},\sO_X
  \]
  which is commutative whenever ${\cal L}'_{-2}\cong {\cal L}'\otimes
  {\cal L}_{-2}$.  Moreover, intersection theory forces $q:={\cal
    L}^{\prime{-}1}\otimes {\cal L}_{-2}^{-1}\otimes {\cal L}'_{-2}$ to be
  degree 0 on every component, and thus be in $\Pic^0(Q)$ as required.  The
  resulting commutative Hirzebruch surface is the one obtained by using
  ${\cal L}_{-2}^{-1}\otimes q^{-1}$ to determine a degree 2 morphism $\pi:Q\to
  \P^1$ and then taking the $\P^1$-bundle corresponding to $\pi_*({\cal
    L}^{\prime{-}1}\otimes q^{-1})$.

  Finally, if we blow up a point, the effect on the sequence of objects in
  $D^b_{\coh} Q$ is to (derived) pull back all of the existing objects to
  the new anticanonical curve $Q'$ and then prepend the appropriate object
  to the beginning of the sequence.  If $p$ is a smooth point of $Q$, this
  is simply a point sheaf, while in general it is given by the (perfect!)
  complex $\sO_{Q'}\to \sO_{Q'}(e)$, which can be computed inside the
  commutative surface.  This complex is trivial on the generic point of the
  blowup $\widetilde{Q}$, and is thus invariant under twisting by $q$,
  while for the other sheaves we note that tensoring by $q$ then pulling
  back is the same as pulling back then tensoring by the pullback of $q$
  (which is again of degree 0 on every component).  We thus find that the
  effect of blowing up $X$ is the same as that of blowing up the
  corresponding point of $X'$ and then pulling back $q$.
\end{proof}

\begin{rems}
  Given any triple $(X',Q,q)$ where $X'$ is a smooth projective rational
  surface, $Q\subset X$ is an anticanonical curve, and $q$ is a line bundle
  in the identity component of $\Pic(Q)$, there is a corresponding
  (dg-enhanced) triangulated category obtained by using $q$ to glue
  $\sO_{X'}^\perp$ to $\sO_{X'}$, such that taking $q=\sO_Q$ recovers
  $\sO_{X'}$.  Of course, to obtain the noncommutative rational surfaces
  themselves, one needs to determine the $t$-structure, which turns out not
  to be possible without imposing some additional structure on $X'$.  This
  can, however, be done inductively given a choice of blowdown structure on
  $X'$ (i.e., a sequence of monoidal transformations blowing it down to
  either a Hirzebruch surface or $\P^2$), by insisting that the
  corresponding functors be pseudo-canonical.  As usual, it appears to be
  hard to prove independently that this gives a $t$-structure, let alone
  that it is the derived category of its heart.  It follows, of course,
  from the above considerations, and thus we can use this to construct a
  family of noncommutative surfaces parametrized by the relative $\Pic^0$
  of the universal anticanonical curve of the moduli stack of anticanonical
  rational surfaces with blowdown structure, see below.
\end{rems}

\begin{rems}
  In the fully noncommutative case, the gluing data is again in the form of
  a pair of adjoint functors, since $q\in \sO_{X'}^\perp$.  On the other hand,
  when $q=\sO_Q$, the above description is no longer of that form.  This
  can be rectified, however, since for $M\in \sO_{X'}^\perp$, one has
  \[
  R\Hom_{X'}(M,\omega_{X'})\cong R\Hom_{X'}(\sO_X,M[2])^*=0,
  \]
  and thus
  \[
  R\Hom_{X'}(M,\sO_{X'})\cong R\Hom_{X'}(M,\cone(\tr)[-1])
  \]
  where $\tr$ is the trace map $\sO_{X'}\to \omega_{X'}[2]$.  Since
  $R\Hom(\sO_{X'},\tr)$ is the identity and $\cone(\tr)[-1]\in
  \sO_{X'}^\perp$, this induces the desired functor.
\end{rems}

\begin{rems}
  It follows from \cite[Thm. 8.5]{KuznetsovA:2009} that the Hochschild
  cohomology of $\sO_{X'}^\perp$ is the hypercohomology of $\sO_{X'}\oplus
  T_{X'}$, so that (since $\sO_{X'}$ is acyclic) the deformation theory of
  $\sO_X^\perp$ is given by the commutative deformation theory of $X'$.
  Similarly, the deformations of the gluing functor (i.e., the deformations
  of $X$ with $\sO_{X}^\perp$ fixed) are controlled by the deformations
  of the image of $q$ in $\sO_{X'}^\perp$.  Thus if $q\not\cong \sO_Q$,
  then the deformation theory is controlled by the hypercohomology of
  $\sO_Q\oplus \sO_Q(Q)$, while for $X=X'$, the deformation theory is
  controlled by the hypercohomology of $\wedge^2 T_{X'}\cong
  \omega_{X'}^{-1}$.  This implies at the very least that any first-order
  deformation of $\perf(X)$ is a noncommutative rational surface of the
  above form, and the agreement of obstruction spaces suggests that
  this should extend to arbitrary formal deformations.
\end{rems}

In particular, any isomorphism $X'\cong Y'$ of rational surfaces with
blowdown structures induces a derived equivalence between the corresponding
noncommutative surfaces, giving the desired derived equivalences
corresponding to $\P^1\times \P^1\cong \P^1\times \P^1$ and the
representation of $F_1$ as a blowup of $\P^2$.  In each case, the resulting
equivalence will be an abelian equivalence as long as the functor $R\Gamma$
is pseudo-canonical, since this functor is clearly preserved by the equivalence.

\begin{prop}
  The image of the restriction map $Li^*:\sO_X^\perp\to \perf(Q)$ generates
  $\perf(Q)$.
\end{prop}

\begin{proof}
  Without loss of generality, $X$ is a commutative rational surface.  As in
  Lemma \ref{lem:blowup_generates_correct_curve}, we find that
  $Li^*\perf(X)$ generates $\perf(Q)$, so it will suffice to show that
  $\sO_Q=Li^*\sO_X$ is in the Karoubian triangulated subcategory
  generated by $Li^*\sO_X^\perp$.  Define an element $M\in \sO_X^\perp$
  by the distinguished triangle
  \[
  M\to \sO_X\to \theta\sO_X[2]\to
  \]
  coming from the natural element of $\Ext^2(\sO_X,\theta \sO_X)$.
  Since $\Ext^2$ vanishes for line bundles on $Q$, this map is annihilated
  by $Li^*$, producing a splitting:
  \[
  Li^*M\cong \theta \sO_Q[1]\oplus \sO_Q.
  \]
  But this is what we needed to show.
\end{proof}

\smallskip

For quasi-ruled surfaces, we note the following.

\begin{lem}
  Let $X$ be a quasi-ruled surface over $C_0,C_1$, and let $Q$ be the
  corresponding curve of points.  Then there are morphisms $\pi_i:Q\to C_i$
  such that $Li^*L\rho_0^*\cong L\pi_0^*$ and
  \[
  Li^*L\rho_1^*\cong {\cal L}_s^{-1}\otimes L\pi_1^*.
  \]
  for some line bundle ${\cal L}_s$ having degree $1$ on every vertical
  component.
\end{lem}

\begin{proof}
  Let $\bar\pi_i:\bar{Q}\to C_i$ be the natural degree 2 morphisms, let
  $\pi:Q\to\bar{Q}$ be the morphism contracting the vertical fibers, and
  define $\pi_i=\pi\circ \bar\pi_i$.  Then for any line bundle ${\cal L}$
  on $C_i$,
  \[
  Li^*\rho_i^*{\cal L}\cong Li^*\rho_i^*\sO_{C_i}\otimes \pi_i^*
  {\cal L},
  \]
  since twisting by ${\cal L}$ gives another quasi-ruled surface.  The
  claim for $i=0$ follows immediately, and for $i=1$ reduces to showing
  that $i^*\rho_1^*\sO_{C_i}$ is an invertible sheaf having degree $-1$
  on every vertical component.  Writing this object as ${\cal L}_s^{-1}$,
  we find
  \[
  R\Hom_X(\rho_1^*M,\rho_0^*N)
  \cong
  R\Hom_Q(\pi_1^*M,{\cal L}_s\otimes \pi_0^*N)
  \cong
  R\Hom_{C_1}(M,N\otimes_{C_0} R(\pi_0\times\pi_1)_*{\cal L}_s),
  \]
  so that $R(\pi_0\times\pi_1)_*{\cal L}_s$ is the sheaf bimodule from
  which we constructed the quasi-ruled surface.  Since $(\pi_0\times
  \pi_1)_*$ has at most $1$-dimensional fibers over its image, it has
  cohomological dimension $\le 1$, and thus the standard spectral sequence
  implies that ${\cal L}_s$ must itself be a sheaf.  We similarly find that
  the adjoint sheaf bimodule is given by
  \[
  R(\pi_0\times\pi_1)_*({\cal L}_s^{-1}\otimes \pi_0^!\sO_{C_0})
  \]
  so that ${\cal L}_s^{-1}$ is also a sheaf.  Thus ${\cal L}_s$ is
  reflexive, and since it has rank 1 away from the vertical fibers, must
  have rank 1.  If ${\cal L}_s$ had degree $\ge 2$ on any vertical fiber, it
  would have a subsheaf isomorphic to the structure sheaf of that fiber,
  and thus its direct image would fail to be pure $1$-dimensional.  On the
  other hand, ${\cal L}_s^{-1}$ must have negative degree on every vertical
  fiber, since the pullback of a point sheaf agrees with that point sheaf
  in cohomological degree $\ge 0$, and thus the corresponding map
  $i^*\pi_1^*\sO_{p_1}\to i^*\pi_0^*\sO_{p_0}$ must be injective
  with cokernel $\sO_p$.
\end{proof}

\begin{prop}
  Let $X$ be an iterated blowup of a noncommutative quasi-ruled surface
  over $(C_0,C_1)$, with curve of points $Q$, and let $L\rho_0^*:D^b_{\coh}
  C_0\to D^b_{\coh} X$ be the corresponding embedding.  Then there is a
  commutative rationally ruled surface $\rho:X'\to C_1$ such that there is
  an equivalence
  \[
  \kappa:(\rho_0^*D^b_{\coh} C_0)^\perp\cong (\rho^*D^b_{\coh} C_1)^\perp
  \]
  satisfying $\kappa(M)|_Q\cong i^* M$, and thus
  \[
  R\Hom(M,\rho_0^*N)\cong R\Hom_Q((\kappa M)|_Q,\pi_0^*N).
  \]
\end{prop}

\begin{proof}
  If $X$ is a noncommutative quasi-ruled surface, then we can take
  $X'$ to be $\P(\pi_{1*}{\cal L}_s)$, so that $Q$ embeds in $X'$
  in such a way that $\sO_{X'}(s)|_Q\cong {\cal L}_s$.  Blowups then work
  as in the rational case.
\end{proof}

\begin{rem}
  Note that when $X$ is rational, the commutative surface obtained in this
  way will not in general be isomorphic to that associated to
  $\sO_X^\perp$, although this just involves a minor twist: the above
  equivalence extends to an isomorphism $\rho_0^*\sO_{\P^1}(1)^\perp\cong
  \rho^*\sO_{\P^1}(1)^\perp$ that still respects the restriction to $Q$,
  and there is an abelian equivalence taking $\rho_0^*\sO_{\P^1}(1)$ to the
  structure sheaf of a slightly different noncommutative rational surface.
\end{rem}

Again, this gives derived equivalences corresponding to any isomorphism of
commutative ruled surfaces, giving the desired equivalences corresponding to
elementary transformations.  Since this respects the projection
$\rho_{0*}\alpha_*$ to $\im\rho_0^*$, it will be an abelian equivalence
whenever that projection is pseudo-canonical.

There is a similar statement associated to iterated blowups of a
noncommutative surface: if $\pi:X\to Y$ is such a blowup such that the
curve of points $Q\subset Y$ is a surface curve, then there is a
commutative blowup $\pi':X'\to Y'$ with an isomorphism $(\pi^*D^b_{\coh}
Y)^\perp \cong (\pi^*D^b_{\coh} Y')^\perp$ respecting restriction to $Q$.
Indeed, we need simply embed $Q$ in a smooth projective surface $Y'$ and
then perform the corresponding sequence of blowups.  In particular, the
results in \cite{poisson} on the category of sheaves/objects annihilated by
$R\pi_*$ carry over immediately to the noncommutative case.  (In fact, it
suffices for $Q$ to be Gorenstein, as it still locally embeds in a surface,
and the categories only depend on how $Q$ behaves near the points being
blown up.)

\section{Derived equivalences}

Of course, we can also use these semiorthogonal decompositions to construct
derived equivalences which are not abelian equivalences.  The simplest
instance of such a derived equivalence is the following duality.  Note that
in the following proposition, only the category $D^b_{\coh} X_q$ is
determined by the data, as we have not chosen a blowdown structure on $X'$.
For quasi-ruled surfaces, such a duality can be obtained from the adjoint
involution discussed above, and the duality we give is, in fact, closely
related to that involution.

\begin{prop}
  Let $(X',Q)$ be an anticanonical rational surface, and for each $q\in
  \Pic^0(Q)$, let $X_q$ denote the corresponding (derived) noncommutative
  rational surface.  Then there is a family of equivalences $\ad_q:(D^b_{\coh}
  X_q)^{\text{op}}\cong D^b_{\coh} X_{q^{-1}}$ such that
  $\ad_{q^{-1}}\ad_q\cong \text{id}$.
\end{prop}

\begin{proof}
  On $X'$ we have the natural (Cohen-Macaulay) duality functor
  $M^D:=R\sHom(M,\omega_{X'})$.  This takes $\sO_{X'}$ to $\omega_{X'}$,
  and if $M\in \sO_{X'}^\perp$, then
  \[
  R\Hom_{X'}(\sO_{X'},M^D)\cong R\Hom_{X'}(M^D,\omega_{X'}[2])^*
  \cong R\Hom(\sO_{X'},M[2])^* = 0,
  \]
  so that $\_^D$ restricts to a duality on $\sO_{X'}^\perp$.  We claim that
  there is a duality $\ad_q:(D^b_{\coh} X_q)^{\text{op}}\to D^b_{\coh}
  X_{q^{-1}}$ such that $\ad_q\sO_{X_q}\cong \theta \sO_{X_{q^{-1}}}$ and
  \[
  \ad_q \kappa_q^{-1}M = \kappa_{q^{-1}}^{-1} M^D.
  \]
  This transforms the semiorthogonal decomposition as
  $(\sO_{X_q}^\perp,\sO_{X_q})\to (\theta
  \sO_{X_{q^{-1}}},\sO_{X_{q^{-1}}}^\perp)$, so it remains only to show
  that it respects the gluing functor.  In other words, we need to show
  that
  \[
  R\Hom(\ad_q\sO_{X_q},\ad_qM)\cong R\Hom(M,\sO_{X_q})
  \]
  for any $M\in \sO_{X_q}^{\perp}$.  Since
  \[
  R\Hom_{X_{q^{-1}}}(\theta \sO_{X_{q^{-1}}},\kappa_{q^{-1}}^{-1} M^D)
  \cong
  R\Hom_{X_{q^{-1}}}(\cone(\tr_{q^{-1}})[-2],\kappa_{q^{-1}}^{-1} M^D),
  \]
  where $\tr_{q^{-1}}:\sO_{X_{q^{-1}}}\to \theta \sO_{X_{q^{-1}}}[2]$
  is the natural trace map associated to the Serre functor, and
  $\cone(\tr_{q^{-1}})\in \sO_{X_{q^{-1}}}^\perp$, we find that
  \begin{align}
  R\Hom_{X_{q^{-1}}}(\theta \sO_{X_{q^{-1}}},\kappa_{q^{-1}}^{-1} M^D)
  &\cong
  R\Hom_{X'}(\kappa_{q^{-1}}\cone(\tr_{q^{-1}})[-2],M^D)\notag\\
  &\cong
  R\Hom_{X'}(M,(\kappa_{q^{-1}}\cone(\tr_{q^{-1}})[-2])^D),
  \end{align}
  so that it remains only to compute
  $\kappa_{q^{-1}}\cone(\tr_{q^{-1}})[-2]$.
  For any $M\in \sO_{X'}^\perp$, we have
  \[
  R\Hom_{X'}(\kappa_{q^{-1}}^{-1}M,\cone(\tr_{q^{-1}})[-2])
  \cong
  R\Hom_{X'}(\kappa_{q^{-1}}^{-1}M,\sO_{X_{q^{-1}}}[-1])
  \cong
  R\Hom_{X'}(M,q^{-1}[-1]),
  \]
  so that $\kappa_{q^{-1}}\cone(\tr_{q^{-1}})[-2]\cong q^{-1}[-1]$
  (adjusted accordingly when $q=\sO_Q$) and dualizing gives $q$ as
  required.

  Since $\kappa_q \ad_{q^{-1}}\ad_q\kappa_q^{-1}=\_^{DD}\cong \text{id}$,
  to show that this is an involution it remains only to check that
  $\ad_{q^{-1}}\theta \sO_{X_{q^{-1}}}\cong \sO_{X_q}$.  But this follows
  from the fact the Serre functor is intrinsic, and thus any contravariant
  equivalence satisfies $S \ad_q S\cong \ad_q$.
\end{proof}

This duality respects the blowup and ruled surface decompositions.

\begin{lem}
  Let $\alpha:\tilde{X}'\to X'$ be a monoidal transformation centered at a
  point of $Q$, and let $\tilde{X}_{q}$, $X_{q}$ be the corresponding
  noncommutative families, with $\tilde{X}_q$ the blowup of $X_q$.  Then
  $\ad_q\sO_e(-1)\cong \sO_e(-1)[-1]$, $\ad_q L\alpha^*M\cong
  L\alpha^!\ad_q M$, and $\ad_q R\alpha_* M\cong \alpha_* \ad_q M$.
\end{lem}

\begin{proof}
  Since $\kappa_q \sO_e(-1)\cong \sO_e(-1)$, the action of the duality on
  $\sO_e(-1)$ reduces to the commutative case.  For the action on
  $L\alpha^*M$, we first note that if $M\in \sO_{X_q}^\perp$, then
  $L\alpha^*M\in \sO_{\tilde{X}_q}^\perp$, and thus
  $\ad_q L\alpha^*M$ can be computed inside $\tilde{X}'$, giving
  \[
  \kappa_{q^{-1}} \ad_q L\alpha^* M
  \cong
  \ad_q L\alpha^* \kappa_q M
  \cong
  L\alpha^! \ad_q\kappa_q M
  \cong
  L\alpha^! \kappa_{q^{-1}} \ad_q M
  \cong
  \kappa_{q^{-1}} L\alpha^! \ad_q M.
  \]
  Since
  \[
  \ad_q L\alpha^*\sO_{X_q}\cong \ad_q \sO_{\tilde{X}_q}\cong \theta
  \sO_{\tilde{X}_{q^{-1}}} \cong L\alpha^!\theta \sO_{X_{q^{-1}}} \cong
  L\alpha^! \ad_q \sO_{X_q},
  \]
  the action on $L\alpha^*M$ is as described.  Finally, we find
  \begin{align}
  R\Hom(N,R\alpha_* \ad_q M)
  &\cong
  R\Hom(L\alpha^* N,\ad_q M)\notag\\
  &\cong
  R\Hom(M,\ad_q L\alpha^* N)\notag\\
  &\cong
  R\Hom(M,L\alpha^! \ad_q N)\notag\\
  &\cong
  R\Hom(R\alpha_* M,\ad_q N)\notag\\
  &\cong
  R\Hom(N,\ad_q R\alpha_* M)
  \end{align}
  so that $R\alpha_* \ad_q M\cong \ad_q R\alpha_* M$ as required.
\end{proof}

\begin{lem}
  Let $(X',Q')$ be an anticanonical Hirzebruch surface with $\rho:X'\to
  \P^1$.  Then on $X_q$ one has $\ad_q \rho_1^*M\cong \rho_1^* M^D$ and
  $\ad_q \rho_0^*M\cong \theta \rho_0^* \theta^{-1} M^D$.
\end{lem}

\begin{proof}
  Since $\im\rho_1^*\subset \sO_{X_q}^\perp$, we can compute the duality
  inside $\im\rho_1^*$ on $X'$, where it becomes
  \[
  \ad_q (\rho^*M(-s))\cong \rho^*(M^D)(-s),
  \]
  easily checked on line bundles.  Similarly, a commutative computation gives
  $\ad_q \rho_0^*\sO_{\P^1}(-1)\cong \theta \rho_0^*\theta^{-1} \sO_{\P^1}(-1)$,
  and the claim for $\sO_{\P^1}$ follows from the value of $\ad_q$ on $\theta
  \sO_X$.
\end{proof}

This suggests that there should be a similar duality on a general iterated
blowup of a quasi-ruled surface.  Although this could in principle be
constructed using the $\rho_0^*D^b_{\coh} C_0$-based semiorthogonal
decomposition, the calculation involves some somewhat tricky adjoint
computations, and thus we use an inductive approach instead.

\begin{prop}
  Let $C_0$, $C_1$ be smooth projective curves, let ${\cal E}$ be a sheaf
  bimodule of birank $(2,2)$ on $C_0\times C_1$, and let $X$ be the
  corresponding noncommutative $\P^1$-bundle.  Similarly, let $X^{\ad}$ be
  the noncommutative $\P^1$-bundle associated to the sheaf bimodule
  $\sExt^1_{C_0\times C_1}({\cal E},\omega_{C_0\times C_1})$.  Then
  $(D^b_{\coh} X)^{\text{op}}\cong D^b_{\coh} X^{\ad}$
\end{prop}

\begin{proof}
  Per the rational case, we want a duality defined in terms of the duality
  on $C_0$, $C_1$ by
  \[
  \ad \rho_1^*M\cong \rho_1^*\sHom(M,\omega_{C_1})\qquad\text{and}\qquad
  \ad \theta \rho_0^*M\cong \rho_0^*\sHom(M,\sO_{C_0}).
  \]
  Since $(\im\theta \rho_0^*,\im\rho_1^*)$ gives a semiorthogonal
  decomposition, this certainly defines a contravariant equivalence to {\em
    some} glued triangulated category $X^{\ad}$, so we just need to show that
  the gluing map corresponds to the given sheaf bimodule.  We compute
  \begin{align}
  R\Hom_{X^{\ad}}(\rho_1^* M,\rho_0^*N)
  &\cong
  R\Hom_{X^{\ad}}(\ad \rho_1^*\sHom(M,\omega_{C_1}),\ad \theta
  \rho_0^*\sHom(N,\sO_{C_0}))\notag\\
  &\cong
  R\Hom_X(\theta \rho_0^*\sHom(N,\sO_{C_0}),\rho_1^*\sHom(M,\omega_{C_1})).
  \end{align}
  Using the functorial distinguished triangle
  \[
  \theta \rho_0^*\theta^{-1}\to \rho_1^* (\_\otimes_{\sO_{C_0}}{\cal
    E})\to \rho_0^*\to,
  \]
  we find that
  \begin{align}
  R\Hom_X(\theta \rho_0^*\sHom(N,\sO_{C_0}),{}&\rho_1^*\sHom(M,\omega_{C_1}))\notag\\
  &\cong
  R\Hom_X(\rho_1^* (\sHom(N,\sO_{C_0})\otimes_{\sO_{C_0}}{\cal E}),
  \rho_1^* \sHom(M,\omega_{C_1}))\notag\\
  &\cong
  R\Hom_{C_1}((\sHom(N,\sO_{C_0})\otimes_{\sO_{C_0}}{\cal
    E}),\sHom(M,\omega_{C_1})\notag\\
  &\cong
  R\Hom_{C_1}(M,\sHom(\sHom(N,\sO_{C_0})\otimes_{\sO_{C_0}}{\cal
    E},\omega_{C_1})).
  \end{align}
  Using the definition of the tensor product of sheaf bimodules, the
  functor being applied to $N$ has the form
  \begin{align}
  N
  \mapsto
  \sHom(\pi_{1*}(\pi_0^* \sHom(N,\sO_{C_0})\otimes {\cal E}),\omega_{C_1})
  &\cong
  \pi_{1*}\sExt^1(\pi_0^* \sHom(N,\sO_{C_0})\otimes {\cal E},\omega_{C_0\times C_1})\notag\\
  &\cong
  \pi_{1*}(\pi_0^* N\otimes \sExt^1({\cal E},\omega_{C_0\times C_1}))\notag\\
  &=
  N\otimes_{\sO_{C_0}} \sExt^1({\cal E},\omega_{C_0\times C_1})
  \end{align}
  as required.
\end{proof}
  
\begin{rem}
  This duality is of course related to the adjoint involution defined
  above in terms of the sheaf $\Z$-algebra: any module over the
  $\Z$-algebra has a dual which is a module over the opposite algebra, and
  an appropriate degree shift and twist by a line bundle recovers the
  duality.  In particular, when applied to a $1$-dimensional sheaf
  corresponding to a difference or differential equation, this is just a
  (cohomological) shift of the usual duality (i.e., coming from the dual
  representation of $\GL_n$ or $\mathfrak{gl}_n$ as appropriate).
\end{rem}

This duality also respects the curve of points, in the following way.

\begin{prop}
  Let $X$ be a quasi-ruled surface with curve of points $Q$, and let
  $X^{\ad}$ be its dual.  Then the curve of points of $X^{\ad}$ can be
  identified with $Q$ in such a way that $\ad[1]$ restricts to the
  Cohen-Macaulay duality on $D^b_{\coh} Q$.
\end{prop}

\begin{proof}
  Any $S$-point $p$ of $X$ arises from a short exact sequence
  \[
  0\to \rho_1^*\sO_{p_1}\to \rho_0^*\sO_{p_0}\to \sO_p\to 0.
  \]
  Applying $\ad$ and computing the relevant commutative duals gives
  a distinguished triangle
  \[
  \ad \sO_p\to \theta \rho_0^*\sO_{p_0}[-1]\to \rho_1^*\sO_{p_1}[-1]\to.
  \]
  It follows from this that $(\ad \sO_p)[2]$ is an $S$-point sheaf on
  $X^d$, and thus that the corresponding functors (so moduli
  schemes) are isomorphic.  Moreover, one has
  \[
  \ad \sO_p\cong \sO_p[-2]\cong \sHom_Q(\sO_p,\omega_Q)[-1]
  \]
  so that $\ad[1]$ agrees with the Cohen-Macaulay duality on point sheaves.
  It then suffices to verify that it behaves correctly on the structure
  sheaf.  We have the short exact sequence
  \[
  0\to \sO_X(-Q)\to \sO_X\to \sO_Q\to 0
  \]
  and the adjoint gives
  \[
  0\to \theta \sO_{X^{\ad}}\to \theta \sO_{X^{\ad}}(Q)\to \ad\sO_Q[1]\to 0,
  \]
  But $\theta \sO_X(Q)|_Q\cong \omega_Q$ by Proposition
  \ref{prop:Serre_for_quasiruled}.
\end{proof}

This claim about the restriction is crucial to let us make blowups work.
  
\begin{lem}
  Let $X$ be a noncommutative surface with curve of points $Q$ such that
  $X$ is generated by weak line bundles along $Q$, and let $\tilde{X}$ be
  the blowup of $X$ at $p$.  If $X^{\ad}$ is another noncommutative surface
  with the same curve of points and there is a duality $\ad:(D^b_{\coh}
  X)^{\text{op}}\to D^b_{\coh} X^{\ad}$ such that $\ad(M(-Q))\cong (\ad M)(Q)$
  and $\ad[1]$ restricts to the Cohen-Macaulay duality on $Q$, then there
  is a duality $\widetilde\ad:(D^b_{\coh} \tilde{X})^{\text{op}}\to D^b_{\coh}
  \tilde{X}^{\ad}$, where $\tilde{X}^{\ad}$ is the blowup of $X^{\ad}$ at $p$,
  and $\widetilde\ad[1]$ satisfies the same conditions with respect to the new
  curve of points.
\end{lem}

\begin{proof}
  Following the rational case, we aim to define $\widetilde\ad$ so that
  \[
  \widetilde\ad \sO_e(-1)\cong \sO_e(-1)[-1]\qquad\text{and}\qquad
  \widetilde\ad L\alpha^! M\cong L\alpha^* \ad M.
  \]
  Consistency of the gluing reduces to computing $\ad$ on the appropriate
  point sheaf, which follows by a computation in $Q$.

  It remains only to show that $\widetilde\ad[1]$ interacts correctly with
  the new curve of points $Q^+$.  That $\widetilde\ad(M(-1))\cong
  (\widetilde\ad M)(1)$ follows from the description of $M\mapsto \theta
  M(1)$ in terms of $\theta M(-Q)$ and the fact that $\widetilde\ad$ and
  $\ad$ both invert $\theta$ while $\ad$ inverts $\_(-Q)$.  Any line
  bundle on $Q$ pulls back via $\alpha^*$ to a line bundle on $Q^+$, and
  one has
  \[
  (\ad \alpha^* {\cal L})[1]
  \cong
  \alpha^! (\ad {\cal L})[1]
  \cong
  \alpha^! \sHom({\cal L},\omega_Q)
  \cong
  \alpha^* \sHom({\cal L},\omega_Q) \otimes \sO_{Q^+}(e).
  \]
  This depends only on the commutative morphism $\alpha:Q^+\to Q$, and may
  thus be computed inside a commutative surface, where it reduces to the
  corresponding fact for the Cohen-Macaulay dual.

  It follows more generally that for any $d\in \Z$, one has
  \[
  \ad (\alpha^*{\cal L}(d))[1]
  \cong
  \sHom(\alpha^*{\cal L},\omega_{Q^+})(-d)
  \cong
  \sHom(\alpha^*{\cal L}(d),\omega_{Q^+})
  \]
  But the line bundles on $Q^+$ of this form are ample, and thus any object
  $M$ in $D^b_{\coh} Q^+$ is quasi-isomorphic to a right-bounded complex in
  which the terms are sums of such line bundles.  This complex is acyclic
  for $\ad[1]$, so $\ad M[1]$ is quasi-isomorphic to the left-bounded
  complex obtained by applying $\ad[1]$ term-by-term.  It follows that $\ad
  M[1]\cong \sHom_{Q^+}(M,\omega_{Q^+})$ as required.
\end{proof}

This gives rise to an inductive construction of a duality $\ad$ on any
iterated blowup of a quasi-ruled surface or noncommutative plane, which
agrees with the previous definition in the rational surface case.

\smallskip

Although this duality cannot be expected to be exact (any more than the
Cohen-Macaulay duality it deforms), we can in fact control fairly well how
it interacts with the $t$-structure.  The key idea is that there is a
particularly nice class of generators of $\qcoh X$ that is preserved by
duality.

\begin{defn}\label{defn:line_bundle}
  A {\em line bundle} on a rational or rationally quasi-ruled surface is a
  member of the smallest class of sheaves on such surfaces such that (a)
  any sheaf isomorphic to a line bundle is a line bundle, (b) on a
  noncommutative plane, $\sO_X$, $\sO_X(-1)$ and $\sO_X(-2)$ are line
  bundles, (c) on a quasi-ruled surface, $\rho_d^*L$ is a line bundle for
  any line bundle $L\in C_d$, and (d) the image of a line bundle under
  $\theta$ or $\alpha_i^*$ is a line bundle.
\end{defn}

\begin{rem}
  This differs from other notions of line bundles in the literature; in
  addition to the weak notion considered above (the restriction to $Q$ is a
  line bundle), one might also consider those weak line bundles which are
  reflexive under $h^0\ad$ (this was considered in
  \cite{NevinsTA/StaffordJT:2007} for noncommutative planes).  The above
  notion has the merit that the corresponding family of moduli spaces is
  flat; in particular, in a sufficiently large family of rationally ruled
  surfaces, the line bundles are deformations of line bundles on the
  commutative fibers of the family.  Although it in principle depends on
  the way in which we described the surface as an iterated blowup, this is
  not the case, see Corollary \ref{cor:lbs_are_intrinsic} below.  It is
  unclear whether there is a corresponding notion for vector bundles (i.e.,
  as a subclass of reflexive coherent sheaves).
\end{rem}

One key fact about line bundles is that there is a notion of twisting by
line bundles.  To be precise, for any line bundle $L$ on $X$, there is
another rational (resp. rationally quasi-ruled) surface $X'$ such that
there is a Morita equivalence $\coh X\cong \coh X'$ taking $L$ to
$\sO_{X'}$, taking point sheaves to point sheaves, and taking line
bundles to line bundles.  This is trivial for noncommutative planes, is
straightforward for quasi-ruled surfaces, and follows by an easy
induction in general.  (For a blowup, we simply take the composition of a
power of $\theta$ with the result of applying such an equivalence to the
original surface, making sure to keep track of the point being blown up.)
When $X$ is commutative, this is agrees with the usual notion of twisting
by a line bundle, but in general this twist is not an autoequivalence.

Line bundles give an alternate description of the $t$-structure.

\begin{lem}
  An object $M\in D_{\qcoh}(X)$ is in $D_{\qcoh}(X)^{\ge 0}$ iff
  $\Hom(L[p],M)=0$ whenever $p$ is a negative integer and $L$ is a line
  bundle.
\end{lem}

\begin{proof}
  This is automatic for a plane, while on a quasi-ruled surface, this
  follows via adjunction from the corresponding fact for line bundles on
  commutative curves; for rationally quasi-ruled surfaces, it follows by an
  easy induction.
\end{proof}

We also have the following important fact, which in particular shows that
line bundles generate $\coh X$, so $\qcoh X$.

\begin{prop}
  If $M\in \coh(X)$, then there is a line bundle $L$ such that
  $\Ext^p(L,M)=0$ for $p>0$ and $L\otimes \Hom(L,M)\to M$ is surjective.
\end{prop}

\begin{proof}
  This is certainly true by construction for planes and quasi-ruled
  surfaces and for blowups follows by an easy induction from the fact that
  for $l\gg 0$, $\alpha_{m*}\theta^{-l}M$ is a sheaf and
  $\alpha_m^*\alpha_{m*}\theta^{-l}M\to M$ is globally generated.  Indeed,
  if $L'$ is the line bundle satisfying the conclusion for the sheaf
  $\alpha_{m*}\theta^{-l}M$ on $X_{m-1}$, then we may take $L=\theta^l
  \alpha_m^* L'$.
\end{proof}

\begin{rem}
  We will prove a number of refinements of this fact below.  One easy
  refinement is that we need not use every line bundle here; in particular,
  if we fix points $x_0\in C_0$, $x_1\in C_1$, we need only use those
  bundles coming from $\sO_{C_d}(nx_d)$ for $d,n\in \Z$.  (This refinement
  also applies to the description of the $t$-structure.)
\end{rem}
 
\begin{rem}
  More generally, given a finite collection of coherent sheaves, we can
  find a single line bundle that acyclically globally generates all of
  them: simply apply the Proposition to the direct sum.
\end{rem}

Since $\ad \theta\cong \theta^{-1}\ad$ and $\ad \alpha^*\cong \alpha^!\ad$,
it is easy to verify that for any line bundle $L$, $\ad L$ is also a line
bundle.  We may thus view the above derived duality as the derived functor
of a contravariant functor $\ad$ of abelian categories, and will thus in
the sequel denote it as $R\ad$, particularly when applying it to sheaves.

\begin{cor}
  The duality $\ad$ is left exact of homological dimension $\le 2$.
\end{cor}

\begin{proof}
  Since line bundles generate $\coh X$ and $\ad$ takes line bundles to line
  bundles, it is left exact.  Now, for any line bundle $L$ and coherent
  sheaf $M$, we have
  \[
  R\Hom(L,R\ad M)\cong R\Hom(M,\ad L)\cong R\Hom(\ad L,\theta M[2])^*,
  \]
  and thus $h^p R\Hom(L,R\ad M)=0$ for $p\notin \{0,1,2\}$.  For each $p'$,
  we may choose $L$ so that the first $p'+1$ cohomology sheaves of $R\ad M$
  are acyclically globally generated by $L$.  But then the spectral
  sequence gives $h^i R\Hom(L,R\ad M) = \Hom(L,R^i\ad M)$ for $0\le i<p'$,
  and thus $R^i\ad M=0$ for $2\le i<p'$.  Since this is true for all $p'$,
  we conclude that $\ad$ has homological dimension 2 as required.
\end{proof}

\medskip

The other family of true derived equivalences we consider are related to
derived autoequivalences of commutative elliptic surfaces.  If $C$ is a
smooth curve and $X/C$ is the minimal proper regular model of an elliptic
curve over $k(C)$, then there is an associated action of $\SL_2(\Z)$ on
$D^b_{\coh} X$.  If $X$ has only smooth fibers, this is just a relative
version of the action of $\SL_2(\Z)$ on the derived category of an elliptic
curve, and in general a similar construction works (albeit with some issues
of nonuniqueness when $X$ has reducible fibers).

Since the deformation theory of an abelian category depends only on the
associated derived category, the autoequivalences of the derived category
act on the space of infinitesimal deformations.  It is therefore not
unreasonable to expect this to extend to an action on algebraic
deformations, such that two deformations in the same orbit should be
derived equivalent.  This is of course purely heuristic, but works for both
types of elliptic surface that have a rational ruling, and thus admit
noncommutative deformations via the above construction.

A rationally ruled surface over a curve of genus $>1$ does not contain any
curves of genus 1 (a vertical curve has genus 0, while a horizontal curve
maps to the base of the ruling, so has genus $>1$).  If the base curve has
genus 1, then every fiber of the elliptic surface contains a smooth genus 1
curve isogenous to the base, and thus the minimal proper regular model must
be a constant family, so that the only rationally ruled case is $E\times
\P^1$.  In the genus 0 case, we have a rational elliptic surface, and it is
well-known that such a surface, if relatively minimal, must have $K^2=0$,
with the elliptic fibration given by the anticanonical pencil $|-K|$.
These cases then deform to (a) noncommutative ruled surfaces over $E$ in
the same component of the moduli stack as $E\times \P^1$, and (b)
noncommutative rational surfaces with $\theta\sO_Q\in \Pic^0(Q)$.

For the genus 1 case, we recall that a sheaf bimodule on $E\times E$
corresponding to a noncommutative ruled surface necessarily has reducible
(or nonreduced) support, and thus may be viewed as an extension of two rank
$(1,1)$ sheaf bimodules, both algebraically equivalent to the diagonal.
The moduli space of such rank $(1,1)$ sheaf bimodules can itself be
identified with $E\times \Pic(E)$: the bimodule is given by a line bundle
in $\Pic^0(E)\cong E$ supported on the graph of a translation.  So as to
obtain a surface sufficiently similar to $E\times \P^1$, we insist that the
line bundle be degree 0, so that the moduli space is $E\times E$.  This
moduli space has a universal family, which in turn induces a Fourier-Mukai
functor $D^b_{\coh} (E\times E)\to D^b_{\coh} (E\times E)$ taking point sheaves
to their corresponding sheaf bimodules.  This is actually an equivalence,
allowing us to identify the relevant component of the moduli stack of rank
$(2,2)$ sheaf bimodules with the moduli stack $(E\times E)^{\langle
  2\rangle}$ of 0-dimensional sheaves on $E\times E$ of Euler
characteristic 2.  (This stack is a mild thickening of the 2-point Hilbert
scheme of $E\times E$, gluing in sheaves of the form $\sO_p^2$.)

We thus obtain a family of noncommutative ruled surfaces over this stack.
Any automorphism of the scheme $E\times E$ has an induced action on the
stack, and we see that the translation subgroup does not actually change
the noncommutative ruled surface.  (The effect on the sheaf bimodule is to
modify the identification between the two copies of $E$ by a translation,
and then twist by the pullback of a line bundle on $E$.)  Thus we should
really view this as a family of noncommutative ruled surfaces over the
quotient $(E\times E)^{\langle 2\rangle}/E\times E$.  There still remains a
group $\Aut(E\times E)$, and the subgroup $U(E\times E)$ (automorphisms
which are unitary with respect to the involution induced by the Poincar\'e
bundle), which is generically $\SL_2(\Z)$, acts as derived equivalences on
the fiber $E\times \P^1$ over the point corresponding to $\sO_p^2$.

\begin{prop}
  The action of $U(E\times E)$ on $(E\times E)^{\langle
    2\rangle}/E\times E$ lifts to an action by equivalences on the derived
  categories of the corresponding noncommutative ruled surfaces.
\end{prop}

\begin{proof}
  By \cite{PolishchukA:1996}, there is an action of $U(E\times E)$ on
  $D^b_{\coh} E$, and thus on both subcategories in the semiorthogonal
  decomposition.  One thus obtains a derived equivalence to a glued
  category in which the gluing functor has been conjugated by the given
  element of $\Aut(D^b_{\coh} E)$.  This reduces to computing the action on
  the sheaf bimodules corresponding to $E\times E$, which again follows by
  \cite{PolishchukA:1996}.
\end{proof}

\begin{rem}
  More generally, if $A$ is an abelian variety, the same argument gives a
  family of rank $(2,2)$ sheaf bimodules on $A\times A$ parametrized by the
  stack analogue $(A\times A^\vee)^{\langle 2\rangle}$ of the $2$-point
  Hilbert scheme, and the action of $U(A\times A^\vee)$ on $D^b_{\coh} A$
  induces derived equivalences between the corresponding noncommutative
  $\P^1$-bundles, deforming the action on $A\times \P^1$.
\end{rem}

\medskip

In the rational case, the elliptic surfaces we consider are obtained by
blowing up the base points of a pencil of cubic curves on $\P^2$.  (We say
``base points'' rather than ``base locus'', as the latter will in general
be singular; to obtain an elliptic surface, we should repeatedly blow up a
base point of the pencil until the resulting anticanonical pencil is
base point free.)  If we instead blow up only 8 of the base points, the
result is a (possibly degenerate) del Pezzo surface $Y$ of degree 1, and
the elliptic surface is obtained by blowing up the unique base point of the
anticanonical pencil on $Y$.  If we fix an anticanonical curve $Q$ on $Y$,
then $Q$ blows up to a fiber of the elliptic surface.  Since $Q$ comes
equipped with a choice of smooth point (the base point of the anticanonical
pencil), we can identify the corresponding component of the smooth locus of
$Q$ with $\Pic^0(Q)$.  This, then, lets us construct a family of
deformations parametrized by $\Pic^0(Q)^2$: we first blow up a point $z\in
\Pic^0(Q)$ viewed as the ``identity'' component of the smooth locus of $Q$,
and then take the noncommutative deformation corresponding to $q\in
\Pic^0(Q)$.  As a family of noncommutative surfaces, this depends on a
choice of blowdown structure on $Y$, but the corresponding derived
categories do not, as we have a semiorthogonal decomposition
\[
(\langle \sO_e(-1)\rangle,\sO_Y^\perp,\langle \sO_X\rangle)
\]
in which the gluing data is determined by $z$ and $q$, with
$\sO_e(-1)|_Q\cong \sO_z$.

If we consider the family of all such surfaces (i.e., a $2$-dimensional
fiber bundle over the stack of degenerate del Pezzo surfaces with blowdown
structure), we expect that any derived equivalence of the generic elliptic
fiber should extend to the full family.  The relevant group has the form
$\Z^2.\Lambda_{E_8}^2\rtimes \SL_2(\Z)$, where the center $\Z^2$ is
generated by the shift and the Serre functor, one copy of $\Lambda_{E_8}$
is translation by the Mordell-Weil group, and the other is the quotient by
the canonical bundle of the group of line bundles on $X$ that have degree 0
on $Q$.  The action of $\SL_2(\Z)$ intertwines these two actions, and one
of its generators is the twist by $\sO_X(e)$ where $e$ is the built-in
section (the exceptional curve of the final blowup).  We have already seen
that twists by line bundles extend to the full family (with a nontrivial
$q$-dependent action on the base of the family), and thus need only one
more generator to generate the full group.  One should note here that the
``correct'' group is really $\Z^2.\Lambda_{E_8}^2\rtimes W(E_8)\times
\SL_2(\Z)$, since we should include the action of $W(E_8)$ by changing the
blowdown structure of the del Pezzo surface.  Also, though every element of
this group extends, the group itself does not; even for $W(E_8)$, the
extension to fibers with $-2$-curves involve a spherical twist, and there
are choices involved which cannot be made in a uniformly consistent
fashion.

\medskip

To get the other generator, we will need to consider how modifications of the
semiorthogonal decomposition interact with the restriction to $Q$.
Given objects $M$ and $N$, let $T_M(N)$ be the cone
\[
M\otimes R\Hom(M,N)\to N\to T_M(N)\to.
\]
If $M$ is exceptional, this is precisely the operation appearing in
modifications of semiorthogonal decompositions involving $\langle M\rangle$
(more precisely, those modifications not changing $M$), while if $M$ is
spherical, this is an autoequivalence (the spherical twist).  (Since we are
really working with a dg-enhanced triangulated category, we can compute the
cone functor $T_M$ inside the dg-category, and thus sidestep the fact that
cones of natural transformations of triangulated functors need not exist.)

\begin{lem}
  Let $X$ be a noncommutative surface with a Serre functor and a curve $C$
  embedded as a divisor, and let $E\in D^b_{\coh} X$ be an exceptional object
  such that $S E \cong E(-C)[d]$ for some $d\in \Z$.  Then for $M\in
  {}^\perp E$, one has
  \[
  T_E(M)|_C\cong T_{E|_C}(M|_C).
  \]
\end{lem}

\begin{proof}
  We recall that $R\Hom(E,M)\cong R\Hom_C(E|_C,M|_C)$, and thus
  one can restrict the distinguished triangle defining $T_E(M)|_C$ to $C$
  to obtain
  \[
  E|_C\otimes R\Hom_C(E|_C,M|_C)\to M|_C\to T_E(M)|_C\to,
  \]
  which is precisely the distinguished triangle defining $T_{E|_C}(M|_C)$.
\end{proof}

\begin{rems}
  If $X$ is generated by weak line bundles along $C$, then one can compute
  the functor $T_{E|_C}$ on $D^b_{\coh} C$ by computing it on sheaves
  $M|_C$.  In particular, the fact that we can undo the modification forces
  $T_{E|_C}$ to be an autoequivalence, and thus makes $E|_C$ a spherical
  object.  Something similar holds for more general subcategories in a
  semiorthogonal decomposition: as long as the corresponding
  restriction-to-$C$ functor has adjoints, it will be spherical, and the
  modification acts as the corresponding twist.  This is primarily of
  interest for quasi-ruled surfaces, where the corresponding spherical
  twists are essentially given by the involutions corresponding to double
  covers $Q\to C_i$.
\end{rems}  

Let $X_{z,q}$ be the family defined above associated to some fixed
degenerate del Pezzo surface $Y$.  Note that by construction, any object
in $D^b_{\coh} X_{z,q}$ restricts to a {\em perfect} object in $D^b_{\coh} Q$,
and thus there is a well-defined map $K_0(X_{z,q})\to \Pic(Q)$ given by
$[M]\mapsto \det(M|_Q)$.  This gives us the following identification.

\begin{prop}
  We have $\det(\sO_Q|_Q)\cong z$, $\det(\sO_x|_Q)\cong q$.
\end{prop}

\begin{proof}
  For $\det(\sO_x|_Q)$, we first observe that all point sheaves are
  equivalent in $K_0(X_{z,q})$; indeed, a numerically trivial class in
  $K_0(X_{z,q})$ is trivial (since the Mukai pairing is independent of $q$
  and this is true for commutative rational surfaces), and any two point
  sheaves are algebraically equivalent.  We may thus assume that $x$ does
  not lie on any of the exceptional curves, and thus reduce to a
  calculation in $\P^2$ or a Hirzebruch surface, where we can compute it
  via numerical considerations.  Similarly, for $\det(\sO_Q|_Q)$, we can
  replace $\sO_Q$ by any sheaf $w\in \Pic^0(Q)$, and in particular may assume
  that $w$ has no global sections, so lies in $\sO_X^\perp$.
  Since $w|_Q$ has rank 0, we find that $\det(w|_Q)$ is independent of $q$,
  so equals $z$ by a commutative calculation.
\end{proof}

Now, starting with the semiorthogonal decomposition $(\sO_X^\perp,\sO_X)$
of $D^b_{\coh} X_{z,q}$, let us first modify it to $(\sO_X,{}^\perp\sO_X)$ and
then rotate to obtain the semiorthogonal decomposition
$({}^\perp\sO_X,\theta^{-1}\sO_X)$.  Let $\nu:\sO_{X_{z,0}}^\perp\to
{}^\perp\sO_X$ be the corresponding inclusion functor, and note that
\[
\nu(M)|_Q\cong T_q^{-1}(M|_Q)\otimes q^{-1}.
\]
Moreover, we have
\[
\theta^{-1}\sO_X|_Q\cong
\sO_X(Q)|_Q
\cong
\det(\sO_Q(Q))
\cong
\det(\sO_Q(Q)|_Q)
\cong
z,
\]
since $\sO_X(Q)|_Q$ is a line bundle, $\det(\sO_Q)\cong \sO_Q$, and
$\sO_Q(Q)\in \Pic^0(Q)$.  We have thus changed the gluing data (as functors
to $D^b_{\coh} Q$) from $(M\mapsto M|_Q,q)$ to $(M\mapsto
T_q^{-1}(M|_Q),qz)$.  Composing both functors with the autoequivalence
$T_q$ gives the equivalent data $(M\mapsto M|_Q,T_q(qz))$.  But
$T_q(qz)\cong qz$: this is true if $z\not\cong \sO_X$ since then
$R\Hom(\sO_q,\sO_{qz})=0$, and if $z\cong \sO_X$ since $T_q(q)\cong
q\otimes \Ext^1(q,q)\cong q$.  In other words, the resulting gluing data
has the form $(M\mapsto M|_Q,qz)$, which is precisely the gluing data
associated to $X_{z,qz}$.  Since $\theta$ is an autoequivalence, we obtain
the following generalization of a special case of \cite[Thm.~12.3]{generic}
(which dealt with the case $Q$ smooth).

\begin{thm}\label{thm:weird_langlands}
  There is a derived equivalence $\Phi:D^b_{\coh} X_{z,q}\cong D^b_{\coh}
  X_{z,qz}$ acting on $D^b_{\coh} Q$ as the spherical twist $T_q$.  Moreover,
  this equivalence fixes $\sO_X$ and takes $\sO_e(-1)$ to $\theta\sO_X(e)$.
\end{thm}

\begin{proof}
  Everything except for the action on $\sO_e(-1)$ is immediate from the
  above discussion, which also shows that for $M\in \sO_X^\perp$, there is
  a distinguished triangle
  \[
  \theta^{-1}\Phi(M)\to M\to R\Hom(R\Hom(M,\sO_X),\sO_X)\to
  \]
  from which it follows that $\theta^{-1}\Phi(\sO_e(-1))$ is the unique
  non-split extension of $\sO_e(-1)$ by $\sO_X$.
\end{proof}

We of course also have an equivalence $M\mapsto \theta^{-1} M(e)$ taking
$D^b_{\coh} X_{z,q}$ to $D^b_{\coh} X_{qz,q}$, fixing $\sO_e(-1)$ and taking
$\sO_X$ to $\theta^{-1}\sO_X(e)$.  As in the case $Q$ smooth, these
generate a faithful action of $\SL_2(\Z)$ on $K_0^{\text{perf}}(Q)$, and
thus $\Phi$ provides the desired extensions of derived autoequivalences of
elliptic surfaces.

In particular, there is a derived equivalence $D^b_{\coh} X_{z,1}\cong
D^b_{\coh} X_{1,z}$ for any $X$.  By Theorem
\ref{thm:painleve_moduli_spaces} below, $X_{z,1}$ may be interpreted (in a
countably infinite number of different ways) as a moduli space of
difference or differential equations, while $X_{1,z}$ is a noncommutative
deformation of the corresponding moduli space $X_{1,1}$ of ``Higgs
bundles'' (or analogues thereof).  In the differential case, this is thus
precisely the sort of derived equivalence arising in geometric Langlands
theory, and is indeed a nontrivial special case thereof (with nonabelian
structure group $\GL_2$ and allowing the equation to have singularities,
not necessarily Fuchsian, with total complexity corresponding to four
Fuchsian singularities).  This suggests that our $\SL_2(\Z)$ of derived
equivalences is the shadow of a more general theory which would extend the
geometric Langlands correspondence to difference equations.

\medskip

Theorem \ref{thm:weird_langlands} is more difficult to apply than the
corresponding result when $Q$ is smooth, since the derived autoequivalences
of $Q$ are ill understood in more degenerate cases.  We can, however,
obtain the following result, which extends \cite[Thm.~7.1]{me:hitchin} to
the case of nonreduced special fiber (and general $d$).

\begin{cor}
  Let $z$ be an $l$-torsion point of $\Pic^0(Q)$ and let $d$ be an integer.
  Then the minimal proper regular model of the relative $\Pic^d$ of the
  (quasi-)elliptic fibration $X_{z,1}$ is isomorphic to the center of
  $X_{z^a,z^{l/\gcd(l,d)}}$, where $a$ is such that $ad\cong \gcd(l,d)
  \bmod l$.
\end{cor}

\begin{proof}
  Set $b=(ad-\gcd(l,d))/l$, so that
  \[
  \begin{pmatrix}
    a&b\\
    l/\gcd(l,d)&d/\gcd(l,d)
  \end{pmatrix}
  \in
  \SL_2(\Z).
  \]
  Then Theorem \ref{thm:weird_langlands} induces a derived equivalence
  between $X:=X_{z,1}$ and $Y:=X_{z^a,z^{l/\gcd(l,d)}}$.  Moreover,
  since derived equivalences respect Serre functors, it follows
  that the (quasi-)elliptic fibration of $X$ induces a pencil of
  natural transformations $\theta^l\to\text{id}$, and equivalences
  between the triangulated categories of perfect objects on the
  respective fibers.  In particular, a $0$-dimensional sheaf on $Y$
  of degree $\gcd(l,d)$ and disjoint from $Q$ maps to a perfect
  complex of rank $1$ and degree $d$ on the corresponding fiber of
  $X$.  If the fiber on $X$ is smooth, then this perfect complex
  is itself a shifted simple sheaf, so necessarily a line bundle;
  since this is a derived equivalence, it follows that line bundles
  of degree $d$ on smooth fibers of $X$ map to $0$-dimensional
  sheaves on $Y$ of degree $\gcd(l,d)$.  Such a sheaf is supported
  on a unique point of the center $Z$ of $Y$, so that smooth fibers
  of $Z$ may indeed be identified with $\Pic^d$ of the corresponding
  fiber of $Z$.  (We should, of course, work over a family of such
  sheaves on $X$, but there is no difficulty.)  Since $Z$ is a
  minimal proper regular elliptic surface, it is the minimal proper
  regular model of the relative $\Pic^d$ of $X$ as required.

  If $X$ is quasi-elliptic, so there are no smooth fibers, then
  $\Pic^0(Q)\cong \G_a$, and thus we must have $l=\ch(k)=p$, with $p=2$ or
  $p=3$ since the surface is quasi-elliptic.  As $d$ only matters modulo
  $l$ and we can use $R^1\ad$ to negate $d$, we see that we need only
  consider $d=1$ or $d=0$.  The claim is obvious for $d=1$ (the relative
  $\Pic^1$ of $X$ is always $X$!), while for $d=0$, it remains only to show
  that the minimal proper regular model of the relative Jacobian is the
  center of $X_{1,z}$.  But in that case the derived equivalence is simple
  enough that we can directly check that it generically takes line bundles
  on fibers to sheaves.  Indeed, it suffices to consider the image of the
  structure sheaf of a fiber, which reduces (since the derived equivalence
  respects restriction to a fiber) to showing that the derived equivalence
  takes $\sO_X$ to $\sO_{e_8}(-1)$.  Up to a shift, this is precisely
  \cite[Lem.~12.5]{generic}, the proof of which does not use smoothness of
  $Q$.
\end{proof}

\begin{rem}
  If $Q$ is reduced, then the center of $Y$ may be obtained by taking the
  images of the various parameters under the isogeny with kernel generated
  by $z$.  In general, the center can in principle be made explicit via our
  understanding of blowups of maximal orders: we need simply keep track of
  the morphism of anticanonical curves as we blow up points.
\end{rem}

\section{Families of surfaces}
\label{sec:families}

For a number of purposes, we will need to consider not just individual
noncommutative surfaces but also families; e.g., even for results purely
over fields, some of the results below require us to work over dvrs.  The
most satisfying theory applies in the case of a Noetherian base, although
again for technical reasons we will need to consider more general bases.

Here the main technical issue in the direct approach is that the theory of
blowups has only been developed over a field.  Indeed, there are a number
of places in van den Bergh's arguments where he does a case-by-case
analysis depending on the multiplicity of the point in the curve $Y$ or on
the size of its orbit under $\theta$, and neither of these is well-behaved
in families.  And, of course, there are further issues even over fields,
since one might wish to blow up a Galois cycle of points.  (One might also
consider the $p^l$-fold blowup of a point defined over an inseparable
extension, but even in the commutative case, there's no way to resolve the
singularity without a field extension.)  Further issues over fields already
arising in the commutative case are that geometrically ruled surfaces might
only be conic bundles and geometrically rational surfaces might not have a
rational ruling; one must also consider Brauer-Severi schemes of relative
dimension $2$.

Most of these issues can be dealt with via descent, so we can at least
initially consider only $\P^1$-bundles, blowups in single points, etc.
Thus for the moment, let $R$ be a general commutative ring\footnote{We will
  eventually be gluing over the \'etale topology, so there is no need to
  leave the affine case yet.}.  A {\em split} noncommutative plane over $R$
will be a category $\qcoh X$ with choice of structure sheaf $\sO_X$,
defined in the following way.  Let $Q\subset \P^2_R$ be a cubic curve, and
let $q$ be an invertible sheaf on $Q$ which on every geometric fiber is in
the identity component of $\Pic^0(Q)$.  Then we can define a $\Z$-algebra
over $R$ as in \cite{BondalAI/PolishchukAE:1993}, and can take $\qcoh X$ to
be the appropriate quotient of the category of modules over that
$\Z$-algebra.

One difficulty that arises is showing that the result is actually flat.
The main issue is that since the $\Z$-algebra is defined via a
presentation, there is no a priori reason why the $\Z$-algebra should be
flat.  We would also like to have the semiorthogonal decomposition, which
reduces to establishing an exact sequence of the form
\[
0\to \sO_X(d-3)\otimes_R L_d\to \sO_X(d-2)\otimes_R V_d\to
\sO_X(d-1)\otimes_R \Gamma(q^{d+1}(1))
\to \sO_X(d)\to 0
\]
where $V_d$, $L$ are determined by exact sequences
\[
0\to V_d\to \Gamma(q^{d}(1))\otimes_R \Gamma(q^{d+1}(1))
\to \Gamma(q^{2d+1}(2))\to 0
\]
and
\begin{align}
0\to L_d&{}\to \Gamma(q^{d-1}(1))\otimes_R  \Gamma(q^{d}(1))\otimes_R
\Gamma(q^{d+1}(1))\notag\\
&{}\to
\Gamma(q^{2d-1}(2))\otimes_R \Gamma(q^{d+1}(1))
\oplus
\Gamma(q^{d-1}(1))\otimes_R \Gamma(q^{2d+1}(2)).
\end{align}

If $R$ is Noetherian, we can establish both facts by a common induction.
Indeed, the $R$-module $\Hom(\sO_X(-n),\sO_X(d))$ is defined inductively as
the cokernel of the map
\[
\Hom(\sO_X(-n),\sO_X(d-2))\otimes_R V_d\to
\Hom(\sO_X(-n),\sO_X(d-1))\otimes_R \Gamma(q^{d+1}(1))
\]
(with $\Hom(\sO_X(-n),\sO_X(d))=0$ for $d<-n$ and
$\Hom(\sO_X(-n),\sO_X(-n))=R$), and thus flatness and exactness reduce to
showing that the complex
\begin{align}
\Hom(\sO_X(-n),\sO_X(d-3))\otimes_R L_d
&{}\to
\Hom(\sO_X(-n),\sO_X(d-2))\otimes_R V_d\notag\\
&{}\to
\Hom(\sO_X(-n),\sO_X(d-1))\otimes_R \Gamma(q^{d+1}(1))
\end{align}
represents a flat $R$-module.  By induction, this is a complex of finitely
generated flat $R$-modules, and thus represents a (finitely generated) flat
$R$-module iff it represents a vector space over every closed point.  We
may thus reduce to the case of a field, which we can even assume is
algebraically closed (since the desired claim descends through fpqc base
change), and thus reduce to the known case.  Similarly, the requisite
vanishing of Ext spaces reduces to considerations of complexes of f.g. flat
$R$-modules, so again can be checked on geometric fibers.  In this way, we
find that $\sO_X(-2),\sO_X(-1),\sO_X$ is still a full exceptional
collection, with gluing functors as expected.

More generally, this property is inherited by any base change, and as
already mentioned descends through fpqc, and thus smooth base change.  Thus
if there is a smooth cover $R'$ of $R$ such that $(C_{R'},q_{R'})$ is a
base change of a family over a locally Noetherian family, then we still
obtain a flat family of categories, and the derived category still has the
expected exceptional collection.  But {\em any} family has this property!
Indeed, cubic curves in a fixed $\P^2$ are classified by a projective fine
moduli space (isomorphic to $\P^9$, to be precise), and thus pairs $(C,q)$
are classified by a smooth group scheme over that moduli space.  The latter
is not quite a fine moduli space (there is a Brauer obstruction to a point
of $\Pic^0$ being represented by a line bundle), but this can be resolved
by an \'etale base change.  We thus see that any family of pairs $(C,q)$ is
\'etale-locally given by a base change from the universal family over the
Noetherian moduli stack of such pairs, and thus inherits flatness and the
exceptional collection.

We can also recover the embedding of $Q$ in $X$ by considerations from
geometric fibers.  The existence of a map $\Hom_X(\sO_X(d'),\sO_X(d))\to
\Hom_Q(q^{d'(d'+3)/2}(d'),q^{d(d+3)/2}(d))$ is immediate from the
construction, and when $d'=d-3$ is surjective with invertible kernel (since
on geometric fibers it is surjective with 1-dimensional kernel).  This
establishes the requisite natural transformation $\_(-Q)\to\text{id}$, and
we easily see that $\Hom_X(L,L'|_Q)$ behaves as expected when $L,L'$ are
line bundles, so recovers a twisted homogeneous coordinate ring of $Q$.

For inductive purposes, we would like to know that the resulting category
$\qcoh X$ is locally Noetherian.  Unfortunately, the proof of
\cite{ArtinM/TateJ/VandenBerghM:1991} for the field case does not quite
carry over, as it uses valuations on the field in a significant way (in
particular, \cite[Lem.~8.9]{ArtinM/TateJ/VandenBerghM:1991}).  Luckily,
those parts of the argument are in service of showing that the quotient of
the graded algebra by the normalizing element of degree $3$ is Noetherian;
we now recognize that quotient as a twisted homogeneous coordinate ring
\cite{ArtinM/VandenBerghM:1990}, and thus it is Noetherian since its
$\Proj$ is Noetherian.  We thus conclude that any ``split'' noncommutative
plane over a Noetherian base ring is locally Noetherian.  (This is a
stronger form of the ``strongly Noetherian'' condition: any Noetherian base
change is still a family over a Noetherian base, so remains Noetherian.)

\smallskip

The argument for quasi-ruled surfaces is analogous.  Given a pair of
projective smooth curves $C_0/R$, $C_1/R$ with geometrically integral
fibers and a sheaf ${\cal E}$ on $C_0\times_R C_1$ which is a sheaf
bimodule on every geometric fiber, we can use the construction of van den
Bergh to get a global sheaf $\Z$-algebra, and flatness, the semiorthogonal
decomposition, and the embedding of $Q$ as a divisor follow over a
Noetherian base using the analogue of the exact sequence of Lemma
\ref{lem:exact_tri_quasi_ruled}.  (Here when checking that the $R\Hom$
complexes behave correctly, we need only check sheaves of the form
$\rho_d^* L$ and can further arrange for the line bundle on the source side
to be arbitrarily negative, since it suffices to check generators; this
lets us reduce to cases in which there is a single cohomology sheaf.)
Again, the corresponding moduli stack is itself locally Noetherian, so the
claims follow in general.  Moreover, over a Noetherian base, the category
is again locally Noetherian; in this case, the argument from the literature
\cite{ChanD/NymanA:2013} can be adapted with no difficulty (since they were
already considering constant families over general Noetherian
$k$-algebras).  For technical reasons, we will define a ``split'' family of
quasi-ruled surfaces to be a surface obtained from the above data in which
we have also marked points $x_i\in C_i(R)$.  (This ensures that every
component of the moduli space of line bundles on $X$ has a point.  Of
course, this is only relevant data for quasi-ruled surfaces themselves; on
blowups, we can simply take the images of the first point to be blown up!)

\smallskip

The case of blowups is somewhat trickier, as the relevant graded bimodule
algebra is no longer given by a presentation.  However, we can show by a
similar induction (again for $R$ Noetherian) that for any line bundle $L$
on $X$, $\alpha_*(L(d))$ is $R$-flat, and thus since line bundles generate
$\coh X$, the bimodule algebra is again $R$-flat.  (For the inductive
argument, we need to represent $\alpha_*(L(d))$ by a complex of finite sums
of line bundles, apply the appropriate functor to obtain $R\alpha_*L(d+1)$,
and observe by checking geometric fibers that the cone of the natural
transformation $(R\alpha_*L(d+1))(-Q)\to \alpha_*(L(d))$ is an $R$-flat
coherent sheaf, and thus so is $R\alpha_*L(d+1)$.)  The semiorthogonal
decomposition is similar, although now we need analogues of the two triangles
\[
L\alpha^* L\to L\alpha^! L\to \sO_e(-1)\to
\]
and
\[
\sO_e(-1)[1]\to L\alpha^* \sO_{\tau p}\to \sO_e\to
\]
in order to establish that $L\alpha^*\coh X$ and $\sO_e(-1)$ generate the
derived category.  Here the sheaf $\sO_e(-1)$ is essentially defined by the
special case $L=\sO_X$ of the first triangle (with $\sO_e:=\sO_e(-1)(1)$),
and is an $R$-flat sheaf since it is a sheaf on every geometric fiber.
Moreover, semicontinuity shows that $\Hom(\alpha^!L/\alpha^*L,\sO_e(-1))$
is a line bundle on $R$ (presumably the fiber of $L$ over the appropriate
point of $Q$), and we may verify that
\[
(\alpha^!L/\alpha^*L)\otimes_R \Hom(\alpha^!L/\alpha^*L,\sO_e(-1))\to
\sO_e(-1)
\]
is an isomorphism by checking geometric fibers.  The second triangle is
somewhat trickier, but we may rewrite it (up to line bundles on $R$) as
\[
\sO_e(-1)[1]\to L\alpha^*R\alpha_* \sO_e\to \sO_e\to
\]
and represent the latter map as a chain map of complexes of finitely
generated $R$-flat modules, and thus again reduce to checking things
geometrically.  And once more the relevant moduli stack is locally
Noetherian (see Proposition \ref{prop:moduli_stack_of_surfaces_is_small}
below), so the general case follows from the Noetherian case.

Inducting along blowups requires that we show that the blowup is again
locally Noetherian as long as the base ring is Noetherian.  Here the
argument of \cite{VandenBerghM:1998} is more difficult to adapt, as it
involves the quotient by a certain invertible ideal that depends on the
multiplicity of the point being blown up in the ambient curve, and this is
not flat.  Luckily, there is a different (and in many ways more natural)
ideal we can use instead, which being flat allows the argument to carry
over.

\begin{prop}
  Let $X/R$ be a split family of noncommutative rational or rationally
  quasi-ruled surfaces over a Noetherian ring $R$.  Then $\qcoh(X)$ is
  locally Noetherian.
\end{prop}

\begin{proof}
  Per the above discussion, it remains only to show that this is inherited
  by blowups.  We use the notation of \cite{VandenBerghM:1998}, and note
  that the basic properties we need are easily deduced from the fact that
  they hold on geometric points.  Thus $\qcoh(\widetilde{X})$ has been
  represented as $\Proj({\cal D})$, where ${\cal D}$ is a suitable ``graded
  bimodule algebra on $X$'', and the main missing ingredient in the
  argument is an invertible ideal in ${\cal D}$ {\em defined over $S$} such
  that the quotient is Noetherian.  Consider the invertible ideal ${\cal
    D}(-1)$.  If we can show that the $\Proj$ of the quotient is a
  commutative curve over $S$, then the $\Proj$ will be Noetherian.  On any
  geometric fiber, this follows from Proposition
  \ref{prop:nice_divisor_on_blowup}, and thus is true globally.  To lift to
  the graded quotient itself, we need simply observe that since $S$ is
  Noetherian, there is a maximum degree in which the graded algebra on any
  fiber differs from its saturation.  The saturation is certainly
  Noetherian, and every homogeneous component where there is a difference
  is Noetherian, and thus the original graded algebra is Noetherian.
\end{proof}

\begin{rem}
  In fact, one can check more directly that the quotient by ${\cal D}(-1)$
  is equal to its saturation in all but degree 0, where the saturation has
  $o_Y$ in place of $o_X$.  Furthermore, the saturation (which since $Y$ is
  commutative may be rephrased in terms of graded coherent sheaves on
  $Y\times Y$) is essentially a relative version of a twisted homogeneous
  coordinate ring (based on the fact that the autoequivalence
  $M\mapsto M(Y^+)=M(1)$ is relatively ample for $Y^+\to Y$).
\end{rem}

\begin{rem}
  In general, the ideal van den Bergh uses is one of the sequence of
  invertible ideals ${\cal D}(-m)((m-1)Y)$.  This is essentially ascending,
  as there is only one degree where ${\cal D}(-m)((m-1)Y)\not\subset {\cal
    D}(-m-1)(mY)$.  It thus describes a descending chain of closed
  subschemes, the limit of which is $\tilde{Y}$.  (For general $m$, the
  quotient is obtained by removing $\min(m-1,\mu-1)$ copies of the
  exceptional curve from $Y^+$.)  Only the case $m=1$ is flat, however.
\end{rem}

\begin{rem}
  Considering the commonalities between the proofs in the different cases,
  it seems likely that it should be possible to prove a quasi-scheme is
  Noetherian whenever it has a morphism to a Noetherian quasi-scheme, a
  relatively ample autoequivalence $\tau$, and a natural transformation
  $T:\text{id}\to\tau$ such that the category of sheaves with $T_M=0$ is
  also Noetherian.  (It may also be necessary to have $T_{\tau M}=\tau
  T_M$.)  In the three cases, $\tau(M)=M(Q)$, $T$ is the obvious natural
  transformation, and the first morphism is $\Gamma$, $\rho_*$, or the
  blow-down respectively.
\end{rem}

\medskip

Since in each case we retain the semiorthogonal decomposition over a
general base, we can give an alternate description of these categories as
$t$-structures on derived categories of suitable dg-algebras.  Consider
the quasi-ruled case.  For each $i$, the category $\perf(C_i)=D^b_{\coh}
C_i$ is generated by $G_i:=\sO_{C_i}\oplus \sO_{C_i}(-x_i)$.  This lets us
represent $D_{\qcoh}(C_i)$ as (the homotopy category of) the category of
right modules over the finite type dg-algebra $R\End(G_i)$ (which we think
of as acting on $G_i$ on the left).  Then $\rho_0^*G_0\oplus \rho_1^*G_1$
is a compact generator of the quasi-ruled surface, and has endomorphism
dg-algebra of the form
\[
\begin{pmatrix}
  R\End_{C_0}(G_0) & R\Hom_{C_0\times C_1}(\pi_1^*G_1,\pi_0^*G_0\otimes {\cal E})\\
  0 & R\End_{C_1}(G_1)
\end{pmatrix}.
\]
In particular, the derived category of modules over this dg-algebra agrees
with $D_{\qcoh}(X)$, with the best analogue of $D^b_{\coh}(X)$ being the
category of {\em perfect} modules over the dg-algebra.  (Over a general
base, ``coherent'' is not very well-behaved.)

We may deal with blowups similarly; if we have already constructed a sheaf
of dg-algebras $A_{m-1}$ corresponding to a family of surfaces $X_{m-1}$,
then any object $\sO_{x_m}$ in $\perf(A_{m-1})$ gives rise to a new sheaf
of dg-algebras of the form
\[
A_m=
\begin{pmatrix}
  A_{m-1} & R\Hom(\sO_{x_m},R)[-1] \\
      0   &          R
\end{pmatrix},
\]
with associated family $X_m$.  (This corresponds to the gluing functor
$R\Hom_{X_m}(\sO_{e_m}(-1),L\alpha_m^*M)\cong R\Hom_X(M,\sO_{x_m}[1])^*$
for $M\in \perf(X_{m-1})$.)  If every geometric fiber of $\sO_{x_m}$ is the
structure sheaf of a point on a divisor on the corresponding fiber of
$X_{m-1}$, then this dg-algebra is again Morita-equivalent to the blown up
surface.

We can again recover the $t$-structure without reference to the original
abelian category.  Indeed, the definition of line bundles extends easily
from the field case (the only change being that anything \'etale locally
isomorphic to a line bundle should be considered a line bundle), and we see
that line bundles generate the abelian category, and further that $M\in
D_{\qcoh}(X)^{\ge 0}$ iff $R\Hom(L,M)\in D_{\qcoh}(R)^{\ge 0}$ for all line
bundles $L$ (or, say, only those coming from line bundles
$\sO_{C_i}(nx_i)$).  Equivalently, $D_{\qcoh}(X)^{\le 0}$ is the smallest
``cocomplete preaisle'' containing the given set of line bundles; that is,
it is the smallest subcategory containing $L[p]$ for all $p<0$ and closed
under extensions and arbitrary direct sums.  By
\cite[Thm.~A.1]{TarrioLA/LopezAJ/SalorioMJS:2003}, a cocomplete preaisle
generated by a set of compact generators is an ``aisle'' (the negative part
of an actual $t$-structure), and thus in particular this holds for the
aisle generated by line bundles.  (Of course, we already knew this, but
this argument does not require that we have a direct construction of the
abelian category.)

One useful consequence (which again could be shown using the construction
of the abelian category, but is particularly simple using the dg-algebra
description) is that if $M,N\in \perf(X)$, then $R\Hom(M,N)\in \perf(R)$,
and thus satisfies semicontinuity.

\begin{prop}
  Let $\perf(X)$ be the dg-category of perfect complexes on a flat family
  of rational or rationally quasi-ruled surfaces over $R$.  Then for any
  $M,N\in \perf(X)$ and any integer $p$, the map
  \[
  s\mapsto \dim\Ext^p_{X_s}(M\otimes^L k(s),N\otimes^L k(s))
  \]
  is an upper semicontinuous function of $s\in \Spec(R)$.  Moreover, if
  $\Ext^q(M,N)=0$ for $q>p$ and $\dim\Ext^p_{X_s}(M\otimes^L
  k(s),N\otimes^L k(s))$ is constant, then $\Ext^p(M,N)$ is a locally free
  $R$-module.
\end{prop}

\begin{proof}
  We need simply note that
  \[
  \Ext^p_{X_s}(M\otimes^L k(s),N\otimes^L k(s))
  \cong
  \Ext^p_X(M,N\otimes^L k(s))
  \cong
  h^p(R\Hom(M,N)\otimes^L k(s)),
  \]
  and thus reduce to standard semicontinuity results in $\perf(R)$.
\end{proof}

Of course, this is only useful when combined with an understanding of the
perfect objects.  An object $M\in D_{\qcoh}(X)$ is perfect iff its
projections to the various pieces of the semiorthogonal decomposition are
perfect.  Each of these lives on a scheme which is smooth and projective
over $R$, and thus we may apply the following.

\begin{prop}
  Let $Y$ be a smooth projective scheme over $R$ of relative dimension $n$.
  An object $M\in D_{\qcoh}(Y)$ is perfect iff there are integers $[a,b]$
  such that for any point $s\in \Spec(R)$, $M\otimes^L_R k(s)$ has
  coherent cohomology supported on the interval $[a,b]$.
\end{prop}

\begin{proof}
  An object is perfect iff $M\otimes^L_{\sO_Y} k(y)$ has coherent
  cohomology of uniformly bounded amplitude as $y$ ranges over points of
  $Y$.  If $\pi:Y\to \Spec(R)$ is the structure map, then we may write
  \[
  M\otimes^L_{\sO_Y} k(y)
  \cong
  M\otimes^L_R k(\pi(y))
  \otimes^L_{\sO_{Y_{\pi(y)}}} k(y).
  \]
  Since $Y_s$ is a projective scheme over a field, $k(y)$ has $\Tor$
  dimension at most $n$.  Thus if $M\otimes^L_R k(\pi(y))$ has coherent
  cohomology supported on $[a,b]$, then $M\otimes^L_{\sO_Y} k(y)$ has
  coherent cohomology supported on $[a-n,b]$, giving one implication.  For
  the converse, note that for any point $s\in \Spec(R)$, if
  $h^p(M\otimes^L_R k(s))$ is the nonzero cohomology sheaf of lowest
  degree and $y$ is a generic point of a component of its support in $Y_s$,
  then $h^{p-\text{codim}(y)}(M\otimes^L_{\sO_Y} k(y))\ne 0$.
\end{proof}

In our case, every fiber has a $t$-structure, and we have the following.

\begin{prop}
  Let $X/R$ be a family of noncommutative rational or rationally
  quasi-ruled surfaces as above.  Then an object $M\in D_{\qcoh}(X)$ is
  perfect iff there are integers $[a,b]$ such that $M\otimes^L_R k(s)\in
  D_{\qcoh}(X_s)$ has coherent cohomology supported on the interval $[a,b]$
  for all $s\in \Spec(R)$.
\end{prop}

\begin{proof}
  On each fiber and for each piece of the semiorthogonal decomposition, the
  projection and the inclusion functors have finite cohomological
  amplitude, and this amplitude is uniformly bounded on $\Spec(R)$.  Thus
  $M\otimes^L_R k(s)$ has uniformly bounded amplitude iff its projections
  have uniformly bounded amplitude.
\end{proof}

\begin{cor}
  If $M\in D_{\qcoh}(X)$ is a flat family of coherent sheaves (i.e.,
  $M\otimes^L_R k(s)$ is a coherent sheaf for all $s\in \Spec(R)$), then
  $M\in \perf(X)$.
\end{cor}

\begin{rem}
  In particular, flat families of coherent sheaves satisfy semicontinuity.
  When $R$ is Noetherian, $\qcoh(X)$ is locally Noetherian, and thus we can
  define $\coh(X)$ to be the subcategory of Noetherian objects, and
  automatically have $\perf(X)\subset D^b_{\coh}(X)$.  Over a more general
  base, ``coherent'' is both tricky to define and of unclear use.
\end{rem}

There is of course no real difficulty in gluing the above construction in
the Zariski topology.  We thus make the following definition.  Note that
the data we used to construct a family of noncommutative planes came with a
choice of curve $Q$, and we make a similar choice as necessary for
quasi-ruled surface, which we include as part of the data.  On fibers which
are not commutative, we can recover $Q$ from the surface, so it does not
give extra data, but commutative fibers come with a choice of anticanonical
curve.

\begin{defn}
  A split family of noncommutative rationally quasi-ruled surfaces over a
  scheme $S$ is an $\sO_S$-linear category $\qcoh(X)$ (or $D_{\qcoh}(X)$
  with a $t$-structure) which is Zariski locally of the above form.
\end{defn}

\begin{defn}
  A split family of noncommutative rationally ruled surfaces is a split
  family of noncommutative rationally quasi-ruled surfaces equipped with an
  isomorphism $C_0\cong C_1$ identifying $x_0$ with $x_1$ and making
  $\overline{Q}$ algebraically equivalent to the double diagonal; in addition,
  $Q$ may not be geometrically integral on any fiber such that $g(C_0)=1$.
\end{defn}

\begin{defn}
  A split family of noncommutative rational surfaces is either an 
  iterated blowup of a family of noncommutative planes as constructed
  above, or a split family of noncommutative rationally ruled surfaces such
  that $g(C_0)=0$.
\end{defn}

\begin{rem}
  Since we do not deal with any other kind of noncommutative surface, we
  will feel free to abbreviate this to ``split family of noncommutative
  surfaces'', only adding ``rational'', ``ruled'', or ``quasi-ruled'' when
  we are imposing further restrictions.
\end{rem}

\begin{rem}
  As we have already mentioned, we will show below that a blowup of a
  noncommutative plane is a noncommutative Hirzebruch surface, and thus
  a noncommutative rational surface is either a noncommutative rationally
  ruled surface over $\P^1$ or a noncommutative plane.
\end{rem}
  
\medskip

As we have already discussed, the above construction is clearly too
restrictive to be the right notion of family.  An obvious further approach
is to consider surfaces that can locally be put in the above form relative
to some suitable topology.  To ensure we make the correct choice of
topology here, we should first determine the correct choice in the
commutative case.

We first observe that $C_0$ and $C_1$ have sections \'etale locally
(indeed, {\em any} smooth surjective morphism has a section \'etale
locally), so that $x_0$ and $x_1$ can be chosen \'etale locally.  Moreover,
they only played a role in constructing a set of generating line bundles on
the surface, and thus the $t$-structure, but the resulting $t$-structure
was independent of the choices, so descends.  For the same reason, a flat
family of smooth surfaces geometrically isomorphic to $\P^2$ again has an
\'etale local section, and is thus \'etale locally isomorphic to $\P^2$,
and the $t$-structure defined using that isomorphism descends.

A somewhat more subtle question has to do with the fact that a
geometrically ruled surface need not be a projective bundle.  Note that
although we are assuming a particular choice of ruling, there is no loss of
generality when $g(C)>0$, since then the canonical map $X\to \Alb^1(X)$ to
the Albanese torsor factors through $C$ on any geometric fiber, so that we
actually have a canonical ruling.

\begin{prop}
  Let $C/S$ be a smooth proper curve, and let $\rho:X\to C$ be a smooth
  proper curve which over geometric points of $C$ is isomorphic to $\P^1$.
  Then any geometric point $s\in S$ has an \'etale neighborhood over which
  $X$ is isomorphic to a projective bundle.
\end{prop}

\begin{proof}
  First note that $C$ \'etale locally has a section, so may be assumed
  projective, and $-K_X$ is relatively ample over $C$, so that $X$ is then
  also projective.  The anticanonical embedding of $X_s/C_s$ is a conic,
  and thus by Tsen's Theorem has a section $\sigma:C_s\to X_s$.  This
  induces a line bundle $\sO_{X_s}(\sigma(C_s))$ such that $X_s\cong
  \P(\rho_*\sO_{X_s}(\sigma(C_s))$.  Let $\Pic^\sigma_{X/S}$ denote the
  component of the Picard scheme $\Pic_{X/S}$ containing the isomorphism
  class of this line bundle.  This is a torsor over $\Pic^0_{C/S}$, thus
  smooth and surjective, so has a section over an \'etale neighborhood $S'$
  of s in S, and the isomorphism class of line bundles has a representative
  ${\cal L}$ defined over an \'etale neighborhood $T$ of $s$ in $S'$.  The
  sheaf of graded algebras over $C$ corresponding to ${\cal L}$ is
  naturally isomorphic to the symmetric algebra of $\rho_*{\cal L}$, so
  that $X_T\cong \P_C(\rho_*{\cal L})$ as required.
\end{proof}

\begin{rem}
  It is tempting to simply specify the component $\Pic^\sigma(X/S)$ by
  numerical considerations (i.e., specifying the Hilbert polynomials of
  ${\cal L}$ and ${\cal L}(f)$).  The issue is that the resulting component
  could be (and, half the time, is!) empty; we need a splitting over some
  geometric point to pick out a good component.
\end{rem}

The hardest case is when there is no longer a unique way to interpret the
commutative surface as an iterated blowup of a ruled surface.  This is
likely to be especially difficult to deal with for rational surfaces, as in
that case there may in fact be infinitely many such interpretations, and
thus the moduli stack of anticanonical rational surfaces is not algebraic!
(If one imposes a blowdown structure, this is no longer an issue, see
\cite{me:hitchin}, generalizing a construction of \cite{HarbourneB:1988}.)
To get a feel for the issue, consider the case of a {\em projective} family
over a connected base with generic geometric fiber $\P^1\times \P^1$.  If
the family has geometric fibers of the form $F_d$ for $d>0$, then the ample
bundle must have bidegree $(d,d')$ with $d>d'$, and this distinguishes the
two rulings on the generic fiber, so that there is a unique choice of
global ruling.  If all fibers are $\P^1\times \P^1$, then the line bundle
of bidegree $(1,1)$ extends \'etale locally, and gives an embedding of the
surface as a smooth quadric surface.  Over a non-closed field, a typical
such surface has no ruling, and thus the only rational line bundles are
multiples of the bidegree $(1,1)$ bundle.  However, the family has a
natural double cover over which it has a ruling, which in odd
characteristic simply takes the square root of the determinant.  In
particular, by removing the singular surfaces (which correspond to smooth
surfaces on which the given bundle is not ample), we have made this cover
\'etale.  (If we had not removed the singular surfaces, the cover would be
ramified, but the family would not be smooth, and does not have a global
desingularization.)  Over the cover, the surface is isomorphic to a product
of two conics, and thus at worst two further \'etale covers make it
globally $\P^1\times \P^1$.

More generally, as long as the family is projective, the ample divisor
$D_a$ cuts out a finite set of blowdown structures on the geometric fibers
of $X$, namely those for which $D_a$ is in the fundamental chamber (see the
discussion in Section \ref{sec:effnef} below).  Moreover, the resulting set
of blowdown structures is a torsor over an appropriate finite Coxeter
group, and thus a blowdown structure exists \'etale locally.

\medskip

With the above discussion in mind, we define a family of noncommutative
rational or rationally quasi-ruled surfaces (again abbreviated to ``family
of noncommutative surfaces'') over $S$ to be a family $D_{\qcoh}(X)$ of
dg-categories over $S$ such that (a) the dg-category is \'etale locally a
split family of noncommutative surfaces, and (b) the associated
$t$-structures are compatible, in that the two ways of pulling back to the
fiber square of the cover agree.

Note that the locally Noetherian condition easily descends through
\'etale covers, and thus we have the following, over a Noetherian base.

\begin{thm}\label{thm:noetherian_base_implies_noetherian}
  Let $X/S$ be a family of noncommutative rational or rationally
  quasi-ruled surfaces over a Noetherian base.  Then $\qcoh(X)$ is locally
  Noetherian.
\end{thm}
  
In principle, we could construct such a category by specifying the specific
data $(C_0,C_1,\dots)$ on each piece of the \'etale cover, and giving
suitable gluing data.  (Indeed, by fpqc descent, quasicoherent sheaves can
always be glued along \'etale covers.)  For rational or rationally ruled
surfaces, this is not too difficult: we have an associated family of
commutative surfaces, so as long as that family is projective, there is an
\'etale cover $T/S$ over which there is a global choice of blowdown
structure.  Since the associated tuple $(C_0,C_1,\dots)$ over $T$ satisfies
the natural compatibility conditions, the result indeed glues to give a
dg-category with $t$-structure over $S$.

It turns out that the dg-category itself is easy enough to construct
directly.  Let us consider first the case that $X$ is rational.  Then the
exceptional object $\sO_X$ of $\perf(X\times_S T)$ certainly descends, and
thus induces a semiorthogonal decomposition over $S$, suggesting the
following construction.  Given a family $Y/S$ of commutative rational
surfaces (i.e., $Y$ is smooth and proper over $S$ of relative dimension
$2$, and its fibers are geometrically rational), not only does $\perf(Y)$
have a generator over $S$, but so does the subcategory $\sO_Y^\perp$;
indeed, if $G$ generates $Y$, then the cone $G'$ of $\sO_Y\otimes
R\Hom(\sO_Y,G)\to G$ generates $\sO_Y^\perp$.  If $Q$ is an anticanonical
divisor on $Y$ and $q\in \Pic^0(Q)$ (that is, $q$ is a line bundle on $Q$
which has degree 0 on every component of every geometric fiber of $Q$),
then we may consider the sheaf of triangular dg-algebras
\[
\begin{pmatrix}
  \sO_S & R\Hom_Q(G'|^{\bf L}_Q,q)\\
     0  & R\End(G')
\end{pmatrix}
\]
This {\em almost} gives a fully general construction of $D_{\qcoh}(X)$ for
noncommutative rational surfaces.  The only remaining issue is that all we
actually obtain from the \'etale local description is a point of
$\Pic_{(Q/S)(\text{\'et})}(S)$ which is in the identity component of every
fiber; in particular, it defines a line bundle on the base change to some
\'etale cover such that the two pullbacks are isomorphic, but the
isomorphisms need not be compatible.  (Basically, the problem is that $q$
is normally given as the determinant of the restriction of a point on the
noncommutative surface, but $Q$ may not have rational points, so we only
know $q$ up to isomorphism!)  This is not too hard to fix as long as $Q$ is
projective (e.g., if $Y$ is projective); the idea is that although $q$ may
not descend, the objects $q\otimes R\Gamma(q^{-1}(d))\in \perf(Q_{T})$ and
$R\End(R\Gamma(q^{-1}(d)))$ descend for any $d$.  Taking $d\gg 0$ gives a
well-defined sheaf of dg-algebras
\[
\begin{pmatrix}
  \sO_S & R\Hom(G'|^{\bf L}_Q,q\otimes \Gamma(q^{-1}(d)))\\
  0 & R\End(G')\otimes \End(\Gamma(q^{-1}(d))).
\end{pmatrix}
\]
This is clearly Morita-equivalent to the original dg-algebra over $T$, and
thus the category of (perfect) modules over this dg-algebra may indeed be
viewed as a family of (derived categories of) noncommutative rational
surfaces, in that every geometric fiber is such a surface.  Moreover, any
category $\perf(X)$ which is \'etale locally equivalent to the category of
perfect objects on an iterated blowup of a noncommutative Hirzebruch
surface (or noncommutative plane) can be obtained via the above
construction.

So we may instead construct families of noncommutative rational surfaces by
specifying a triple $(Y,Q,q)$ (with $q\in \Pic(Q/S)(S)$) along with a
splitting over some \'etale cover $T/S$ such that both pullbacks of the
$t$-structure agree.  If $T$ is Noetherian, then we will see below that
$\qcoh(X)$ is equivalent to the $\Proj$ of a sheaf of $\Z$-algebras.

For ruled surfaces which are not rational, Proposition
\ref{prop:f_is_unique_if_quasi-ruled} below tells us that the functor
$\alpha_m^*\cdots \alpha_1^*\rho_0^*$ descends (i.e., noncommutative ruled
surfaces of genus $>0$ also have canonical rulings), and thus we may use
the corresponding semiorthogonal decomposition to construct $\perf(X)$ from
a commutative family $Y$ of geometrically rationally ruled surfaces over a
curve $C$ of genus $>0$, an anticanonical divisor $Q$, and a morphism $Q\to
C$ which contracts vertical fibers and is in the same component of the
moduli stack of such morphisms as the morphism coming from the ruling on
$Y$.

We should note that even when $Y$ is geometrically ruled, the subcategory
$\perf(C)^\perp$ is not in general equivalent to $\perf(C)$.  Indeed, if
$Y/C$ is a nonsplit conic bundle, then $\perf(C)^\perp$ is the category of
perfect objects in the derived category of the corresponding (quaternion)
Azumaya algebra over $C$.

We ignore families of surfaces which are rationally quasi-ruled but not
rationally ruled, as the fibers themselves are not particularly interesting
as noncommutative surfaces (and there is no intrinsic difficulty in
constructing them in any case; Proposition
\ref{prop:f_is_unique_if_quasi-ruled} implies that the ruling
is again canonical (though the fibers may only be conics)).  Note that we
cannot quite dismiss these as being maximal orders on families of
rationally ruled commutative surfaces, as the degree over the center can
fail to be constant: see the discussion above of quasi-ruled surfaces of
``2-isogeny'' type.

\medskip

The one issue with the above notion of family, though more natural than the
split case, is that the corresponding moduli problem is not well-behaved.
We can resolve this by putting back part of the splitting data.  A {\em
  blowdown structure} on a family $X/S$ of noncommutative ruled surfaces is
an isomorphism class of sequences
\[
X=X_m\to X_{m-1}\to\cdots\to X_0\to C
\]
of morphisms of families, in which all but the last map is a blowdown, and
$X_0\to C$ is a ruling.  (Note that when $C$ has genus 1, it might only be
an algebraic space.)  This is essentially combinatorial data; as we will
see, the blowdown structures on a noncommutative ruled surface over an
algebraically closed field form a discrete set.  Note however, that in the
rational case, this set can be infinite, and in general, the set is far
from flat (in a typical family, the size changes on dense subsets!).

\begin{prop}\label{prop:moduli_stack_of_surfaces_is_small}
  The moduli stack of noncommutative rationally ruled surfaces with
  blowdown structures is an algebraic stack locally of finite type over
  $\Z$, and the components classifying $m$-fold blowups of ruled surfaces
  of genus $g$ have dimension $m+3$ if $g=0$, $m+1$ if $g=1$, and $m+2g-2$
  otherwise.
\end{prop}

\begin{proof}
  We first note that if $X/S$ is a family of noncommutative rationally
  ruled surfaces with a blowdown structure, then the splittings of $X$
  compatible with the blowdown structure are essentially classified by a
  smooth scheme over $S$.  More precisely, the splittings are classified by
  $C\times_S \Pic(C)$, with $C$ recording the marked point and $\Pic(C)$
  capturing the fact that the sheaf bimodule is only determined modulo
  $\pi_1^*\Pic(C)$.  If we insist that the Euler characteristic of the
  bimodule be $0$ or $-1$, then this eliminates all but one component.
  Thus it will suffice to construct the corresponding stack for split
  surfaces.
  
  For ruled surfaces, we first note that the curve $C$ is classified by
  the union over $g$ of the classical moduli stacks ${\cal M}_g$ of
  smooth genus $g$ curves with a marked point.

  Relative to the ample divisor $\sO_C(x)\boxtimes \sO_C(x)$ on $C\times
  C$, a sheaf bimodule of rank 2 has Hilbert polynomial $2t+\chi$ for some
  $\chi$.  The condition for a sheaf on $C\times C$ with that Hilbert
  polynomial to have Chern class algebraically equivalent to the double
  diagonal is open and closed, while having no subsheaf support on a fiber
  of $\pi_1$ or $\pi_2$ is open.  We thus find that the stack ${\cal
    M}_{sb}$ of sheaf bimodules corresponding to ruled surfaces is smoothly
  covered by open substacks of
  \[
  \Quot_{C\times C/{\cal M}_{g,1}}(\sO_C(-dx)\boxtimes \sO_C(-dx),2t+\chi).
  \]

  For a family without commutative fibers, we can recover the anticanonical
  curve as the Quot scheme $\Quot_{C\times C/{\cal M}_{sb}}({\cal E},1)$.
  Over commutative fibers, this is not a curve, but we can resolve this as
  follows.  The condition that ${\cal E}$ be a vector bundle of rank 2 on a
  curve is a closed condition, and thus we may blow this up to obtain a new
  stack ${\cal M}_0$.  The $\Quot$ scheme classifying the anticanonical
  curve is still badly behaved over the exceptional locus, but now the bad
  behavior comes in the form of an extra component that we can remove to
  obtain a family of curves of genus 1.  We thus find that ${\cal M}_0$ is
  the desired moduli stack of noncommutative ruled surfaces.  (We easily
  check that when $g=1$, the anticanonical curve is either two disjoint
  copies of $C$ or a double copy of $C$ with some number of fibers, so in
  either case is not integral.)

  The universal anticanonical curve over ${\cal M}_0$ embeds in the
  commutative ruled surface
  \[
  X':=\Quot_{C/{\cal M}_0}(\pi_{1*}{\cal E},1),
  \]
  which lets us construct the moduli stack in general.  For each $m>0$, ${\cal
    M}_m$ is the universal anticanonical curve over ${\cal M}_{m-1}$, and
  the universal anticanonical curve on ${\cal M}_m$ is obtained by blowing
  up the corresponding point of the universal commutative ruled surface
  and then removing one copy of the exceptional curve from the pullback of
  $Q$.

  To compute the dimension, we first note that each blowup adds $1$ to the
  dimension, so it suffices to consider the ruled surface case, and for
  $g\ge 1$ since the dimension for $g=0$ was effectively computed in
  \cite[Thm.~5.8]{me:hitchin}.  We may also compute the dimension for $C$
  fixed and add $3g-3$.  For $g\ge 2$, the sheaf bimodule is supported on
  the double diagonal, and is either invertible (a $0$-dimensional stack
  since the double diagonal has trivial dualizing sheaf), a torsion-free
  sheaf determined by where it fails to be invertible, or a vector bundle.
  The invertible case has dimension $0$, but we must the mod out by the
  action of the Picard stack of $C$ (twisting by $\rho_1^*{\cal L}$), so
  the correct dimension is $1-g$, giving the total $2g-2$ as stated.

  For the non-invertible cases, pick a point where ${\cal E}$ is not
  invertible (there being only finitely many such points on non-commutative
  geometric fibers).  Then $Q$ has a component lying over that point, and
  thus we obtain a 1-parameter family of blowups in a point of that
  component and not on any other component.  Each of those blowups has a
  unique blowdown (using Theorem \ref{thm:elem_xform_works}) to a surface
  with one fewer point of non-invertibility.  We can recover the original
  surface and its blowup from the resulting $Q'$ along with the point that
  got blown up, and thus we find that the two pieces of the moduli stack
  have the same dimension.

  In the genus 1 case, the nonreduced support case picks up an additional
  dimension (which reduced curve it is supported on) which is then
  cancelled out by the fact that $C$ has infinitesimal automorphisms, and
  we need to mod out by the choices of isomorphism $C_1\cong C_0$.  A
  bimodule with reduced support is supported on the union of two graphs of
  translations ($2$ dimensions), is a line bundle on each ($2-2$ more
  dimensions), modulo a line bundle on $C_1$ ($1-1$ fewer dimensions) and
  automorphisms of $C_1$ ($-1$ dimensions).
\end{proof}

\begin{rem}
  In the rational case, we could also construct this as the relative
  $\Pic^0(Q)$ (the scheme, not the stack!) over the moduli stack of
  commutative rational surfaces (with blowdown structure), which also deals
  with the case of noncommutative planes.  It follows from
  \cite[Thm.~5.8]{me:hitchin} that the stack of noncommutative rational
  surfaces is a local complete intersection.  Something similar is true in
  higher genus, though the fibers are more complicated to describe; we omit
  the details.  In either case, we find that the moduli stack of
  noncommutative surfaces retracts to the substack of anticanonical
  commutative surfaces via a smooth morphism.
\end{rem}

\begin{rem}
  One can also compute the expected dimension of this moduli stack, and
  find that it is {\em highly} obstructed for $g\ge 2$.  (That there are
  some obstructions should not be surprising: for $g\ge 2$, the moduli
  stack has nonreduced components!)  Indeed, the moduli stack of sheaf
  bimodules has expected dimension $(2\Delta_{C\times C})^2=8-8g$, and thus
  for $g\ge 2$, the expected dimension is $m+(8-8g)+(1-g)+(3g-3)=m+6-6g$;
  thus there must be at least $8g-8$ obstructions.  (For $g=1$, we should
  subtract one to get the expected dimension, and thus see that despite the
  moduli stack being reduced, there are still at least 2 obstructions.)
  This suggests (along with the fact that we are already working with
  dg-algebras) that one should consider the {\em derived} moduli stack of
  noncommutative surfaces.  It should be possible to construct this: the
  moduli stack of commutative rational or rationally ruled surfaces is
  smooth if we do not choose an anticanonical curve, and thus one could
  consider the derived linear system $|{-}K_X|$ over this stack.  The
  derived moduli stack of noncommutative surfaces would then presumably be
  some sort of smooth 1-dimensional fiber bundle over this stack.
\end{rem}

In \cite{me:hitchin}, we considered a decomposition of the moduli stack of
rational surfaces based on the structure of the anticanonical curve (i.e.,
the divisor classes and multiplicities of the components along with whether
$\Pic^0(Q)$ is elliptic, multiplicative, or additive), and showed that in
each case, the corresponding substack was irreducible and smooth.  (Indeed,
there is an \'etale covering making it a smooth scheme, in which we specify
an ordering within each set of components of $Q$ with the same divisor
class and multiplicity.)  This of course carries over immediately to the
moduli stack of noncommutative rational surfaces, since $\Pic^0(Q)$ is also
irreducible and smooth.

The basic idea also works for noncommutative rationally ruled surfaces of
higher genus.  There is an easy induction reducing to $X_0$: each time we
blow up a point, the new combinatorics either determines the multiplicity
with which each component contained that point or gives a smooth open curve
over which the point varies.  For $X_0$ itself, if $g=1$ and $Q$ is smooth,
then we need simply specify the pair of line bundles and translations, and
otherwise we are in the differential case, so need simply specify the
torsion-free sheaf on the double diagonal.  Such a sheaf is either
invertible, so by the constraint on the Euler characteristic is classified
by $\Pic^0(Q)\cong \G_a$, or strictly torsion-free, in which case it is
uniquely determined by the divisor on which it fails to be torsion-free, or
equivalently the points of $C_0$ (with multiplicities) over which the
components of $Q$ of class of $f$ lie.  In addition, we note that the
degree of the Fitting scheme of the sheaf bimodule is semicontinuous, so
for $g>1$ each subfamily is locally closed, and the same follows for $g=1$
once we observe that the support being reduced is an open condition.

One immediate consequence is that characteristic 0 points are dense in
the moduli stack of noncommutative surfaces; i.e, that any surface over a
field of characteristic $p$ can be lifted (possibly after a separable field
extension) to characteristic $0$.  Another is that for any family $X/S$,
there is an induced decomposition of $S$ into locally closed subsets, which
must be finite if the base is Noetherian.

A natural further question is how the closures of these substacks interact.
Over $\Z$, there are already examples showing in the rational case that
this is not a stratification in general, though it is still open whether
this phenomenon can occur if one excludes points of finite characteristic.
Unfortunately, despite some questions being easier for ruled surfaces with
$g>0$, this appears unlikely to be one of them.

\section{Dimension and Chern classes of sheaves}

In \cite{ChanD/NymanA:2013}, a definition was given for a ``non-commutative
smooth proper $d$-fold'', and it was shown there that ruled surfaces (more
precisely, quasi-ruled surfaces in which the two curves are isomorphic)
satisfy their definition (with $d=2$, of course).  Our objective is to show
that this continues to hold for iterated blowups of arbitrary quasi-ruled
surfaces.  In this section, our focus is on those of their conditions
related to dimensions of sheaves, though in the process we will also be
considering the Grothendieck groups in greater detail.  This will give us
the information required to finish our considerations of birational
geometry; we will then prove the remaining Chan-Nyman axiom in Section
\ref{sec:quot} below.

Thus suppose $X=X_m$ is an $m$-fold blowup of a quasi-ruled surface $X_0$
(over an algebraically closed field), with intermediate blowups
$X_1$,\dots,$X_{m-1}$ and associated morphisms $\alpha_1$,\dots,$\alpha_m$.
We wish to show that $X$ satisfies the Chan-Nyman axioms.  (Note that they
implicitly include ``strongly Noetherian'' as an unnumbered axiom, but this
of course follows as a special case of Theorem
\ref{thm:noetherian_base_implies_noetherian}.)

The first two of their axioms are automatic: the Gorenstein condition
simply states that there is a Serre functor such that $S[-2]$ is an
autoequivalence, which follows from the above descriptions of the derived
categories and $t$-structures; this in turn immediately implies that
$\Ext^p$ vanishes for $p\notin \{0,1,2\}$, giving their ``smooth, proper of
dimension $2$'' condition.  Thus the first case that requires some work is
the existence of a suitable dimension function (satisfying their axioms
3,4,6 as well as continuity and finite partitivity).

To define this, we recall that the description of the derived category
immediately gives us a description of the Grothendieck group of $\coh X$:
it is the direct sum of $K_0(C_0)$, $K_0(C_1)$, and a copy of $\Z$ for each
point being blown up.  We define the {\em rank} of a class in the
Grothendieck group to be the rank of the image in $K_0(C_0)$ minus the rank
of the image in $K_0(C_1)$.  This then immediately defines a rank for any
coherent sheaf (or object in $D^b_{\coh} X$) as the rank of its class in
the Grothendieck group.

\begin{lem}
  If $m>0$, then for any class $[N]\in K_0(X_{m-1})$,
  $\rank(\alpha_m^*[N])=\rank([N])$, while for any $[M]\in K_0(X)$,
  $\rank(\alpha_{m*}[M])=\rank([M])$.
\end{lem}

\begin{proof}
  The first claim is immediate from the definition, while the second claim
  follows from the first together with the fact that
  $\alpha_m^*\alpha_{m*}[M]-[M]$ is a multiple of the rank 0 class
  $[\sO_{e_m}(-1)]$.
\end{proof}

Applying the construction of Proposition \ref{prop:nice_divisor_on_blowup}
inductively gives an embedding of a commutative curve $Q$ as a divisor,
with morphism $\iota:Q\to X$.  Note that there is a natural map $L\iota^*$
from $D^b_{\coh} X$ to the derived category $\perf(Q)$ of perfect complexes
on $Q$, and thus an induced map from $K_0(X)$ to $K_0^{\perf}(Q)$.  The
latter of course also has a well-defined rank function, which is almost
everywhere on $Q$ given by the alternating sums of the ranks of the
cohomology sheaves.  In particular, we may compute the rank by considering
the restriction to $\hat{Q}$ (the scheme-theoretic union of component(s) of
$Q$ which are not components of fibers).

\begin{lem}
  For any class $[M]\in K_0(X)$, $\rank([M])=\rank(\iota^*[M])$.
\end{lem}

\begin{proof}
  It suffices to check this on each summand of $K_0(X)$.  Each of the
  summands other than $K_0(C_0)$ and $K_0(C_1)$ is generated by a sheaf
  $\sO_{e_i}(-1)$ of rank 0 that meets $\hat{Q}$ transversely, and thus has
    rank 0 restriction as required.  Similarly, $K_0(C_0)$ and $K_0(C_1)$
    are generated by line bundles, and in either case the restriction to
    $\hat{Q}$ is again a line bundle, so has rank 1 as required.
\end{proof}

\begin{cor}
  For any sheaf $M\in \coh Q$, $\rank(\iota_*M)=0$.
\end{cor}

\begin{proof}
  We have $\rank(\iota_*M)=\rank(L\iota^*\iota_*M)$,
  and the cohomology sheaves of $L\iota^*\iota_*M$ are $M$ and the twist of
  $M$ by a suitable invertible sheaf, so almost everywhere have the same
  rank.
\end{proof}

\begin{cor}
  For any class $[M]\in K_0(X)$, $\rank(\theta[M])=\rank([M])$.
\end{cor}

\begin{proof}
  We have $[M]-\theta[M]=\iota^*\iota_*[M]$.
\end{proof}

\begin{prop}
  For any object $M\in \coh X$, $\rank(M)\ge 0$.
\end{prop}

\begin{proof}
  Since $\alpha_{m*}$ preserves the rank, we may reduce to the case $m=0$,
  so that $X=X_0$ is a quasi-ruled surface.  Consider the function
  $r(l,[M]):=\rank \rho_{l*}[M]$ on $\Z\times K_0(X_0)$.  This is clearly
  linear in $[M]$ (since rank and the action of $\rho_{l*}$ on the
  Grothendieck group are both linear), while in $l$ it satisfies (as an
  immediate consequence of Lemma \ref{lem:exact_tri_quasi_ruled})
  \[
  r(l+1,[M])+r(l-1,[M])=2r(l,[M]),
  \]
  so that $r(l,[M])-r(0,[M])$ is bilinear.   Since
  \[
  \rank([M])=r(0,[M])-r(-1,[M]),
  \]
  we find that $r(l,[M]) = r(0,[M])+l\rank([M])$.
  For $M$ a sheaf, $\rho_{0*}\theta^{-l}M$ is a sheaf for $l\gg 0$, and
  thus
  $\rank(\rho_{0*}\theta^{-l}M)=r(2l,M)\ge 0$
  for $l\gg 0$, implying $\rank(M)\ge 0$ as required.  (Here we use the
  fact that $R\rho_{0*}\theta^{-l}M$ differs from $R\rho_{(2l)*}M$ via a
  twist by a line bundle.)
\end{proof}

The rank of course only depends on the {\em algebraic} class in $K_0(X)$,
or in other words on the quotient $K_0^{\num}(X)$ of $K_0(X)$ by the
identity component $\Pic^0(C_0)\oplus \Pic^0(C_1)$ of $K_0(C_0)\oplus
K_0(C_1)$.  (The notation is justified by the fact that numeric and
algebraic equivalence turn out to agree in this case, as we will see.)  Let
$[\pt]\in K_0^{\num}(X)$ denote the class of a point of $Q$; this is
well-defined since the corresponding map $Q\to K_0(X)$ is a map from a
connected scheme, so lies inside a single coset of the identity component.
This can also be described in terms of the semiorthogonal decomposition: it
is the class of a point in $K_0^{\num}(C_0)$ minus the class of a point in
$K_0^{\num}(C_1)$, and is 0 in every other component.  This is dual to the
definition of rank, and indeed one finds that $\rank([M])=\chi([M],[\pt])$,
where $\chi$ denotes the Mukai pairing (which on a pair of complexes is the
alternating sum of $\Ext$ dimensions, and is bilinear on $K_0(X)$).  We of
course have $\rank([\pt])=0$, and thus we may use $\rank$ and $[\pt]$ to
define a filtration of the Grothendieck group, and in turn define the
dimension of a coherent sheaf.

\begin{defn}
  Given a nonzero sheaf $M\in \coh(X)$, the dimension $\dim(M)$ of $M$ is
  defined to be $2$ if $\rank(M)>0$, $0$ if $[M]\propto [\pt]$, and $1$
  otherwise.  We say that $M$ is {\em pure} $d$-dimensional if $\dim(M)=d$
  while any nonzero proper subsheaf of $M$ has dimension $<d$.  A sheaf is
  {\em torsion-free} if it is either $0$ or pure $2$-dimensional.
\end{defn}

To show that this is truly a dimension function, we need to show that a
subsheaf or quotient of a $d$-dimensional sheaf has dimension at most $d$.
(Chan and Nyman also ask for compatibility with the Serre functor, but this
follows immediatetly from $\rank(\theta[M])=\rank([M])$.)  For $d=2$ this
is trivial, while for $d=1$ it reduces to the fact that the rank is
additive and nonnegative.  For $d=0$, this will require some additional
ideas.  We first note that the nonnegativity of the rank has an analogue
for $0$-dimensional sheaves.

\begin{prop}
  If $M\in \coh X$ is a sheaf of numeric class $d[\pt]$, then $d\ge 0$, with
  equality iff $M=0$.
\end{prop}

\begin{proof}
  If $m>0$, then $\alpha_{m*}\theta^{-l}M$ is a sheaf for $l\gg 0$, and we
  readily verify that it also has class $d[\pt]$.  Thus by induction $d\ge
  0$, and if $d=0$, then $\alpha_{m*}\theta^{-l}M=0$ for $l\gg 0$; since
  $\alpha_m^*\alpha_{m*}\theta^{-l}M\to \theta^{-l}M$ is surjective for
  $l\gg 0$, we conclude that $M=0$ as required.  Similarly, for $m=0$,
  $\rho_{l*}M$ has class $d[\pt]$ and is a sheaf for $l\gg 0$, so that we
  reduce to the corresponding claims in $C_0$ and $C_1$.
\end{proof}

\begin{cor}
  If $M$, $N$ are coherent sheaves with $[M]=[N]$ in $K_0^{\num}(X)$, then
  any injective or surjective map $M\to N$ is an isomorphism.
\end{cor}

\begin{proof}
  The quotient resp. kernel would be a sheaf with trivial class in
  $K_0^{\num}(X)$, and thus 0 by the Proposition.
\end{proof}

\begin{defn}
  The {\em N\'eron-Severi lattice} $\NS(X)$ is the subquotient
  $\ker(\rank)/\Z[\pt]$ of $K_0^{\num}(X)$.  Given two classes $D_1,D_2\in
  \NS(X)$, their {\em intersection number} $D_1\cdot D_2$ is given by
  $-\chi(D_1,D_2)$.  The {\em canonical class} $K_X\in \NS(X)$ is the class
  $[\theta \sO_X]-[\sO_X]$
\end{defn}

\begin{rem}
  Since $\chi(D,[\pt])=\chi([\pt],D)=0$ iff $\rank(D)=0$, this pairing is
  indeed well-defined.
\end{rem}

By mild abuse of notation, we will refer to the classes in $\NS(X)$ as
``divisor'' classes.  This might more properly be reserved for the
extension of $\NS(X)$ obtained as the quotient of the subgroup of
$1$-dimensional classes in $K_0(X)$ by the subgroup of $0$-dimensional
classes, but this quotient is badly behaved when $X$ is not a maximal
order, since the subgroup of $0$-dimensional classes need not be closed.

\begin{defn}
  The (numeric first) Chern class $c_1:K_0(X)\to \NS(X)$ is defined
  by $c_1([M])=[M]-\rank([M])[\sO_X]$.
\end{defn}

In particular, a $1$-dimensional sheaf is $0$-dimensional iff its Chern
class is 0.

\begin{prop}
  The Picard lattice has rank $m+2$, and the intersection pairing is a
  nondegenerate symmetric bilinear form of signature $(+,-,-,\dots,-)$.
\end{prop}

\begin{proof}
  The semiorthogonal decomposition makes it straightforward to determine
  the Fourier-Mukai form on $K_0^{\num}(X_m)$ from that of
  $K_0^{\num}(X_{m-1})$, and we in particular find that the induced form on
  $\NS(X_m)$ satisfies
  \[
  (\alpha_m^*D_1+r_1 e_m)\cdot (\alpha_m^*D_2+r_2 e_m)
  =
  (D_1\cdot D_2) -r_1r_2,
  \]
  so that we may reduce to the case $m=0$.  (Here we have defined $e_m$ to
  be the Chern class of the exceptional sheaf $\sO_{e_m}(-1)$.)

  For $m=0$, we have $K_0^{\num}(C_i)\cong \Z^2$ (determined by rank and
  degree), and thus $K_0^{\num}(X_0)\cong \Z^4$ and $\NS(X_0)\cong \Z^2$.
  We may then define $f$ to be the (numeric) class of a point of $C_0$
  (noting that the numeric class of a point in $C_1$ gives the same class
  in $\NS(X)$, since they differ by $[\pt]$), and let $s$ be any class
  which is rank $1$ in both $C_0$ and $C_1$.  We then readily compute that
  $s\cdot f=f\cdot s=1$, giving the required symmetry.  Moreover, we have
  $f^2=0$, making the signature $(+,-)$ as required.
\end{proof}

We will often refer to the basis $s,f,e_1,\dots,e_m$ of $\NS(X)$
arising from this proof.  The class $s$ is only defined modulo $f$, but
since $(s+df)^2=s^2+2d$, we can fix this by insisting that $s^2\in
\{-1,0\}$ to obtain a canonical basis of $\NS(X_m)$.  We should
caution that this basis is only canonical relative to the given
representation of $X_m$ as an iterated blowup of a quasi-ruled surface;
e.g., if $X_0$ is a sufficiently nondegenerate deformation of $\P^1\times
\P^1$, then it can be represented in two ways as a ruled surface, and the
two resulting bases differ by interchanging $s$ and $f$.

\begin{cor}
  The radical of the Fourier-Mukai pairing is $\Pic^0(C_0)\times
  \Pic^0(C_1)$, and thus numerically equivalent classes in $K_0(X)$ are
  algebraically equivalent.
\end{cor}

\begin{proof}
  Since the Fourier-Mukai pairing is constant in flat families, the
  identity component of $K_0(X)$ is certainly contained in the radical, and
  we have just shown that the induced pairing on $[\pt]^\perp/[\pt]$ is
  nondegenerate, so that the radical of the induced pairing on
  $K_0^{\num}(X)$ is contained in $\Z[\pt]$.  Since
  $\chi([\sO_X],[\pt])=1$, it follows that the pairing on $K_0^{\num}(X)$
  is nondegenerate.
\end{proof}

We have the following by an easy induction, where we recall the notion of
line bundle from Definition \ref{defn:line_bundle}.

\begin{cor}
  For any divisor class $D$, there is a line bundle $L$ with $c_1(L)=D$,
  and any two line bundles with the same Chern class have the same class in
  $K_0^{\num}(X)$.
\end{cor}

\begin{rem}
  The reader should bear in mind that line bundles do not form a group in
  any reasonable sense!  Indeed, the map from the set of isomorphism
  classes of line bundles to $\NS(X)$ (which, of course, {\em is} a
  group) may not even have constant fibers: half of the fibers are
  $\Pic^0(C_0)$-torsors, and half are $\Pic^0(C_1)$-torsors, depending on
  the parity of $D\cdot f$.
\end{rem}

Since $\sO_X$ has Chern class $0$, and $[\pt]$ can be expressed as a linear
combination of classes of line bundles, we conclude the following.

\begin{cor}
  $K_0^{\num}(X)$ is generated by classes of line bundles.
\end{cor}

For a split family of noncommutative surfaces, the above considerations
imply that the family $K_0^{\num}(X_s)$ of abelian groups is locally
constant as $s$ varies, letting us make the following statement.

\begin{cor}
  If $X/S$ is a split family of noncommutative surfaces, then for any $M\in
  \perf(X)$, the class of $M|^{\bf L}_s$ in $K_0^{\num}(X_s)$ is locally
  constant on $S$.
\end{cor}

\begin{proof}
  The class in $K_0^{\num}(X_s)$ of an object $M$ is uniquely determined by
  the linear functional $\chi(\_,M)$ coming from the Fourier-Mukai pairing,
  and thus on the values $\chi(L,M)$.  But $\chi(L|^{\bf L}_s,M|^{\bf
    L}_s)$ is locally constant since $R\Hom(L,M)$ is perfect.
\end{proof}

This implies the ``no shrunken flat deformations'' axiom of
\cite{ChanD/NymanA:2013}: if $M$ and $N$ are distinct fibers of a flat family of
coherent sheaves over a connected base, then any injective or surjective
morphism between them is an isomorphism, since the cokernel or kernel is
then a sheaf with trivial class in $K_0^{\num}(X)$, and thus 0.


\begin{cor}
  A class in $K_0^{\num}(X)$ is uniquely determined by its rank, Chern
  class and Euler characteristic $\chi(M):=\chi(\sO_X,M)$.  The
  Fourier-Mukai pairing is given in these terms by
  \begin{align}
    \chi(M,N) 
    = {}&-\rank(M)\rank(N)\chi(\sO_X)\notag\\
  &+\rank(M)\chi(N)+\rank(N)\chi(M)\notag\\
  &-c_1(M)\cdot (c_1(N)-\rank(N)K_X).
  \end{align}
\end{cor}

\begin{proof}
  We first observe that Serre duality implies
  \[
  \chi([\sO_X],\theta[\sO_X])
  =
  \chi([\sO_X],[\sO_X])
  \]
  and
  \begin{align}
  \chi(M,[\sO_X])
  &=
  \chi(M,[\sO_X]-\theta[\sO_X])
  +
  \chi(M,\theta[\sO_X])\notag\\
  &=
  \chi(M-\rank(M)[\sO_X],[\sO_X]-\theta[\sO_X])
  +
  \chi([\sO_X],M)\notag\\
  &=
  c_1(M)\cdot K_X + \chi(M),
  \end{align}
  while the definition of the intersection pairing gives
  \[
  -c_1(M)\cdot c_1(N) = \chi(M-\rank(M)[\sO_X],N-\rank(N)[\sO_X]).
  \]
  Expanding the right-hand side via bilinearity and simplifying gives the
  desired result.
\end{proof}

In the above expression, $\chi(\sO_X)$ is easy to compute, as $\sO_X$ lives
entirely inside the component $D^b_{\coh} C_0$ of the semiorthogonal
decomposition; we thus conclude that
$\chi(\sO_X)=\chi(\sO_{C_0})=1-g(C_0)$.  More generally, if $L$ is a line
bundle, one has $\chi(L,L)=1-g(C_{c_1(L)\cdot f})$, from which one can
solve for $\chi(L)$:
  \[
  \chi(L) = 1 - \frac{g(C_0)+g(C_{c_1(L)\cdot f})}{2} + \frac{c_1(L)\cdot
    (c_1(L)-K_X)}{2}.
  \]
  This, of course, agrees with the standard Riemann-Roch formula for the
  Euler characteristic in the case of a ruled surface (as it must: $\chi(L)$ is
  locally constant as we vary the surface, so for ruled surfaces may be
  computed on the commutative fiber).

For $K_X$, the situation is slightly more complicated, but we have the
following.

\begin{prop}
  In terms of the standard basis of $\NS(X)$, we have
  \[
  K_X = \begin{cases}
    -2s-(2-g(C_0)-g(C_1))f+e_1+\cdots+e_m, & s^2 = 0\\
    -2s-(3-g(C_0)-g(C_1))f+e_1+\cdots+e_m, & s^2 = -1
  \end{cases}.
  \]
\end{prop}

\begin{proof}
  We readily reduce to the case $m=0$, where the description of the action
  of $\theta$ tells us that
  \[
    [\theta\sO_X]+[\rho_0^*\omega_{C_0}] = \rho_1^*V
  \]
  for some rank 2 vector bundle $V$ on $C_1$.  It follows that $K_X =
  -2s+df$ for some $d$, and since $K_X^2$ is linear in $d$, we reduce to
  showing $K_X^2 = 4(2-g(C_0)-g(C_1))$.  Since $\chi(\rho_1^*V,\rho_1^*V)
  =\chi(\End(V)) = 4(1-g(C_1))$ by Riemann-Roch, we find
  \[
  \chi([\theta\sO_X]+[\rho_0^*\omega_{C_0}],[\theta\sO_X]+[\rho_0^*\omega_{C_0}])
  =
  4-4g(C_1).
  \]
  The left-hand side expands as a sum of 4 terms, three of which simplify
  via Serre duality to a calculation inside $D^b_{\coh} C_0$ and the fourth of
  which can be simplified using the general expression for $\chi(M,N)$.  We
  thus find
  \[
  4-4g(C_1)
  =
  K_X^2 - c_1(\rho_0^*\omega_{C_0})\cdot K_X.
  \]
  Since $c_1(\rho_0^*\omega_{C_0}) = (2g(C_0)-2)f$, this gives $K_X^2 =
  4(2-g(C_0)-g(C_1))$ as required.
\end{proof}

\begin{rem}
  We similarly find $[\sO_Q]=-K_X + (g(Q)-1)f$.
\end{rem}

\begin{cor}
  The action of $\theta$ on $K_0^{\num}(X)$ is given by
  \begin{align}
    \rank(\theta M) &= \rank(M)\\
    c_1(\theta M) &= c_1(M)+\rank(M)K_X\\
    \chi(\theta M) &= \chi(M) + c_1(M)\cdot K_X.
  \end{align}
\end{cor}

\begin{proof}
  For each $M$, $\chi(M,N)=\chi(N,\theta M)$ by Serre duality.  Both sides
  are linear functionals on $K_0^{\num}(X)$, and comparing coefficients
  gives the desired expressions.
\end{proof}

The action of the duality $R\ad$ is also straightforward to compute, since
we know how it acts on line bundles.

\begin{prop}
  The action of $R\ad$ on $K_0^{\num}(X)$ is given by
  \begin{align}
    \rank(R\ad M) &= \rank(M)\\
    c_1(R\ad M) &= -c_1(M)+\rank(M)K_X\\
    \chi(R\ad M) &= \chi(M).
  \end{align}
\end{prop}  

\begin{proof}
  Indeed, this is linear and acts correctly on line bundles.
\end{proof}

When $X$ is a maximal order, we are also interested in how the direct image
and pullback functors relate the two Grothendieck groups.  The nicest
answer regards the N\'eron-Severi groups.

\begin{prop}\label{prop:NS_of_center}
  Let $X$ be a rationally quasi-ruled surface, and suppose that $X\cong
  \Spec{\cal A}$ where ${\cal A}$ is a maximal order of rank $r^2$ over the
  commutative surface $Z$, with associated morphism $\pi:X\to Z$.  Then $X$
  and $Z$ have the same parity, and in terms of the standard bases of
  $\NS(X)$ and $\NS(Z)$, the maps $\pi^*$ and $\pi_*$ are both
  multiplication by $r$.
\end{prop}

\begin{proof}
  Since $\pi^*\pi_*M\cong M\otimes_{\sO_Z} {\cal A}$ and ${\cal A}$ is a
  vector bundle of rank $r^2$, we see that $c_1(\pi^*\pi_*M)=r^2 c_1(M)$.
  Adjunction then gives $\pi_*D_1\cdot \pi_*D_2 = r^2 D_1\cdot D_2$, and
  the same holds for $\pi^*$ since their product is multiplication by $r^2$.

  We next address the question of parities, which is of course really about
  quasi-ruled surfaces.  An elementary transformation flips both parities,
  and thus we may assume that we are in the untwisted case with
  $\hat{Q}=\bar{Q}$.  As we discussed when considering point sheaves, in
  this case there is a natural morphism $\rho_1^*\sO_{C_1}\to
  \rho_0^*\sO_{C_0}$ the cokernel of which is a sheaf of class in $s+\Z f$
  disjoint from $Q$.  The direct image of this sheaf is a vector bundle of
  rank $r$ on its support, a section of $Z$.  Since we have shown that the
  direct is a similitude relative to the intersection form, we conclude
  that the classes in $s+\Z f$ on the respective surfaces have the same
  self-intersection, and thus the surfaces have the same parity, letting us
  identify their N\'eron-Severi lattices via the standard bases.

  It remains to show that $\pi_*$ and $\pi^*$ are multiplication by $r$
  relative to these bases.  For $m\ge 1$, we note that $\pi_*\sO_{e_m}(-1)$
  is supported on the corresponding exceptional curve of $Z$, and thus has
  Chern class a multiple of $e_m$, which must be $r e_m$ by the isometry
  property.  Similarly, for any point $x\in C'$, we have
  $\pi^*\rho^*\sO_x\cong \rho_0^* \phi^*\sO_x$, where $\phi$ is the
  projection $C_0\to C'$, and thus $\pi^*(f) = rf$.  Finally, $\pi_*s$ is
  orthogonal to each $e_i$ and has intersection $r$ with $f$, so has the
  form $\pi_*s = rs+df$, with $d=0$ then forced by $(\pi_*s)^2=r^2$.
\end{proof}

In fact, we can give the full map between the Grothendieck groups, though
this is more complicated.

\begin{prop}
  If the rationally quasi-ruled surface $X$ is a maximal order over $Z$,
  then using the standard bases to identify their N\'eron-Severi lattices,
  one has
  \begin{align}
  \rank(\pi_*M) &= r^2\rank(M)\notag\\
  c_1(\pi_*M) &= r c_1(M)+\rank(M) \frac{r^2 K_Z - r K_X}{2}\notag\\
  \chi(\pi_*M) &= \chi(M)
  \end{align}
  and
  \begin{align}
  \rank(\pi^*M) &= \rank(M)\notag\\
  c_1(\pi^*M) &= r c_1(M)\notag\\
  \chi(\pi^*M) &= r^2\chi(M)+\frac{c_1(M)\cdot (r^2K_Z-rK_X)}{2}+\rank(M)(\chi(\sO_X)-r^2\chi(\sO_Z)).
  \end{align}
\end{prop}

\begin{proof}
  Since $\pi^*\sO_Z\cong \sO_X$, we have $\chi(\pi_*M)=\chi(M)$, and since
  $\pi_*\sO_X={\cal A}$, $\rank(\pi_*\sO_X) = r^2$.  Moreover, our Chern
  class computation tells us how $\pi_*$ acts on the Chern class of
  $1$-dimensional sheaves.  This almost determines how $\pi_*$ acts on a
  basis of $K_0^{\num}(X)$, except that we still need to determine
  $c_1(\pi_*\sO_X)$.  Since $\pi_*\theta \sO_X\cong \sHom(\pi_*\sO_X,\omega_Z)$,
  we find that
  \[
  c_1(\pi_*\theta \sO_X) = r^2 K_Z - c_1(\pi_*\sO_X)
  \]
  and thus
  \[
  2c_1(\pi_*\sO_X) = r^2 K_Z - r K_X.
  \]
\end{proof}  

\medskip

In the commutative setting, the notion of an effective divisor is of course
quite crucial.  In the commutative setting, a divisor on a smooth surface
is effective if it can be represented by a curve; although curves
themselves do not make sense in the noncommutative case (apart from
components of $Q$), we are still led to the following definition.

\begin{defn}
  A divisor class is {\em effective} if it is the Chern class of a
  $1$-dimensional sheaf.
\end{defn}

\begin{rem}
Note that the effective classes form a monoid, since we can always take the
direct sum of the corresponding $1$-dimensional sheaves.
\end{rem}

\begin{rem}
  In the commutative case, one often uses the equivalence between a divisor
  being effective and the corresponding line bundle having a global
  section.  As stated, this does not hold in the noncommutative setting;
  there are cases (e.g., $e_1-e_2$ when one blew up two points in the same
  orbit) of effective divisors such that no line bundle with the given
  Chern class has a global section.  The situation is somewhat better if
  one allows both line bundles to vary, fixing the difference of Chern
  classes, but even then it is unclear whether the resulting classes form a
  monoid.  (The monoid they generate is, however, correct, at least for
  rational or rationally ruled surfaces.)
\end{rem}

If a $0$-dimensional sheaf had a $1$-dimensional subsheaf, then both the
subsheaf and the quotient would have nonzero effective Chern classes, and
thus to rule this out, we need to show that the effective monoid intersects
its antipode only in 0.

\begin{prop}\label{prop:neg_of_eff_not_eff}
  If $D\in \NS(X)$ is such that both $D$ and $-D$ are effective,
  then $D=0$.
\end{prop}

\begin{proof}
  Let $M$ be a $1$-dimensional sheaf with Chern class $D$, and let $N$ be a
  $1$-dimensional sheaf with Chern class $-D$.  If $m>0$, then for $l\gg
  0$, $\alpha_{m*}\theta^{-l}M$ and $\alpha_{m*}\theta^{-l}N$ are both
  sheaves, and their Chern classes add to 0 since $\alpha_{m*}$ induces a
  well-defined homomorphism $\NS(X_m)\to \NS(X_{m-1})$ with
  kernel $\Z e_m$.  Thus by induction, we have $D=d e_m$ for some $d\in
  \Z$.  But then $[\alpha_{m*}\theta^{-l}M]$ and
  $[\alpha_{m*}\theta^{-l}N]$ are multiples of $[\pt]$ for all $l$, with the
  coefficient depending linearly on $l$; since they must be nonnegative for
  $l\gg 0$, both linear terms must be nonnegative, implying that $d,-d\ge
  0$ as required.

  For $m=0$, suppose $D=ds+d'f$ (relative to the standard basis discussed
  above).  Then $\rank(\rho_{l*}M)=d=-\rank(\rho_{l*}N)$, and since both
  are sheaves for $l\gg 0$, we must have $d=0$.  But then $[\rho_{l*}M]$
  and $[\rho_{l*}N]$ are proportional to the class of a point for all $l$,
  again depending linearly on $l$, with coefficients $\pm d'$, so that
  $d'=0$ as well.
\end{proof}

It follows immediately that our dimension function is {\em exact} in the
sense of \cite{ChanD/NymanA:2013}: given an exact sequence
\[
0\to M'\to M\to M''\to 0,
\]
we have $\dim(M)=\max(\dim(M'),\dim(M''))$.  (We have also implicitly shown
that the shifted Serre functor preserves dimension.)

\medskip

The next two results essentially say that $X$ is irreducible.

\begin{lem}
  If $L_1$, $L_2$, $L_3$ are line bundles on $X_m$ and $\phi_1\in
  \Hom(L_1,L_2)$, $\phi_2\in \Hom(L_2,L_3)$ are morphisms such that
  $\phi_2\circ \phi_1=0$, then $\phi_1=0$ or $\phi_2=0$.
\end{lem}

\begin{proof}
  For $m>0$, we may write $L_i = \theta^{l_i}\alpha_m^* \theta^{-l_i}L'_i$
  for line bundles on $X_m$, and we find that there are
  composition-respecting injections
  \[
  \Hom(L_1,L_2)\subset \Hom(L'_1,L'_2)\qquad\text{and}\qquad
  \Hom(L_2,L_3)\subset \Hom(L'_2,L'_3).
  \]
  (Indeed, by the construction of the blowup, $\Hom(L_1,L_2)$ is
  essentially defined to be the subspace of $\Hom(L'_1,L'_2)$ satisfying an
  appropriate condition of the form ``vanishes to multiplicity $l_1-l_2$'')
  We thus reduce by induction to the case $m=0$, where it is simply the
  fact that the $\Z$-algebra $\bar{\cal S}$ corresponding to a quasi-ruled
  surfaces is a domain.
\end{proof}

\begin{prop}
  Any nonzero morphism between line bundles is injective.
\end{prop}

\begin{proof}
  Suppose otherwise, so that $\phi_2:L_2\to L_3$ is a morphism of line
  bundles with nonzero kernel.  Then there is a line bundle $L_1$ such that
  $L_1\otimes \Hom(L_1,\ker(\phi_2))\to \ker(\phi_2)$ is surjective, and
  thus $\Hom(L_1,\ker(\phi_2))\ne 0$.  A nonzero homomorphism $L_1\to
  \ker\phi_2$ induces a nonzero homomorphism $\phi_1:L_1\to L_2$ such that
  $\phi_2\circ \phi_1=0$, giving a contradiction.
\end{proof}

\begin{cor}
  Line bundles are torsion-free.
\end{cor}

\begin{proof}
  If $M\subset L$ is a rank 0 subsheaf of a line bundle $L$, then as
  before, there exists a nonzero morphism $L'\to M$ for some line bundle
  $L'$ and thus the composition $L'\to L$ has image a nonzero subsheaf of
  $M$.  But this composition is nonzero, so injective, and thus its image
  has rank 1, giving a contradiction.
\end{proof}

Serre duality immediately gives the following.

\begin{cor}
  If $L$ is a line bundle and $M$ a $\le 1$-dimensional sheaf, then
  $\Ext^2(L,M)=0$.
\end{cor}

We would like more generally to know that the ``cohomology'' of a coherent
sheaf $M$ (i.e., $R\Gamma(M):=R\Hom(\sO_X,M)$) vanishes in degree
$>\dim(M)$.  For $2$-dimensional sheaves, this follows trivially from the
global bound on $\Ext$ groups, while for $1$-dimensional sheaves, this is a
special case of the corollary.  For $0$-dimensional sheaves, it will be
handy to use the following dichotomy.  We note that although the
restriction map $K_0(X)\to K_0^{\perf}(Q)$ is not well-defined on
$K_0^{\num}(X)$, the rank is certainly well-defined, as is the class of the
determinant in $\Pic(Q)/\Pic^0(C_0)\times \Pic^0(C_1)$.  In particular,
there is a well-defined class $q$ in the latter group obtained as the
determinant of the restriction of $[\pt]$.  Note that if $q^r\not\sim
\sO_Q$, then a sheaf of class $r[\pt]$ has nontrivial restriction to $Q$;
in particular, if $q\not\sim \sO_Q$, then any sheaf of class $[\pt]$ is a
point of $Q$.

\begin{prop}
  The class $q$ is torsion iff $X$ is a maximal order on a smooth
  commutative projective surface $Z$.
\end{prop}

\begin{proof}
  If $X$ is a maximal order in a central simple algebra of degree $r$, then
  there are $0$-dimensional sheaves on $X$ of class $r[\pt]$ which are
  disjoint from $Q$, and thus the determinant of their restriction is
  trivial, implying $rq\sim \sO_Q$.

  Now, suppose that $q$ is torsion.  The conclusion is inherited under
  blowing up, so we may assume $m=0$, and since the conclusion is automatic
  for surfaces which are quasi-ruled but not ruled, we may assume $X=X_0$
  is a ruled surface, and not of the hybrid type (since the conclusion also
  holds there) or commutative.  In the differential case, we have
  $\Pic^0(Q)/\Pic^0(C_0)\times \Pic^0(C_1)\cong \G_a$ and find that $q=1$,
  which is torsion only in finite characteristic, where the conclusion
  holds.  In the difference cases, we can express $q$ in terms of the
  action of $s_0s_1$, and find that the latter is torsion iff $q$ is
  torsion.
\end{proof}

\begin{rem}
  The same calculation shows that for a rationally ruled surface, $q$ is
  trivial iff $X$ is commutative.
\end{rem}

\begin{prop}
  If $M$ is $0$-dimensional of class $d[\pt]$, then for any line bundle
  $L$, $\Ext^p(L,M)=0$ for $p>0$ and $\dim\Hom(L,M)=d$.
\end{prop}

\begin{proof}
  If $X$ is a maximal order, then the claim follows immediately from the
  corresponding fact on $Z$, so we suppose that $q$ is not torsion.

  Since $\Ext^2(L,M)=0$, we find that
  $\dim\Hom(L,M)-\dim\Ext^1(L,M)=\chi(L,M)=d$, and thus any line bundle has
  a nonzero morphism to $M$.  The image and cokernel (if nonzero) of such a
  morphism will again be $0$-dimensional, and thus vanishing of
  $\Ext^1(L,M)$ will follow by induction in $d$.  We may thus assume that
  $M$ has no proper sub- or quotient sheaf.  Since $\det(M|^{\dL}_Q)\ne 0$, we
  conclude that $M|^{\dL}_Q\ne 0$, and thus that $M$ is actually supported on
  $Q$.  In other words, we may write $M=\iota_* N$ for some $0$-dimensional
  sheaf $N$ on $Q$, and find
  \[
  R\Hom(L,M)\cong R\Hom(\iota^*L,N),
  \]
  and vanishing follows since $\iota^*L$ is an invertible sheaf on $Q$.
\end{proof}

\begin{rem}
  By Serre duality, it follows that $\Ext^1(M,L)=0$.
\end{rem}

\begin{cor}
  If $M$ is $\le d$-dimensional, then $R^p\ad M=0$ for $p>2$ and $p<2-d$.
\end{cor}

\begin{proof}
  We have already shown that $\ad$ is left exact of homological dimension
  $2$.  Since this is bounded, there is a line bundle $L$ on $X^{\ad}$ such
  that each cohomology sheaf is acyclic and globally generated relative to
  $L$.  On the other hand, we have
  \[
  \Hom(L,R^p\ad M)
  \cong
  h^p(R\Hom(L,R\ad M))\cong \Ext^p(M,\ad L)\cong \Ext^{2-p}(\ad L,\theta M)^*
  \]
  and thus $R^p\ad M$ vanishes if $2-p>\dim(\theta M)=\dim(M)$.
\end{proof}

\begin{prop}
  For any coherent sheaf $M$, $\ad M$ is torsion-free.
\end{prop}

\begin{proof}
  We have a surjection of the form $L^n\to M$ and thus by left exactness
  $\ad M$ is a subsheaf of the torsion-free sheaf $\ad L^n$.
\end{proof}

Call a coherent sheaf $M$ ``reflexive'' if $R\ad M$ is a sheaf.

\begin{cor}
  A coherent sheaf $M$ is reflexive iff there exists a sheaf $N$ such that
  $M\cong \ad N$.
\end{cor}

\begin{proof}
  If $M$ is reflexive, then $M\cong R\ad \ad M=\ad(\ad M)$, so it remains to
  show that $\ad N$ is always reflexive.  But this follows from the
  spectral sequence $R^p\ad R^q\ad N\Rightarrow N$: if $R^p\ad \ad N\ne 0$ for
  $p\in \{1,2\}$, then this would contribute to the positive cohomology of
  the limit of the spectral sequence.
\end{proof}

\begin{cor}
  For any coherent sheaf $M$, $R^p\ad M$ has dimension $\le 2-p$.
\end{cor}

\begin{proof}
  We similarly conclude using the negative part of the spectral sequence
  that $R^p\ad R^2\ad M=0$ for $p<2$, and that for any line bundle $L$,
  $R^2\ad R^2\ad M$ is acyclic for $R\Hom(L,\_)$.  It follows that $R^2\ad
  R^2\ad M$ is $0$-dimensional, and thus (dualizing yet again) that $R^2\ad
  M$ is $0$-dimensional.

  For $R^1\ad M$, vanishing of the negative cohomology only directly tells
  us that $\ad R^1\ad M$ injects in $R^2\ad R^2\ad M$, but since $\ad
  R^1\ad M$ is torsion-free, its injection in a $0$-dimensional sheaf
  forces it to vanish.  But then
  \[
  \rank(R^1\ad M) = \rank(R\ad R^1\ad M) = \rank(R^2\ad R^1\ad
  M)-\rank(R^1\ad R^1\ad M),
  \]
  and since $R^2\ad R^1\ad M$ is $0$-dimensional by the previous paragraph,
  we find that $\rank(R^1\ad M)+\rank(R^1\ad R^1\ad M)=0$.  Since sheaves
  have nonnegative rank, it follows that $R^1\ad M$ has rank 0, so is $\le
  1$-dimensional as required.
\end{proof}

\begin{cor}
  A nonzero $\le 1$-dimensional sheaf $M$ is pure $1$-dimensional iff
  $R^2\ad M=0$.
\end{cor}

\begin{proof}
  For any $0$-dimensional sheaf $N$, we have
  \[
  \Hom(N,M)\cong h^0(R\Hom(R\ad M,R^2\ad N))\cong \Hom(R^2\ad M,R^2\ad N),
  \]
  and thus there is a nonzero morphism $N\to M$ iff there is a nonzero
  morphism $R^2\ad M\to R^2\ad N$.  In other words, if $N$ is a
  nonzero $0$-dimensional subsheaf of $M$, then $R^2\ad N$ is a nonzero
  quotient of $R^2\ad M$, while if $N$ is a nonzero quotient of $R^2\ad M$,
  then $R^2\ad N$ is a nonzero $0$-dimensional subsheaf of $M$.
\end{proof}

\begin{cor}
  If $M$ is pure 1-dimensional, then $R^1\ad R^1\ad M\cong M$.
\end{cor}

This gives an easy proof of the ``finitely partitive'' property considered
in \cite{ChanD/NymanA:2013}.

\begin{prop}
  Let $M=M_0\supset M_1\supset\cdots $ be a descending chain of coherent
  sheaves such that $M_i/M_{i+1}$ has dimension $d=\dim(M)$ for all $i\ge
  0$.  Then the chain must be finite.
\end{prop}

\begin{proof}
  If $d=2$, then this follows from the fact that each subquotient has
  positive rank, so the chain terminates in at most $\rank(M)$ steps, while
  if $d=0$, we similarly must terminate in at most $\chi(M)$ steps.  For
  $d=1$, we first replace each $M_i$, $i>0$ by a slightly larger sheaf
  $\widetilde{M}_i$ defined by letting $\widetilde{M}_i/M_i$ be the maximal
  $0$-dimensional subsheaf of $\widetilde{M}_{i-1}/M_i$.  This makes each
  subquotient a pure $1$-dimensional sheaf without affecting the length of
  the chain.  Furthermore, if $N$ is the maximal $0$-dimensional subsheaf
  of $M$, then $N$ is contained in each $\widetilde{M}_i$, and thus we may
  quotient the chain by $N$ without affecting the dimensions of the
  subquotients.  We have thus reduced to the case that $M$ and every
  subquotient is pure $1$-dimensional.  But then applying $R^1\ad$ gives
  an {\em ascending} chain
  \[
  0\subset R^1\ad(M_0/M_1)\subset R^1\ad(M_0/M_2)\subset\cdots\subset R^1\ad M_0
  \]
  which terminates since $\coh X$ is Noetherian.
\end{proof}

\begin{rem}
  Of course, once we have an understanding (and have established the
  existence!) of ample divisors, we will be able to sharpen the
  $1$-dimensional considerably; the length of the chain is bounded by
  $D_a\cdot c_1(M)$ for any ample divisor $D_a$.
\end{rem}

\medskip

The above results also hold with minor changes in the case of an iterated
blowup of a noncommutative plane.  The inductive steps are all the same,
with the only difference being the base case.  For a noncommutative plane,
the semiorthogonal decomposition shows
\[
K^0(X)\cong \Z[\sO_X]\oplus \Z[\sO_X(-1)]\oplus \Z[\sO_X(-2)],
\]
from which (together with constancy on flat deformations) it follows that
$\chi([\sO_X(-d)],[M])$ is a quadratic polynomial in $d$ and $[M]$ is
uniquely determined by this polynomial.  Moreover, this polynomial is
nonnegative for $d\gg 0$ (since $\theta^{-1}$ is ample), with equality iff
$M=0$.  The rank of $M$ is given by the coefficient of $d^2/2$ in this
polynomial, while $[pt]$ is represented by the polynomial $1$, and
$NS(X)\cong \Z$ is generated by the class $h$ (with $h^2=1$) corresponding
to the polynomials with leading term $d$, with $K=-3h$.

That morphisms between line bundles are injective can be readily
established by choosing an integral component $C$ of $Q$ and observing that
if a product of morphisms of line bundles vanishes, then so does the
product of their restrictions to $C$, and thus one of the two morphisms
vanishes on $C$; the result then follows by induction on the overall
difference of Chern classes.

This is all we need for the proof of Theorem
\ref{thm:blowup_of_plane_is_F1}, which states that a one-point blowup of a
noncommutative plane is a noncommutative Hirzebruch surface.  (We need to
know that $H^0$ of a nonzero $0$-dimensional sheaf is nonzero, but this
follows from vanishing of $H^2$ and positivity of the Euler
characteristic.)  Given that, we can then reduce the remaining claims to
statements on such surfaces.  In particular, that a noncommutative plane is
a maximal order iff $q$ is torsion (and commutative iff $q$ is trivial)
follows by choosing points in distinct orbits, and observing that the
maximal order property follows from the fact that both blowups are maximal
orders.

\section{Birational geometry}

We can now finish the proof that the usual isomorphisms in the birational
geometry of commutative surfaces carry over to the noncommutative setting.
We recall that in each case we have constructed a derived equivalence and a
pair of left exact functors related by the derived equivalence, and thus to
show that the derived equivalence induces an abelian equivalence, it
suffices to show that the functors $f$ satisfy the property that for any
nonzero sheaf $M$, $f_*\theta^l M\ne 0$ for some $l\in \Z$.

For elementary transformations, the functor is $\rho_{0*}\alpha_*$,
so suppose $M$ is a nonzero sheaf such that $\rho_{0*}\alpha_*\theta^l M=0$
for all $l\in Z$.  Since
\[
-\rank(R^1(\rho_{0*}\alpha_*)\theta^l M)
=
\rank(R\rho_{0*}R\alpha_*\theta^l M)
=
\rank(R\rho_{0*}R\alpha_*M)-l\rank(M),
\]
we must have $\rank(M)=0$, so that $M$ is a (pure) $1$-dimensional sheaf.
(Here we have used the fact that $\rho_{0*}\alpha_*$ has cohomological
dimension 1.)  We then have
\[
f\cdot c_1(M)
=
-\rank(R^1(\rho_{0*}\alpha_*)M).
\]
If this were not 0, then there would be line bundles on $C_0$ such that
$\Ext^2(\alpha^*\rho_0^*L,M)\ne 0$, which is impossible for a
$1$-dimensional sheaf, and thus $f\cdot c_1(M)=0$, so that $c_1(M)\in \Z
f+\Z e_1$.  Moreover, since now $h^1(R\rho_{0*}R\alpha_*\theta^l M)$ is
$0$-dimensional, we find that $R\Hom(\sO_{C_0},R\rho_{0*}R\alpha_*\theta^l
M)$ is supported in degree $1$, so its Euler characteristic is negative.
It follows that $c_1(M)\cdot K=0$, and thus $c_1(M)=a(f-2e_1)$ for some
nonzero $a\in \Z$.

Choose such an $M$ that minimizes $|a|$.  We have $\chi(M,M)=4a^2\ge 4$,
and thus $\dim\Hom(M,M)+\dim\Ext^2(M,M)\ge 4$.  If $\dim\Hom(M,M)>1$, then
the endomorphism ring of $M$ contains zero divisors, and thus there are
nonzero endomorphisms with nontrivial kernel.  The kernel and cokernel are
both subsheaves of $M$, so have nonzero Chern class proportional to
$f-2e_1$, which by minimality of $|a|$ forces them to have opposite signs,
contradicting Proposition \ref{prop:neg_of_eff_not_eff}.  We thus conclude
that $\dim\Hom(M,M)=1$ and thus $\dim\Ext^2(M,M)\ge 3$.  But this is
equivalent to $\dim\Hom(M,\theta M)=3$, and the same argument tells us that
any such morphism must be injective.  Since $M$ and $\theta M$ have the
same class in the numeric Grothendieck group, such a morphism must be an
isomorphism, and we get the desired contradiction.  We have thus shown the
following.

\begin{thm}\label{thm:elem_xform_works}
  Let $X$ be a quasi-ruled surface and $x\in Q$ a point.  Then the blowup
  of $X$ at $x$ is also a blowup of the quasi-ruled surface obtained by
  replacing the sheaf bimodule ${\cal E}$ by the kernel of the point in
  $\Quot({\cal E},1)$ associated to $x$.
\end{thm}

\begin{rem}
  The construction of the derived equivalence via the associated
  commutative ruled surface shows that this operation is an involution.
  The above description may appear to violate this (both steps make ${\cal
    E}$ smaller!), but this is an illusion coming from the fact that ${\cal
    E}$ is not uniquely determined by the quasi-ruled surface.  Indeed,
  what one finds is that the sheaf bimodule resulting after two steps is
  isomorphic to the tensor product of ${\cal E}$ with the line bundle
  $\pi_1^*{\cal L}(-\pi_1(x))$, and thus determines the same quasi-scheme
  $(X,\sO_X)$.
\end{rem}  

We also note the following useful fact we proved in the course of the above
argument.

\begin{prop}\label{prop:f_is_nef}
  If $D\in \NS(X)$ is effective, then $D\cdot f\ge 0$.
\end{prop}

Since $K\cdot f=-2$, we immediately conclude the following.

\begin{cor}
  The class $K$ is ineffective.
\end{cor}

This has the following useful consequence.

\begin{cor}
  If $M$, $N$ are torsion-free sheaves of rank 1, then at least one of
  $\Hom(M,N)$ and $\Ext^2(M,N)$ vanishes.
\end{cor}

\begin{proof}
  A nonzero morphism $M\to N$ must be injective, since otherwise the image
  would be a $\le 1$-dimensional subsheaf of $N$.  The cokernel is thus a
  $1$-dimensional sheaf of Chern class $c_1(N)-c_1(M)$, which must
  therefore be effective whenever $\Hom(M,N)\ne 0$.  Since
  $\Ext^2(M,N)\cong \Hom(N,\theta M)^*$, we similarly find that if
  $\Ext^2(M,N)\ne 0$ then $c_1(M)-c_1(N)+K$ is effective.  But these
  divisor classes cannot be simultaneously effective, since their sum $K$
  is ineffective.
\end{proof}

\medskip

For the remaining (rational) cases, the functor we consider is just the
usual global sections functor $\Gamma(M):=\Hom(\sO_X,M)$.  Here we have
some additional complications coming from the fact that this functor has
cohomological dimension 2.  Luckily, this is only a possibility when $M$ is
$2$-dimensional, and we can rule that case out.

\begin{lem}
  Let $M$ be a sheaf on a noncommutative rational or rationally quasi-ruled
  surface $X$.  Then $\Hom(M,\theta^l \sO_X)=0$ for $l\gg 0$.
\end{lem}

\begin{proof}
  It suffices to show this when $M$ is a line bundle, since in general
  there is a surjection to $M$ from a power of a line bundle.  On a
  noncommutative plane, the line bundles have the form $\sO_X(d)$, and we
  have $\Hom(\sO_X(d),\theta^l \sO_X)=0$ for $l>-d/3$.  Otherwise, if
  $\Hom(L,\theta^l\sO_X)\ne 0$, then $c_1(L)+l K_X$ is effective, but
  $(c_1(L)+lK_X)\cdot f=c_1(L)\cdot f-2l$ is eventually negative.
\end{proof}

\begin{rem}
  By Serre duality, this implies that $H^2(\theta^{-l}M)=0$ for $l\gg 0$.
\end{rem}

This implies that if $M$ is a sheaf on a noncommutative Hirzebruch surface
or a one-point blowup of a noncommutative plane such that $H^0(\theta^l
M)=0$ for all $l$, then $M$ must be $1$-dimensional.  Indeed, if
$\rank(M)>0$, then then we have $\chi(\theta^{-l}M) = 4\rank(M) l^2 +
O(l)$, and thus $\chi(\theta^{-l}M)>0$ for $l\gg 0$, but this is impossible
since $H^2(\theta^{-l}M)=0$ for $l\gg 0$.

We then find as in the elementary transformation case that $c_1(M)\cdot
K=0$, which uniquely determines $c_1(M)$.  For $F_1$ and the one-point
blowup of $\P^2$, we have $\chi(M,M)\ge 8$, and obtain a
contradiction as before.

\begin{thm}\label{thm:blowup_of_plane_is_F1}
  Let $X$ be a noncommutative Hirzebruch surface such that the
  corresponding commutative surface is $F_1$.  Then $X$ is the one-point
  blowup of a noncommutative plane.  Conversely, any one-point blowup of a
  noncommutative plane is the noncommutative Hirzebruch surface associated
  to a point of $\Pic^0(Q)$ with $Q$ an anticanonical curve in $F_1$.
\end{thm}

The $\P^1\times \P^1$ case is somewhat more subtle, as the inequality we
obtain is only that $\chi(M,M)\ge 2$, which is in fact possible.  However,
this does at least pin down that $c_1(M)=\pm (s-f)$, and $-s+f$ can be
ruled out since $(-s+f)\cdot f<0$.  If $M$ is not transverse to $Q$ (i.e.,
if $M|^{\bf L}_Q$ is not a sheaf), then this gives a nonzero subsheaf of
$\theta M$ supported on $Q$, which must have the same Chern class, and thus
some component of $Q$ is a $-2$-curve on the corresponding commutative
surface (and we may take $M=\sO_{Q'}(-1)$ where $Q'$ is that component).
Otherwise, since $c_1(M)\cdot Q=0$, we must have $M|^{\bf L}_Q=0$ and thus
$\det(M|^{\bf L}_Q)\cong \sO_Q$.  In particular, the image of $s-f$ in
$\Pic^0(Q)$ must be a power of $q\in \Pic^0(Q)$, and we conclude the
following.

\begin{thm}
  Let $X$ be the noncommutative Hirzebruch surface associated to a triple
  $(\P^1\times \P^1,Q,q)$ with $Q\subset \P^1\times \P^1$ an anticanonical
  curve and $q\in \Pic^0(Q)$.  If the line bundle $\sO_{\P^1\times
    \P^1}(1,-1)|_Q$ is not a power of $q$, then swapping the two rulings
  gives an isomorphic noncommutative Hirzebruch surface.
\end{thm}

Note that when the ``not a power of $q$'' hypothesis fails, then the
commutative surface associated to some twist of $X$ will be $F_2$.  There
we have the following corollary of Proposition \ref{prop:weak_reflection}.

\begin{prop}
  Suppose that $X$ is the noncommutative Hirzebruch surface associated to a
  triple $(F_2,Q,q)$ with $Q\subset F_2$ an anticanonical curve disjoint
  from the $-2$-curve and $q\in \Pic^0(Q)$ of order $r\in [1,\infty]$.
  Then there is an irreducible $1$-dimensional sheaf $\sO_{s-f}(-1)$ of
  Chern class $s-f$ and Euler characteristic 0, and if $0\le b-a\le r$,
  there is a short exact sequence
  \[
  0\to \sO_X(-bs-af)\to \sO_X(-as-bf)\to \sO_{s-f}(-1)^{b-a}\to 0.
  \]
\end{prop}

\begin{proof}
  Let $x\in Q$ be any point.  If we blow up $X$ in $x$, then an elementary
  transformation lets us interpret $\widetilde{X}$ as a blowup of a
  noncommutative $F_1$, and thus as a two-fold blowup of a noncommutative
  $\P^2$, in which the same point is blown up each time.  The exceptional
  classes on the latter are $e'_1=s-e_1$ and $e'_2=f-e_1$, and thus there
  is a sheaf $\sO_{e'_1-e'_2}(-1)$, the direct image of which is the
  desired sheaf $\sO_{s-f}(-1)$ on $X$.  Moreover, there is a short exact
  sequence
  \[
  0\to \sO_{\widetilde{X}}(-be'_1-a'e_2)\to \sO_{\widetilde{X}}(-ae'_1-be'_2)
  \to \sO_{e'_1-e'_2}(-1)^{b-a}\to 0
  \]
  or equivalently
  \[
  0\to \sO_{\widetilde{X}}(-bs-af+(a+b)e_1)
  \to \sO_{\widetilde{X}}(-as-bf+(a+b)e_1)
  \to \sO_{e'_1-e'_2}(-1)^{b-a}\to 0
  \]
  The result follows by taking direct images down to $X$.
\end{proof}

\begin{rem}
  If $Q$ is not disjoint from the $-2$-curve, then it contains the
  $-2$-curve, so that one may define $\sO_{s-f}(-1)$ as the appropriate
  line bundle on that component of $Q$, and one again obtains a short exact
  sequence for $0\le b-a\le 1$.
\end{rem}

\bigskip

The other thing we want to show vis-\`a-vis birational geometry is that
``rationally quasi-ruled'' and ``birationally quasi-ruled'' are essentially
the same, or in other words the following.

\begin{thm}\label{thm:birationally_qr_is_rationally_qr}
  Suppose that $X$ is a noncommutative surface such that some iterated
  blowup of $X$ is an iterated blowup of a quasi-ruled surface with
  associated curves $(C_0,C_1)$.  Then either $X$ is an iterated blowup of a
  quasi-ruled surface over the same pair of curves, or $C_0\cong C_1\cong
  \P^1$ and $X$ is a noncommutative plane.
\end{thm}

By an easy induction, this reduces to showing that if a one-fold blowup of
$X$ is an iterated blowup of a quasi-ruled surface or noncommutative
plane, then the same is true for $X$, and thus to understanding when we
can blow down a $-1$-curve on a rationally quasi-ruled surface.

Thus let $\alpha:\widetilde{X}\to X$ be a van den Bergh blowup with exceptional
sheaf $\sO_e(-1)$, such that $\widetilde{X}$ is rational or rationally
quasi-ruled.  We first need to deal with the technical issue that $X$ may
not have a well-defined dimension function, and thus it is not immediately
obvious that the exceptional sheaf is $1$-dimensional.  This is not too
difficult to show: $X$ inherits a Serre functor from $\widetilde{X}$, and if
$M_l:=\alpha^*(\sO_X(lC))(-l)$ (where $C$ is the relevant curve on $X$),
then we have short exact sequences
\[
0\to M_l\to M_{l+1}\to \sO_e(-l-1)\to 0,
\]
since $\sO_X(lC)$ is a weak line bundle along $C$.  An easy induction from
$\theta^{-1}\sO_e(-1)\cong \sO_e$ shows that $\sO_e(-l-1)\cong
\theta^{-l}\sO_e(-1)$ and thus has the same rank, implying that
$\rank(M_l)=\rank(\sO_e(-1))l+O(1)$.  Since ranks of coherent sheaves are
nonnegative, we conclude that $\rank(\sO_e(-1))=0$ as desired, making $e$
$1$-dimensional.

Let $e$ denote the Chern class of $\sO_e(-1)$.  As one might expect, this
behaves numerically like a $-1$-curve on a commutative surface.  Since
$\sO_e(-1)$ is exceptional, $-e^2 = \chi(\sO_e(-1),\sO_e(-1))=1$.  (In
particular, $\widetilde{X}$ cannot be a noncommutative plane, since
$\NS(\widetilde{X})\cong \NS(\P^2)$ has no classes of self-intersection $-1$.)
In addition, since
\[
1 = \chi(\sO_{\tau p}) =
\chi(\sO_e)-\chi(\sO_e(-1))=\chi(\sO_e)-\chi(\theta\sO_e),
\]
we have $e\cdot K=-1$.  (Note that we can compute $\chi(\sO_{\tau p})$
inside $\coh(\widetilde{C})$.)  Moreover, since $R\Gamma(\sO_e(-1))=0$,
$\sO_e(-1)$ must be pure $1$-dimensional.  

Since $\rank(\sO_e(d))=0$, $\theta\sO_e(d)$ and $\sO_e(d)$ have the same
Chern class, so that $c_1(\sO_e(d))=e$.  It then follows that
$\sO_e/\sO_e(-1)$ has numerical class $\chi(\sO_e/\sO_e(-1))[\pt]=[\pt]$.
Thus this sheaf is a point, and the same holds (by algebraic equivalence)
for any point $x\in\widetilde{C}$.  In particular, for any $x\in C$,
$[L\alpha^*\sO_x]=[\pt]$.


Suppose $M$ is a quotient of $\sO_e(-1)$ by a proper subsheaf $N$.  Then
the long exact sequence gives $R^1\alpha_*M=0$ and $\alpha_*M\cong
R^1\alpha_*N$, so that $\alpha_*M$ is an extension of point sheaves on $C$,
thus $[L\alpha^*\alpha_*M]\propto [\pt]$ and $[M]$ is in the span of
$[\sO_e(-1)]$ and $[\pt]$.  Since $c_1(M)+c_1(N)=e$ is a sum of effective
divisors, the only possibilities are $c_1(M)=e$ or $c_1(M)=0$, and the
former case can be ruled out since $\sO_e(-1)$ is pure $1$-dimensional.  We
thus conclude that $\sO_e(-1)$ is irreducible in a suitable sense: any
proper quotient is $0$-dimensional.

When $X_0$ (the surface of which $\tilde{X}$ is an iterated blowup) is
strictly quasi-ruled, there is an additional constraint on $e$, namely that
$e\cdot Q=1$.  (This constraint is equivalent to $e\cdot K=-1$ if $X_0$ is
ruled, and is actually redundant unless $g(C_0)+g(C_1)=1$!)  Since $K+Q$ is
a nonnegative multiple of $f$, Proposition \ref{prop:f_is_nef} tells us
that $e\cdot Q\ge e\cdot (-K)\ge 1$, so it suffices to show that if
equality fails, then $e\cdot f=0$.  Suppose $e\cdot Q>1$.  Then $e$ meets
$Q$ in at least 2 points (possibly by being a component of $Q$), and thus
one finds that $\dim\Hom(\sO_e(-1),\sO_e)>1$.  The kernel of any such
homomorphism is a point sheaf, and thus there is an induced family of point
sheaves parametrized by $\P^1$.  Taking images to $X_0$ gives a family of
point sheaves on $X_0$ and thus a morphism from $\P^1$ to $\Quot(X_0,1)$
(which is well-defined since $X_0$ is a maximal order).  Since $X_0$ is
strictly quasi-ruled, its curve of points is the union of an irreducible
curve $\hat{Q}$ of positive genus and a collection of fibers, and thus the
morphism must map to a fiber, so that $e\cdot f=0$ (which we can compute as
an intersection in the center).

Thus in general, given an iterated blowup $X$ of a quasi-ruled surface, we
need to classify the elements $e\in \NS(X)$ such that $e^2=e\cdot K=-e\cdot
Q=-1$ and there exists an irreducible $1$-dimensional sheaf $\sO_e(-1)$
with Chern class $e$ and vanishing $R\Gamma$.  Call such a class a ``formal
$-1$-curve'' on $X$; our objective is to show that any such class can be
blown down, and the result is as described in the Theorem.

Note first that if $m=0$ with $s^2=-1$ (if $s=0$ any self-intersection is
even), then the equations $e^2=e\cdot K=-1$ imply $g(C_0)+g(C_1)\in
\{0,2\}$, as one must solve a quadratic equation with discriminant
$9-4g(C_0)-4g(C_1)$.  If $g(C_0)+g(C_1)=2$, the unique solution is $e=-s$,
which is ineffective since then $e\cdot f<0$, while if $g(C_0)=g(C_1)=0$,
then the unique solution is $e=s$, which can be blown down (to a
noncommutative plane) precisely when there is an irreducible sheaf
$\sO_e(-1)$ (i.e., when the corresponding commutative surface is $F_1$). So
we may reduce to the case $m>0$.

A key observation is that we have shown that $X$ is not in general uniquely
representable as an iterated blowup of a quasi-ruled surface.  Changing the
representation of course has no effect on $X$, but {\em does} change the
basis we are using for $\NS(X)$.  Thus to show that formal $-1$-curves can
be blown down and the result is again rationally quasi-ruled (or rational),
it suffices to show that in {\em some} representation of $X$ as an iterated
blowup, the expression for $e$ in the resulting basis is $e_m$.

\begin{prop}\label{prop:formal_-1_can_be_blown_down}
  Let $X_m$ be an iterated blowup of a quasi-ruled surface $X_0$, with
  $m>0$, and let $e$ be a formal $-1$-curve on $X_m$.  Then there is a
  sequence of interchanges of commuting blowups, elementary
  transformations, and (if $X$ is rational) swappings of rulings on
  $\P^1\times \P^1$ such that in the resulting basis of $\NS(X_m)$,
  $e$ is represented by $e_m$.
\end{prop}

\begin{proof}
  Consider the representation of $e$ in the associated basis of $\NS(X_m)$.
  If $e\cdot (e_i-e_{i+1})<0$ for some $i$, then there are two
  possibilities: either we may interchange the two blowups, and thus obtain
  a basis in which $e\cdot (e_i-e_{i+1})>0$ (essentially reflecting in
  $e_i-e_{i+1}$), or the blowups do not commute, in which case there is a
  $1$-dimensional sheaf $M$ on $X_{i+1}$ of class $e_i-e_{i+1}$, which is
  either a component of $Q$ or has irreducible pullback to $X_m$.  But then
  the fact that the intersection number is negative forces
  $\Hom(\sO_e(-1),\alpha_m^*\cdots \alpha_{i+2}^*M)\ne 0$ or
  $\Hom(\theta^{-1}\alpha_m^*\cdots \alpha_{i+2}^*M,\sO_e(-1))\ne 0$,
  neither of which can happen by irreducibility of $\sO_e(-1)$.  (An easy
  induction shows that any irreducible constituent of the pullback of $M$
  has Chern class of the form $e_j-\sum_{k\in S} e_k$ where $j\ge i$ and
  $|S|$ consists of integers greater than $j$.)  Similarly, if $D\cdot
  e_1\ge D\cdot f/2$, then we may perform an elementary transformation to
  subtract $D\cdot e_1$ from $D\cdot f$.

  Conjugation by an elementary transformation preserves the finite Weyl
  group $W(D_m)$, and thus this process will eventually terminate in a
  basis such that $e\cdot e_m\le e\cdot e_{m-1}\cdots\le e\cdot e_1\le
  (e\cdot f)/2$.  If $e\cdot e_m<0$, then irreducibility gives $e=e_m$ and
  we are done, so we may assume $e\cdot e_m\ge 0$, and thus we have an
  expression
  \[
  e = ds+d'f-\lambda_1e_1-\cdots-\lambda_me_m
  \]
  in which $\lambda$ is an ordered partition with largest part at most $d/2$.

  We now split into three cases, depending on the ``genus''
  $g:=(g(C_0)+g(C_1))/2$ appearing in the expression for $K_X$.
  If $g\ge 1$, then we may write
  \[
  ds+d'f = -(d/2)K_{X_0} + ((2g-2)d+n/2)f
  \]
  for some integer $n$ (with parity depending on the parity of $2g$ and
  $s^2$), in terms of which we have
  \[
  (ds+d'f)^2 = (2g-2)d^2+dn
  \qquad\text{and}\qquad
  (ds+d'f)\cdot K_{X_0} = -n.
  \]
  Since $e\cdot K = -1$, we find that $\lambda$ is a partition of $n-1$
  into at most $m$ parts all at most $d/2$ (and thus in particular $n\ge
  1$).  It follows that
  \[
  \sum_i \lambda_i^2 \le d(n-1)/2,
  \]
  and thus
  \[
  e^2 = (2g-2)d^2+dn-\sum_i \lambda_i^2
  \ge (2g-2)d^2+d(n+1)/2
  \ge 0,
  \]
  giving a contradiction.

  If $g=1/2$, then $X_0$ is strictly quasi-ruled, and not of the
  $4$-isogeny type, so that $Q=-K+lf$ for some $l>0$.  Since $e\cdot
  (Q+K)=0$, we conclude that $e\cdot f=0$, and thus $d=0$, forcing
  $\lambda=0$ and $e^2=0$.

  In the $g=0$ case (i.e., rational surfaces), we must consider an
  additional reflection, namely the reflection in $s-f$ when $X_0$ is an
  even Hirzebruch surface or in $s-e_1$ when $X_0$ is an odd Hirzebruch
  surface.  Again, we can perform this reflection precisely when the
  corresponding class is not effective, and thus irreducibility of
  $\sO_e(-1)$ again ensures that we can perform the reflection whenever we
  would want to.  Moreover, if we have already put $e$ into standard form
  relative to interchanges of blowups and elementary transformations, then
  the reflection in $s-f$ or $s-e_1$ will decrease $e\cdot f$, and since
  $e$ is effective (thus $e\cdot f\ge 0$), the resulting procedure
  (alternating between the two reductions) will necessarily terminate.
  (Note that it is no longer the case that there is a finite Coxeter group
  forcing termination!)

  For convenience, we then perform an elementary transformation if needed
  to make $X_0$ an even Hirzebruch surface, so obtain an expression of the
  form
  \[
  e = af + b(s+f)+c(s+f-e_1)+d(2s+2f-e_1-e_2) - \sum_i \lambda_i e_{i+2}
  \]
  in which $a,b,c\ge 0$ and $\lambda$ is a partition with $\lambda_1\le d$.
  (Of course, if $m=1$, we must have $\lambda=d=0$ and find $e^2\ge 0$) We
  have $|\lambda| = 2a+4b+3c+6d-1$ and thus
  \begin{align}
  e^2 &= 2ab+2ac+4ad+2b^2+4bc+8bd+c^2+6cd+6d^2
  - \sum_i \lambda_i^2\notag\\
  &\ge 2ab+2ac+4ad+2b^2+4bc+8bd+c^2+6cd+6d^2
  - \lambda_1|\lambda|\notag\\
  &\ge 2ab+2ac+2ad+2b^2+4bc+4bd+c^2+3cd+d\notag\\
  &\ge 0,
  \end{align}
  again giving a contradiction.
\end{proof}  

\begin{rem}
  It is worth noting that this only applies to the van den Bergh blowup; in
  the case of a maximal order, the underlying commutative surface $Z$ may
  very well have $-1$-curves that cannot be blown down via the above
  Proposition.  The issue is that the local structure of the resulting
  order on the blown down surface at the base point may be ``singular'' in
  a suitable sense.  One example of this comes from biquadratic extensions
  of $\P^1$.  Any invertible sheaf on such an extension produces a
  quasi-ruled surface corresponding to a quaternion order on a Hirzebruch
  surface, and when the Euler characteristic of the invertible sheaf has
  the correct parity (depending on the various genera), the Hirzebruch
  surface is, at least generically, $F_1$.  In characteristic 0, one can
  represent this as the Clifford algebra associated to a quadratic form,
  and find that the blown down algebra is the Clifford algebra associated
  to a quadratic form on a vector bundle on $\P^2$ such that the form
  vanishes identically at the base point of the blowup.  In particular, the
  base change of the order to the complete local ring at that point is
  local but not of global dimension 2.  Thus the blow down in this case
  more closely resembles the contraction of a $-2$-curve than the
  contraction of a $-1$-curve.
\end{rem}

\begin{rem}
  When $g>0$ and thus the relevant group is finite, we may conclude that in
  the original basis for $\NS(X)$, we have $e=e_i$ or $e=f-e_i$ for
  some $1\le i\le m$.  It is worth noting that these classes are always
  effective and satisfy all of the numerical conditions to be a formal
  $-1$-curve; the only condition they can violate is irreducibility, which
  happens iff $e_i-e_j$ or $f-e_i-e_j$ (respectively) is effective for some
  $j>i$.
\end{rem}

If $X$ is a noncommutative rationally ruled surface, a ``blowdown
structure'' on $X$ is an explicit isomorphism between $X$ and an iterated
blowup of a quasi-ruled surface (which in the rational case should be
expressed {\em as} a ruled surface; thus a noncommutative $\P^1\times \P^1$
has two blowdown structures).  The ``parity'' of a blowdown structure is
the parity of the self-intersection of the resulting class $s$.  The above
discussion shows that a typical rationally quasi-ruled surface has many
blowdown structures, but also gives us a fair amount of control.
Specifying a blowdown structure is essentially equivalent to specifying a
basis of $\NS(X)$ in standard form, and thus any two blowdown structures
specify a change of basis matrix which respects both $K$ and the
intersection pairing.  The atomic birational transformations all act in
relatively simple ways on the basis; in particular, the commutation of
blowups and the exchange of rulings on $\P^1\times \P^1$ both act as
reflections relative to the intersection pairing, with corresponding roots
$e_i-e_{i+1}$ and $s-f$ respectively.  Moreover, although an elementary
transformation does not have such a description (since it changes the
parity and thus the intersection matrix), the conjugate of the reflection
in $e_1-e_2$ by an elementary transformation is the reflection in
$f-e_1-e_2$.  (Similarly, if $X_0$ is an odd Hirzebruch surface that blows
down to $\P^2$, then reflection in $s-e_1$ acts as commutation on the
two-point blowup of $\P^2$.)  In each case, the condition for the image
under the reflection to still be a blowdown structure is precisely that the
corresponding root is ineffective.  With this in mind, we call the divisor
classes $f-e_1-e_2, e_1-e_2,\dots, e_{m-1}-e_m$ ``simple roots'', along
with $s-f$ when $X_0$ is an even Hirzebruch surface and $s-e_1$ when $X_0$
is an odd Hirzebruch surface.  Note that these are the simple roots of a
Coxeter group of type $D_m$ for $g>0$ and $E_{m+1}$ for $g=0$.  This,
together with the reduction of Proposition
\ref{prop:formal_-1_can_be_blown_down}, motivates the following definition.

\begin{defn}
  Let $X$ be a noncommutative rationally quasi-ruled surface with $m\ge 1$.
  Then the {\em fundamental chamber} in $\NS(X)$ (relative to a given
  blowdown structure) is the cone of divisors having nonnegative
  intersection with every simple root and with $e_m$, as well as $f-e_1$
  when $m=1$.
\end{defn}

\begin{prop}
  Let $X$ be a noncommutative rationally quasi-ruled surface.  Then any two
  blowdown structures on $X$ with $X_0$ of the same parity are related by a
  sequence of reflections in ineffective simple roots.
\end{prop}

\begin{proof}
  Let $s,f,e_1,\dots,e_m$ and $s',f',e'_1,\dots,e'_m$ be the bases
  corresponding to the two blowdown structures.  For $m\ge 1$, we have
  shown that there is a sequence of simple reflections and elementary
  transformations taking the second blowdown structure to one in which
  $e'_m=e_m$.  For $m\ge 2$, this is unchanged by an elementary
  transformation, and since the conjugate of a simple reflection by an
  elementary transformation is a simple reflection, we may arrange to
  eliminate all of the elementary transformations.  We thus see that the
  second blowdown structure is related by a sequence of reflections in
  ineffective simple roots to one coming from a blowdown structure on
  $X_{m-1}$ of the same parity, so that the claim follows by induction.

  For $m=1$, a formal $-1$-curve determines the parity of the blowdown
  structure, and the only case in which there is more than one of the
  correct parity is when $X_0$ is an odd Hirzebruch surface and both $s$
  and $e_1$ are formal $-1$-curves.  But in that case $s$ is a formal
  $-1$-curve iff $s-e_1$ is ineffective, and thus again they are related by
  a reflection.  Similarly, for $m=0$, the ruling is uniquely determined
  unless $X_0$ is an even Hirzebruch surface on which $s-f$ is ineffective,
  in which case the corresponding reflection swaps the rulings.
\end{proof}

\begin{prop}\label{prop:f_is_unique_if_quasi-ruled}
  Let $X$ be a noncommutative rationally quasi-ruled surface of genus $>0$.
  Then the functor $R\rho_{0*}R\alpha_{1*}\cdots R\alpha_{m*}$ is independent
  of the choice of blowdown structure.
\end{prop}

\begin{proof}
  Reflection in $e_i-e_{i+1}$ preserves the three factors
  \[
  R\rho_{0*}R\alpha_{1*}\cdots R\alpha_{(i-1)*},
  R\alpha_{i*}R\alpha_{(i+1)*},
  R\alpha_{(i+2)*}\cdots R\alpha_{m*},
  \]
  so preserves the composition, while an elementary transformation
  preserves the two factors
  \[
  R\rho_{0*}R\alpha_{1*},
  R\alpha_{2*}\cdots R\alpha_{m*}.
  \]
  Since any two blowdown structures are related by a sequence of simple
  reflections and elementary transformations, the claim follows.
\end{proof}

\begin{rem}
  An alternate approach is to observe that $f$ is the unique effective
  class such that $f^2=0$, $f\cdot K=-2$, and that $C_0$ is the moduli
  space of irreducible $1$-dimensional sheaves of Chern class $f$ and Euler
  characteristic 1, with the morphism to $C_0$ following as in
  \cite{ChanD/NymanA:2013}.
\end{rem}

\section{Effective, nef, and ample divisors}
\label{sec:effnef}

Since the surfaces we are deforming are projective, not just proper, we
would like to know that they have some analogue of ample divisors, and
ideally be able to control the set of such divisors.  It is unreasonable to
hope that ample divisors will correspond to graded algebras (let alone
quotients of noncommutative projective spaces), but we can still hope for
some version of Serre vanishing and global generation.  By Proposition
\ref{prop:NS_of_center}, this is at least in principle easy when $X$ is a
maximal order (though the lack of an anticanonical curve in the general
quasi-ruled case makes it difficult to make everything explicit) as then we
may identify the effective and nef cones of $X$ and its center, so that
anything in the interior of the nef cone is ample.

We will thus restrict our attention to the ruled surface case.  Note here
that we are focused on {\em noncommutative} surfaces, and thus any
commutative fiber will come with a choice of anticanonical curve.  (It is
also worth noting that our arguments do not actually work for the excluded
case of a (characteristic 2) commutative ruled surface of genus 1 with
integral anticanonical curve!)

In the commutative setting, the ample divisors on a surface are best
understood as the integer interior points of the cone of nef
divisors, for which we certainly have a well-behaved analogue in the
noncommutative setting: a divisor class is nef iff it has nonnegative
intersection with every effective divisor class.

Of course, to understand the nef cone, we will first need to fully
understand the effective cone.  To describe this cone, we will need one
more set of effective classes.

\begin{defn}
  A formal $-2$-curve is a $\theta$-invariant class of self-intersection
  $-2$ which is the Chern class of an irreducible sheaf.
\end{defn}

\begin{prop}
  If $\alpha$ is a formal $-2$-curve, then there is a blowdown structure in
  which $\alpha$ is a simple root.
\end{prop}

\begin{proof}
  This is by the same argument as in Proposition
  \ref{prop:formal_-1_can_be_blown_down}, with only minor changes to the
  final inequalities.
\end{proof}

\begin{thm}
  Let $X$ be a rationally ruled surface with $m\ge 1$.  Then the effective
  monoid of $X$ is generated by components of $Q$, formal $-1$-curves, and
  formal $-2$-curves.
\end{thm}

\begin{proof}
  Let $\mu$ denote the monoid so generated, which is clearly contained in
  the effective monoid.  We will proceed in two steps: first showing that
  the effective monoid is generated by $\mu$ and its dual, then showing
  that $\mu$ contains its dual.

  For the first claim, let $M$ be a pure $1$-dimensional sheaf.
  If $c_1(M)\cdot Q_i<0$ for some component $Q_i$ of $Q$, then
  $\chi(M|^{\bf L}_{Q_i})<0$ and thus $h^{-1}(M|^{\bf L}_{Q_i})\ne 0$.
  Since this gives a subsheaf of $M(-Q_i)$, it must have positive rank, and
  thus there is a subsheaf of $M$ of Chern class $Q_i$.  Similarly, if
  $D\cdot e<0$ for some formal $-1$-curve $e$, then $\chi(\sO_e(-1),M)>0$,
  so that either $\Hom(\sO_e(-1),M)\ne 0$ or $\Hom(M,\theta \sO_e(-1))\ne
  0$, giving either a subsheaf of class $e$ or a quotient sheaf of class
  $e$.  Finally, if $\alpha$ is a formal $-2$-curve with $D\cdot \alpha<0$,
  then $\chi(M_\alpha,M)>0$ for any irreducible sheaf $M_\alpha$ of the
  appropriate Chern class, and thus again there is a sub- or quotient sheaf
  of Chern class $\alpha$.

  We may thus obtain a descending sequence of subquotients of $M$ by
  repeatedly choosing sub- or quotient sheaves of Chern class a component
  of $Q$ or a formal $-1$- or $-2$-curve.  The corresponding pair of
  chains of subsheaves of $M$, one ascending, one descending, both
  terminate, and thus this process terminates.  The only way it can
  terminate is with a subquotient of $M$ having nonnegative intersection
  with all generators of $\mu$, giving an expression for $c_1(M)$ as the
  sum of an element of $\mu$ and an element of the dual monoid as required.

  To proceed further, we note first that any effective simple root (for any
  blowdown structure!) is either a formal $-2$-curve or a sum of components
  of $Q$ (with negative self-intersection) and formal $-1$-curves, so is in
  $\mu$.  Moreover, if $e$ is in the $W(D_m)$ or $W(E_{m+1})$-orbit of
  $e_m$, then when applying the procedure of Proposition
  \ref{prop:formal_-1_can_be_blown_down} to move $e$ back to $e_m$, each
  reflection in an effective simple root subtracts a positive multiple of
  that root, and thus $e$ is a sum of effective roots and a formal
  $-1$-curve, so again in $\mu$.  In particular, both $e_1$ and $f-e_1$ are
  in $\mu$, so that $f$ is $\mu$.  Furthermore, the pullback of a class in
  $\mu(X_{m-1})$ is in $\mu(X_m)$.

  Thus if $D$ is in the dual of $\mu$ and $\alpha$ is a simple root such
  that $D\cdot \alpha<0$, then $\alpha$ is ineffective and we may reflect
  in $\alpha$, and iterating this process will eventually terminate since
  there can be only a bounded number of steps in a row which do not
  decrease $D\cdot f$ and $D\cdot f$ cannot become negative.  We may thus
  assume that $D$ is in the fundamental chamber.

  If $g=0$, we conclude that the dual of the monoid is contained in the
  simplicial monoid generated by (relative to an even blowdown structure)
  \[
  f,s+f,s+f-e_1,2s+2f-e_1-e_2,\dots,2s+2f-e_1-\cdots-e_m.
  \]
  (This is just the dual of the monoid generated by simple roots and
  $e_m$.)  For each $i\ge 2$, the class $2s+2f-e_1-\cdots-e_i$ is certainly
  a sum of components of $Q$ on $X_i$, and thus (being a pullback) is
  in $\mu$.  We have already shown that $f$ is in $\mu$, and since $s-e_1$
  is in the orbit of $f-e_1$, it is also in $\mu$, so that
  $(s-e_1)+(f-e_1)+(e_1)$ and $(s-e_1)+(f-e_1)+2(e_1)$ are in $\mu$.
  It follows that any element of the dual monoid is in $\mu$ as required.

  If $g>0$, we note that $X_0$ (which we again take to be even) has at
  least one component of $Q$ of class $s-df$ for $d\ge 0$, which is again
  in $\mu$ (as a pullback).  We thus conclude that $D\cdot (s-df)\ge 0$,
  telling us that $D$ is an element of the simplicial monoid with
  generators
  \[
  f,s+df,s+df-e_1,2s+df-e_1-e_2,\dots,2s+df-e_1-\cdots-e_m.
  \]
  It is again easy to see that each of these classes is in $\mu$,
  using
  \[
  Q = -K_X=2s+(2-2g)f-e_1-\cdots-e_m
  \]
  for all but the first three generators and the expressions
  \begin{align}
    f &= (f-e_1)+(e_1)\\
    s+df &= (s-df) + d(f-e_1)+d(e_1)\\
    s+df-e_1 &= (s-df)+d(f-e_1)+(d-1)(e_1).
  \end{align}
  If $d=0$, this last case fails, but then $Q$ on $X_0$ has no vertical
  components, and thus the point being blown up must lie on one of the at
  most two components of class $s$, so that $s-e_1$ is in $\mu$.
\end{proof}

\begin{rem}
  This fails if $m=0$, but is easy enough to work around since a class on
  $X_0$ is effective iff its pullback is effective.  We find that the
  effective monoid of $X_0$ is generated by $s'$ and $f$ where $s'$ is the
  effective class of minimal self-intersection such that $s'\cdot f=1$, and
  the only case in which $s'$ is not a formal $-1$- or $-2$-curve or a
  component of $Q$ is when $X_0$ is noncommutative $\P^1\times \P^1$ and
  $s'=s$.  Similarly, the effective monoid of a noncommutative plane is
  the same as the nef monoid, and is generated by the class
  $h:=c_1(\sO_X(1))$.
\end{rem}

\begin{rem}
  It is worth noting that effectiveness is a {\em numerical} condition;
  there will be {\em some} $1$-dimensional sheaf with the given Chern
  class, but we cannot expect to have any particular control over the
  continuous part of its class in $K_0(X)$.  Thus, for instance, if $g>1$,
  then $f$ is effective, but not every line bundle of Chern class $f$ has a
  global section.  (This can also be an issue for effective simple roots in
  general.)
\end{rem}

\begin{cor}
  Any nef divisor class is effective.
\end{cor}

\begin{proof}
  In the course of the above proof, we showed that the dual of the
  effective monoid is contained in the effective monoid.
\end{proof}

It is important to note that some of these generators are redundant.

\begin{prop}
  For $m\ge 1$, the effective monoid is generated by formal $-1$- and
  $-2$-curves and those components $Q_i$ of $Q$ such that $Q_i\cdot Q\le
  0$, unless $g=0$, $m=7$, and $Q$ is integral.
\end{prop}

\begin{proof}
  We need to show that if $Q_i$ is a component of $Q$ such that $Q_i\cdot
  Q>0$, then $Q_i$ is in the stated span.  If $Q_i$ is a vertical component
  with positive intersection with $Q$, then it is either exceptional or the
  strict transform of a fiber, and thus is either $f$ or a $-1$-curve, and
  is redundant in either case.

  If $g>1$ or in the differential case for $g=1$, any horizontal component
  has negative intersection with $Q$.  In the remaining (elliptic
  difference) case with $g=1$, there are two disjoint components,
  $s-df-\sum_{i\notin I} e_i$ and $s+df-\sum_{i\in I} e_i$, and only the
  latter can have positive intersection with $Q$.  We moreover must have
  $|I|<2d$.  The difference of the two components may then be written as
  \[
  (2d-|I|)(f) + \sum_{i\in I} (f-e_i) + \sum_{i\notin I} (e_i),
  \]
  each term of which is a sum of formal $-1$- and $-2$-curves.
  
  We are thus left with the rational case.  When $Q$ is integral, we have
  $Q^2=8-m$, and thus need only establish that $Q$ is redundant for $1\le
  m\le 6$.  For $m=1$, this is easy: $2s+2f-e_1=2(f-e_1)+2(s-e_1)+e_1$,
  while for $2\le m\le 6$, there is a commutative rational surface with
  anticanonical curve an $8-m$-gon of $-1$-curves, again giving an
  expression of $-K_X$ as a sum of elements in the orbit of $e_m$.

  When $Q$ is not integral, then its components are rational, and $Q_i\cdot
  Q>0$ implies $Q_i^2\ge -1$.  If $Q_i^2=-1$, then it is a formal
  $-1$-curve, so we may assume $Q_i^2\ge 0$.  As in the proof of
  Proposition \ref{prop:formal_-1_can_be_blown_down}, we may assume that
  $Q_i$ is in the fundamental chamber.  We then find (by the analogous
  inequality) that $Q_i\cdot Q-2=Q_i^2\ge Q_i\cdot Q$ unless $Q_i\cdot
  e_2=0$, so that we reduce to the case $m=1$.  The generators of the
  relevant monoid are contained in the monoid generated by $e_1$, $f-e_1$,
  $s-e_1$, so the claim follows.
\end{proof}

\begin{rem}
  As in the commutative case, when $g=0$, $m=7$ and $Q$ is irreducible, the
  generator $Q$ is {\em nearly} redundant, since $2Q$ {\em can} be
  expressed as a sum of $-1$-curves on a suitable commutative surface.
\end{rem}

Our understanding of the effective monoid also gives us the following
stronger version of the finitely partitive property.

\begin{cor}
  For any divisor class $D$, there are only finitely many $D'$ such
  that both $D'$ and $D-D'$ are effective.
\end{cor}

\begin{proof}
  Since any formal $-1$- or $-2$-curve can be expressed as a nonnegative
  linear combination of simple roots and $e_m$, we find that the effective
  monoid is contained in a finitely generated monoid (generated by
  components of $Q$, simple roots, and $e_m$).  Moreover, all of the
  generators are lexicographically positive (relative to the ordering $s$,
  $f$, $e_1$,\dots, $e_m$ of the standard basis), and thus any linear
  dependence between the generators has two coefficients of opposite sign.
  It follows that the dual cone has dimension $m+2$, and thus its interior
  contains integer elements.  An element of the interior induces a grading
  on the larger monoid with respect to which it is finite and thus there
  are only finitely many elements of the monoid of degree strictly smaller
  than $D$.
\end{proof}

\begin{rem}
It is worth noting that this result continues to hold even if we allow the
surface to vary over a Noetherian base, and simply ask for $D'$ and $D-D'$
to be effective on {\em some} surface in the family.  Indeed, there are
still only finitely many possible divisor classes of components of $Q$
(this reduces to a statement about families of commutative surfaces),
giving a finitely generated monoid containing the effective monoid of any
fiber and having dual of full dimension.
\end{rem}

As for commutative rational surfaces, we can adapt the above arguments to
give an essentially combinatorial algorithm for determining whether a class
is effective (resp. nef).  Indeed, if we {\em start} by reducing to the
fundamental chamber, each step either just changes the blowdown structure
(and possibly reduces $D\cdot f$) or exhibits a formal $-2$-curve or
component of $Q$ having negative intersection with $D$, which we may
subtract without changing whether $D$ is effective.  Thus after a finite
number of steps, we will either have put $D$ into the fundamental chamber
or made $D\cdot f$ negative (in which case it was not effective).  If the
result has nonnegative intersection with every component of $Q$, then it is
effective, and otherwise, we may subtract the offending component and
rereduce to the fundamental chamber as necessary.  Each such step subtracts
a nontrivial effective divisor from $D$, and thus again we can only do this
finitely many times before eventually reaching a divisor having negative
intersection with some element of the interior of the dual of the cone
generated by simple roots, $e_m$ and components of $Q$.

To determine if $D$ is nef, the algorithm is even simpler: any time we
would have wanted to subtract an effective divisor class from $D$, we
simply terminate saying that $D$ was not nef.

\medskip

Having mostly settled the question of what the effective and nef divisor
classes look like, we now turn to understanding {\em ample} divisors.
Since our divisor classes are only numerical, the two things we would like
an ample divisor $D_a$ to satisfy are (a) that for any sheaf $M$, there is
an integer $B$ such that for any line bundle $L$ of Chern class $-bD_a$
with $b\ge B$, $\Ext^p(L,M)=0$ for $p>0$, and (b) that there is a similar
bound guaranteeing that $L\otimes \Hom(L,M)\to M$ is surjective.  (If both
hold, we say that $M$ is {\em acyclically globally generated} by $L$.)  We
will not entirely characterize the divisor classes satisfying these
conditions (there are some technical issues coming from the fact that there
can be $1$-dimensional sheaves disjoint from $Q$ without proper
subsheaves), but will show that anything in the interior of the nef cone
will do.  (It is also easy to see that (a) and (b) fail for divisors which
are not nef, so the only issues arise on the boundary.)

A key observation is that global generation for general sheaves follows if
we can establish global generation for line bundles; there is also a more
subtle reduction to line bundles for acyclic global generation (see below).
Since the only thing we know about $D_a$ is certain numerical inequalities,
we are led to look for a statement along the following lines: if $L$ and
$L'$ are line bundles such that $c_1(L)-c_1(L')$ satisfies certain
inequalities, then $L$ is acyclically globally generated by $L'$.  The
presence of an anticanonical curve turns out to be extremely helpful in
this regard; the condition will essentially be that $c_1(L)-c_1(L')$ is nef
and that the ratio of their restrictions to $Q$ is acyclic and globally
generated.  The latter condition is somewhat different depending on whether
$X$ is rational or has genus $\ge 1$.  Note that in either case $Q$ may be
embedded as an anticanonical curve on a commutative ruled surface.

\begin{lem}
  Let $Q$ be an anticanonical curve on a (commutative) rational surface and
  let $L$ be a line bundle having nonnegative degree on every component of
  $Q$.  Then $L$ is acyclic if $\deg(L)\ge 1$ and globally generated if
  $\deg(L)\ge 2$.
\end{lem}

\begin{proof}
  When the ambient surface is $\P^2$ or $\P^1\times \P^1$, this was shown
  in \cite[\S7]{ArtinM/TateJ/VandenBerghM:1991}.  In fact, the only way in
  which the ambient surface was used in those proofs was in showing two
  numerical connectivity properties of $Q$: first, that if $Q=A+B$ with
  $A,B$ nonzero effective divisors on the ambient surface, then $A\cdot
  B>0$, and second that if $Q=A+B+C$ with $A,B$ nonzero effective and $C$ a
  smooth irreducible component of $Q$, then $A\cdot B>0$.  The first was
  shown for any anticanonical curve on a rational surface in
  \cite[Lem.~5.13]{me:hitchin}, which further showed that $A\cdot B$ is
  even.  Thus in the latter situation, we find that $A\cdot (B+C)$ and
  $B\cdot (A+C)$ are both positive even integers, and thus their sum
  \[
  2(A\cdot B) + (A+B)\cdot C\ge 4.
  \]
  Since $C$ is a smooth rational curve, $(A+B)\cdot C=2$, and thus $A\cdot
  B\ge 1$ as required.
\end{proof}

\begin{prop}
  Let $X$ be a noncommutative rational surface, and let $L$, $L'$ be line
  bundles on $X$ such that $D=c_1(L)-c_1(L')$ is nef.  If $D\cdot Q\ge 1$,
  then $\Ext^p(L',L)=0$ for $p>0$, and if $D\cdot Q\ge 2$, then
  $\Hom(L',L)\otimes L'\to L$ is surjective.
\end{prop}

\begin{proof}
  Since the nef cone is invariant under twisting $X$ by a line bundle, we
  may as well assume that $L'\cong \sO_X$, so that we need to show that for
  any nef divisor $D$ and line bundle $L$ of class $D$, $L$ is acyclic if
  $D\cdot Q\ge 1$ and globally generated if $D\cdot Q\ge 2$.  If $m\ge 1$,
  choose a blowdown structure on $X$ such that $D$ is in the fundamental
  chamber.  In particular, $D$ is in the monoid generated by
  \[
  f,s+f,s+f-e_1,2s+2f-e_1-e_2,\dots,2s+2f-e_1-\cdots-e_m.
  \]
  Acyclicity and global generation are inherited by pullbacks, so if $D$ is
  a pullback (i.e., if $D\cdot e_m=0$) then we may reduce to $X_{m-1}$ and
  induct.  This also applies when $m=1$, except that now the reduction
  applies if either $D\cdot e_1$ or $D\cdot (f-e_1)$ is 0, and we may need
  to perform an elementary transformation before blowing down.  Similarly,
  for the odd $m=0$ case, if $D\cdot s=0$, then $s$ must be a $-1$-curve
  and we can reduce to $m=-1$.

  Thus suppose that $m\ge 1$ and $D\cdot e_m>0$ (with $D\cdot (f-e_1)>0$ if
  $m=1$).  To show that $L$ is acyclic (resp. and globally generated), it
  will suffice to show that $L(-Q)$ and $L|_Q$ are acyclic (resp. and
  globally generated).  The latter holds since $L|_Q$ has nonnegative
  degree on every component (since $L$ is nef) and has total degree $D\cdot
  Q$.  For the former, we need to show that $D-Q$ is nef and either
  vanishes or satisfies the same inequality.  The result is certainly still
  in the fundamental chamber, so that for nefness we need only show that it
  has nonnegative intersection with the components of $Q$ having
  nonnegative intersection with $Q$.  If $Q$ is reducible and $Q_i$ is such
  a component, then $(D-Q)\cdot Q_i=D\cdot Q_i-(Q\cdot Q_i)\ge D\cdot Q_i$
  as required.  It thus remains only to consider when $(D-Q)\cdot Q\ge 1$
  (or $2$, as appropriate).  If $Q^2=8-m\le 0$, there is again no
  difficulty, while if $m<7$, then any generator of the standard monoid has
  intersection $\ge 2$ with $Q$, and $\sO_X$ is certainly acyclic and
  globally generated.

  The only case in which $D-Q$ fails to satisfy the requisite inequality is
  when $m=7$ and $D=2Q$.  In that case, we still have acyclicity, and need
  only establish global generation.  One global section of $\sO_X(2Q)$ is
  easy to understand: simply take the composition of the maps $\sO_X\to
  \sO_X(Q)$ and $\sO_X(Q)\to \sO_X(2Q)$ coming from the anticanonical
  natural transformation.  The cokernel of this global section is an
  extension
  \[
  0\to \sO_X(Q)|_Q\to \sO_X(2Q)/\sO_X\to \sO_X(2Q)|_Q\to 0.
  \]
  Since $\sO_X$ is acyclic and globally generated, $\sO_X(2Q)$ is globally
  generated iff $\sO_X(2Q)/\sO_X$ is globally generated, and since this is
  an extension of sheaves on $Q$ with acyclic kernel, it is globally
  generated iff its restriction to $Q$ is globally generated.  But
  $\sO_X(2Q)|_Q$ is globally generated by degree considerations.

  It remains only to consider the base cases $m=0$ and $m=-1$.  Here, the
  relevant inequalities force $D=0$ or $D\cdot Q\ge 2$, so we always want
  acyclic global generation.  For $m=-1$, any effective line bundle is
  acyclically globally generated, essentially by construction.  For $m=0$,
  the effective monoid is generated by $f$ and either $s$ or the most
  negative horizontal component of $Q$, so if $D\cdot f\ge 2$, then $D-Q$
  will again be nef and either be 0 or have sufficiently large degree on
  $Q$.  For $D\cdot f=1$, acyclicity and global generation reduce to the
  corresponding properties for ${\cal E}$, so again reduce to what happens
  on $Q$, while for $D\cdot f=0$, so $D\propto f$, we find that $\sO_X(D)$
  is the pullback of an acyclic and globally generated sheaf on $\P^1$.
\end{proof}

The argument for higher genus surfaces is analogous, the main difference
being in the specific numeric condition along $Q$.

\begin{lem}
  Let $Q$ be an anticanonical curve on a commutative rationally ruled
  surface of genus $g\ge 1$ (with $Q$ not integral), and let $L$ be a line
  bundle having nonnegative degree on every component of $Q$.  Then $L$ is
  acyclic if it has degree at least $2g-1$ on the horizontal component(s)
  of $Q$ and globally generated if it has degree at least $2g$ on the
  horizontal component(s).
\end{lem}

\begin{proof}
  Since $0$-dimensional sheaves are acyclic and globally generated and both
  ``acyclic'' and ``acyclic and globally generated'' are closed under
  extensions, we may feel free to replace $L$ by any line bundle it
  contains that still satisfies the numerical conditions.  In particular,
  we may as well assume that $L$ has degree $0$ on every vertical component
  of $Q$.  But then we may as well contract that component before checking
  the desired conditions, and we may achieve this effect by blowing down to
  a ruled surface and performing a suitable sequence of elementary
  transformations.  We may thus reduce to the case that $Q$ has only
  horizontal components.  If $Q$ is reduced, then $g=1$, $Q$ is smooth, and
  $L$ is clearly acyclic (resp. globally generated) on each component.
  If $Q$ is nonreduced, then we have a short exact sequence
  \[
  0\to L'\to L\to L|_{Q^{\red}}\to 0
  \]
  where $L'$ is a line bundle on $Q^{\red}\cong C$.  That $L|_{Q^{\red}}$
  is acyclic (resp. acyclic and globally generated) follows by standard
  facts about line bundles on smooth curves, and $L'$ has degree at least
  that of $L|_{Q^{\red}}$, so is also acyclic.
\end{proof}

\begin{rem}
  Since a component of $Q$ is vertical iff $Q_i\cdot f=0$ and horizontal
  iff $Q_i\cdot f=1$, we may rephrase the numerical condition for
  $\sO_X(D)|_Q$ as saying that $D-(2g-1)f$ (resp. $D-2gf$) has nonnegative
  intersection with every component of $Q$.  Since any formal $-1$- or
  $-2$-curve on a higher genus rationally ruled surface is orthogonal to
  $f$, we may in fact rewrite this as saying that $D-(2g-1)f$
  (resp. $D-2gf$) is nef.
\end{rem}

\begin{prop}
  Let $X$ be a noncommutative rationally ruled surface of genus $g\ge 1$,
  and let $L$, $L'$ be line bundles on $X$.  If $c_1(L)-c_1(L')-(2g-1)f$ is
  nef, then $\Ext^p(L',L)=0$ for $p>0$, and if $c_1(L)-c_1(L')-2gf$ is nef,
  then $L$ is acyclically globally generated by $L'$.
\end{prop}

\begin{proof}
  We may again assume that $L'=\sO_X$ and choose a blowdown structure such
  that $D=c_1(L)-(2g-1)f$ (resp. $c_1(L)-2gf$) is in the fundamental
  chamber with $D\cdot e_m\ge 1$ if $m\ge 1$ (and $D\cdot (f-e_1)\ge 1$ if
  $m=1$).  For $m\ge 1$, there is no difficulty since if $D$ is nef with
  $D\cdot e_m>0$, then $D-Q$ is again nef, and thus we readily reduce to
  $m=0$.  In that case, we again find that $D-Q$ remains nef as long as
  $D\cdot f\ge 2$, and thus reduce to the cases $D\cdot f=1$ and $D\cdot
  f=0$.  The latter is immediate (since then $L$ is the pullback of an
  acyclic or acyclic and globally generated line bundle on $C$), while the
  former reduces to the corresponding claim for ${\cal E}$ and thus for
  $Q$.
\end{proof}

\begin{thm}
  Let $X$ be a noncommutative rational or rationally ruled surface, and let
  $D_a$ be an integral element of the interior of the nef cone.  Then for
  any sheaf $M$ on $X$, there is an integer $B$ such that for any line
  bundle $L$ of Chern class $-bD_a$ with $b\ge B$, $M$ is acyclically
  globally generated by $L$.
\end{thm}

\begin{proof}
  There is certainly {\em some} line bundle $L_0$ that acyclically globally
  generates $M$, and a line bundle $L_1$ that acyclically globally
  generates the kernel, etc., giving us a resolution
  \[
  \cdots\to L_1^{n_1}\to L_0^{n_0}\to M.
  \]
  For any other line bundle $L$, we may use this resolution to compute
  $R\Hom(L,M)$, and find that the only possible contributions to
  $\Ext^p(L,M)$ for $p\ge 0$ come from $\Ext^2(L,L_1)$ and $\Ext^p(L,L_0)$
  for $p\in \{1,2\}$.  Thus if $c_1(L_1)-c_1(L)$ and $c_1(L_0)-c_1(L)$ are
  nef divisors satisfying the appropriate additional inequality, then
  $\Ext^p(L,M)=0$ for $p>0$.  Similarly, if $L_0$ is globally generated by
  $L$, then so is $M$.  Since $D_a$ is in the interior of the nef cone, it
  has positive intersection with every effective class, and thus there is
  some $B$ such that for $b\ge B$, $bD_a+c_1(L_0)$ and $bD_a+c_1(L_1)$ are
  both nef and have intersection at least $2$ with $Q$ or $2g$ with the
  horizontal components of $Q$ as appropriate.
\end{proof}

\begin{rem}
  Since we have shown that the nef cone contains a subcone of full
  dimension, there are indeed ample divisors.
\end{rem}

\begin{cor}
  If we choose for each $b\in \Z$ a line bundle of Chern class $bD_a$, then
  $X$ may be recovered as the $\Proj$ of the corresponding full subcategory
  of $X$, viewed as a $\Z$-algebra.
\end{cor}

Now, suppose that $X/S$ is a family of noncommutative surfaces over some
Noetherian base $S$.  Let $Y/S$ be the corresponding family of commutative
surfaces, and let $D$ be an ample divisor on that family.  Since $D^2>0$,
it follows that $\langle D,Q\rangle^\perp$ is negative definite, and thus
that it meets the root system $E_{m+1}$ or $D_m$ (as appropriate) in a
finite subsystem.  If we first base change so that $Y$ has a blowdown
structure, then change the blowdown structure so that $D$ is in the
fundamental chamber, then $\langle D,Q\rangle^\perp$ meets the ambient root
system in a finite {\em parabolic} subsystem, and it follows that the set
of such blowdown structures is a torsor over the corresponding finite
Coxeter group.  (None of these roots can be effective on $Y$, since $D$ is
ample.)  It follows that there is a finite \'etale (and Galois) cover of
$S$ over which $Y$ has a blowdown structure.  Over this larger family, we
may choose a new divisor $D'$ which is not only ample on every fiber but in
the interior of the fundamental chamber.  Although this divisor class
itself will typically not descend to $X$, we note that since the
$t$-structure is invariant under the action of monodromy, the monodromy
preserves the ample cone.  Thus the trace (in the obvious sense) of $D'$
will still be ample, and now descends to a class in $\NS(X)$ over $S$.
Moreover, any blowdown structure on $Y$ that puts $\Tr(D')$ into the
fundamental chamber will give rise to the same $t$-structure on $X$.  The
significance of this is that $\Tr(D')$ makes sense even without first
assuming that $\perf(X)$ has a rational $t$-structure, and thus we may use
it to define such a $t$-structure whenever it exists.

In fact, we may use such a divisor to construct a model of $X$ as the
$\Proj$ of a sheaf of $\Z$-algebras on $S$.  Let $D_a\in \NS(X)$ be a
divisor which is ample on every fiber.  The main difficulty is that we need
to choose a line bundle of each Chern class which is a multiple of $D_a$,
and there may not be any such line bundle over $S$.  The main issues are
that the corresponding $\Pic^0_{C/S}$-torsor could fail to have points over
$S$, and that a point over $S$ need not correspond to an actual line
bundle.  For the first issue, we note that the $\Pic^0_{C/S}$-torsor lies in
an appropriate component of $\Pic_{Q/S}$, and is in the image of the natural
map
\[
\Gamma(S;\Pic_{Q/S}/\Pic^0_{C/S})\to H^1(S;\Pic^0_{C/S}).
\]
The element in $\Gamma(S;\Pic_{Q/S}/\Pic^0_{C/S})$ can be computed as the
image under the determinant-of-restriction map $K_0^{\num}(X)\to
\Pic_{Q/S}/\Pic^0_{C/S}$ of the numeric class of a line bundle of Chern
class $bD_a$.  Since the class of such a line bundle depends quadratically
on $b$, we obtain a quadratic map $\Z\to \Gamma(S;\Pic_{Q/S}/\Pic^0_{C/S})$
and thus a quadratic map $\Z\to H^1(S;\Pic^0_{C/S})$.  Moreover, since each
component of $\Pic_{Q/S}$ is quasi-projective, so is each
$\Pic^0_{C/S}$-torsor, and thus the $\Pic^0_{C/S}$-torsors are projective.
This implies that the resulting classes in $H^1(S;\Pic^0_{C/S})$ are
torsion, and thus that we can make the map to $H^1(S;\Pic^0_{C/S})$ trivial
by replacing $D_a$ by a sufficiently divisible positive multiple.  Moreover,
we find that not only does the map $\Z\to
\Gamma(S;\Pic_{Q/S}/\Pic^0_{C/S})$ map to the image of $\Pic_{Q/S}(S)$, but
we may choose a lifting of the map which is still quadratic.  (Indeed, we
may simply choose arbitrary preimages for the images of $\pm 1$, and take
the image of $0$ to be trivial.)  The obstruction to being able to lift
this map from $\Pic_{Q/S}(S)$ to $\Pic(Q)$ is a quadratic map $\Z\to
H^2(S;G_m)$, and again the image is torsion since $Q$ is projective.  So
multiplying $D_a$ by another sufficiently divisible positive integer makes
the Brauer obstruction vanish and allows us to lift.  We thus obtain a
system of line bundles ${\cal L}_b$, $b\in \Z$ satisfying ${\cal L}_b\cong
{\cal L}_1^{b(b+1)/2}\otimes {\cal L}_{-1}^{b(b-1)/2}$ and such that for
each $b\in \Z$, there is a line bundle $L_b$ on a suitable base change of
$X$ which has Chern class $bD_a$ and satisfies $L_b|_Q\cong {\cal L}_b$.
Moreover, by including this isomorphism, we rigidify $L_b$ and guarantee
that it descends to $S$.  (Indeed, $L_b$ is unique up to automorphisms, and
its automorphisms are scalars, so change the isomorphism to ${\cal L}_b$.)
The full subcategory of $\qcoh X$ with objects $L_b$, $b\in \Z$ is a sheaf
of $\Z$-algebras on $S$ and every fiber of $X$ agrees with the $\Proj$ of
the corresponding fiber of this sheaf of $\Z$-algebras.

By mild (but very suggestive) abuse of notation, we will denote the line
bundles $L_b$ so constructed by $\sO_X(bD_a)$.  The reader should be
cautioned that these sheaves do not descend to the moduli stack of
noncommutative surfaces, as their construction required some additional
choices.

\medskip

In the commutative case, one normally says that a line bundle is effective
iff it has a global section.  This does not quite agree as stated with our
definition above, but leads one to ask not only when a line bundle has
global sections, but how many global sections it has.  Although it is
unclear how to answer this question for general ruled surfaces, we can give
a complete answer in the rational case, and in fact (just as in the case of
a smooth anticanonical curve \cite{generic}) can give an essentially
combinatorial algorithm to compute the dimensions of Ext groups between any
two line bundles.

Let $D_1,D_2\in \Pic^{\num}(X)$ be a pair of divisor classes, and suppose
that we wish to compute $\dim \Ext^i(\sO_X(D_1),\sO_X(D_2))$ for $i\in
\{0,1,2\}$.  We know the alternating sum of these dimensions, and thus it
suffices to compute $\dim\Hom(\sO_X(D_1),\sO_X(D_2))$ and
$\dim\Ext^2(\sO_X(D_1),\sO_X(D_2))=\dim\Hom(\sO_X(D_2),\theta
\sO_X(D_1))$.  If $D_2-D_1$ is not effective, then
$\dim\Hom(\sO_X(D_1),\sO_X(D_2))=0$ while if $K_X+D_1-D_2$ is not
effective, then $\dim\Ext^2(\sO_X(D_1),\sO_X(D_2))=0$; since $K_X$ is
ineffective, at least one of these must be the case.  We thus reduce to the
question of computing $\dim\Hom(\sO_X(D_1),\sO_X(D_2))$ when $D_2-D_1$
is effective.  This in turn reduces to computing $\dim\Gamma(\sO_{X'}(D_2-D_1))$
where $X'$ is a suitable twist of $X$, so that we may as well take $D_1=0$,
$D_2=D$.

As in the algorithms above, a key step is to reduce to the case that $D$ is
nef and in the fundamental chamber.  If there is a component $Q_1$ of $Q$
such that $D\cdot Q_1<0$, then $\sO_{Q_1}(D)$ has no global
sections, and thus
\[
\Gamma(\sO_X(D-Q_1))\cong \Gamma(\sO_X(D)),
\]
allowing us to replace $D$ by $D-Q_1$.  We may then iterate this until $D$
has nonnegative inner product with every component of $Q$.  (Since $D$ is
assumed effective, this process terminates just as in the algorithm for
testing whether $D$ is effective.)

The same argument lets us subtract $e_m$ from $D$ whenever $D\cdot e_m<0$,
and thus it remains to consider the case that $D$ has negative intersection
with a simple root.  If the simple root is ineffective, then we may apply
the corresponding reflection, but this of course fails when the simple root
is effective.  Suppose that the simple root is $e_i-e_{i+1}$.  We have
already disposed of the possibility that $e_i-e_{i+1}$ is an anticanonical
component on $X_{i+1}$, and thus the only way it can be effective is when
the two points being blown up are smooth points of the anticanonical curve
in the same orbit.  We may then for simplicity twist by $e_{i+1}$ to ensure
that the points are the same, at the cost of changing the required
computation to $\dim\Hom(\sO_X(le_{i+1}),\sO_X(D+le_{i+1}))$ for some $l$.
If $q$ is non-torsion, then there is a unique such $l$, while if $q$ is
torsion of order $r$, we choose $l$ so that $-r\le (D+le_{i+1})\cdot
(e_i-e_{i+1})<0$.

If $(D+le_{i+1})\cdot (e_i-e_{i+1})\ge 0$, so that $q$ is non-torsion and
$l=(le_{i+1})\cdot (e_i-e_{i+1})>0$, then we may apply Proposition
\ref{prop:weak_reflection} to suitable twists to obtain short exact
sequences
\[
0\to \sO_X(le_{i+1}) \to \sO_X(le_i) \to \sO_{e_i-e_{i+1}}(-1)^l \to 0
\]
and
\[
0\to \sO_X(D+le_{i+1})
 \to \sO_X(s_i(D)+le_i)
 \to \sO_{e_i-e_{i+1}}(-1)^{(D+le_{i+1})\cdot (e_i-e_{i+1})}
 \to 0,
\]
where $s_i(D)$ is the image of $D$ under the reflection.  We may thus
compute the spherical twists of $\sO_X(le_{i+1})$ and $\sO_X(D+le_{i+1})$,
allowing us to conclude that
\[
\Hom(\sO_X(le_{i+1}),\sO_X(D+le_{i+1}))
\cong
\Hom(\sO_X(le_i),\sO_X(s_i(D)+le_i)),
\]
so that $\sO_{X'}(D)$ and $\sO_{s_i(X')}(s_i(D))$ have isomorphic spaces of
global sections, where $X'$ is the appropriate twist of $X$.  Similarly, if
$l\le 0$, then either $q$ is non-torsion or $-r<l\le 0$, so that we may use
Proposition \ref{prop:weak_reflection} to compute the opposite spherical
twists and again conclude that
\[
\Hom(\sO_X(le_{i+1}),\sO_X(D+le_{i+1}))
\cong
\Hom(\sO_X(le_i),\sO_X(s_i(D)+le_i)).
\]

The remaining case is that $l>0$ but $(D+le_{i+1})\cdot (e_i-e_{i+1})<0$.
In that case, we still have a short exact sequence
\[
0\to \sO_X(s_i(D)+le_i)
 \to \sO_X(D+le_{i+1})
 \to \sO_{e_i-e_{i+1}}(-1)^{-(D+le_{i+1})\cdot (e_i-e_{i+1})}
 \to 0,
\]
but now have
\[
\Hom(\sO_X(le_{i+1}),\sO_{e_i-e_{i+1}}(-1))=0,
\]
so that
\[
\Hom(\sO_X(le_{i+1}),\sO_X(s_i(D)+le_i))
\cong
\Hom(\sO_X(le_{i+1}),\sO_X(D+le_{i+1})).
\]
Setting $D'=s_i(D)+le_i-le_{i+1}$, we find that $D-D'$ is a positive
multiple of $e_i-e_{i+1}$, so that this reduces $-D\cdot (e_i-e_{i+1})$.

We thus find that in each case such that $D\cdot (e_i-e_{i+1})<0$ with
$e_i-e_{i+1}$ effective, we have an isomorphism $\Gamma(\sO_X(D))\cong
\Gamma(\sO_{X'}(D'))$ where $D-D'$ is a positive multiple of $e_i-e_{i+1}$.
Similar calculations also hold for the simple roots $f-e_1-e_2$ (related to
$e_1-e_2$ by an elementary transformation) and $s-f$.  Thus in general if
$D$ has negative intersection with a simple root, we may reduce the problem
to the corresponding calculation on a related surface with smaller
(relative to the effective cone) $D$, sandwiched
between the original $D$ and its image under the reflection.  In
particular, we can perform only finitely many such reductions before ending
up in the fundamental chamber.

It remains to determine what happens when $D$ is nef and in the fundamental
chamber; we may also assume that $D\cdot e_m>0$ as otherwise we may reduce
to the blown down surface.  If $D=0$ or $D\cdot Q>0$, then $\sO_X(D)$ is
acyclic, and thus we have $\dim\Gamma(\sO_X(D))=
\chi(\sO_X(D))=1+(D\cdot (D+Q))/2$.  It remains only to consider what
happens when $D\ne 0$ with $D\cdot Q=0$.  We again have a short exact
sequence
\[
0\to \sO_X(D-Q)\to \sO_X(D)\to \sO_X(D)|_Q\to 0.
\]
If $\sO_X(D)|_Q$ is nontrivial (so has no global sections) or $\sO_X(D-Q)$
is acyclic, then this induces a short exact sequence
\[
0\to \Gamma(\sO_X(D-Q))\to \Gamma(\sO_X(D))\to \Gamma(\sO_X(D)|_Q)\to 0,
\]
again allowing us to reduce the dimension calculation.

At this point, the only way $D$ can avoid all of our reductions is to have
$m=8$, $Q^2=0$ and $D\propto Q$, such that any rational component of $Q$
has self-intersection $-2$.  We thus need to compute $\Gamma(\sO_X(lQ))$
for $l>0$, where $Q$ is of this form and $\sO_Q(lQ)\cong \sO_Q$.  If $l$ is
minimal with this property, then $\sO_X((l-1)Q)$ is acyclic, so that we may
compute
\[
\dim\Gamma(\sO_X(lQ))=\chi(\sO_X((l-1)Q))+1=2.
\]
Since these global sections are generically nonzero on $Q$, but there is a
1-dimensional subspace vanishing on $Q$, they are algebraically
independent, and thus we conclude that
\[
\dim\Gamma(\sO_X(alQ))\ge a+1
\]
for any integer $a\ge 0$.  The long exact sequence tells us that
\[
\dim\Gamma(\sO_X((a+1)lQ))\le 1 + \dim\Gamma(\sO_X(((a+1)l-1)Q))
= 1 + \dim\Gamma(\sO_X(alQ)),
\]
from which it follows that $\dim\Gamma(\sO_X(alQ))=a+1$, finishing off our
calculation of dimensions of global sections.

\medskip

Although the construction of $X$ via the $\Z$-algebra associated to an
ample divisor is in principle explicit, it is still not the nicest model
one might hope for.  In the ideal situation, we could represent $X$ as the
$\Proj$ of an actual graded algebra.  This is too much to hope for in
complete generality, as it requires us $\coh X$ to have an ample
autoequivalence, and this is not the case in general.  There is one case in
which can guarantee such an autoequivalence, namely when $X$ is Fano, i.e.,
when $Q$ is ample.  More generally, as long as $Q$ is nef with $Q^2\ge 1$, we
can hope to obtain a graded algebra deforming the anticanonical coordinate
ring of a degenerate del Pezzo surface.  (Such rings may be of particular
interest as models of noncommutative surfaces with rational singular
points.)

This leads us to the question of what form the presentation of the
anticanonical graded algebra $\bigoplus_i \Gamma(\theta^{-i}\sO_X)$ takes.
A key observation is that there is a central element of degree $1$
corresponding to the anticanonical natural transformation, and the quotient
by the corresponding principal ideal is the graded algebra
\[
\bigoplus_i \Gamma(\theta^{-i}\sO_X|_Q),
\]
which is essentially a twisted homogeneous coordinate ring of $Q$, except
that if $Q$ has components of self-intersection $-2$, its $\Proj$ will
contract those components.  Indeed, for any $i$, $\theta^{-i}\sO_X|_Q$ has
degree 0 on any such component, and thus we may as well work over the
contracted curve.  Moreover, we may think of the anticanonical ring itself
as a filtered deformation of this twisted homogeneous coordinate ring
(i.e., a filtered ring with that associated graded); in particular,
apart from the additional generator of degree 1 and the relations that it
be central, the structure of the presentation is the same.

The case of degree $1$ del Pezzo surfaces is particularly nice, as in that
case the contracted curve is just the Weierstrass model of $Q$.  In general,
let $L_i$ be the line bundle on $Q$ given by $\theta^{-i}\sO_X|_Q$.

\begin{lem}
  If $Q$ is nef with $Q^2=1$ and $q$ is nontrivial, then the anticanonical
  ring is a filtered deformation of a ring generated by elements of degree
  $1$ and $2$, with relations of degree $4$, $5$, and $6$.
\end{lem}

\begin{proof}
  This is really just a statement about the twisted homogeneous coordinate
  ring of a Weierstrass curve relative to a degree 1 line bundle and a
  translation $q$.  We first observe that there is an exact sequence
  \[
  0\to 
  L_{-6}\to
  L_{-5}\oplus L_{-4}\to
  L_{-2}\oplus L_{-1}\to \sO_Q\to 0
  \]
  of sheaves on $Q$, which may be obtained as the Yoneda product of
  \[
  0\to L_{-6}\to L_{-5}\oplus L_{-4}\to L_{-2}\otimes L_{-1}\to 0
  \]
  and
  \[
  0\to L_{-2}\otimes L_{-1}\to L_{-2}\oplus L_{-1}\to \sO_Q\to 0
  \]
  Here the maps $L_{-1}\to \sO_Q$ and $L_{-6}\to L_{-5}$ come from the
  degree $1$ generator $v_1$ (which is unique modulo scalars) and similarly
  for the maps $L_{-2}\to \sO_Q$ and $L_{-6}\to L_{-4}$, and the
  identification of $L_{-2}\otimes L_{-1}$ is via determinant
  considerations.  It follows in particular that the algebra is generated
  by the elements of degree $1$ and $2$ coming from this exact sequence.

  The correspondence between sheaves and graded modules over the twisted
  homogeneous coordinate ring is by applying $\theta^{-i}$ and taking
  global sections.  The only degrees in which the corresponding sequence of
  graded modules is not exact are those for which some terms have
  nontrivial $H^1$ (and some terms have global sections), and thus only
  $0\le i\le 6$ need be considered.  For $0<i$, surjectivity at $\sO_Q$ is
  straightforward (given $q$ nontrivial), and implies exactness for $i\in
  \{1,2,3\}$.  We dually have exactness for $i\in \{3,4,5\}$, so that only
  $i=0$ and $i=6$ remain.  For $i=0$, we clearly have cohomology $0,0,0,k$,
  while for $i=6$ the fact that the spectral sequence converges to $0$ lets
  us read off the cohomology of global sections from the cohomology of
  $H^1$, and thus we obtain $0,0,k,0$.  But this implies that the relations
  are generated by the relations of degree $4$ and $5$ coming from the
  complex, except that there is one missing relation of degree $6$.
\end{proof}

\begin{rem}
  In the commutative case, we need an additional generator and relation of
  degree $3$, and the only relation which is not a commutator is the
  relation of degree $6$.
\end{rem}

For degree 2 del Pezzo surfaces, the situation is slightly more complicated
since the contracted curve may be reducible.  It is always a double cover
of $\P^1$, however, and this makes things tractable.

\begin{lem}
  If $Q$ is nef with $Q^2=1$ and $q$ nontrivial, then the anticanonical
  ring is a filtered deformation of a ring generated by two elements of
  degree $1$ with relations of degree $3$, $3$, and $4$.
\end{lem}

\begin{proof}
  Analogous, except that now the relevant exact sequence has the form
  \[
  0\to L_{-4}\to L_{-3}^2\to L_{-1}^2\to \sO_Q\to 0,
  \]
  again obtained as a Yoneda product.
\end{proof}

\begin{rem}
  Again, we need an additional generator and relation of degree $2$ in the
  commutative case, with the only nontrivial relation that of degree $4$.
\end{rem}

For $Q^2\ge 3$, the algebra is generated in degree $1$; this follows from
the general fact that if ${\cal L}_1$, ${\cal L}_2$ are invertible sheaves
on $Q$ with nonnegative degree on every component and total degree at least
$3$, then $\Gamma({\cal L}_1)\otimes \Gamma({\cal L}_2)\to \Gamma({\cal
  L}_1\otimes {\cal L}_2)$ is surjective.  (The proof of
\cite[Lem.~2.3]{generic} extends easily from the smooth case.)

\begin{lem}
  If $Q$ is nef with $Q^2=3$ and $q$ nontrivial, then the anticanonical
  ring is a filtered deformation of a ring generated by three elements of
  degree $1$ with three quadratic relations and one cubic relation.
\end{lem}

\begin{proof}
  Similar but now with
  \[
  0\to L_{-3}\to L_{-2}^3\to L_{-1}^3\to \sO_Q\to 0,
  \]
  which can be obtained by observing that $Q$ embeds in a noncommutative
  $\P^2$ in a way that respects the automorphism.
\end{proof}

Something similar should hold in higher degrees, and one expects the
algebra to be quadratic for $Q^2\ge 4$.

In special cases, one expects additional graded algebras; indeed, if $Q$ is
reducible, then we get a graded algebra corresponding to a singular form of
the surface from any linear combination of components of $Q$ which is nef
and has positive self-intersection.  There are also cases where the
autoequivalence is less canonical; e.g., for a noncommutative $\P^2$,
twisting by $\sO_X(1)$ is not an autoequivalence in general, but does act
trivially on the open substack of the moduli stack where the anticanonical
curve is integral and not an additive curve of characteristic 3.  The
corresponding scheme of isomorphisms is the $\Pic^0(Q)[3]$-torsor
corresponding to $q\in \Pic^0(Q)/\Pic^0(Q)^3$, giving a
$\Pic^0(Q)[3]$-torsor of graded algebra representations.  (This, of course,
gives rise to the classical construction of noncommutative planes via
graded algebras \cite{ArtinM/TateJ/VandenBerghM:1991}.)  In general, for
rational surfaces, a divisor class gives rise to an autoequivalence of $X$
whenever there is an isomorphism of the corresponding commutative surfaces
that identify the two embeddings of $q$ as sheaves on the respective
anticanonical curves as well as the images of the nef cone of $X$.  (The
first condition gives a derived autoequivalence, and the second ensures
that it preserves the $t$-structure.  The latter can be weakened to simply
asking that the chosen divisor lie in the intersection of the images of the
nef cone of $X$ under the action of the group generated by the
autoequivalence on $\Pic(X)$.)

\section{Quot and Hilbert schemes}
\label{sec:quot}

There remains one of Chan and Nyman's axioms we have yet to show: the
``halal Hilbert scheme'' condition, which states that for any coherent
sheaf $M$, the functor $\Quot(M)$ is representable by a separated scheme,
locally of finite type, which is a countable union of projective schemes.
Since $\Quot(M)$ classifies families of sheaves, this suggests that we
should ask for this condition to hold when $M$ itself is a family of
sheaves, or better yet a sheaf on a family of surfaces.  Furthermore, the
commutative and maximal order cases strongly suggest that there should be a
better version of the $\Quot$ scheme condition: if $M$ is a $S$-flat family
of sheaves on $X/S$ (with $S$ Noetherian), then the subfunctor of $\Quot$
classifying quotients of Hilbert polynomial $P$ (relative to some fixed
ample divisor class) should be (and, as we will show, is) projective.

In order to carry out the usual commutative construction of the $\Quot$
functor, we need some boundedness results, as well as an analogue of the
flattening stratification.  A key issue is that we need to be able to force
not just the original family, but various base changes, to be acyclically
generated by line bundles $\sO_X(-bD_a)$ for uniformly bounded
$b$.  In particular, our previous results give this over a field, but we
both need to control how the bound varies with the fiber and need to deal
with the fact that acyclicity and global generation are not a priori
inherited from fibers when the sheaf is not flat over the base.

Luckily, this last fact turns out not to be an issue.  We will need to
start by showing this for curves.

\begin{lem}
  Let $C/S$ be a smooth proper curve over a Noetherian base $S$.  Then a
  sheaf $E\in \coh(C)$ is acyclic iff every fiber is acyclic.
\end{lem}

\begin{proof}
For any point $s$, the complex $R\Gamma(E)\otimes^L k(s)\cong
R\Gamma(E\otimes^L k(s))$ has vanishing cohomology in all degrees $>1$,
and thus $H^1(E)\otimes k(s)\cong H^1(E\otimes k(s))$.
\end{proof}

For global generation, a globally generated sheaf has globally generated
fibers, but the converse need not hold.  Luckily, if we also ask for
acyclicity, this problem goes away.  The key idea is the following.  In the
proofs of the next several results, we abbreviate ``acyclic and globally
generated'' by ``a.g.g.''.

\begin{lem}
  Let $C/S$ be a smooth proper curve over a Noetherian base $S$, and let
  $\Delta\subset C\times_S C$ be the diagonal, with projections $\pi_1$,
  $\pi_2$.  Then a sheaf $M\in \coh(C)$ is acyclic and globally generated
  iff $\pi_2^*M(-\Delta)$ is $\pi_{1*}$-acyclic.
\end{lem}

\begin{proof}
  We first note that $M$ is globally generated iff
  \[
  \pi_{1*}\pi_2^*M\to \pi_{1*}(\pi_2^*M|_{\Delta})
  \]
  is surjective, as this is nothing other than the map
  $\Gamma(M)\otimes_S\sO_C\to M$.  Let $N$ be the kernel of the
  (surjective) natural map $\pi_2^*M\to \pi_2^*M|_{\Delta}$.  Then we have
  an exact sequence
  \[
  0\to \pi_{1*}N\to \pi_{1*}\pi_2^*M\to \pi_{1*}(\pi_2^*M|_{\Delta}) \to
  R^1\pi_{1*}N\to R^1\pi_{1*}\pi_2^*M\to 0,
  \]
  from which it follows immediately that $M$ is a.g.g.~iff $R^1\pi_{1*}N=0$.
  We then note the short exact sequence
  \[
  0\to \Tor_1(\pi_2^*M,\sO_{\Delta})\to \pi_2^*M(-\Delta)\to N\to 0.
  \]
  Since the kernel is supported on $\Delta$, it is $\pi_{1*}$-acyclic, and
  thus $N$ is $\pi_{1*}$-acyclic iff $\pi_2^*M(-\Delta)$ is
  $\pi_{1*}$-acyclic.
\end{proof}

\begin{cor}
  The set of points $s\in S$ such that $M_s$ is acyclic and globally
  generated is open in $S$.
\end{cor}

\begin{proof}
  Indeed, it is the complement of the image of the support of
  $R^1\pi_{1*}\pi_2^*M(-\Delta)$.
\end{proof}

\begin{cor}
  $M$ is acyclic and globally generated iff every fiber of $M$ is acyclic
  and globally generated, and then every Noetherian base change of $M$ is
  acyclic and globally generated.
\end{cor}

\begin{proof}
  We may interpret $R^1\pi_{1*}\pi_2^*M(-\Delta)$ as the relative $H^1$ of
  a sheaf on the base change of $C$ to $C$.  Thus
  $R^1\pi_{1*}\pi_2^*M(-\Delta)=0$ iff for every point $x\in C$,
  $H^1(M(-x))=0$.  This is equivalent to asking that $H^1(M_s(-x))=0$ for
  every point $s\in S$ and $x\in C_s$, which is in turn equivalent to
  asking that every fiber be a.g.g.  Acyclic global generation on fibers is
  clearly inherited by base changes, so the remaining claim follows.
\end{proof}
  
\begin{prop}
  Let $X/S$ be a family of noncommutative surfaces over a Noetherian base,
  and let $L$ be a line bundle on $X$.  A sheaf $M\in \coh(X)$ is
  acyclically generated by $L$ iff every fiber of $M$ is acyclically
  generated by $L$, and then every Noetherian base change of $M$ is
  acyclically generated by $L$.
\end{prop}

\begin{proof}
  Acyclic generation descends along \'etale morphisms, so we may
  assume $X$ split.  We can then use twist functors to reduce to the case
  $L=\sO_X$.  If $m>0$, then $M$ is a.g.g.~iff it is a.g.g.~relative to
  $\alpha_m$ and $\alpha_{m*}M$ is a.g.g.  But the argument of Proposition
  \ref{prop:blowup_acyclic} carries over to an arbitrary family and shows
  that $M$ is a.g.g.~relative to $\alpha_m$ iff $\Ext^2(\sO_{e_m},M)=0$.  The
  same argument as in the curve case (since this is the highest possible
  degree of a nonvanishing cohomology sheaf) shows that this is equivalent
  to the same condition on fibers, and thus the result holds by induction.

  If $X$ is a ruled surface, we similarly find that $M$ is a.g.g.~iff it is
  a.g.g.~relative to $\rho_0$ and $\rho_{0*}M$ is a.g.g.  We can again use the
  semiorthogonal decomposition to find that $M$ is a.g.g.~relative to $\rho_0$
  iff $R^1\rho_{-1*}M=0$.  Since $\rho_{-1*}$ has cohomological dimension
  1, we once more find that this is equivalent to the condition on fibers,
  and thus the result reduces to the curve case.

  It remains to consider the case that $X$ is a noncommutative plane.  Take
  the (fppf) base change of $X$ to $Q$, and let $\tilde{X}$ be the blowup
  of this family in the tautological section of $Q$.  The sheaf $\alpha^*M$
  is a.g.g.~iff $\alpha_*\alpha^*M$ is a.g.g.~and $R^1\alpha_*\alpha^*M=0$.
  The spectral sequence associated to $R\alpha_*L\alpha^*M\cong M$ tells us
  that $R^1\alpha_*\alpha^*M$ always vanishes, and that there is a short
  exact sequence
  \[
    0\to R^1\alpha_*L_1\alpha^*M\to M\to \alpha_*\alpha^*M\to 0.
  \]
  The kernel is an extension of sheaves on $Q$ with support finite over
  $S$, and thus is a.g.g.; it follows that $M$ is a.g.g.~iff
  $\alpha_*\alpha^*M$ is a.g.g.

  Thus $M$ is a.g.g iff $\alpha^*M$ is a.g.g.~iff $(\alpha^*M)_q$ is
  a.g.g.~for all $q\in Q$ iff $(\alpha^*M)_s$ is a.g.g.~for all $s\in S$.
  Since $\alpha^*$ is right exact, $(\alpha^*M)_s\cong \alpha^*(M_s)$ and
  thus this is a.g.g.~iff $M_s$ is a.g.g.
\end{proof}

Here and below, we assume that $D_a$ is not only ample, but $D_a\cdot Q\ge
2$ or $D-2gf$ is nef as appropriate.  (This can always be arranged by
replacing $D_a$ by a sufficiently large multiple of $D_a$, without changing
the validity of the final result.  E.g., in the following result, we need
simply combine the bounds for finitely many twists of $M$ by multiples of
the original ample divisor.)

\begin{cor}\label{cor:noetherian_is_bounded}
  Let $X/S$ be a family of noncommutative surfaces over a Noetherian base.
  Then for any $M\in \coh(X)$, there is an integer $b_0$ such that for
  $b\ge b_0$, every Noetherian base change of $M$ is acyclically generated
  by $\sO_X(-bD_a)$.
\end{cor}

\begin{proof}
  It suffices to find a bound that works for $M$ itself.  For each $b$, let
  $M_b$ be the image of the natural map
  \[
  \sO_X(-bD_a)\otimes_{\sO_S}\Hom(\sO_X(-bD_a),M)\to M.
  \]
  Since $\sO_X(-bD_a)$ is acyclically generated by
  $\sO_X(-(b+1)D_a)$, this gives an ascending chain of subsheaves of $M$,
  which therefore terminates and gives a surjection
  \[
  \sO_X(-b_1D_a)\otimes_{\sO_S} N_1\to M
  \]
  for some $N_1\in \coh(S)$.  The kernel is again a coherent sheaf, so we
  can extend this to a resolution
  \[
  \sO_X(-b_0D_a)\otimes_{\sO_S}N_0\to \sO_X(-b_1D_a)\otimes_{\sO_S} N_1\to M\to 0
  \]
  with $b_0\ge b_1$.  Thus $M$ is generated by $\sO_X(-b D_a)$ as long as
  $b\ge b_1$, and is acyclic for $\sO_X(-bD_a)$ as long as $b\ge b_0$.  But
  then the same statement holds on an arbitrary Noetherian base change.
\end{proof}

\begin{rem}
  One consequence is that the support of $M$ over $S$ is closed (and can be
  given a natural closed subscheme structure), as it may be identified with
  the support of the coherent $\sO_S$-module $\Hom(\sO_X(-bD_a),M)$.  This,
  together with the fact that the cohomology sheaves of perfect complexes
  are coherent (since $X/S$ is Noetherian), implies that a perfect complex
  being a sheaf is an open condition (the complement of the unions of the
  supports of the other cohomology sheaves, only finitely many of which can
  be nonzero).
\end{rem}

In particular, this lets us prove the following criterion for flatness.

\begin{lem}
  Let $X/S$ be a family of noncommutative surfaces over a Noetherian base.
  A coherent sheaf $M\in \coh(X)$ is $S$-flat iff for all $b\gg 0$,
  $R\Hom(\sO_X(-bD_a,M))$ is an $S$-flat sheaf.
\end{lem}

\begin{proof}
  Let $b_0$ be such that $M$ is acyclically generated by $\sO_X(-bD_a)$ for
  all $b\ge b_0$.  Then the complex $R\Hom(\sO_X(-bD_a),M)$ is certainly a
  sheaf, and is clearly flat whenever $M$ is flat.  Thus suppose that
  $R\Hom(\sO_X(-bD_a),M)$ is a flat sheaf on $S$ for all $b\ge b_0$.  For
  any closed point $s\in S$, we can find $b'\ge b_0$ such that both
  $\Tor_1(M,k(s))$ and $\Tor_2(M,k(s))$ are acyclically generated
  by $\sO_X(-b'D_a)$.  But then we find that
  \begin{align}
  0
  &=
  \dim(h^{-1} R\Hom(\sO_X(-b'D_a),M)\otimes^L k(s))\notag\\
  &=
  \dim(h^{-1} R\Hom(\sO_X(-b'D_a),M\otimes^L k(s)))\notag\\
  &\ge
  \dim\Hom(\sO_X(-b'D_a,\Tor_1(M,k(s)))),
  \end{align}
  so that $\Tor_1(M,k(s))=0$.
\end{proof}

We define the Hilbert polynomial of a coherent sheaf over a field in the
obvious way, namely as the polynomial $P_M(b)=\chi(\sO_X(-bD_a),M)$.  This
is clearly locally constant in flat families (since it respects field
extensions and on a geometric fiber depends only on the class of $M$ in
$K_0^{\num}(X)$).  The usual argument then gives the following.

\begin{cor}
  If $S$ is integral and Noetherian, then a coherent sheaf $M\in \coh X$ is
  $S$-flat iff all of its fibers have the same Hilbert polynomial.
\end{cor}

\begin{proof}
  For $b\gg 0$, so that every fiber of $M$ is acyclically
  generated by $\sO_X(-bD_a)$, we have
  $\dim\Hom(\sO_X(-bD_a),M_s)=\chi(\sO_X(-bD_a),M_s)$, and thus by
  semicontinuity, $\Hom(\sO_X(-bD_a),M)$ is flat.
\end{proof}

This lets us construct a weak form of flattening stratification.  For any
sheaf $M\in \coh(X)$, suppose that $M$ is acyclically generated by
$\sO_X(-bD_a)$ for all $b\ge b_0$.  Then for each such $b$, the sheaf
$\Hom(\sO_X(-bD_a),M)$ has a universal flattening stratification, and
taking the fiber product over $b\ge b_0$ gives a covering of $S$ by a
countable union of localizations of closed subschemes.  (This fiber product
exists and is affine since each factor is affine.)  This satisfies the
universal property that any $f:T\to S$ such that $f^*M$ is $S$-flat (a
``flattening morphism for $M$'') factors through this fiber product.  Say
that $M$ has ``finite flattening stratification'' if the fiber product is a
finite union of locally closed subschemes.

\begin{lem}  Let $X/S$ be a family of noncommutative surfaces over a
  Noetherian base $S$, and $M\in \coh(X)$.  Suppose that for any integral
  closed subscheme $S'\subset S$, there is an open neighborhood of $k(S')$
  in $S'$ on which $M$ is flat.  Then $M$ has finite flattening
  stratification.
\end{lem}

\begin{proof}
  Let $s$ be the generic point of some component of $S$.  Then by
  hypothesis, there is an open neighborhood of $k(s)$ in its reduced
  closure on which $M$ is flat, and thus in particular the Hilbert
  polynomial of $M$ is constant.  It follows by Noetherian induction that
  $S$ is the disjoint union of finitely many locally closed subsets on
  which the Hilbert polynomial of $M$ is constant.  The corresponding
  disjoint union of reduced subschemes is a flattening morphism, so factors
  through the universal flattening morphism.  Each set in the partition
  maps to at most one component of the flattening morphism, which therefore
  has only finitely many components, each of which must therefore be
  locally closed.
\end{proof}

This is somewhat difficult to prove in general, but a large family of
special cases where generic flatness holds was given in
\cite{ArtinM/ZhangJJ:2001} (a categorical version of the main result of
\cite{ArtinM/SmallLW/ZhangJJ:1999}), implying the following.

\begin{cor}\cite[Thm.~C5.1]{ArtinM/ZhangJJ:2001}
  If $S$ is of finite type, then any sheaf on $X$ has finite flattening
  stratification.
\end{cor}

\begin{rem}
  In fact, it follows from the reference that this holds whenever $S$ is
  ``admissible'', as defined in \cite{ArtinM/SmallLW/ZhangJJ:1999}, but the
  finite type case turns out to be enough to let us bootstrap to a general
  Noetherian base, see Corollary \ref{cor:can_always_flatten} below.
\end{rem}

\medskip

We now turn to the $\Quot$ scheme.  Given a family $X/S$ of noncommutative
surfaces over a Noetherian base and a sheaf $M\in \coh(X)$, we define a
functor $\Quot_{X/S}(M,P)$ on the category of schemes over $S$ by taking
$\Quot_{X/S}(M,P)(T)$ to be the set of isomorphism classes of short exact
sequences
\[
0\to I\to M_T\to N\to 0
\]
such that $N$ is a $T$-flat coherent sheaf all fibers of which have Hilbert
polynomial $P$.  (Recall that a ``$T$-flat coherent sheaf'' is a $T$-flat
element of $\perf(X_T)\cap \qcoh(X_T)$, so makes sense even when $T$ is not
Noetherian, and the usual argument tells us that $\Quot$ respects base
change.)  This is of course a contravariant functor of $T$, and also
satisfies a partial functoriality in $M$.  Indeed, given any surjection
$M'\to M$, we obtain a morphism $\Quot_{X/S}(M,P)\to \Quot_{X/S}(M',P)$ by
pullback of short exact sequences (in particular keeping the same quotient
sheaf).  This is a closed embedding, and thus to show that
$\Quot_{X/S}(M,P)$ is projective, it suffices to show this for some sheaf
surjecting on $M$.  We may thus reduce to considering
$\Quot(\sO_X(-bD_a)\otimes_{\sO_S} M,P)$ for $M\in \coh(S)$, and then by
twisting to the case $b=0$.  (Note that although we have only defined
twisting functors for split families, the twist functor descends as long as
$L$ is defined over $S$.)

In order to carry out the standard commutative argument, we need to show
that we can make a global choice of $b$ such that for any $T$ and any short
exact sequence in $\Quot_{X/S}(\sO_X\otimes_{\sO_S}M,P)(T)$, $\sO_X(-bD_a)$
acyclically generates $I$.  Since this reduces to a question about fibers
(modulo some technical issues coming from the fact that $T$ is not
Noetherian), a bound that works when $T$ is a geometric point of $S$ will
work in general.  In particular, if we can give a bound that works for a
surface over an algebraically closed field and depends only on the
combinatorics of the surface, this will produce a bound for $S$.

Thus let $X$ be a surface over an algebraically closed field $k$, and
consider a quotient of $\sO_X\otimes V$.  A complete flag in $V$ induces a
filtration of $\sO_X\otimes V$, and thus a filtration of the short exact
sequence.  This suggests that we should able to to reduce this problem to
the case $V=k$.

Thus suppose consider a short exact sequence
\[
0\to I\to \sO_X\to M\to 0
\]
over a geometric point of $S$.  Since $\sO_X$ is torsion-free of rank 1,
$I$ either vanishes or is itself torsion-free of rank 1.

\begin{lem}
  Any nonzero homomorphism between torsion-free sheaves of rank 1 is
  injective.
\end{lem}

\begin{proof}
  The image is a nonzero subsheaf of a torsion-free sheaf, so must have
  positive rank, and thus the kernel is a rank 0 subsheaf of a torsion-free
  sheaf, so is 0.
\end{proof}

\begin{cor}
  If $I$, $I'$ are torsion-free sheaves of rank 1, then $\Hom(I,I')=0$
  unless $c_1(I')-c_1(I)$ is effective.
\end{cor}

We will want to understand the moduli space of such sheaves below, so one
natural question is the dimension of that space.

\begin{prop}
  Suppose that $I$ is a torsion-free sheaf of rank 1 on a noncommutative
  rational surface.  Then
  \[
  \chi(I,I)\le 1,
  \]
  with equality iff $I$ is a line bundle.
\end{prop}

\begin{proof}
  Since $\chi(I,I)$ is invariant under twisting, we may as well assume that
  $c_1(I)=0$.  Since $c_1(\theta I)-c_1(I)=K$ is ineffective, we find that
  $\Ext^2(I,I)=\Hom(I,\theta I)=0$, and since any endomorphism of $I$ is
  injective, we must have $\dim \Hom(I,I)=1$.  The inequality follows
  immediately.  If equality holds, then Riemann-Roch gives $\chi(I)=1$, and
  thus $I$ is numerically equivalent to $\sO_X$ and $\Hom(\sO_X,I)\ne 0$,
  forcing $I\cong \sO_X$.
\end{proof}

In particular, the line bundles on a noncommutative rational surface are
precisely the torsion-free sheaves of rank 1 with $\chi(I,I)=1$, and thus
the notion is independent of the blowdown structure.  As this is of
independent interest, we include quasi-ruled surfaces in the next result.

\begin{prop}
  Suppose that $I$ is a torsion-free sheaf of rank 1 on a noncommutative
  rationally quasi-ruled surface over curves of genus $g_0$ and $g_1$.  Then
  \[
  \chi(I,I)\le 1-g_{c_1(I)\cdot f},
  \]
  with equality iff $I$ is a line bundle.
\end{prop}

\begin{proof}
  Again, we may as well assume that $c_1(I)=0$, and may use Riemann-Roch to
  restate the inequality as $\chi(I)\le 1-g_0$.  Let $\rho$ denote the
  composition $\rho_0\circ \alpha_1\circ\cdots \alpha_m$, and note that
  $\rho_*$ has cohomological dimension 1.  A homomorphism from $I$ to a
  line bundle must be injective, and thus if $\Hom(I,L)\ne 0$, then
  $c_1(L)$ is effective.  It follows that for any line bundle ${\cal L}$ on
  $C_0$, we have
  \[
  \Ext^2(\rho^*{\cal L},I)
  \cong
  \Hom(I,\theta \rho^*{\cal L})^*
  =
  0,
  \]
  since $(K+df)\cdot f=-2<0$ for any $d$.  It follows that
  \[
  \Ext^2({\cal L},R\rho_*I)=0
  \]
  for any ${\cal L}$ and thus that $\rank(R^1\rho_*I)=0$.  We then
  compute $\rank(\rho_*I)=\rank(I)=1$ and
  \[
  \chi(I) = \chi(\rho_*I)-\chi(R^1\rho_*I)\le \chi(\rho_*I).
  \]
  In particular, if $\chi(I)\ge 1-g_0$, then $\chi(\rho_*I)\ge 1-g_0$,
  and thus there is a degree 0 line bundle ${\cal L}$ on $C_0$ such that
  $\Hom({\cal L},\rho_*I)\ne 0$, making $\Hom(\rho^*{\cal L},I)\ne 0$.
  Since $\rho^*{\cal L}$ is again torsion-free of rank 1, we again find
  that such a morphism must be injective.  Since the quotient has trivial
  Chern class, it must be $0$-dimensional, and since $\Ext^1$ from a
  $0$-dimensional sheaf to a line bundle always vanishes, the corresponding
  short exact sequence splits.  Since $I$ is torsion-free, the quotient
  must be 0 and thus $I=\rho^*{\cal L}$.
\end{proof}

\begin{rem}
  In the rational or rationally ruled case, we may use Riemann-Roch to
  restate the bound as
\[
\chi(I)\le 1-g + c_1(I)\cdot (c_1(I)-K_X)/2,
\]
again with equality iff $I$ is a line bundle.  An analogous formula of
course exists in the quasi-ruled case, but is significantly messier.
\end{rem}

\begin{cor}\label{cor:lbs_are_intrinsic}
  Let $X$ be a noncommutative rational or rationally quasi-ruled surface
  and suppose that $L$ is a line bundle with respect to some blowdown
  structure on $X$.  Then $L$ is a line bundle with respect to {\em every}
  blowdown structure on $X$.
\end{cor}

\begin{proof}
  Certainly, $L$ is a torsion-free sheaf of rank 1 and saturates the
  relevant bound.  In the rational or rationally ruled case, the bound is
  manifestly independent of the blowdown structure, so the result is
  immediate.  For a strictly quasi-ruled surface, we need merely observe
  that the class $f$ is the same in any blowdown structure, so the bound
  remains invariant.
\end{proof}

We now return to considering only rational and rationally ruled surfaces.
It turns out that our bound on the Euler characteristic of torsion-free
sheaves of rank 1 also gives a bound on certain globally generated sheaves.

\begin{cor}\label{cor:Ox_quotient_bound}
  If $M$ is a quotient of $\sO_X$, then either $M\cong \sO_X$ or
  $\rank(M)=0$, $c_1(M)$ is effective, and
  \[
  \chi(M)\ge -c_1(M)\cdot (c_1(M)+K_X)/2.
  \]
\end{cor}

\begin{proof}
  Let $I$ be the kernel of the map $\sO_X\to M$, and observe that either
  $I=0$ or $I$ is torsion-free of rank 1, with
  $c_1(I)=-c_1(M)$ and $\chi(I)=\chi(\sO_X)-\chi(M)=1-g-\chi(M)$.
\end{proof}

\begin{cor}
  If $M$ is globally generated, then $c_1(M)$ is effective.
\end{cor}

\begin{proof}
  If $\sO_X\otimes \Hom(\sO_X,M)\to M$ is surjective, then each subquotient
  of the filtration induced by a complete flag in $\Hom(\sO_X,M)$ has
  effective Chern class, since it is either $\sO_X$ or $1$-dimensional.
\end{proof}

\begin{cor}\label{cor:finite_hps_for_glob_generated}
  For any integer $n$ and quadratic polynomial $P$, there is a finite
  subset of $K_0^{\num}(X)$, depending only on $D_a$ and the combinatorics
  of $Q$, that contains the class of any quotient of $\sO_X^n$ with Hilbert
  polynomial $P$.
\end{cor}

\begin{proof}
  The Hilbert polynomial determines the rank, the Euler characteristic, and
  $c_1(M)\cdot D_a$.  Since $c_1(M)$ is effective, we reduce to showing
  that there are finitely many classes that are effective on {\em some}
  surface with the given combinatorics and have given intersection with
  $D_a$.  The effective monoid of $X$ is generated by components of $Q$ and
  formal $-1$- and $-2$-curves on $X$, so it will suffice to find a
  finitely generated monoid containing the latter such that every generator
  has positive intersection with $D_a$.  We assume $m\ge 1$ (the cases
  $m\le 0$ are straightforward) so that we may choose a blowdown structure
  such that $D_a$ is in the fundamental chamber.

  Let $W(D_a)$ be the Weyl group generated by the reflections in the simple
  roots orthogonal to $D_a$.  Since $D_a^2>0$ so $D_a^\perp$ is negative
  definite, this is a finite Weyl group.  We then claim that any formal
  $-1$- or $-2$-curve is contained in the monoid generated by
  the $W(D_a)$-orbits of the simple roots not orthogonal to $D_a$ along
  with the $W(D_a)$-orbit of $e_m$ (and $f-e_1$ for $m=1$).  But this
  follows by the usual argument, once we observe that since $D_a$ is ample,
  a simple root orthogonal to $D_a$ cannot be effective.

  Since $D_a$ is $W(D_a)$-invariant, has positive intersection with $e_m$
  (and $f-e_1$ if needed), and nonnegative intersection with any simple
  root, it follows that this finitely generated monoid is indeed finitely
  graded by the intersection with $D_a$, as is the monoid obtained by
  throwing in the components of $Q$.
\end{proof}

Returning to a split family $X$ over a Noetherian base, let
$\Hilb^{n,+}(X)$ denote the functor such that $\Hilb^{n,+}(X)(T)$ is the
set of $T$-flat coherent sheaves such that every geometric fiber is torsion
free, with $\rank(I)=1$, $c_1(I)=0$ and $\chi(I)=1-g-n$.  (We will
eventually show (Theorem \ref{thm:hilb_is_projective}) that this is
represented by a smooth projective scheme over $S$.)  Note that these
conditions on a sheaf are independent of the blowdown structure, and thus
this functor can be defined for a non-split family.

We need to show that this functor is covered by a Noetherian scheme;
i.e., that there exists a family of sheaves $I$ over some base change of
$X$ such that every geometric point of $\Hilb^{n,+}(X)$ is a fiber of $I$.

\begin{prop}
  Let $X$ be a split family over a Noetherian base.  Then
  $\Hilb^{n,+}(X)$ is empty for $n<0$, while for $n>0$, it is covered by
  a Noetherian scheme.
\end{prop}

\begin{proof}
  For noncommutative planes, an explicit construction was given in
  \cite{NevinsTA/StaffordJT:2007} showing that the functor
  $\Hilb^{n,+}(X_{-1})$ is not only representable but projective.

  For $m>0$, the proof of \cite[Lem.~11.43]{generic} applies mutatis
  mutandis to the case of general rationally ruled surfaces, and thus tells
  us that if $I\in \Hilb^{n,+}(X_m)$, then $\alpha_{m*}I(-ne_m)\in
  \Hilb^{n(n+3)/2,+}(X_{m-1})$ and one can recover $I$ from
  $\alpha_{m*}I(-ne_m)$.  It follows that $\Hilb^{n,+}(X_m)$ is a closed
  subfunctor of $\Hilb^{n(n+3)/2,+}(X_{m-1})$.

  For $m=0$, the semiorthogonal decomposition gives a five-term exact
  sequence
  \[
  0\to \rho_1^*M_1\to \rho_0^*M_0\to I\to \rho_1^*N_1\to \rho_0^*N_0\to 0
  \]
  for sheaves $M_i$, $N_i$ on $C$.  Chern class considerations tell us that
  $\Ext^2(\rho_0^*{\cal L},I)=\Ext^2(\rho_{-1}^*{\cal L},I)=0$ for any line
  bundle ${\cal L}$ on $C$, and thus that $\rank(N_0)=\rank(N_1)=0$, so
  that $\rank(M_1)=0$ and $\rank(M_0)=1$.  If $M_0$ had a nonzero
  $0$-dimensional subsheaf, then the corresponding image in $I$ would have
  nontrivial class in $K_0^{\num}(X)$ and rank 0, which is impossible since
  $I$ is torsion-free.  We thus conclude that $M_0$ is a line bundle and
  $M_1=0$.  Moreover, $\chi(M_0)-\chi(N_0)$ and $\chi(N_1)$ are determined
  by $[I]$, and $\deg(M_0)$ is bounded between $-n$ and $-1$.  Thus for
  each $n$, the moduli stack of potential triples $(M_0,N_1,N_0)$ is
  Noetherian, and remains so when we include the map $\rho_1^*N_1\to
  \rho_0^*N_0$ (on which we impose the open condition of surjectivity) and
  then the $\Ext^1$ from the kernel to $M_0$ and the torsion-free condition
  on the extension.  The result is a Noetherian family of sheaves
  surjecting on $\Hilb^{n,+}(X_0)$ as required.
\end{proof}

\begin{cor}
  For any quadratic polynomial $P(t)$, the functor $\Quot_{X/S}(\sO_X,P)$
  is covered by a Noetherian scheme.
\end{cor}

\begin{proof}
  The kernel in any corresponding short exact sequence is a torsion-free
  sheaf of rank 1, and thus up to a twist is just a point in
  $\Hilb^{n,+}(X)$ for suitable $n$.  There are only finitely many possible
  choices for the associated Chern class, and the short exact sequence is
  determined by a choice of injective map from the torsion-free sheaf to
  $\sO_X$.
\end{proof}

\begin{prop}
  For any coherent sheaf $M\in \coh(S)$ and any quadratic polynomial
  $P(t)$, the functor $\Quot_{X/S}(\sO_X\otimes_S M,P)$ is covered by a
  Noetherian scheme.
\end{prop}

\begin{proof}
  Since $S$ is Noetherian, it suffices to prove this on each piece of an
  affine covering of $S$, and on each such covering we can cover $M$ by a
  free module, so that it suffices to consider the case that $M\cong
  \sO_S^n$.

  Consider a subsheaf of $\sO_X^n$ with fixed class $[I_n]\in
  K_0^{\num}(X)$ having the given Hilbert polynomial.  (By Corollary
  \ref{cor:finite_hps_for_glob_generated}, there are only finitely many
  possible values of $[I_n]$.)  The filtration associated to a complete
  flag gives subquotients each of which is either $0$ or torsion-free of
  rank 1. If we ignore the Euler characteristics for the moment, then there
  are only finitely many possibilities for the rank and Chern class of the
  last subquotient, since the last subquotient is either $0$ or has Chern
  class $-D$ such that both $D$ and $-c_1([I])-D$ are effective.  We thus
  find by induction that there are only finitely many choices for the
  entire sequence of ranks and Chern classes of subquotients.  For each
  such choice, every subquotient has an upper bound on its Euler
  characteristic, and since the total Euler characteristic is fixed, there
  are only finitely many possible choices for the sequence of Euler
  characteristics.  In particular, we conclude that there are indeed only
  finitely many possible choices for the numerical invariants of the last
  subquotient, and thus only finitely many possible choices for its Hilbert
  polynomial.

  Fix such a choice $P_n$, and consider the morphism
  $\Quot_{X/S}(\sO_X^n,P)\to \Quot_{X/S}(\sO_X^{n-1},P-P_n)\times
  \Quot_{X/S}(\sO_X,P_n)$.  The fiber over a point corresponding to
  subsheaves $I_{n-1}\subset \sO_X^{n-1}$ and $I\subset \sO_X$ is
  represented by the affine space $\Hom(I,\sO_X^{n-1}/I_{n-1})$.  Since
  this is Noetherian and $\Quot_{X/S}(\sO_X,P-P_n)$,
  $\Quot_{X/S}(\sO_X,P_n)$ are covered by Noetherian schemes, the claim
  follows.
\end{proof}

Applying Corollary \ref{cor:noetherian_is_bounded} to this covering family
gives the following.

\begin{cor}
  Let $X/S$ be a family of noncommutative surfaces over a Noetherian base,
  and let $M$ be a coherent sheaf on $S$.  Then for any quadratic
  polynomial $p$, there is a bound $b_0$ such that for all $T/S$ and $I\in
  \Quot_{X/S}(\sO_X\otimes_{\sO_S} M,P)(T)$, $I$ is acyclically
  generated by $\sO_X(-bD_a)$ for all $b\ge b_0$.
\end{cor}

\begin{proof}
  Since $T$ need not be Noetherian, our previous arguments do not suffice
  to show that $I$ inherits acyclic generation from its fibers.  Acyclicity
  is not a difficulty, since we can represent $I$ by the complex
  \[
  \sO_X\otimes_{\sO_S} M\otimes_{\sO_S} \sO_T\to N
  \]
  with $N$ flat coherent (thus certainly inheriting acyclicity from fibers)
  and $M$ coherent, and thus
  \[
  \Ext^1(\sO_X(-bD_a),I)\otimes k(t)\cong
  \Ext^1(\sO_X(-bD_a),I\otimes k(t))=0
  \]
  for all points $t\in T$, implying $\Ext^1(\sO_X(-bD_a),I)=0$, since
  $\Ext^1(\sO_X(-bD_a),I)$ is a quotient of the (flat) coherent sheaf
  $\Hom(\sO_X(-bD_a),N)$.
  
  For global generation, it suffices to show that for each $b\ge b_0$,
  there is some $b_1$ such that
  \[
  \Hom(\sO_X(-b'D_a),\sO_X(-bD_a))\otimes \Hom(\sO_X(-bD_a),I)
  \to
  \Hom(\sO_X(-b'D_a),I)
  \]
  is surjective for all $b'\ge b_1$.  By the snake lemma, the cokernel is a
  quotient of the locally free coherent sheaf
  \[
  \ker(
  \Hom(\sO_X(-b'D_a),\sO_X(-bD_a))\otimes \Hom(\sO_X(-bD_a),N)
  \to
  \Hom(\sO_X(-b'D_a),N))
  ,
  \]
  so that again it suffices to check vanishing on fibers.  The existence of
  a global $b_1$ that works for every fiber follows using the Noetherian
  covering family.
\end{proof}

As is common in these constructions, the direct approach will only give
quasiprojectivity, so we need to check the valuative condition.

\begin{lem}
  The functor $\Quot_{X/S}$ is proper.
\end{lem}

\begin{proof}
  Let $X/R$ be a family of noncommutative surfaces, where $R$ is a discrete
  (since the base scheme is Noetherian) valuation ring with quotient field
  $K$ and field of fractions $k$.  For $M\in \coh(X)$, consider an exact
  sequence
  \[
  0\to I\to M_K\to N\to 0
  \]
  in $\coh(X_K)$.  We need to show that there is a unique short exact
  sequence
  \[
  0\to I'\to M\to N'\to 0
  \]
  such that $I'_K=I$ and $N'$ is $R$-flat, or equivalently $N'$ is
  $R$-torsion-free (since a module over a dvr is flat iff it is
  torsion-free).  Let $I'$ be the maximal submodule of $M$ such that
  $I'_K=I$.  Then $M/I'$ cannot have torsion, since we could move any
  torsion to the subsheaf without changing the generic fiber, while any
  proper subsheaf $I''\subsetneq I'$ has $I'/I''$ torsion since they have
  the same generic fiber, and thus makes $N''$ fail to be torsion-free.
  It follows that $I'$ gives rise to the desired unique extension.
\end{proof}

\begin{thm}
  Let $X/S$ be a family of noncommutative surfaces over a Noetherian base.
  Then for any $M\in \coh(X)$, $\Quot_{X/S}(M,P)$ is a projective scheme
  over $S$.
\end{thm}

\begin{proof}
  As already discussed, it suffices to prove this for
  $\Quot_{X/S}(\sO_X\otimes_{\sO_S} M,P)$ with $M\in \coh(S)$.  We thus
  find in general that there is a constant $b_0$ such that for any $T$-point
  of $Q:=\Quot_{X/S}(\sO_X\otimes_{\sO_S}M,P)$, the kernel in the
  corresponding short exact sequence is acyclically generated by
  $\sO_X(-b_0D_a)$.  We thus obtain a resolution
  \[
  \sO_X(-b_0D_a)\otimes_{\sO_T} \Hom(\sO_X(-b_0D_a),I)
  \to
  \sO_X\otimes_{\sO_T} M_T
  \to
  N_T
  \to
  0,
  \]
  giving a $T$-point of the Grassmannian
  \[
  G:=\Grass(\Hom(\sO_X(-b_0D_a),\sO_X)\otimes_{\sO_S} M,P(b_0))
  \]
  classifying rank $P(b_0)=\rank(\Hom(\sO_X(-b_0D_a),N_T))$ quotients of
  $\Hom(\sO_X(-b_0D_a),\sO_X)\otimes_{\sO_S} M$.
  
  In the other direction, any point of the Grassmannian determines a
  quotient sheaf of $\sO_X\otimes_{\sO_T}M_T$, which gives a point of $Q$
  iff it has the correct Hilbert polynomial.  It follows that $Q$ is a
  component of the universal flattening stratification of the universal
  quotient sheaf over the Grassmannian.  If we knew that sheaf had finite
  flattening stratification, we could then conclude that $Q$ was locally
  closed in $G$.  Since $Q$ is proper and $G$ is projective, this would
  imply $Q$ projective as required.  In particular, this implies that the
  theorem holds whenever $S$ (and thus $G$) is of finite type.

  For more general base rings, it clearly suffices to show that $Q$ is
  closed in $G$.  In particular, if we can show that $Q$ is fppf locally
  projective, then it follows that the map $Q\to G$ is fppf locally closed,
  so closed, and thus $Q$ must be projective.  As the question is now fppf
  local, we may now make some simplifying assumptions: since $X$ is
  \'etale locally split, we may assume that $X$ is split, and since
  $M$ is Zariski locally a quotient of a free sheaf, we may assume this is
  so.  Since the corresponding map of $\Quot$ functors is closed, we may
  assume $M\cong \sO_X^n$.
  
  Let ${\cal M}$ denote the moduli stack of split surfaces, and note that
  $X/S$ induces a map $S\to {\cal M}$.  Since $S$ is Noetherian and ${\cal
    M}$ is locally of finite type, there is a morphism $U\to {\cal M}$ with
  $U$ of finite type such that the fiber product is smooth and
  surjective over $S$.  Let $T$ be that fiber product, and observe that one
  has
  \[
  \Quot_{X/S}(\sO_X^n,P)\times_S T
  \cong
  \Quot_{X_T/T}(\sO_X^n,P)
  \cong
  \Quot_{X_U/U}(\sO_X^n,P)\times_U T.
  \]
  Since $U$ is of finite type, the corresponding $\Quot$ scheme is
  projective, and thus $\Quot_{X/S}(\sO_X^n,P)$ is fppf locally projective
  as required.
\end{proof}

\begin{rem}
  Of course, in the last part of the argument, what we are really proving
  is that the functor $\Quot_{X/{\cal M}}(\sO_X^n,P)$, appropriately
  defined, is fppf locally projective.
\end{rem}

Although much of the above proof involved working around the potential lack
of finite flattening, this is in some sense unnecessary, as finite
flattening always holds.  Unfortunately, the only proof we have of this
fact uses projectivity of the $\Quot$ scheme!

\begin{cor}\label{cor:can_always_flatten}
  Let $X/S$ be a family of noncommutative surfaces over a Noetherian base.
  Then any sheaf $M\in \coh(X)$ has finite flattening stratification.
\end{cor}

\begin{proof}
  It suffices to show that when $S$ is integral, there is an open
  neighborhood of $k(S)$ on which $M$ is flat, or equivalently on which the
  Hilbert polynomial is constant.  Let $p$ be the Hilbert polynomial of
  $M_{k(S)}$, and consider the Quot scheme $Q=\Quot_{X/S}(M,P)$, with
  associated short exact sequence
  \[
  0\to I\to M_Q\to N\to 0.
  \]
  We first note that the fiber over $k(S)$ consists of the single short
  exact sequence
  \[
  0\to 0\to M_{k(S)}\to M_{k(S)}\to 0,
  \]
  as the Hilbert polynomial constraint forces the map $M\to N$ to become
  an isomorphism.  In particular, since the image of $\pi:Q\to S$ contains
  $k(S)$, the morphism is surjective.  Since $N$ is flat, the restriction
  of the short exact sequence to any point $q\in Q$ will still be a short
  exact sequence, and thus we find that $I_q\ne 0$ precisely when the
  Hilbert polynomial of $M_{\pi(q)}$ is different from $p$.  In other
  words, the closed subset $\pi(\supp(I))$ of $S$ agrees with the set of
  points where the Hilbert polynomial jumps.  It follows that the set of
  points with the same Hilbert polynomial as the generic point is open as
  required.
\end{proof}

\begin{rem}
  This argument in fact shows that the class in the Grothendieck group is
  upper semicontinuous: it is constant on locally closed subsets, and the
  value at a point on the boundary of such a set differs by the class of a
  sheaf over that point.
\end{rem}

\bigskip

We now turn to a closer analysis of the functors $\Hilb^{n,+}(X/S)$ for
families over a Noetherian base.  Note that since any torsion-free sheaf of
rank 1 is stable, we automatically find (via the standard argument of
Langton) that the functor is proper, with smoothness following by
deformation theory, since $\Ext^2(I,I)=0$.  We then find that the dimension
is everywhere $1-\chi(I,I)=2n+g$.

\begin{thm}\label{thm:hilb_is_projective}
  For all $n\ge 0$, the functor $\Hilb^{n,+}(X/S)$ is represented by a
  smooth projective scheme of dimension $2n+g$ over $S$ with geometrically
  integral fibers.
\end{thm}

\begin{proof}
  Since $\Hilb^{n,+}(X/S)$ is covered by a Noetherian family, there is a
  line bundle $L=\sO_X(-bD_a)$ such that any point $I\in \Hilb^{n,+}(X)$ is
  acyclically generated by $L$.  As above, a complete flag in
  $\Hom(L,I)$ determines a filtration of $I$ and there are only finitely
  many possibilities for the sequence of subquotients; it follows that
  there are only finitely many possibilities for the class in
  $\Pic^{\num}(X)$ of a subsheaf of $I$ which is generated by $L$.
  It follows as in \cite[Thm.~11.47]{generic} (following the usual
  commutative arguments) that we can represent $\Hilb^{n,+}(X/S)$ as a
  subscheme of a GIT quotient of $\Quot(L^{\chi(L,[I])},[I])$, making it
  quasi-projective.  Since it is proper, it is projective as required.

  For irreducibility, it suffices to consider the moduli problem over a
  smooth covering of the moduli stack of surfaces; in particular, we may
  assume that the family contains commutative fibers.  When $X$ is
  commutative, the numerical conditions on $I$ force it to be the tensor
  product of a line bundle in $\Pic^0(X)$ by the ideal sheaf of an
  $n$-point subscheme, and thus that fiber is connected, and thus
  (following \cite[Prop.~8.6]{NevinsTA/StaffordJT:2007}) every fiber is
  connected.  Since the fibers are smooth, they are geometrically integral
  as required.
\end{proof}

In the rational case, this is a family of deformations of the Hilbert
scheme of the corresponding commutative surface (except when $n=1$, when it
{\em is} the corresponding commutative surface, just as in \cite{generic}).
To recover something closer to the usual Hilbert scheme in higher genus, we
observe that there is a natural morphism $\Hilb^{n,+}(X)\to \Pic(Q)$ given
by $I\mapsto \det(I|^{\bf L}_Q)$.  Since the numerical class of $I$ is
fixed, the image is contained in a particular coset of $\Pic^0(C)$.  Any
line bundle in $\Pic^0(C)$ induces an autoequivalence of $\coh X$, and we
find that this autoequivalence simply twists the determinant of the
restriction to $Q$ by that same line bundle.  We thus find that the fibers
of the determinant-of-restriction map are all isomorphic, and we thus have
an isomorphism of the form $\Hilb^{n,+}(X)\cong \Hilb^n(X)\times \Pic^0(C)$
where $\Hilb^n(X)$ is any fiber of the determinant-of-restriction map.
(One caveat is that this presumes the image of the
determinant-of-restriction map has a point!)  We conclude that $\Hilb^n(X)$
is smooth, irreducible, and projective of dimension $2n$.  When $X$ is
commutative, the relevant coset of $\Pic^0(C)$ is the trivial coset, and
thus there is a canonical choice of fiber, and we find that the result
agrees with the usual Hilbert scheme of $n$ points.

In the commutative case, there is also a map from the Hilbert scheme of $n$
points to $\Sym^n(C)$.  This too has an analogue in the noncommutative
setting.  We saw in the case of a ruled surface that $I$ always satisfies
$\rho_{-1*}I=0$ and $R^1\rho_{-1*}I$ is 0-dimensional of degree $n$, inducing
a map $\Hilb^{n,+}(X)\to \Sym^n(C)$ which is preserved by the action of
$\Pic^0(C)$, and thus factors through $\Hilb^n(X)$.  In general, one finds
that for $I\in \Hilb^{n,+}(X_m)$, there is an integer $0\le d\le n$ such that
$\alpha_{m*}I$ is in $\Hilb^{d,+}(X_{m-1})$ and $R^1\alpha_{m*}I$ is a
$0$-dimensional sheaf of degree $n-d$.  It follows by induction that
\[
h^1(R\rho_{-1*}R\alpha_{1*}\cdots R\alpha_{m*}I)
\]
is again a 0-dimensional sheaf of degree $n$, so the map to $\Sym^n(C)$
survives.

\section{Moduli of sheaves and complexes}

Now that we have established that our surfaces behave reasonably, the most
natural next object of study is the corresponding moduli spaces of sheaves.
The construction via the derived category suggests that a more natural
first problem would be to study moduli spaces in $\perf(X)$.  The most
general result along those lines is the following (specializing a result of
\cite{ToenB/VaquieM:2007} for fairly general dg-categories).

\begin{prop}\cite{ToenB/VaquieM:2007}
  For any family $X/S$ of noncommutative surfaces, there is a derived stack
  ${\cal M}_{X/S}$ such that for all derived schemes $T/S$, ${\cal
    M}_{X/S}(T)$ is the set of quasi-isomorphism classes of objects in
  $\perf(X_T)$.  This stack is locally geometric and locally of finite
  presentation, and the tangent complex to ${\cal M}_{X/S}$ at a geometric
  point $M$ is $R\sEnd(M)[1]$.
\end{prop}

\begin{proof}
  Per the reference, this holds if $\perf(X)$ is {\em saturated}, i.e., if
  it is locally perfect, has a compact generator, is triangulated, and is
  smooth.  The first three properties follow immediately from the
  description via a dg-algebra in $\perf(S)$.  Smoothness follows from the
  fact that gluing preserves smoothness \cite{LuntsVA/SchnurerOM:2014}, along
  with the fact that $\perf(S)$ and $\perf(C)$ are smooth over $S$.
\end{proof}

\begin{rem}
  Of course, the most natural object to study here is the derived stack
  classifying objects of the universal $\perf(X)$ over the moduli stack of
  surfaces.  Since the moduli stack of surfaces is singular, this suggests
  that one should really define a {\em derived} moduli stack of surfaces,
  and then obtain a moduli stack of objects on the universal derived
  noncommutative surface over that derived moduli stack.
\end{rem}

This, of course, is quite far from a stack in the usual sense, and even the
truncation is only a higher Artin stack.  However, we can use this to
obtain algebraic stacks in the usual sense.  After \cite{LieblichM:2006},
call an object in a geometric fiber ``gluable'' if $\Ext^i(M,M)=0$ for
$i<0$.  By semicontinuity, this is an open condition.  (More precisely, it
is open for each $i$, and $R\sEnd(M)$ is perfect, so locally bounded in
cohomological degree.)  Thus we may define a functor ${\cal M}^{ug}_{X/S}$
such that (for $T/S$ an ordinary scheme) ${\cal M}^{ug}_{X/S}(T)$ is the
set of quasi-isomorphism classes of objects in $\perf(X_T)$ such that every
geometric fiber is gluable, and this functor will be an open substack of
the truncation of ${\cal M}_{X/S}$.

\begin{cor}
  The functor ${\cal M}^{ug}_{X/S}$ can be represented by an algebraic stack
  locally of finite presentation.
\end{cor}

\begin{rem}
  Note that the substack of ${\cal M}_{X/S}$ classifying flat families of
  coherent sheaves is contained in the derived stack of which ${\cal
    M}^{ug}_{X/S}$ is a truncation, since any coherent sheaf on $X/k$ is
  gluable.  We thus find that coherent sheaves on $X$ are classified by an
  algebraic stack locally of finite presentation.  In fact, this is true
  for {\em any} family of $t$-structures on $D_{\qcoh}(X)$.
\end{rem}

Similarly, one calls an object over a geometric fiber ``simple'' if the
natural map
\[
k\to \tau_{\le 0} R\End(M)
\]
is an isomorphism.  Again, this condition is open, so defines an open
substack ${\cal M}^{s}_{X/S}$.  Here there is a further reduction one can
make: since the family of automorphism groups of any family with simple
fibers is just $\G_m$ (so flat and locally of finite presentation), ${\cal
  M}^{s}_{X/S}$ is a gerbe over an algebraic space $\Spl_{X/S}$.  The
corresponding functor is easy enough to define: an element of
$\Spl_{X/S}(T)$ is an equivalence class of simple objects in $\perf(X_U)$
with $U/T$ \'etale, such that the two base changes to $U\times_T U$ are
quasi-isomorphic (but not necessarily canonically so).  Objects $M$ and
$M'$ are equivalent if their base changes to the fiber product \'etale
cover are quasi-isomorphic.

\begin{cor}
  The functor $\Spl_{X/S}$ is represented by a algebraic space locally of
  finite presentation.
\end{cor}

\begin{rem}
  Note that a stable sheaf (relative to any stability condition) is
  automatically simple.
\end{rem}

In the commutative setting
(\cite{BottacinF:1995,HurtubiseJC/MarkmanE:2002b,poisson}; also for the
noncommutative rational surfaces considered in \cite{generic}), the moduli
space of simple {\em sheaves} has an important additional structure, namely
a Poisson structure (defined for algebraic spaces as a consistent family of
Poisson brackets on an \'etale covering by affine schemes).  Such a Poisson
structure is in particular a biderivation, and the construction of the
biderivation (going back to \cite{TyurinAN:1988} in the commutative
setting) carries over to $\Spl_{X/S}$.  To wit, the cotangent sheaf on
$\Spl_{X/S}$ is given by $\sExt^1(M,\theta M)$ where $M$ is the universal
family (the cotangent complex on the corresponding derived stack is
$R\sHom(M,\theta M)[1]$, and the $\sExt^2(M,\theta M)$ term is killed by
the structure morphism of the gerbe), and one has a bivector defined by
\[
\sExt^1(M,\theta M)\otimes \sExt^1(M,\theta M)
\to
\sExt^1(M,M)\otimes \sExt^1(M,\theta M)
\to
\sExt^2(M,\theta M)
\to
\sO_{\Spl_{X/S}},
\]
where the map $\theta M\to M$ comes from a choice of isomorphism $\theta
M\cong M(-Q)$.  (In other words, the map $\theta M\to M$ is a Poisson
structure on $X$ that vanishes on $Q$.  This is unique modulo scalars, so
such Poisson structures on noncommutative surfaces form a line bundle over
the moduli stack of noncommutative surfaces.  The associated biderivation
scales linearly with the choice of Poisson structure on the surface, so is
associated to the same line bundle.) In order to define a Poisson
structure, this bivector needs to be alternating and satisfy the Jacobi
identity.  Since it is defined via the Yoneda product, it is easy to see
that it is antisymmetric, so (apart from characteristic 2), only the Jacobi
identity remains.  This is unfortunately tricky to deal with, though see
below for the case of simple sheaves.

One important part of the known special cases is the identification of the
symplectic leaves of the Poisson structure.  This can be stated without
using the Jacobi identity (or alternation in characteristic 2).  Indeed,
for manifolds, any bivector induces a foliation by manifolds with
nondegenerate $2$-forms, which are closed if the bivector is Poisson.
(In fact, the converse is true, though in the algebraic setting this only
holds for bivectors on reduced schemes.)  This was key to the argument of
\cite{poisson} for vector bundles (extending the result of
\cite{BottacinF:1995} to finite characteristic).  In the algebraic setting,
such a foliation need not exist, but the tangent space at $M$ to the
foliation is always defined, and given by the image of the cotangent space
in the tangent space under the bivector.  In our case, this is the image of
\[
\Ext^1(M,\theta M)\to \Ext^1(M,M),
\]
or equivalently, the kernel of the natural map to $\Ext^1(M,M|^{\bf
  L}_Q)\cong \Ext^1_Q(M|^{\bf L}_Q,M|^{\bf L}_Q)$.  This suggests that the
symplectic leaves should be the fibers of $|^{\bf L}_Q$, and indeed this is
the case for simple sheaves on commutative rationally ruled surfaces
\cite{poisson}.

In general, each simple object $M\in \Spl_{X/k}$ defines a subspace of
simple objects with quasi-isomorphic derived restriction to $Q$, which we
(somewhat abusively) call the symplectic leaf of $M$.  (Note, however, that
contrary to the usual analytic notion of symplectic leaf, these symplectic
leaves need not be connected.)  The bivector on $\Spl_{X/k}$ induces a
perfect pairing on the cotangent spaces of the symplectic leaf, and thus
(if the leaf is smooth!) defines an isomorphism between the tangent and
cotangent spaces, thus a nondegenerate $2$-form.  (We of course conjecture
that this form is closed, as this would follow if the bivector satisfied
the Jacobi identity.)

We thus wish to know that the symplectic leaves are smooth (which in
characteristic 2 implies the pairing is alternating, as the self-pairing
gives the obstruction to extending a first-order deformation).  This is
especially the case in finite characteristic, where Poisson structures
themselves are not nearly as useful, but the symplectic leaves are still
natural moduli spaces to consider.  The first hope would be that the
corresponding obstructions might vanish.  We can compute this using the
derived moduli stack ${\cal M}_{X/k}$.  Indeed, the symplectic leaves
(rather, certain gerbes over the symplectic leaves) are truncations of open
substacks of fibers of the obvious map ${\cal M}_{X/k}\to {\cal M}_{Q/k}$.
We can thus compute the tangent complex $T_M$ to the derived symplectic
leaf containing $M$ as the cone:
\[
T_M\to R\Hom(M,M)[1]\to \tau_{\ge 1}R\Hom_Q(M|^{\bf L}_Q,M|^{\bf L}_Q)[1]\to.
\]
(Indeed, $R\Hom(M,M)[1]$ is the tangent complex in ${\cal M}_{X/k}$, and
the other term is the normal complex of the given point, viewed as a higher
Artin stack.)  This makes it easy to compute the homology of $T_M$:
\[
h^p(T_M) = \begin{cases}
  \Ext^{p+1}(M,M) & p<0\\
  \im(\Ext^1(M,\theta M)\to \Ext^1(M,M)) & p=0\\
  \Ext^{p+1}(M,\theta M) & p>0
\end{cases}
\]
Indeed (unsurprisingly), the tangent complex of the symplectic leaf is
self-dual.  Since $M$ is simple, $T_M$ has amplitude $[-1,1]$, but we have
$h^{-1}T_M\cong k$ and thus by duality $h^1T_M\cong k$, giving an
obstruction.

The dimension of the tangent space at $M$ of its symplectic leaf is
$\dim(h^1(T_M))=\chi(T_M)+2$, and the description of $T_M$ as a cone tells
us that
\[
 -\chi(T_M) = \chi(M,M) - \chi(\tau_{\ge 1}R\Hom_Q(M|^{\bf L}_Q,M|^{\bf
    L}_Q)).
\]
(Compare the ``index of rigidity'' of \cite{poisson}.)  The second term is
clearly constant on the symplectic leaf, while the first term depends only
on $[M]\in K_0^{\num}(X)$, and is thus at least locally constant.  In other
words, the tangent spaces to symplectic leaves have locally constant
dimension, and thus each component of a symplectic leaf is either smooth or
nowhere reduced.  (The latter can actually happen for the commutative
Poisson surfaces of characteristic 2 that are not fibers of the moduli
stack of noncommutative rational or rationally ruled surfaces.)

Luckily, there are two cases in which we can prove smoothness, which
together are large enough that we have been unable to construct an object
to which neither case applies.

\begin{prop}
  Let $X/k$ be a noncommutative rational or rationally ruled surface over an
  algebraically closed field, and let $N\in \perf(Q)$ be any nonzero object.
  Then the fiber over $N$ in $\Spl_{X/k}$ is smooth.
\end{prop}

\begin{proof}
  We need to show that for any local Artin algebra $A$ with residue field
  $k$, any small extension $A^+$ of $A$, and any $M\in \perf(X_A)$ with
  $M|^{\bf L}_Q\cong N\otimes_k A$, there is an extension $M^+$ of $M$ to
  $A^+$ with $M^+|^{\bf L}_Q\cong N\otimes_k A^+$.

  We first consider the problem of extending $M$, ignoring for the moment
  the constraint on its restriction to $Q$.  The obstruction to extending
  $M$ is given by a class in $\Ext^2(M_k,M_k)$, and the image of this class
  in $\Ext^2_Q(N,N)$ is the obstruction to extending $M|^{\bf L}_Q$.  Since
  $M|^{\bf L}_Q$ is a trivial deformation, there {\em is} no such
  obstruction, and thus the obstruction to extending $M$ is in the image of
  $\Ext^2(M_k,\theta M_k)$.  The map $\Ext^2(M_k,\theta M_k)\to
  \Ext^2(M_k,M_k)$ vanishes, since it is dual to $\Hom(M_k,\theta M_k)\to
  \Hom(M_k,M_k)$, which vanishes since $\Hom(M_k,M_k)\to \Hom_Q(N,N)$ is
  injective.  It follows that the obstruction to extending $M$ is trivial,
  and thus it has extensions, classified by a torsor $T$ over
  $\Ext^1(M_k,M_k)$.

  It remains to show that at least one of those extensions restricts to
  $N\otimes_k A^+$.  The map taking an extension of $M$ to its restriction,
  viewed as a deformation of $N$, is a torsor map $T\to \Ext^1_Q(N,N)$
  compatible with the homomorphism $\Ext^1(M_k,M_k)\to \Ext^1_Q(N,N)$, and
  thus its image is a coset of the image of $\Ext^1(M_k,M_k)$.  We thus
  find that the image of $T$ contains $0$ iff the composition
  \[
  T\to \Ext^1_Q(N,N)\to \Ext^2(M_k,\theta M_k)
  \]
  is 0.  The composition
  \[
  \Ext^1_Q(N,N)\to \Ext^2(M_k,\theta M_k)\cong k
  \]
  is dual to the morphism
  \[
  k\to \Hom_Q(N,N)
  \]
  taking $1$ to $\id$, and thus has the alternate expression
  \[
  \begin{CD}
  \Ext^1_Q(N,N) @>\tr>> \Ext^1_Q(\sO_Q,\sO_Q) \to \Ext^2(\sO_X,\theta
  \sO_X)\cong k.
  \end{CD}
  \]
  The trace of a deformation of complexes on a commutative scheme is the
  deformation of its determinant, so that we finally reduce to showing that
  $\det(M^+|^{\bf L}_Q)\otimes \det(N\otimes_k A^+)^{-1}$ is in the image
  of $\Pic^0(C)(A^+)$ for any $M^+$.  But this follows from the fact that
  determinant-of-restriction gives a well-defined map from $K_0^{\num}(X)$
  to $\Pic(Q)/\Pic^0(C)$.
\end{proof}

\begin{rem}
  Much the same argument was used for vector bundles on commutative
  surfaces in \cite{poisson}; the main new complication is that
  $\Ext^2(M,M)$ is no longer guaranteed to vanish.  Of course, since it
  does not actually contribute to the obstructions of the symplectic leaf,
  it is unsurprising that we can show that we never see nontrivial classes
  in $\Ext^2(M,M)$.
\end{rem}

When $M|^{\bf L}_Q=0$, we have $\Ext^2(M,\theta M)\cong\Ext^2(M,M)$, and
thus we can no longer get started in the above argument.  The key idea for
dealing with (nearly all of) these remaining cases is that in any
deformation theory problem, the moduli space of formal deformations is a
quotient of the completion at the origin of the affine space of
infinitesimal deformations, and the obstruction space gives a bound on the
number of elements required to generate the ideal.  In particular, since
our obstruction space is $1$-dimensional, the ideal must be principal
(hopefully 0).  Directly showing that the ideal is 0 can be difficult (as
it is unexpected behavior), but when the ideal is nonzero, controlling the
ideal reduces to showing that certain deformations do {\em not} extend,
giving us an alternate angle of attack.

Of course, since our objective is to prove the ideal is 0, this may not
seem particularly helpful.  However, because we are now asking for $M|^{\bf
  L}_Q$ to be 0, we can embed the symplectic leaf in a family of moduli
problems by allowing the surface to vary.  The argument for $M|^{\bf
  L}_Q\ne 0$ suggests that we should look at $\det M|^{\bf L}_Q$, or rather
its image in $\Pic^0(Q)/\Pic^0(C)$.  This can be computed in two ways.  On
the one hand, it is $\det(0)=\sO_Q \Pic^0(C)$ by our assumption on the
restriction of $M$.  On the other hand, we can also compute $\det(M)$ via
its class in $K_0^{\num}(X)$.  Thus if we allow $X$ to vary in such a way
that $\det([M])$ varies, then the resulting moduli problem {\em will} be
visibly obstructed, but the obstruction will necessarily be in the
direction of the deformation of surfaces.

\begin{lem}
  Let $X/k$ be a noncommutative surface over an algebraically closed field,
  and let $[M]\in K_0^{\num}(X)$ be a nonzero class such that $[M]|_Q$ is
  numerically trivial.  Then there is a dvr $R$ with residue field $k$ (and
  field of fractions $K$) and an extension $X'/R$ (as a split family) of
  $X_l$ to $R$ such that $\det([M]_K|_{Q_K})\notin \Pic^0(C_K)$.
\end{lem}

\begin{proof}
  Of course, if $\det([M]|_Q)\notin \Pic^0(C)$, then we may simply take the
  trivial extension.  Otherwise, we may consider the (smooth!) locally
  closed substack of the moduli stack of surfaces determined by the
  combinatorial type of $Q$ (and, in the rational case, whether $\Pic^0(Q)$
  is elliptic, multiplicative, or additive).  It suffices to show that
  $\det([M]|_Q)\in \Pic^0(C)$ cuts out a nontrivial divisor in this
  substack, since then any point of the divisor can be extended to a dvr
  not contained in the divisor.  In particular, we may restrict our
  attention to the corresponding stack in characteristic 0.
  
  If $c_1([M])\cdot e_m\ne 0$, then $e_m$ cannot be a component of $Q$, so
  that $x_m$ is a smooth point of the anticanonical curve of $X_{m-1}$, and
  varying $x_m$ varies the determinant.  If $c_1([M])\cdot e_m=0$, then we
  may reduce to the blowdown, and thus reduce to ruled surfaces and planes.
  If $c_1([M])\ne 0$, then $X$ must be a ruled surface, and then $c_1([M])$
  is fixed by the condition $c_1([M])\cdot Q=0$.  We must also have
  $\hat{Q}=\overline{Q}$, and by varying the invertible sheaf (again
  possibly including a lift to characteristic 0) on $\bar{Q}$, we may again
  vary the determinant.  Finally, if $c_1([M])=0$, then the only way
  $\det([M]|_Q)$ can be 0 is if $X$ has center of order dividing
  $\chi([M])\ne 0$, which is cut out by a nontrivial codimension 1
  condition on $q$.
\end{proof}

One useful observation is that $M|^{\bf L}_Q=0$ is an open condition on any
Noetherian substack of $\Spl_{X/k}$ (the complement of the supports of the
cohomology sheaves, only finitely many of which can be nonzero on the
substack).

\begin{prop}
  Let $X/k$ be a noncommutative surface over an algebraically closed field.
  The fiber of $\Spl_{X/k}$ over $0$ is smooth, except possibly on those
  components with trivial class in the Grothendieck group.
\end{prop}

\begin{proof}
  Choose $M\in \Spl_{X/k}$ with $[M]\ne 0$ and $M|^{\bf L}_Q=0$, let $X'/R$
  be the extension of $X$ given by the Lemma, and consider the localization
  ${\cal S}$ of $\Spl_{X'/R}$ at $M$.  Then by standard deformation theory,
  one has
  \[
  \dim({\cal S})\ge 1+\dim\Ext^1(M,M)-\dim\Ext^2(M,M)=\dim\Ext^1(M,M).
  \]
  Since $\det([M]|_Q)$ is nontrivial over $K$ and $M|^{\bf L}_Q=0$ on
  ${\cal S}$, we conclude that ${\cal S}(K)=0$, and thus that ${\cal S}$ is
  annihilated by some power of the maximal ideal of $R$.  It follows that
  $\dim({\cal S})=\dim({\cal S}_k)$, and thus $\dim({\cal S}_k)\ge
  \dim\Ext^1(M,M)$.  Since $\Ext^1(M,M)$ is the tangent space at $M$,
  ${\cal S}_k$ is smooth at $M$ as required.
\end{proof}

\begin{rem}
  It is not clear if it is even possible for a simple object to have
  trivial class in the Grothendieck group (let alone also have trivial
  restriction to $Q$).  Note in particular that such an object satisfies
  $\chi(M,M)=\chi(0,0)=0$, and thus any such connected component of
  $\Spl_{X/k}$ is either a smooth surface or a nowhere reduced curve.  Note
  that to show smoothness in finite characteristic, it would suffice to
  show smoothness in characteristic $0$: if $M$ lifts to characteristic
  $0$, then semicontinuity of fiber dimension forces its component to be
  smooth, and if it does not lift, the deformation theory argument shows
  that its component is smooth!
\end{rem}

\medskip

For simple {\em sheaves}, the proof of the Jacobi identity in
\cite{generic} can be carried out for more general {\em rational}
noncommutative surfaces (in fact, a density argument lets one deduce the
general rational case from the elliptic case), but the argument depends
quite strongly on the existence of the $\sO_X^\perp$ semiorthogonal
decomposition, and thus fails for higher genus ruled surfaces.  There is a
general approach (via shifted Poisson structures, see below) for arbitrary
simple objects that should work for general surfaces, but in the case of
sheaves, we can hope for a more elementary approach.  (The resulting
argument is actually somewhat simpler even in the rational case!)

As long as we restrict ourselves to sheaves, we can apply the reduction of
\cite{HurtubiseJC/MarkmanE:2002b}, which as observed in \cite{generic}
applies in the noncommutative setting.  This has the effect of replacing
$M$ by the cokernel of a map $\Hom(L,M)\otimes L\to M$ for a suitable line
bundle $L$, and gives a locally closed embedding of a neighborhood of $M$
in $\Spl_{X/S}$ in a neighborhood of the kernel, respecting the pairings on
the cotangent bundles (up to a sign).  Now, since $\ad$ has homological
dimension 2, and $\coh X$ is generated by $\ad$-acyclic (i.e., reflexive)
sheaves (namely the line bundles), performing this reduction twice gives a
locally closed bivector-preserving embedding of the neighborhood of $M$ in
the moduli space $\Refl_{X/S}$ classifying reflexive sheaves.  (Note that
being reflexive is an open condition, as a family fails to be reflexive
precisely on the union of the supports of $R^1\ad M$ and $R^2\ad M$.)  Thus
if we can show that the bivector on $\Refl_{X/S}$ is Poisson, then it will
follow for the moduli space of simple sheaves.  The Jacobi identity is a
closed condition (it corresponds to the vanishing of a certain trivector),
and thus it suffices (at least when $S$ is reduced) to prove this for a
dense subset of the universal moduli stack $\Refl$ (classifying pairs
$(X,M)$ with $M$ reflexive).

\begin{lem}
  The points of $\Refl$ over finite fields are dense, and this remains true
  if we exclude finitely many characteristics.
\end{lem}

\begin{proof}
  We first note that the characteristic $0$ points of $\Refl$ are dense;
  given any point $(X,M)$ over a field of characteristic $p$, we can lift a
  finite separable base change of $X$ to a surface over a mixed
  characteristic dvr, and then by smoothness can lift $M$ over an \'etale
  extension of the dvr.  Since being reflexive is an open condition, the
  lifted sheaf will still be reflexive, and thus gives a characteristic 0
  point having $(X,M)$ in its closure.

  Now, suppose $(X,M)$ is defined over a field of characteristic $0$.
  Since the moduli stack of noncommutative surfaces is locally of finite
  type, we may write $X$ as the base change of a surface over a field $l$
  of finite transcendence degree over $\Q$, and then similarly for $M$.  If
  $l$ is not a number field, then we may choose a subfield $k\subset l$
  which is algebraically closed in $l$ and such that $l/k$ has
  transcendence degree 1.  Then we may view $l$ as the function field of a
  smooth projective curve $C_l/k$, and find that $X$ extends to a surface
  over an open subset of $C_l$, and any extension of $M$ will be reflexive
  over an open subset of that open subset.  Since the closed points of that
  open subset are dense in $C_l$, we see that $(X,M)$ is in the closure of
  a set of points defined over fields of lower transcendence degree.  We
  thus reduce to the case that $l$ is a number field, where the same
  argument applies with an arbitrary open subset of $\Spec(O_l)$ in place
  of $C_l$, allowing us to exclude finitely many characteristics.
\end{proof}

Since noncommutative surfaces over finite fields are maximal orders, we
conclude the following.

\begin{cor}
  The points of $\Refl$ such that the surface is a maximal order are dense,
  and this remains true if we exclude finitely many characteristics.
\end{cor}

\begin{rem}
  This fails in general if we ask for the maximal order to be
  characteristic 0, for the simple reason that a surface with $g\ge 2$ is a
  maximal order iff it has finite characteristic.  In particular, even if
  one only wishes to know that the characteristic 0 moduli spaces are
  Poisson, the argument below will require working in finite
  characteristic.  (In contrast, for $g=0$ or $g=1$, one can show that
  characteristic 0 elliptic difference cases with $q$ torsion are dense.)
\end{rem}

\begin{rem}
  The same argument shows that the characteristic 0 localization of
  $\Spl_{X/S}$ is contained in the closure of the maximal orders with
  characteristic avoiding any finite set.  Showing that the characteristic
  0 points are dense is more difficult, since $\Spl_{X/S}$ is no longer
  smooth, though in most cases this could be avoided by lifting along the
  symplectic leaf (which requires a careful choice of the lift of $X$ when
  $M|_Q=0$ and may fail if $M|_Q=0$ and $[M]=0$).
\end{rem}

Thus to show that the sheaf locus in $\Spl_{X/S}$ ($S$ reduced) is Poisson,
it will suffice to prove that $\Refl_{X/k}$ is Poisson whenever $X$ is a
maximal order of characteristic prime to $6$.  (Here we avoid
characteristic 3 for convenience and characteristic 2 by necessity.)  Thus
suppose that our surface $X/k$ is $\Spec({\cal A})$ for a maximal order
${\cal A}$ on the smooth surface $Z$.

\begin{prop}
  A module $M\in \coh X$ is reflexive iff it is locally free as an ${\cal
    A}$-module.
\end{prop}

\begin{proof}
  For any point $z\in Z$, ${\cal A}_z={\cal A}\otimes \sO_{Z,z}$ has finite
  global dimension, and thus (applying $\ad$ to a sufficiently long
  resolution of $\ad M$ in powers of line bundles pulled back from $Z$) a
  reflexive sheaf is locally projective.  Since ${\cal A}_z$ is semilocal,
  $M$ is free of rank $n$ over ${\cal A}_z$ iff for any simple ${\cal
    A}_z$-module $N$, $\dim_{k(z)}\Hom(M,N)=\dim_{k(z)}\Hom(\sO_X^n,N)$.
  Now, $N$ is a $0$-dimensional sheaf defined over $k(z)$, and thus for
  $p>0$,
  \[
  \Ext^p(M,N)=\Ext^{p+2}(R^2\ad N,\ad M)=0,
  \]
  so that
  \[
  \dim_{k(z)}\Hom(M,N)=\chi(M,N) =
  \rank(M)\chi(N)=\chi(\sO_X^n,N)=\dim_{k(z)}\Hom(\sO_X^n,N),
  \]
  and thus $M$ is free (of rank $\rank(M)$) over ${\cal A}_z$.

  Since a sheaf is reflexive iff it is locally reflexive (localize at a
  prime associated to the support of $R^2\ad M$ or $R^1\ad M$) and
  $\sO_X^n$ is reflexive, it follows that a locally free sheaf is reflexive.
\end{proof}

\begin{rem}
  Note that ``locally free of rank $1$'' is not a synonym for ``line
  bundle'' as we have defined the latter.  Indeed, for $M\in
  \Hilb^{n,+}(X)$, the map $M\to \ad\ad M$ is injective with
  $0$-dimensional cokernel of Euler characteristic at most $n$, so that if
  $n<\ord(q)$, then $M$ is reflexive iff its restriction to $Q$ is a line
  bundle.
\end{rem}

We may thus represent a reflexive sheaf $M$ of rank $n$ via a $1$-cocycle
in $Z^1_{\et}(\GL_n({\cal A}))$, with the symplectic leaf condition that
$M|_Q$ be constant corresponding to the requirement that the image of the
cocycle in $Z^1_{\et}(\GL_n(\sO_Q))$ be constant.  (To be precise, this
condition is that the image in $H^1_{\et}(\GL_n(\sO_Q))$ be constant, but
any gauge transformation in $\GL_n(\sO_Q)$ making the cocycle constant can
be lifted to $\GL_n({\cal A})$.)  We thus obtain a family of cocycles
taking values in the corresponding cosets of $\GL_n({\cal A},{\cal I}_Q)$,
where ${\cal I}_Q$ is the ideal sheaf of $Q$ and $\GL_n(A,I)$ denotes the
kernel of $\GL_n(A)\to \GL_n(A/I)$.

To proceed further, we need to understand the isomorphism ${\cal I}_Q\cong
\omega_X$.  As an ${\cal A}$-module, $\omega_X$ is represented by
$\sHom_{\sO_Z}({\cal A},\omega_Z)$, and the bimodule map
\[
  \alpha:{\cal I}_Q\cong \sHom_{\sO_Z}({\cal A},\omega_Z)
\]
may be expressed as a map $\tau:{\cal I}_Q\to \omega_Z$ such that
$\tau(ab)=\tau(ba)$ for $a$ a local section of ${\cal A}$ and $b$ a local
section of ${\cal I}_Q$, and such that the induced bilinear form
$(a,b)\mapsto \tau(ab)$ is a perfect pairing.  Since the generic fiber is a
central simple algebra, the condition $\tau(ab)=\tau(ba)$ implies that
$\tau$ has the form $\tau(b) = \Tr(b) \omega$ where $\Tr$ is the reduced
trace and $\omega$ is some meromorphic 2-form on $Z$.  Now, since ${\cal
  I}_Q^r = {\cal A}\otimes\sO_Z(-Q')$, any element of ${\cal I}_Q$ is
nilpotent on $\sO_{Q'}$, and thus $\Tr{\cal I}_Q\subset \sO_Z(-Q')$.  But
multiplication by $\omega$ must be an isomorphism from $\Tr{\cal I}_Q\to
\omega_Z$, and thus $\Tr{\cal I}_Q=\sO_Z(-Q')$ and $\tau(b) = \alpha'
\Tr(b)$ where $\alpha':\sO_Z(-Q')\cong \omega_Z$ is a Poisson structure on
$Z$ vanishing along $Q'$.

At this point, the proof of \cite[\S3.1]{poisson} extends mutatis mutandis
(the main change being to replace the trace on $\Mat_n(\sO_X)$ by the
reduced trace on $\Mat_n({\cal A})$) to show that the corresponding
$2$-form is closed, and thus we have finished proving the following.

\begin{thm}
  Let $X/S$ be a noncommutative rational or rationally ruled surface over a
  reduced locally Noetherian base.  Then the subspace of $\Spl_{X/S}$
  classifying simple sheaves has a natural Poisson structure, and if $S$ is
  a field, the fibers of derived restriction to $Q$ are a foliation by
  smooth symplectic leaves.
\end{thm}

\begin{rem}
  Since the moduli stacks of noncommutative rational surfaces and
  noncommutative rationally ruled surfaces of genus 1 are both reduced, the
  Theorem continues to hold in those cases even when the base scheme is
  nonreduced.
\end{rem}

\medskip

The ease with which the proof of the Jacobi identity extended suggests that
it should be possible to generalize things further.  With this in mind, let
$X/k$ be a smooth projective surface over a field of characteristic not 2
and let $G/X$ be a smooth group scheme.  (The construction actually works
over more general base schemes on which $2$ is invertible, as long as the
global sections of $Z(G)$ are flat over the base.)  Then the moduli stack
$\Tors_G$ classifying \'etale-locally trivial $G$-torsors on $X$ (i.e.,
splittings of the map $[X/G]\to X$) has an open substack corresponding to
``simple'' torsors, those with automorphism group $\Gamma(Z(G))$.  In
particular, the isotropy substack is flat and locally of finite
presentation, and thus this substack is a gerbe over an algebraic space
$\Spl_G$.

Now, suppose $H$ is a smooth normal subgroup of $G$ having the same
generic fiber, and let $\tau$ be a $G$-invariant perfect pairing
\[
\tau:{\mathfrak g}\otimes {\mathfrak h}\to\omega_X
\]
which is symmetric on ${\mathfrak h}$.  Then for any $G$-torsor $V$, $\tau$
induces a perfect pairing
\[
H^1({\mathfrak g}(V))\otimes H^1({\mathfrak h}(V))
\to
H^2(\omega_X),
\]
so that the cotangent sheaf at $V$ may be identified with $H^1({\mathfrak
  h}(V))$ and we obtain a bivector as the composition
\[
H^1({\mathfrak h}(V))\otimes H^1({\mathfrak h}(V))
\to
H^1({\mathfrak g}(V))\otimes H^1({\mathfrak h}(V))
\to
H^2(\omega_X).
\]
Moreover, an infinitesimal deformation is tangent to a symplectic leaf iff
it is in the kernel of the map $H^1({\mathfrak g}(V))\to H^1(({\mathfrak
  g}/{\mathfrak h})(V))$ and thus the symplectic leaves are the fibers
of the map $\Spl_G\to \Tors_{G/H}$ taking $V$ to $V/H$.

Then a mild variation on the argument of \cite[\S3.1]{poisson} shows that
the above bivector induces a symplectic structure on any smooth symplectic
leaf.  In particular, this implies that $\Spl_G$ is Poisson as long as (a)
it is reduced, and (b) the generic symplectic leaf is smooth.

Since the requisite changes to the argument are not completely trivial, we
spell them out.  Let $V/R$ be a family of simple $G$-torsors over a regular
affine $k$-scheme $\Spec(R)$, and suppose that $V$ lies in a symplectic
leaf, i.e., that $V/H$ is \'etale-locally on $R$ the pullback of a
$G/H$-torsor over $k$.  We may then just as well replace $R$ by the
relevant \'etale cover so that $V/H$ is isomorphic to a constant family.
We may thus represent $V$ by a cocycle
\[
g_{ij}\in \Gamma(U_{ij}\times R;G_R)
\]
relative to some \'etale covering $U_i$ of $X$, with the property that each
$g_{ij}H\in \Gamma(U_{ij};G/H)$.  (All we get a priori is that $g_{ij}H$
represents a constant class in $H^1_{\et}(X;G/H)$, but since $H$ is smooth
we can \'etale locally on $X_R$ lift the gauge transformations needed to make
the cocycle itself constant.)  The Kodaira-Spencer map of this family is
then represented by the cocycle
\[
\beta_{ij} = d_R g_{ij} g_{ij}^{-1} \in Z^1({\mathfrak h}(V)\otimes \Omega_R),
\]
which since $\Omega_R\cong T_R^*$ may be viewed as a map from $T_R$ to
cocycles.  Then the pairing associated to the bivector is defined on a pair
of vectors in $T_R$ as
\[
(t_1,t_2)\mapsto \tau(\beta_{ij}(t_1)\cup \beta_{ij}(t_2))
= \tau(\beta_{ij}(t_1),\Ad(g_{ij})\beta_{jk}(t_2))
\in Z^2({\mathfrak h}(V)).
\]
Since this pairing is alternating on cohomology, the cocycle
\[
\frac{1}{2} \tau(\beta_{ij}(t_1),\Ad(g_{ij})\beta_{jk}(t_2))
-
\frac{1}{2} \tau(\beta_{ij}(t_2),\Ad(g_{ij})\beta_{jk}(t_1))
\]
represents the same class.  (That the resulting cohomology class is
independent of the choices made in defining $\beta$ follows by its relation
to the well-defined bivector, but is also reasonably straightforward to
verify directly.)  We may write the corresponding $2$-form by viewing
$\beta$ as a cocycle in the tensor product algebra $U({\mathfrak h})\otimes
\wedge^* \Omega_R$, and observing that $\tau$ gives a well-defined function
on the elements of the universal enveloping algebra of degree $\le 2$.  We
thus obtain a cocycle
\[
\gamma_{ijk} = \frac{1}{2}\tau(\beta_{ij}\Ad(g_{ij})\beta_{jk})
\in Z^2(\omega_X)\otimes \Omega^2_R,
\]
and the corresponding class in $H^2(\omega_X)\otimes \Omega^2_R=\Omega^2_R$
is precisely the $2$-form we need to show is closed.

We first note that rewriting $\beta_{jk}$ using the cocycle condition
\[
\beta_{ij}-\beta_{ij}+\Ad(g_{ij})\beta_{jk} = 0
\]
gives
\[
\gamma_{ijk}=\frac{1}{2}\tau(\beta_{ij}\beta_{ik}-\beta_{ij}^2).
\]
We may also compute the exterior derivative of $\beta$ as
\[
d_R \beta_{ij} = -\beta_{ij}^2,
\]
which makes sense since the square of an element of ${\mathfrak h}\otimes
\Omega_R$ lies in ${\mathfrak h}\otimes \Omega^2_R$.  (This calculation may be
performed in any faithful representation of $G$, so reduces to the
analogous statement in $\GL_n$.)  We thus find that
\[
d_R \gamma_{ijk} = \frac{1}{2}\tau(d_R (\beta_{ij}\beta_{ik})) =
-\frac{1}{2}\tau(\beta_{ij}^2\beta_{ik})-\frac{1}{2}\tau(\beta_{ij}\beta_{ik}^2)
= -\frac{1}{6}\check{d}\tau(\beta_{ij}^3),
\]
where the \v{C}ech coboundary calculation reduces to checking
that if $\beta$, $\beta'$ are two elements of ${\mathfrak h}\otimes
\Omega^1_R$, then
\[
\tau(\beta^2\beta') = \tau(\beta\beta'\beta) = \tau(\beta'\beta^2).
\]

It follows that $d_R\gamma_{ijk}$ is the coboundary of a class in
$Z^1(\omega_X)\otimes \Omega^3_R$, and thus the 2-form represented by
$\gamma$ is indeed closed.

\begin{rem}
  Although we divided by $3$ above, the calculation can still be made to
  work in characteristic 3.  The key observation is that the alternating
  ternary form associated to $\frac{1}{3}\tau(\beta^3)$ may be expressed as
  $(t_1,t_2,t_3)\mapsto \tau(\beta(t_1),[\beta(t_2),\beta(t_3)])$.
\end{rem}

\medskip

This calculation does not quite suffice to show that the bivector is
Poisson, as to finish the argument requires that $\Spl_G$ be reduced and
the symplectic leaves be smooth.  (In general, all we can conclude from the
above calculation is that the trivector arising from the Jacobi identity
vanishes on the Zariski closure of the points which are smooth in their
symplectic leaf.)

Since $\Spl_G$ is reduced iff it is generically smooth, we reduce to
understanding smoothness at a point $V$.  The obstructions at $V$ to
$\Spl_G$ and the corresponding symplectic leaf lie in $H^2({\mathfrak
  g}(V))$ and $H^2({\mathfrak h}(V))$ respectively, so we need to
understand those spaces.  These are dual to $H^0({\mathfrak h}(V))$ and
$H^0({\mathfrak g}(V))$, and simplicity of $V$ implies that the natural map
$H^0(Z({\mathfrak g}))\to H^0({\mathfrak g}(V))$ is an isomorphism, thus so
is $H^0(Z({\mathfrak g})\cap {\mathfrak h})\to H^0({\mathfrak h}(V))$.  (In
particular, the obstruction space to $\Spl_G$ at $V$ is 0 iff
$\Gamma(Z(G)\cap H)$ is discrete, and for the symplectic leaf iff
$\Gamma(Z(G))$ is discrete, but of course the spaces can have nontrivial
obstructions yet still be smooth.)  Now, suppose $N$ is a smooth connected
normal subgroup of $G$ such that $\tau({\mathfrak n},\Gamma(Z({\mathfrak
  g})\cap {\mathfrak h}))=0$.  Then by pairing with $H^0({\mathfrak
  h}(V))=H^0(Z({\mathfrak g})\cap {\mathfrak h})$, we find that the natural
map $H^2({\mathfrak n}(V))\to H^2({\mathfrak g}(V))$ is 0, and thus
$H^2({\mathfrak g}(V))\cong H^2(({\mathfrak g}/{\mathfrak n})(V))$.  The
latter is precisely the obstruction space to $\Spl_{G/N}$ at $V/N$, and we
thus conclude that a deformation of $V$ extends in $\Spl_G$ iff the
corresponding deformation of $V/N$ extends in $\Spl_{G/N}$, so we may
reduce to understanding smoothness on the latter.  Similarly, if $N$ also
satisfies $\tau(\Gamma(Z({\mathfrak g})),{\mathfrak n}\cap {\mathfrak
  h})=0$, then $H^2({\mathfrak h}(V))\cong H^2(({\mathfrak h}/{\mathfrak
  h}\cap {\mathfrak n})(V))$ and thus a deformation of $V$ in its
symplectic leaf extends iff the corresponding deformation of $V/N$ has an
extension which is trivial as an extension of $V/HN$.  We should also note
that each condition holds for $N_1N_2$ if it holds for $N_1$ and $N_2$, and
thus we may, if we choose, take $N$ to be the unique maximal smooth
connected normal subgroup satisfying the conditions.

These conditions on $N$ are somewhat subtle in complete generality, but in
most cases it turns out that we can take $N$ to be the derived subgroup
$G'$ of $G$ (i.e., the smallest normal subgroup of $G$ with abelian
quotient).  We first note that since $N$ and $H$ are smooth, it suffices to
check the conditions over the generic point of $X$, which in particular
shows that the second condition is strictly stronger than the first.
Moreover, if $z\in Z({\mathfrak g}_{\overline{k(X)}}))$, $x\in {\mathfrak
  g}_{\overline{k(X)}}$, $g\in G(\overline{k(X)})$, then
\[
\tau((\Ad(g)-1)x,z) = \tau(x,(\Ad(g^{-1})-1)z) = 0,
\]
and thus
\[
\tau(\langle (\Ad(G)-1){\mathfrak g}\rangle,\Gamma(\Lie(Z(G))) = 0.
\]
In most cases (in characteristic 0, in particular), the Lie algebra of the
derived subgroup is precisely this span, so that we may take $N=G'$.  In
that case, we find that $G/N=G^{\text{ab}}$ is abelian.  So if the
algebraic group $\Spl_{G^{\text{ab}}}$ is reduced, then $\Spl_G$ is smooth,
and if the kernel of the homomorphism $\Spl_{G^{\text{ab}}}\to
\Spl_{G/HG'}$ is smooth, then the symplectic leaves of $\Spl_G$ are smooth.

We thus see that in characteristic 0 (or if $\Gamma(Z(G))$ is discrete),
the above construction always gives a smooth Poisson moduli space with
smooth symplectic leaves, while in most other cases this typically reduces
to a question about torsors over abelian group schemes.  (This is of course
directly analogous to what we found above for simple complexes.)

\medskip

In particular, if ${\cal A}$ is the maximal order corresponding to a
noncommutative rational or rationally ruled surface, then we may take
$G=\GL_n({\cal A})$, $H=\GL_n({\cal A},{\cal I}_Q)$ above, so that the
moduli space of simple reflexive ${\cal A}$-modules may be identified with
$\Spl_Q$, and the bivector coming from the group scheme construction agrees
with that coming from the interpretation as sheaves on a noncommutative
surface.  Modifying the groups slightly leads to other examples of Poisson
moduli spaces of torsors.  Indeed, $\tau$ induces a perfect pairing on
$\gl_n(\sO_Q)$ taking values in $(\omega_X\otimes k(Z))/\omega_X$, or, by
taking residues, in $\omega_Q$.  Thus if $G$, $H$ are two smooth subgroups
of $\GL_n(\sO_Q)$ such that ${\mathfrak g}={\mathfrak h}^{\perp_\tau}$,
then we may use their preimages in $\GL_n({\cal A})$ in the above
calculation.  In particular, we may take $G$ associated to the intersection
of a parabolic subgroup of $\GL_n({\cal A}\otimes k(X))$ with $\GL_n({\cal
  A})$, in which case ${\mathfrak g}^{\perp_\tau}$ is the Lie algebra of
the group coming from the corresponding unipotent radical.  In that case,
the smoothness conditions reduce to the conditions for the original case,
so we again obtain a Poisson moduli space with smooth symplectic leaves
given by the fibers over $\Tors_{G/H}$.  This corresponds to reflexive
sheaves on $X$ with a filtration of $M|_Q$ by locally free sheaves of the
appropriate ranks, with the symplectic leaves given by the fibers over the
map to the associated graded of the filtration.  (Compare
\cite{BottacinF:2000}.)  This suggests in general that there should be a
well-behaved (and Poisson) moduli space of simple objects in $\perf(X)$
equipped with a filtration of $M|_Q$ by perfect objects.  (Note that this
is not quite fibered over $\Spl_{X/S}$, as the notion of ``simple''
includes only those automorphisms that preserve the filtration.  However,
there {\em is} a fibration of derived moduli stacks and the fibers depend
only on $M|_Q$, so this morally ought to reduce to a purely commutative
question.)

\medskip

As mentioned, the above argument only applies to sheaves, and thus new
ideas are likely to be required for more general objects.  Given that the
symplectic leaves of $\Spl_{X/S}$ are most naturally interpreted as fibers
of a map to a {\em derived} stack this suggests (to B. Pym, who then
suggested it to the author!) that one should look more closely at the
morphism ${\cal M}_X\to {\cal M}_Q$.  One can show (work in progress with
Pym based on a result of To\"en \cite{ToenB:2018}) that when $S$ is a field
of characteristic 0, this morphism is in fact a Lagrangian map to the
$1$-shifted symplectic derived stack ${\cal M}_Q$.  This induces a
$0$-shifted Poisson structure on ${\cal M}_X$, and further gives
$0$-shifted symplectic structures on the fibers of the map.  This in
particular implies that the symplectic leaves of the {\em derived} stack of
simple objects are indeed ($0$-shifted) symplectic, but the presence of
obstructions means that this does not trivially imply the corresponding
fact for the algebraic space.  This does, however, give an alternate proof
for the case of reflexive sheaves in characteristic 0, and thus, via the
usual reductions, for general sheaves.

One should also note that the notions of shifted symplectic and Poisson
structures have only been defined in characteristic 0, and the existence of
commutative Poisson surfaces in characteristic 2 with nonreduced symplectic
leaves implies that the main results being used in the construction cannot
even be true as stated in arbitrary characteristic.  That said, a full
proof over general Noetherian base schemes of characteristic 0 would imply
Poisson-ness on the Zariski closure of the characteristic 0 locus of the
universal moduli space $\Spl_{X/{\cal M}}$.

Even this is not necessarily the ``right'' proof of the Poisson structure;
in light of the discussion following Theorem \ref{thm:weird_langlands}, we
expect that the symplectic structure on the moduli space of 1-dimensional
sheaves disjoint from $Q$ is the semiclassical limit of a noncommutative
deformation of the moduli space (with parameter $\Pic^0(Q)/\Pic^0(C)$),
and that there should be derived equivalences between these deformations
extending the derived autoequivalences of the abelian fibration arising in
the fully commutative case.

\medskip

If $F:\coh X\to \coh X'$ is an equivalence of categories, then there is a
canonical natural isomorphism $F\theta \cong \theta F$, and thus we can
transport the anticanonical natural transformation $\theta\to \text{id}$
through $F$.  In particular, we obtain a notion of a {\em Poisson}
equivalence, in which the natural transformations $\theta F\to F$ and
$F\theta\to F$ agree modulo the canonical natural isomorphism.  We then
find that a Poisson equivalence of categories induces a Poisson isomorphism
between the moduli spaces of simple objects.  (Presumably something similar
applies for the derived moduli stacks.  Also, we should really be saying
``bivector-preserving''.)  When $X=X'$, we note that the anticanonical
natural transformation is uniquely determined by the restriction-to-$Q$
functor and a linear functional on $H^1(\sO_Q)/H^1(\sO_C)$ (the image of
the map $H^1(\sO_Q)\to H^2(\omega_X)$), and thus an autoequivalence of $X$
is Poisson iff the induced automorphism of $Q$ acts trivially on
$H^1(\sO_Q)/H^1(\sO_C)$.  Note that there will typically be no nontrivial
automorphisms of $Q$ fixing $q$, and thus the map being Poisson will be
automatic.

Now, given a finite group $G$, we may consider an action of $G$ on $\coh X$
in the following weak sense: to each $g\in G$, we associate a Poisson
autoequivalence $F_g$ of $\coh X$, in such a way that $F_e=\text{id}$ and
$F_gF_h$ is isomorphic to $F_{gh}$.  Then $G$ acts on $\Spl_X$ preserving
the Poisson structure, and the fixed subspace of this action will inherit a
Poisson structure.  (Note that without additional consistency conditions on
the natural isomorphisms, this does {\em not} give an action of $G$ on the
derived moduli stack or its truncation.)

In the commutative case, such an action of $G$ on $\coh X$ is determined by
a combination of its action on point sheaves and a class in
$H^1(G;\Pic(X))$, determining its action on the structure sheaf.  Then a
simple sheaf $M$ corresponds to a $G$-fixed point of $\Spl_X$ iff for each
$g\in G$, there is an isomorphism $M\cong F_gM$.  If we fix a system of
such isomorphisms (which should be the identity for $g=e$), then we have an
induced isomorphism $F_{gh}M\cong M\cong F_h M\cong F_gF_hM$ for each $g$,
$h$, and those isomorphisms will satisfy the obvious consistency
conditions.  In particular, since we can write $F_g M\cong
(g^{-1})^*M\otimes {\cal L}_g$ with ${\cal L}_{gh}\cong {\cal L}_g\otimes
(g^{-1})^*{\cal L}_h$, we find that there are induced choices for the
latter morphisms making them consistent.  In other words, any $G$-fixed
point of $\Spl_X$ promotes the class in $H^1(G;\Pic(X))$ to an equivariant
gerbe (with trivial underlying gerbe), and this makes the corresponding
sheaf $M$ a twisted $G$-equivariant sheaf.  (This choice is not unique, as
we can twist the isomorphisms $M\cong F_g M$ by any class in $H^1(G;k^*)$.)
When the twisting is trivial, these are sheaves on the orbifold quotient
$[X/G]$, which is itself (in characteristic prime to $|G|$) derived
equivalent to a commutative projective surface via the derived McKay
correspondence \cite{BridgelandT/KingA/ReidM:2001}.  (More precisely, it is
derived equivalent to the minimal desingularization of the scheme-theoretic
quotient $X/G$.)

Something similar holds in the noncommutative setting: the only natural
automorphisms of the identity functor are scalars, and thus the obstruction
to making the natural isomorphisms $F_gF_h\cong F_{gh}$ compatible is a
class in $H^3(G;k^*)$, which vanishes as long as there is {\em any} fixed
point in $\Spl_X$.  If that class vanishes, then the different compatible
choices form a torsor over $H^2(G;k^*)$, and each such choice gives a
disjoint (possibly empty) subset of $(\Spl_X)^G$.  Each fixed point then
corresponds to a collection of (twisted) $G$-equivariant sheaves, which
themselves form a torsor over $H^1(G;k^*)$.  We thus see that these fixed
subspaces are closely related to the moduli space of simple (twisted)
$G$-equivariant sheaves, though in addition to the $H^1(G;k^*)$ action, we
must also bear in mind that a simple equivariant sheaf need not have simple
underlying sheaf, and thus the fixed subspaces are at best quotients of
open subsets of the ``correct'' moduli spaces.  (This issue goes away if we
consider the stack version, where we must fix a consistent family of
natural isomorphisms, but should then get the full moduli stack of
$G$-equivariant objects, suitably defined.)  Presumably there is also an
analogue of the derived McKay correspondence in this case as well.

We can also obtain Poisson automorphisms of $\Spl_X$ associated to {\em
  contravariant} equivalences.  To construct such automorphisms, we first
note that for any $d\in \Z$, we have a contravariant derived equivalence
$M\mapsto (R\ad M)[d]$.  This in general changes the surface, by inverting
$q$, but we can sometimes fix this by composing with an abelian equivalence
$\coh \ad X\cong \coh X$, which as above will tend to give a Poisson
autoequivalence.

\bigskip

We finally turn to the more classical types of moduli spaces, classifying
{\em stable} (or semistable) sheaves.  The usual argument of Langton
\cite{LangtonSG:1975} shows that the stable moduli space (an open subspace
of $\Spl_{X/S}$) is separated and the semistable moduli space is proper,
but of course one expects in general that they should be
(quasi-)projective.  The standard construction involves inequalities proved
by induction involving general hyperplane sections, and thus is difficult
to extend in general.  (Though of course for rank 1 sheaves, we have
already shown this above!)  In the elliptic rational case, this issue was
finessed for $1$-dimensional sheaves, and it turns out that the basic ideas
carry over to general ruled surfaces.  (The question of projectivity of the
moduli space of semistable sheaves of rank $>1$ remains open.)

Looking at the argument used in \cite{generic}, we see that the key
results are Corollary 11.64 and Proposition 11.65 op.~cit.  The first result
in turn reduces to Lemma 11.42 op.~cit., which holds on general rational or
rationally ruled surfaces in the following form, with the same proof.

\begin{lem}\cite[Lem.~11.42]{generic}
  Let $X/k$ be a noncommutative rational or rationally ruled surface of
  genus $g>1$, and let $M\in \coh X$ be globally generated.  Then for any
  line bundle $L$ such that $c_1(M)+c_1(L)-Q$ is ineffective and
  $\Ext^p(L,\sO_X)=0$ for $p>0$, we have $\Ext^p(L,M)=0$ for $p>0$.
\end{lem}

Of course, the condition that $\Ext^p(L,\sO_X)=0$ holds for rational
surfaces if $-c_1(L)$ is nef and $-c_1(L)\cdot Q>0$, and for rationally ruled
surfaces of higher genus if $-c_1(L)-(2g-1)f$ is nef.

We then find the following by the same argument as in the elliptic case.
Call an ample divisor class $D_a$ ``strongly ample'' if $D_a-2gf$ is nef
and, if $g=0$, $D_a\cdot Q\ge 2$.  Clearly, every ample divisor class has a
positive multiple which is strongly ample, and replacing it by the latter
leaves the associated stability condition unchanged.

\begin{cor}\cite[Cor.~11.64]{generic}
  Let $D_a$ be a strongly ample divisor class.  If $M$ is a semistable
  1-dimensional sheaf on $X$ with $c_1(M)\cdot D_a=r>0$ and $\chi(M)>r^2(r+1)$,
  then $M$ is acyclic and globally generated.
\end{cor}

Similarly, for the other result we need, the argument carries over and
reduces to the following fact.

\begin{lem}\cite[Lem.~11.41]{generic}
  If $M$ is a pure $1$-dimensional sheaf of Chern class $df$ on a
  noncommutative quasi-ruled surface, then there is a filtration $M_i$ of
  $M$ such that each subquotient $M_{i+1}/M_i$ is a pure 1-dimensional
  sheaf of Chern class $f$ and such that
  \[
  \chi(M_1)\ge \chi(M_2/M_1)\ge\cdots\ge\chi(M/M_{d-1}).
  \]
\end{lem}

\begin{proof}
  The one part of the proof op.~cit.~that needs to be established is that a
  globally generated sheaf of rank 0 and Chern class $df$ has Euler
  characteristic at least $d$.  If $D$ is a divisor class such that both
  $D$ and $df-D$ are effective, then (since $f$ is nef and $f^2=0$) $D\cdot
  f=0$, and thus $D\propto f$, so that we may reduce to the case of a
  quotient of $\sO_X$, where it is immediate from Corollary
  \ref{cor:Ox_quotient_bound}.
\end{proof}

\begin{cor}\cite[Prop.~11.65]{generic}
  Let $D$ be a strongly ample divisor class on the noncommutative rational
  or rationally ruled surface $X$.  Then for any pure $1$-dimensional sheaf
  $M$ and any line bundle $L$ of class $-D$, there is a homomorphism
  $\phi:L\to \sO_X$ such that $\Hom(\coker\phi,M)=0$.
\end{cor}

\begin{proof}
  As in \cite{generic}, we may reduce to the case that $M$ is irreducible
  and disjoint from $Q$ (in particular ruling out noncommutative planes and
  odd Hirzebruch surfaces).  Moreover, by considering those $\phi$ that
  vanish to suitable order on $Q$, we may arrange to have $D\cdot e_m=0$
  (in some blowdown structure), and then reduce to the corresponding
  question on the blowdown.  (One caveat is that for ruled surfaces of
  positive genus, the conclusion does not actually hold for $D=0$ (take $L$
  to be the pullback of an ineffective degree 0 line bundle on $C_0$), but
  luckily this case is only needed for the $g=0$, $m=7$, $D=2Q$ case of the
  reduction.)  This reduces to ruled surfaces, where we note that the
  existence of a sheaf disjoint from $Q$ implies the parity is even and
  either $g=0$ and $Q$ is integral of class $2s+2f$, or $g\ge 0$ and $Q$ is
  the sum of two (possibly identical) curves of class $s+(1-g)f$.  Thus by
  considering $\phi$ that vanishes on an integral component of $Q$, we may
  similarly reduce to the case that either $D\cdot f=0$ or, when $g=0$ and
  $Q$ is integral, $D\cdot f=1$.  Since $M$ has Chern class a positive
  multiple of $s+(g-1)f$ and the kernel of a nonzero morphism
  $\coker\phi\to M$ has Chern class $D-c_1(M)$, we can rule out the $D\cdot
  f=0$ case entirely, while the remaining case may be dealt with as in the
  case that $Q$ is elliptic: the support of $\coker\phi|_Q$ is the same as
  that of $\ker\phi|_Q$, which is easy to control since $\ker\phi$ is an
  extension of sheaves of class $f$ and $Q$ is integral.
\end{proof}

These results imply an analogue of the Le Potier-Simpson bound and from that a
proof of projectivity.

\begin{lem}\cite[Lem.~11.66]{generic}
  Let $D_a$ be a strongly ample divisor class.  If $M$ is a semistable
  $1$-dimensional sheaf with $c_1(M)\cdot D_a=r$, then $ h^0(M)\le
  \max((r+1)^3+\chi(M),0)$.
\end{lem}

\begin{thm}\cite[Thm.~1.67]{generic}
  Let $D$ be an effective divisor class and let $D_a$ be an ample divisor
  class.  Then the moduli problem of classifying $S$-equivalence classes of
  $D_a$-semistable sheaves $M$ with $\rank(M)=0$, $c_1(M)=D$ and
  $\chi(M)=l$ is represented by a projective scheme $M_{D,l}$.  The stable
  locus of $M_{D,l}$ is a Poisson subscheme, which is smooth if $D\cdot
  Q>0$, as is the Zariski closure of the sublocus with $M|_Q=0$ if $D\cdot
  Q=0$.  If $\Z[M]\subset K_0^{\num}(X)$ is a saturated sublattice, then
  the stable locus admits a universal family.
\end{thm}

\medskip

As in \cite{generic}, there are a number of cases in which we can identify
the (semi)stable moduli space as a particular projective rational surface.
The proofs in \cite{generic} used the corresponding special case of the
derived equivalences coming from Theorem \ref{thm:weird_langlands}, but the
arguments are difficult to translate directly, as they depended strongly on
the fact that when $Q$ is smooth, any simple object in $D^b_{\coh}(Q)$ is a
shift of a stable sheaf.

Luckily, we can mostly replace this by the following fact.

\begin{lem}
  Let $X/k$ be a noncommutative rational surface over an algebraically
  closed field, with anticanonical curve $Q$ such that $Q^2=0$ and no
  component of $Q$ has self-intersection $<-2$, and let $r,l$ be relatively
  prime integers with $r>0$.  Also let $D_a$ be an ample divisor class such
  that any $D_a$-semistable 1-dimensional sheaf with $c_1(M)=rQ$,
  $\chi(M)=l$ is $D_a$-stable.  Then the moduli space of such sheaves
  supported on $Q$ is isomorphic to $Q$, and the Fourier-Mukai functor
  associated to the universal family is a derived autoequivalence of $Q$.
\end{lem}

\begin{proof}
  If $Q$ is smooth, this follows as in \cite{generic}.  Otherwise, we
  choose a blowdown structure such that $D_a$ is in the fundamental
  chamber.  We then claim that we can embed $Q$ as a fiber in a commutative
  rational elliptic surface $X'/k$ (with generically smooth fibers and a
  section), in such a way that there is an ample divisor on $X'$ inducing
  the same stability condition on $Q$.

  Embedding $Q$ as a fiber of some $X'$ is straightforward; indeed, in
  characteristic not 2 or 3, we may simply obtain $X'$ from $X$ by passing
  to the corresponding commutative surface, blowing down a $-1$-curve, then
  blowing up the base point of the anticanonical linear system.  In
  characteristic 2 or 3, we observe that the quasi-elliptic fibrations form
  a closed substack of the moduli stack of rational elliptic surfaces, and
  thus to show that the substack having $Q$ as a fiber contains elliptic
  fibrations, it suffices to show this for the closure of the substack
  (letting us degenerate $Q$).  The maximally degenerate Kodaira types on a
  rational elliptic surface are $I_9$, $I_4^*$, and $I\!I^*$; the first
  cannot appear on a quasi-elliptic fibration, while the latter two can
  only appear in characteristic 2 or 3 respectively.  It thus suffices to
  observe that the elliptic curve
  \[
  E/\F_2(t):y^2+xy=x^3+tx^2+x
  \]
  has rational minimal proper regular model with a fiber of type $I_4^*$, while
  the elliptic curve
  \[
  E/\F_3(t):y^2=x^3+x+t
  \]
  has rational minimal proper regular model with a fiber of type $I\!I^*$.

  Having chosen $X'$, let $D'$ be a divisor on $X'$ that has the same
  restriction to $Q$ as a positive integer multiple of $D_a$.  (In the
  direct construction of $X'$, we may of course take $D'=D_a$, but note
  that this need not be ample on $X'$.)  If we apply the algorithm to put
  $D'$ in the fundamental chamber, then any time we would wish to reflect
  in an effective root, that root cannot be a component of $Q$, and thus
  performing that reflection will have no effect on the induced stability
  condition.  We may thus assume that $D'$ has nonnegative intersection
  with every simple root, and then by adding a suitable multiple of $Q$,
  that it has positive intersection with $e_8$.  If it is not ample, then
  this is because it has intersection $0$ with some component of a fiber
  other than $Q$; if we multiply $D'$ by 3 and then add all components of
  other fibers that it {\em does} intersect, this again has no effect on
  the stability condition, but strictly decreases the set of bad
  components.  We thus need do this only finitely many times before
  eventually getting an ample divisor as desired, which we can then perturb
  as needed to eliminate any remaining strictly semistable sheaves of the
  desired invariants.
  
  Now, the surface $X'$ has a pencil of generically smooth anticanonical
  curves, so we may choose a smooth anticanonical curve $Q'$ in $X'$.  Thus
  by \cite[Thm.~11.70]{generic}, the moduli space of $D'$-semistable
  sheaves on $X'$ with the specified invariants is isomorphic to $X'$, and
  the isomorphism respects the elliptic fibration.  It follows therefore
  that the subscheme corresponding to sheaves supported on $Q$ is
  isomorphic to $Q$ as required.  Moreover, the universal family on $X'$
  induces a derived autoequivalence by the proof op.~cit., and this
  restricts to a derived autoequivalence of $Q$.
\end{proof}

We recall the notation used above in Theorem \ref{thm:weird_langlands}:
we fix a (possibly degenerate) commutative del Pezzo surface $Y$ with
anticanonical curve $Q$, and let $X_{z,q}$ be the 2-parameter family of
noncommutative surfaces obtained from $Y$ by blowing up a point and taking
the noncommutative deformation, with parameters $z\cong \det(\sO_Q|_Q)$ and
$q\cong \det([\pt]|_Q)$.

\begin{prop}\label{prop:irreds_on_fibers_exist}
  Let $r,d$ be relatively prime integers with $r>0$, and let $z$, $q$ be
  such that $z^r q^d$ is torsion, of exact order $l$.  Then the moduli
  space of stable 1-dimensional sheaves on $X_{z,q}$ with Chern class $lrQ$
  and Euler characteristic $ld$ is nonempty and contains irreducible
  sheaves.
\end{prop}

\begin{proof}
  Let $R$ be a dvr with residue field $k$ and field of fractions $K$ with
  $\ch(K)=0$, and let $Y^+/R$ be a lifting of $Y$ to $R$ as an
  anticanonical del Pezzo surface in such a way that the generic fiber of
  $Q^+$ is smooth.  Then (up to an \'etale extension of $R$), we may lift
  $q$ and $z$ as well (inside the smooth scheme where the torsion condition
  is satisfied), so as to obtain a lift $X^+_{z,q}$.  (Moreover, the ample
  line bundle on the special fiber lifts to an ample line bundle.)  The
  semistable moduli space of the generic fiber of the lift is nonempty by
  \cite[Prop.~12.8]{generic}, and the component containing stable sheaves
  disjoint from $Q$ is a smooth surface.  The limit in the special fiber of
  any sheaf in that component satisfies $\theta M\cong M$, and thus
  dimension considerations (and properness) imply that the semistable
  moduli space of the special fiber also contains sheaves disjoint from
  $Q$.  That the generic such sheaf is irreducible (thus stable) follows as
  in the proof op.~cit.~for the generic fiber.
\end{proof}

\begin{rem}
  If $l$ is a nontrivial multiple of the order of $q$, then the generic
  semistable sheaf on $X^+$ is a direct sum, so is not simple, and thus
  semicontinuity shows that this is inherited by all semistable sheaves,
  including on the special fiber, so that there are no stable sheaves in
  this case.
\end{rem}

\begin{rem}
  In \cite{generic}, it was also shown that the obstruction to a existence
  of a sheaf over the generic point of the stable moduli space is given by
  the generic fiber of an Azumaya algebra derived equivalent to $X_{z,q}$
  via Theorem \ref{thm:weird_langlands}.  This would require a better
  understanding of these derived equivalences, and in particular how they
  act on points of the identity component of $Q$.
\end{rem}

\begin{thm}\label{thm:painleve_moduli_spaces}
  With $Y$, $Q$ fixed, let $w\in \Pic^0(Q)$.  Let $r,d$ be relatively prime
  integers with $r>0$, and let $D_a$ be an ample divisor class such that
  any semistable $1$-dimensional sheaf on $X_{w^{-d},w^r}$ with
  $c_1(M)=rQ$, $\chi(M)=d$ is stable.  Then the corresponding semistable
  moduli space is isomorphic to $X_{w,1}$.
\end{thm}

\begin{proof}
  It follows as in the proof of \cite[Thm.~11.70]{generic} for the $Q$
  smooth case that any component $M$ of the stable moduli space
  containing sheaves disjoint from $Q$ (which exist by the Proposition) is
  derived equivalent to $X_{w^{-d},w^r}$.  Composing with the derived
  equivalence coming from Theorem \ref{thm:weird_langlands} gives a derived
  equivalence $D^b_{\coh} M\cong D^b_{\coh} X_{w,1}$ preserving rank
  and Euler characteristic.  If $w$ is not torsion, this already induces
  the desired isomorphism, by \cite[Thm.~3.2]{KawamataY:2002}.  Otherwise,
  $X_{w,1}$ is a genus 1 fibration (with one fiber of the form
  $\ord(w)Q$), and the derived equivalence respects this fibration.  It
  follows that it takes smooth points of integral fibers to shifts of
  smooth points of integral fibers, and thus the support of the
  corresponding Fourier-Mukai kernel is $2$-dimensional, so that the proof
  of \cite{KawamataY:2002} still applies.

  Since the subscheme of the semistable moduli space consisting of sheaves
  supported on $Q$ is geometrically isomorphic to $Q$, so geometrically
  connected, it follows that there is at most one component of the
  semistable moduli space containing such sheaves, and thus $M$ is
  the only component of the semistable moduli space, as required.
\end{proof}

\bigskip

One thing which is not quite settled by the above discussion is when the
moduli space of simple (or stable) sheaves (with fixed class in
$K_0^{\num}(X)$, say) is nonempty.  This is most interesting in the case of
$1$-dimensional sheaves, as these are the ones that correspond to
differential or difference equations, and in particular in the case of
$1$-dimensional sheaves disjoint from $Q$.  For such sheaves, we could also
ask a harder question, namely whether there are any {\em irreducible}
sheaves, i.e., such that any nonzero subsheaf has the same Chern class.
Such a sheaf corresponds to an equation which is irreducible in a quite
strong sense: there is no gauge transformation over $k(C)$ that makes the
equation become block triangular, or its singularities become simpler.
(The singularities becoming simpler corresponds to having a sub- or
quotient sheaf of Chern class $f-e_i-e_j$ or $e_i-e_j$.)

For simple sheaves in general, the above smoothness results give the
following reduction (analogous to the one we used above for reflexive sheaves).

\begin{prop}
  Let $X/R$ be a split noncommutative rationally ruled surface over the dvr
  $R$ with field of fractions $K$, and let $[M]\in K_0^{\num}(X)$ be a
  class such that $\rank([M])\ne 0$ or $c_1([M])\cdot Q\ne 0$.  If there is
  a separable extension $l/k$ and a simple sheaf $M_l$ of this class on
  $X_l$, then there is a flat family of sheaves on the base change of $X$
  to an \'etale extension of $R$ such that the special fiber is isomorphic
  to $M_l$ and the generic fiber is simple.
\end{prop}

\begin{proof}
  Since any separable extension $l/k$ lifts to an \'etale extension of $R$,
  we may assume $l=k$.  The conditions on $[M]$ ensure that the moduli
  space of simple sheaves is unobstructed, and thus the point corresponding
  to $M_k$ extends to an \'etale neighborhood, and is in turn \'etale
  locally represented by a sheaf.
\end{proof}

This has a variant in the disjoint-from-$Q$ case.

\begin{prop}
  Let $X/R$ be a split noncommutative rationally ruled surface over the dvr
  $R$ with field of fractions $K$, and let $[M]\in K_0^{\num}(X/K)$ be a
  1-dimensional class such that $\det([M]|^{\bf L}_Q)\in \Pic^0(C_K)$.  If
  there is a separable extension $l/k$ and a simple sheaf $M_l$ of this
  class on $X_l$ with $M_l|_Q=0$, then there is a flat family of sheaves on
  the base change of $X$ to an \'etale extension of $R$ such that the
  special fiber is isomorphic to $M_l$ and the generic fiber is simple and
  disjoint from $Q$.
\end{prop}

\begin{proof}
  Although the moduli space of sheaves is obstructed, the determinant
  condition ensures that the obstruction is ``orthogonal'' to $R$, and thus
  we still have the desired \'etale local extensions.
\end{proof}

\begin{prop}
  Let $X/R$ be a noncommutative rationally ruled surface over the dvr $R$,
  and let $M$ be an $R$-flat coherent sheaf such that $M_k$ is an
  irreducible $1$-dimensional sheaf.  Then so is $M_K$.
\end{prop}

\begin{proof}
  Suppose $M_K$ is not irreducible, and let $N_K$ be an irreducible pure
  $1$-dimensional quotient of $M_K$.  Then $N_K$ extends to $R$ (albeit not
  necessarily as a quotient of $M$), and we may further arrange for that
  extension $N$ to have pure $1$-dimensional fibers.  (Indeed, by
  properness, we may arrange for $N_k$ to be semistable!)  Then
  semicontinuity tells us that $\dim \Hom(M_k,N_k)\ge \dim\Hom(M_K,N_K)>0$,
  so that there is a nonzero homomorphism from $M_k$ to $N_k$.  But the
  kernel and image of this homomorphism have nonzero Chern class,
  contradicting irreducibility of $M_k$.
\end{proof}

In particular, if we start with a geometrically irreducible (thus simple)
sheaf on $k$ and use the previous Propositions to lift it to a family of
simple sheaves on an \'etale cover of $R$, the generic fiber of that family
will also be geometrically irreducible.

In general, if we are given a surface $X/K$ and a class in $K_0^{\num}(X)$
and want to know whether simple sheaves (or simple sheaves disjoint from
$Q$) of the given class exist (over $\bar{K}$, but then by smoothness over
$K^{\text{sep}}$), these results suggest the following strategy.  (This is
inspired by the strategy used in \cite[\S8]{ArtinM/TateJ/VandenBerghM:1990}
to prove noncommutative planes are Noetherian.)  First, since the moduli
stack of surfaces is locally of finite type over $\Z$, the point
corresponding to $X$ factors through a field of finite transcendence degree
over its prime field, and thus $X$ itself is the base change of a surface
defined over such a field.  We thus reduce to the case that $K$ itself has
finite transcendence degree.  Any irreducible Cartier divisor on the
closure of $X$ in the moduli stack induces a valuation on $K$ and an
extension of $X$ to an \'etale extension of the corresponding valuation
ring.  By the Propositions, if the desired simple sheaves exist on the
special fiber of this family, then they exist on some separable cover of
$K$, and thus it will suffice in general to settle the question in the case
that $K$ is finite!  (It will be helpful to make one small refinement: we
may choose the smooth neighborhood in the moduli stack in such a way that
every surface in the closure of $X$ has an anticanonical curve with the
same combinatorial structure, and in the rational case remains smooth or
nodal when appropriate.)

The advantage of having $K$ be finite is that since
$\Pic^0(Q_K)/\Pic^0(C_K)$ is finite, $q$ will necessarily have finite
order, and thus any surface over a finite field is a maximal order.  This
opens up the ability to use commutative techniques in constructing the
desired sheaves.  This is particularly powerful in the case of irreducible
$1$-dimensional sheaves, since we can then consider the support of the
sheaf.

We thus consider the following problem.  For any Chern class $D\in \NS(X)$
and Euler characteristic $\chi\in \Z$, we want to know if there are
irreducible sheaves on $X/K^{\text{sep}}$ with the given invariants (and,
if $D\cdot Q=0$, disjoint from $Q$).  The usual reductions clearly apply:
if $D$ has negative intersection with a component of $Q$, then it must be
that component to be irreducible, so we may as well assume that all such
intersections are nonnegative.  We may then attempt to reduce $D$ to the
fundamental chamber.  If this fails because $D\cdot f$ becomes negative or
because we attempt to reflect in an effective simple root, this shows that
a sheaf of class $D$ cannot be irreducible, unless $D$ is actually equal to
the offending simple root.  (In that case, we know that an irreducible
sheaf exists whenever $D$ and $\chi$ satisfy the appropriate determinant
condition, and the result will be disjoint from $Q$.)

We thus reduce to the case that $D$ is in the fundamental chamber, letting
us ignore the possibility that the specializations we apply to $K$ may
introduce additional $-2$-curves.  (We have already assumed that they do
not change the structure of $Q$.)  When $X$ is given by a maximal order
${\cal A}$ of degree $l$, we may identify the N\'eron-Severi groups of $X$
and its center $Z$, and find by Proposition \ref{prop:NS_of_center} that
the support of a sheaf of Chern class $D$ on $X$ is a divisor of class $D$
on $Z$.  Conversely, given an {\em integral} curve $C_D$ on the center
which has this class and is transverse to the anticanonical curve, the
fiber of $X$ over $K(C_D)$ will be a central simple algebra.  The Brauer
group of $K^{\text{sep}}(C_D)$ is trivial (since it has transcendence
degree 1 over a separably closed field), and thus there is some separable
extension $L/K$ such that ${\cal A}\otimes L(C_D)\cong \Mat_l(L(C_D))$.  We
may WLOG assume that $L=K$ and thus that ${\cal A}\otimes K(C_D)$ has a
simple module of dimension $l$ as a vector space over $K(C_D)$.  If we view
this as an $\sO_{C_D}$-module, we may certainly extend it to a torsion-free
coherent sheaf $M'$ on $C_D$, and can then obtain a coherent ${\cal
  A}|_{C_D}$-module with the same generic fiber as the image of ${\cal
  A}\otimes M'\to M'\otimes K(C_D)$.  This ${\cal A}|_{C_D}$-module is of
course also a coherent ${\cal A}$-module $M$.  Any nonzero subsheaf of $M$
has the same generic fiber over $C_D$, so has $0$-dimensional quotient, and
thus $M$ is an irreducible sheaf of the desired Chern class.  Directly
controlling the Euler characteristic of this sheaf is somewhat tricky, but
is unnecessary.  Indeed, if the support is transverse to $Q$ but not
disjoint, then we can add any integer to the Euler characteristic by
modifying the sheaf at such a point.  If the support is disjoint from $Q$,
we can still add any multiple of $l$ to the Euler characteristic, and this
is enough to give any Euler characteristic satisfying the requisite
determinant condition.

\medskip

In the rational case, we obtain the following for irreducible sheaves
disjoint from $Q$.  Note that since changing the Euler characteristic
multiplies the determinant of restriction by a power of $q$, we have a
natural restriction map $\NS(X)\to \Pic(Q)/\langle q\rangle$.  Let $Q_l$,
$0\le l\le m$ denote the pullback to $X$ of the class of the anticanonical
curve on $X_l$.

\begin{prop}
  Let $X/k$ be a noncommutative rational surface, and let $D\in \NS(X)$,
  $\chi\in \Z$ be such that (a) $D$ is nef and in the fundamental chamber,
  and (b) if $[M]\in K_0^{\num}(X)$ is the class of rank 0 sheaves with
  Chern class $D$ and Euler characteristic $\chi$, then $\det([M]|^{\bf
    L}_Q)\cong \sO_Q$.  Then there is an
  irreducible sheaf $M$, disjoint from $Q$, of class $[M]$ unless either
  (a) $D=rQ_8$ and $Q_8|_Q\in \Pic(Q)/\langle q\rangle$ has order $r'$
  strictly dividing $r$, or (b) $D=rQ_8+e_8-e_9$ and $Q_8|_Q\in \langle
  q\rangle$.
\end{prop}

\begin{proof}
  If $X$ is a maximal order, then the given conditions ensure that the
  corresponding linear system on the center contains irreducible divisors
  disjoint from the anticanonical curve (\cite[Thm. 4.8]{me:hitchin}), and
  thus per the discussion above ensures the existence of suitable
  irreducible sheaves on $X$.  If $D\notin \Z Q_8\cup (e_8-e_9)+\Z Q_8$,
  then this is already enough to give the desired existence result on
  general surfaces over fields, by induction along valuations.  To make
  this work in the remaining cases, we need merely observe that we can
  always choose a valuation on a separable extension of $K$ in such a way
  that the order of $Q_8|_Q$ in $\Pic(Q)/\langle q\rangle$ does not change,
  and thus when we reach a finite field will again have the requisite
  irreducible curves.
\end{proof}

\begin{rem}
  Analogous results of course apply to the case that $D\cdot Q>0$, since we
  again \cite{HarbourneB:1997,me:hitchin} know the possible ways such a
  divisor class on the center can fail to be generically irreducible.
\end{rem}

The situation in the two remaining cases is more subtle.  The first case
was dealt with in Proposition \ref{prop:irreds_on_fibers_exist} (and the
remark following), completely determining in which subcases irreducible
sheaves can exist.  In the second case, irreducible sheaves in fact never
exist.  We may assume $q$ non-torsion (since otherwise the support makes
sense and is reducible) and first note that by twisting by a suitable
multiple of $e_7$, we may ensure that there are irreducible sheaves
disjoint from $Q$ with Chern class $Q_8$ and Euler characteristic 0;
twisting by $e_9$ then leaves this condition alone and allows us to set
$\chi=1$, and ensures the existence of a sheaf $\sO_{e_8-e_9}$.  The
argument of \cite[Prop.~11.68]{generic} carries over to show that the
generic irreducible sheaf with the given invariant is acyclic, so has a
unique global section.  The image of that global section must have the same
Chern class, but then must have the same Euler characteristic by
disjointness from $Q$, and thus the sheaf is globally generated.  Let $I$
be the kernel of the global section, and note that $I$ is torsion-free of
rank $1$, Chern class $-rQ_8-e_8+e_9$, and Euler characteristic 0 (and in
fact $R\Gamma(I)=0$).  Moreover, since $\theta M\cong M$ is acyclic,
$\Hom(I,\sO_X)\cong \Hom(\sO_X,\sO_X)$.  To obtain a contradiction, it will
suffice to show that $\Hom(I,\sO_X(e_9-e_8))\ne 0$, since then the unique
map $I\to \sO_X$ factors through $\sO_X(e_9-e_8)$, implying that $M$ has a
nonzero morphism to $\sO_{e_8-e_9}$.

It is of course equivalent to ask for $\Hom(\theta \sO_X(e_8-e_9),\ad I)\ne
0$.  It will thus (since torsion-free sheaves of rank 1 are reflexive when
$q$ is torsion) suffice to show that if $I$ is a torsion-free sheaf of rank
1, Chern class $rQ_8+e_8-e_9$ and $R\Gamma(I)=0$, then
$\Hom(\sO_X(-Q_8+e_8-e_9),I)\ne 0$.  By semicontinuity, we may reduce to
the generic surface $X$ having sheaves disjoint from $Q$ with Chern class
$Q_8$ and Euler characteristic 0 (but with no sheaf of the form
$\sO_{e_8-e_9}(d)$).  Since both sheaves are in $\sO_X^\perp$, we may
transport the question to the corresponding commutative surface $X'$, and
since the generic surface has $Q$ smooth, we may apply
\cite[Prop.~11.53]{generic} to see that $I$ maps to a torsion-free sheaf on
$X'$.  Twisting reduces to showing that if $I'$ is the ideal sheaf of $r$
points in $X'$, then there is a morphism from $\sO_X(-rQ_8)$ to $I'$, and
this is a trivial consequence of the fact that $X'$ is elliptic.

\smallskip

It should of course be possible to give similar results for higher genus
noncommutative ruled surfaces, with the main obstruction being the
requirement to understand precisely when effective classes in the
N\'eron-Severi group of a {\em commutative} anticanonical ruled surface are
generically irreducible.

\medskip

We finish with some remarks on (semi)stable moduli spaces of torsion-free
sheaves.  For sheaves of rank $>1$, the situation is more complicated, as
the proof in the commutative case of the requisite inequalities involves an
intersection of the sheaf with a generic {\em pair} of hyperplanes, and it
is unclear what the noncommutative analogue of such a configuration might
be.  However, we can still prove a number of results about such moduli
spaces even without knowing they are quasiprojective (or schemes).  The
main result along those lines is the following, a direct analogue of
\cite[Thm. 11.78]{generic}, with the same proof.

\begin{thm}
  Let $D_a$ be an ample divisor class on the noncommutative rational
  surface $X$, let $r>0$, $D\in \Pic(X)$, $\chi\in \Z$ be such that
  $\gcd(r,D\cdot D_a,\chi)=1$, and let ${\cal M}$ be the corresponding
  stable moduli space, with universal family $M_{\cal M}$.  Then ${\cal M}$
  is either connected or empty, and its Chow ring is generated by the Chern
  classes of $R\Hom(N,M_{\cal M})$ where $N$ ranges over a system of
  generators of $K_0(X)$.
\end{thm}

This should be taken as including a number of other results, e.g., that the
moduli space is bounded and the $\gcd$ condition ensures the existence of a
universal family (both of which carry over to more general surfaces); see
\cite[\S 11.5]{generic} for a discussion of the case that $X$ has smooth
anticanonical curve, all of which carries over without much difficulty to
the general case.  (It also includes an alternate proof that the moduli
spaces of rank 1 sheaves are connected.)  Note that since ${\cal M}$ is
smooth (since it is unobstructed), connectivity implies irreducibility.
Also, by considering the degree 1 part of the Chow ring, we obtain a
surjective homomorphism $K_0(X)\to \Pic({\cal M})$.  In addition, we find
that the resulting presentation of the Chow ring is locally constant as we
allow $X$ to vary (subject to the condition that ${\cal M}$ is nonempty).

It is unclear whether an analogous result holds for more general rationally
ruled surfaces.  It remains the case in general that the class of the
diagonal in the Chow group of ${\cal M}\times {\cal M}$ may be computed
using Porteous' formula as $c_{\dim {\cal M}}(R\End(M_{\cal M})[1])$.  The
argument in the rational case then uses a full exceptional collection to
expand this in terms of the given generators, and thus fails in general.
However, in the general case we also have a nontrivial morphism from ${\cal
  M}$ to the identity component of the Grothendieck group, and might
reasonably hope to resurrect the result as a statement about the fibers.

Note that unlike projectivity, which we only know for rank $\le 1$, this
result on the Chow ring is only known for positive rank, and it is unclear
whether we should expect it to hold for $1$-dimensional sheaves.

\section{Differential and difference equations}

The construction of noncommutative ruled surfaces via differential and
difference (or hybrid) operators means that we can think of sheaves on such
surfaces as analogous to $D$-modules, with the caveat that the category of
$D$-modules is a deformation of a {\em quasi-}projective surface (the
cotangent bundle of the curve).  In particular, we expect that we should be
able to identify a class of sheaves corresponding to differential (or
difference) equations, or more precisely to meromorphic (possibly discrete)
connections on vector bundles.  (We set aside the hybrid and commutative
cases, as the action by operators is sufficiently far from faithful to make
it difficult to interpret sheaves in this way.  Though of course the
commutative case can be interpreted in terms of Higgs bundles or their
discrete analogues, see \cite{me:hitchin}.)

This is simplest to achieve in the untwisted case with $\bar{Q}=Q$.  Let
$M$ be a pure $1$-dimensional sheaf on such a ruled surface.  Then the
semiorthogonal decomposition gives a long exact sequence
\[
0\to \rho_1^*\rho_{{-}1*}M
\to \rho_0^*\rho_{0*}M\to M
\to \rho_1^*R^1\rho_{{-}1*}M
\to \rho_0^*R^1\rho_{0*}M
\to 0
\]
Suppose that $M$ is $\rho_{{-}1*}$-acyclic (so $\rho_{0*}$-acyclic), and
that $W=\rho_{{-}1*}M$ and $V=\rho_{0*}M$ are vector bundles.  Then $M$ is
uniquely determined by the corresponding element of
\[
\Hom(\rho_1^*W,\rho_0^*V)
\cong
\Hom_{\bar{Q}}(\pi_1^*W,\pi_0^*V).
\]
In the differential case, $\bar{Q}$ is the double diagonal in $C\times C$,
and thus there is a natural map
\[
\Hom_{\bar{Q}}(\pi_1^*W,\pi_0^*V)
\to
\Hom_C(W,V)
\to
\Hom_{\bar{Q}}(\pi_1^*W,\pi_1^*V).
\]
If the map from $W$ to $V$ is injective (i.e., if $M$ is transverse to
$Q$), then it is generically invertible, and thus we may compose with its
inverse to get a meromorphic map from $\pi_1^*V$ to $\pi_0^*V$ such that
the corresponding meromorphic endomorphism of $V$ is the identity.  But
this is precisely a meromorphic connection on $V$!  Conversely, given a
meromorphic connection on $V$, we can choose a subsheaf $W\subset V$ with
torsion quotient such that the induced map $\pi_1^*W\to \pi_0^*V$ is
holomorphic, and thus obtain a representation as above.  This sheaf is not
unique (we can clearly replace $W$ by any subsheaf of $W$ with torsion
quotient), but the property is preserved under taking sums, and thus there
is a maximal such subsheaf, giving us a canonical representation of the
meromorphic connection as a sheaf.  (We further see as in \cite{me:hitchin}
that $W$ is maximal iff for any sheaf of the form $\sO_f(-1)$ (i.e., a pure
1-dimensional sheaf with Chern class $f$ and Euler characteristic 0),
$\Hom(\sO_f(-1),M)=0$.  Note that the present convention differs by that of
\cite{me:hitchin} in that it is covariant rather than contravariant, but
this can be fixed by applying $R^1\ad$ and a twist by a line bundle.)

A similar calculation in the nonsymmetric difference cases gives a
discrete analogue of a meromorphic connection, namely a meromorphic map
\[
A:\tau_q^*V\ratto V,
\]
where $\tau_q$ is the appropriate automorphism of $C$.  (When $\tau_q$ is
translation by $1\in \G_a$, this is a ``d-connection'' as in
\cite{ArinkinD/BorodinA:2006}.)  Again, there is a canonical way to take
such a discrete connection and turn it back into a sheaf on the ruled
surface.  The symmetric case is only slightly more complicated: the
discrete connection lives on $Q$ rather than $C$, and gives a meromorphic
map
\[
A:\tau_q^*\pi_0^*V\ratto \pi_0^*V
\]
such that $s_1(A)A = 1$, with the map $\pi_1^*W\to \pi_0^*V$ coming from a
factorization of $A$ as guaranteed by Hilbert's Theorem 90.  (See the
discussion of the commutative case in \cite{me:hitchin}.)

In the symmetric difference case, since the connection lives on $Q$, we may
feel free to twist the interpretation by any automorphism of $Q$ (which
will, of course, act on all the various parameters including $q$).  This
freedom survives in a somewhat modified version in the nonsymmetric
difference case, as there is still a hidden dependence on $Q$: we had to
choose a component of $Q$ to fix the isomorphism between $C_0$ and $C_1$
and thus to determine the discrete connection.  In particular, the discrete
connection is more naturally viewed as living on that component of $Q$.  We
may thus again twist the interpretation by automorphisms of $Q$ that
preserve that component.  (Of course, if the automorphism {\em fixes} that
component, then it will not affect the interpretation.)  There are also
automorphisms that swap the two components, the most natural of which is
the deck transformation corresponding to the degree 2 map $Q\to C_0$.  In
the $q$-difference case, this transformation takes
\[
v(qz)=A(z)v(z)\qquad\text{to}\qquad v(z/q)=A(z/q)^{-1}v(z)
\]
when $V$ is trivial, and the other nonsymmetric difference cases are
analogous.  (In fact, the deck transform also has this effect in the
symmetric difference cases, this being precisely the condition to {\em be}
a symmetric difference equation.)  Similarly, in the differential case, we
most naturally have a connection on $C=Q^{\red}$, with the deck transform
acting trivially since it fixes $C$.

When the sheaf bimodule is no longer just the structure sheaf of $Q$, the
situation is more complicated.  In the invertible sheaf case, we may view
it as a twisted connection; this is not too difficult to deal with in the
difference cases, but is already quite subtle in the differential case.
And although many of the non-invertible (but torsion-free) cases are in
principle not twisted, and thus come with an explicit representation via
untwisted operators, the identification as an untwisted sheaf is only up to
nonunique isomorphism, and thus we have {\em multiple} such
representations.

Consider the differential case again.  The issue is that we want in the end
to have a meromorphic map $\pi_1^*V\ratto \pi_0^*V$ on the double diagonal,
but instead are given a meromorphic map $\pi_1^*V\ratto \pi_0^*V \otimes
{\cal E}$.  So to fix this in a consistent way, we need to choose an
invertible meromorphic map $\sO_Q\ratto {\cal E}$.  But such a map can be
specified via a suitable sheaf on the ruled surface!  In other words, to
translate meromorphic connections to sheaves on the ruled surface, we need
only specify a $\rho_{{-}1*}$-acyclic sheaf $\Sigma$ transverse to $Q$ with
$\rank(\Sigma)=0$, $\rho_{0*}\Sigma\cong \sO_C$.  This forces
$\rank(\rho_{1*}\Sigma)=1$, and thus $c_1(\Sigma)\cdot f=1$, so that we may
think of $\Sigma$ as a noncommutative analogue of a section of the ruling.
(Note that we can still get an interpretation when $\rho_{0*}\Sigma$ is a
nontrivial line bundle, but will get a connection on $V\otimes
\rho_{0*}\Sigma^{-1}$.)

This of course works more generally; given a ruled surface $X$ and a
``section'', we obtain a correspondence between suitable $1$-dimensional
sheaves and meromorphic (discrete) connections.  Moreover, changing the
section simply tensors the connection with the inverse of the connection
corresponding to the new section relative to the old section, so
corresponds to a scalar gauge transformation.  Note that the conditions on
$M$ are the same as in the commutative case: to correspond to a (discrete)
connection (with maximal $W$), $M$ must be transverse to $Q$ and satisfy
$\Hom(\sO_f(-1),M)=\Hom(M,\sO_f(-1))=0$ for all sheaves $\sO_f(-1)$.  (The
first vanishing is needed for maximality of $W$, and implies that $V$ and
$W$ are torsion-free, while the second vanishing corresponds to
$\rho_{{-}1*}$-acyclicity of $M$.)  We of course must insist that $\Sigma$
satisfy these conditions as well.

From this perspective, the reason why the untwisted invertible case has a
canonical interpretation as connections is that it has a canonical choice
of $\Sigma$.  Even this is not unique, however: e.g., in the differential
case, we can add any holomorphic differential to the canonical flat
connection on $\sO_C$.  (In this case, the automorphism group of the ruled
surface acts transitively on such sections.)  Similarly, twisting by line
bundles ${\cal L}_i$ with $\pi_0^*{\cal L}_0\cong \pi_1^*{\cal L}_1$ has
the effect of tensoring with a holomorphic connection on ${\cal L}_0$; if
${\cal L}_0$ is degree 0, then such a connection exists, but there is no
canonical choice unless ${\cal L}_0$ is trivial.  Beyond that, we can of
course add any {\em meromorphic} differential to the canonical flat
connection, which now corresponds to a section that is no longer disjoint
from $Q$.

Once we include the section in the data, we see that the interpretation as
connections is preserved by elementary transformations.  That is, if we
start with $X_0$ with section $\Sigma$, then blowing up a point gives a
surface $X_1$ and the minimal lift $\alpha_1^{*!}\Sigma$.  (We recall from
\cite{poisson,generic} that if $M$ has no $0$-dimensional subsheaf, then
$\alpha_i^{*!}M$ is defined to be the image of the natural map
$\alpha_i^*M\to \alpha_i^!M$, and similarly for minimal lifts through
iterated blowups.)  This has no map to or from $\sO_{e_1}(-1)$, and thus
when we blow down $f-e_1$ to get $X'_0$, the image of $\alpha_1^{*!}\Sigma$
will still be a section.  It moreover agrees with the original section on
the generic fiber over $C$, so is obtained from the original meromorphic
map $\sO_{\bar{Q}}\ratto {\cal E}$ by using the natural identification of
the generic fibers of ${\cal E}$ and ${\cal E}'$.  This lets us give a more
geometric interpretation of how the section comes into play: we simply
perform elementary transformations until $\Sigma$ is disjoint from $Q$,
which forces the resulting surface to be the standard untwisted surface,
with $\Sigma$ determined up to a holomorphic differential (in the
differential case) or a scalar factor (in the nonsymmetric elliptic
difference case), as appropriate.

Since the interpretation as a meromorphic connection essentially depends
only on the generic fiber over $C$, we find the following: a map $\phi:M\to
N$ of pure 1-dimensional sheaves such that the kernel and cokernel have
Chern class proportional to $f$ induces a gauge transformation between the
corresponding connections.  Indeed, this map is an isomorphism on the
generic fiber, and thus induces an isomorphism between the generic fibers
of $\rho_{0*}M$ and $\rho_{0*}N$ compatible with the two connections.  But
this is the same as specifying a meromorphic map $\rho_{0*}M\ratto
\rho_{0*}N$ that gauges the connection on $\rho_{0*}M$ to the connection on
$\rho_{0*}N$.  Such a gauge transformation is essentially a discrete
isomonodromy transformation, with the caveats that it may introduce or
remove apparent singularities, and that ``isomonodromy'' (though
traditional) is not really the right word when the equation has irregular
singularities.  Rather, the key property of the transformation is that any
local meromorphic solution of the equation corresponding to $\rho_{0*}M$
induces a local meromorphic solution of the equation corresponding to
$\rho_{0*}N$, and vice versa.  In the differential case, this not only
implies that their monodromy representations agree, but that they have the
same Stokes data at irregular singularities.

An important source of such transformations comes from the minimal lift
operation and twisting by line bundles.  Indeed, if we lift to a blowup,
twist by a line bundle of the form $\sO_X(\pm e_i)$, then take the direct
image, then the generic fiber of the result is unchanged, so the effect on
the connection is a gauge transformation (corresponding to a modification
of $V$ at the point lying under $e_i$).  One consequence is that following
\cite[Lem.~6.8]{poisson}, we can always arrange by a sequence of such
canonical gauge transformations to get a sheaf and a blowup such that the
minimal lift is disjoint from $Q$.  This process seems to correspond to one
of removing apparent singularities, and at the very least has the effect
that any further twists by exceptional line bundles preserve disjointness
from $Q$.  (Note that if we specify the section $\Sigma$ as the image of a
sheaf on the blowup, then twisting $\Sigma$ by exceptional line bundles
will not preserve the condition that its image on $C$ be the structure
sheaf, so will introduce an overall twist by a line bundle as mentioned
above.)

This mostly describes the action of twisting by line bundles; the only
remaining detail for exceptional bundles is the description of the specific
subspace of $V|_{\pi(e_i)}$ along which one should modify $V$, but this is
relatively straightforward in most cases: we simply take the sum of the
subspaces corresponding to indecomposable singularities (see below) for
which the separating blowup involves $e_i$.  (The answer is slightly more
complicated in the more special case in which $e_i$ came from blowing up a
smooth point of $Q$ that gets blown up more than once in the separating
blowup; in that case, the full space associated to the singularity has a
nilpotent matrix $N$ acting on it (coming from the Jordan block structure),
and the subspace has the form $\ker(N^l)$ if it corresponds to the $l$th
time that point was blown up.)

Twisting by a line bundle of Chern class $f$ is also easy to understand,
once we remember to keep track of the action on $\Sigma$: we thus find that
twisting the surface by any line bundle in $\rho_0^*\Pic(C)$ acts trivially
on the corresponding (discrete) connection.

Before considering twisting by $s$, we first consider the action of
$R^1\ad$.  This commutes with elementary transforms, so we may reduce to
the untwisted case.  We should note, of course, that the adjoint does not
preserve the canonical section $\Sigma$, so that it will be somewhat more
convenient to include a twist by a line bundle, and thus consider $M\mapsto
R^1\ad M\otimes \omega_C^{-1}$.  It moreover suffices to consider how this
behaves on a suitable localization of $C$, since we can use the known
effect on $\Sigma$ to guide the gluing on $M$.  In particular, we may as
well consider only how $M$ restricts to the {\em generic} point of $C$, so
that in particular all of the vector bundles are trivial and $M$
corresponds to a module over a $\Z$-algebra of the form $\overline{S}$.  We
know how the adjoint acts on such modules, and thus find that this acts via
the composition of the (na\"{i}ve) dual on the corresponding equation and
the deck transform over $C$.  To be precise, if $M$ corresponds to the
$q$-difference equation
\[
v(qz)=A(z)v(z),
\]
then $R^1\ad M\otimes \omega_C^{-1}$ corresponds to
\[
v(z/q)=A(z/q)^t v(z)
\]
and similarly for the other types of difference equations.  Similarly, in
the differential case, the operation takes
\[
v'(z) = A(z)v(z) \qquad\text{to}\qquad v'(z)=-A(z)^t v(z).
\]
In either case, composing by the deck transform of $Q\to C_0$ recovers the
standard duality on equations, e.g., taking $v(qz)=A(z)v(z)$ to
$v(qz)=A(z)^{-t}v(z)$.  Here we should note that although this duality
comes from the natural duality on $\GL_n$-torsors, most other functors on
$\GL_n$-torsors do not have nice translations into the sheaf framework, for
the simple reason that they tend to greatly increase the complexity of the
singularities (not to mention acting nonlinearly on the order of the
equation).

To understand the twist by $s$, we consider another contravariant symmetry,
namely taking the transpose of the morphism $B:\pi_1^*W\to \pi_0^*V$ to
get a morphism
\[
B^t:
\pi_0^*\sHom(V,\sO_C)
\to
\pi_1^*\sHom(W,\sO_C)
\]
corresponding to a sheaf on the noncommutative ruled surface arising by
swapping the roles of $C_0$ and $C_1$ in the sheaf bimodule.  We thus
obtain a (discrete) meromorphic connection on $W^*$ from the original
(discrete) meromorphic connection on $V$, and this operation is clearly a
contravariant involution.  This is closely connected to the adjoint, since
if $M$ has presentation given by $B$, then $R^1\ad M$ has a presentation of
the form
\[
0\to \rho_2^*\Hom(V,\omega_C)\to \rho_1^*\Hom(W,\omega_C)\to R^1\ad M\to 0
\]
in which the map is essentially given by $B^t$.  We thus see that (in the
untwisted case) transposing $B$ takes $M$ to $R^1\ad M\otimes L$, where $L$
is a suitable line bundle of Chern class $s+(1-g)f$ determined by the
condition that $\Sigma\cong R^1\ad \Sigma\otimes L$.  So to twist by $s$,
we need merely take the adjoint and then apply the above operation.  Note
that the combined operation is again expressible as an isomonodromy
transformation, the one caveat being that in the symmetric case, it changes
the symmetry.

The above calculations suffice to allow one to express the twist by any
line bundle as an isomonodromy transformation.  One case worth singling out
is twisting by $\theta \sO_X$.  This twist is not quite given by $\theta$,
for the simple reason that twisting by $\theta\sO_X$ changes the
identification of $Q$ with the anticanonical curve.  We thus find that
twisting by $\theta \sO_X$ is given by the pullback through a translation,
or in other words takes
\[
v(qz)=A(z)v(z)\qquad\text{to}\qquad
v(qz)=A(qz)v(z).
\]
(Note that since $A(qz)=A(qz) A(z) A(z)^{-1}$, this is indeed an
isomonodromy transformation.)

\medskip

The discussion of \cite{me:hitchin} in the commutative case suggests that
for a $1$-dimensional sheaf $M$ on a ruled surface (which we may suppose
untwisted with $Q=\bar{Q}$), the singularities of the corresponding
connection are determined by the sheaf $M|_Q$.  This is somewhat tricky to
deal with in general, but the previous discussion suggests that we should
instead consider sheaves on iterated blowups that are disjoint from $Q$, in
which case the singularities should be determined by the surface $X_m$ and
the Chern class of $M$.  In the commutative case, the precise
correspondence to singularities was determined in \cite[\S5.4]{me:hitchin},
but this required fairly explicit calculations in affine patches of
blowups, making it nontrivial to carry out the analogous calculation in the
noncommutative setting.

In the case of a rational surface, this can be avoided in the following
way.  We first note that the blowups needed to separate a sheaf $M$ on
$X_0$ from $Q$ depend only on local information, and thus we may feel free
to modify $M$ at points not in the orbit of the support of $M|_Q$, or in
other words to apply an isomonodromy transformation which is invertibly
holomorphic on the orbit.  (This modifies $M$ by sheaves of Chern class
$\propto f$ with direct image outside the orbit.)  This makes the
separating blowup slightly more complicated, but we can recognize which
blowups were needed on the original surface from the images of the
corresponding points on $Q$, so do not lose information.  In particular, by
a suitable such modification, we may arrange for $V$ to be isomorphic to
$\sO_C^n$.  We then find that $R\Gamma(M(-f))=0$, and thus $M$ induces a
family $M_q$ of objects as we vary $q$ leaving the commutative surface
associated to $X(-f)$ unchanged.  Moreover, changing $q$ has no effect on
either $W$ or the map $\rho_1^*W\to \rho_0^*V$, and thus $M_q$ is actually
a family of {\em sheaves} (injectivity does not depend on $q$), and
corresponds to a family of equations $\hbar v'(z) = A(z) v(z)$,
$v(z+\hbar)=A(z)v(z)$, or $v(qz)=A(z)v(z)$ with $A$ independent of $\hbar$
or $q$.  In addition, the minimal lift of $M$ to the separating blowup also
satisfies $R\Gamma(M(-f))=0$, and thus the separating blowup is essentially
independent of $\hbar$ or $q$.  (We may need to exclude countably many
values of the noncommutative parameter where $M$ picks up apparent
singularities.)  As a result, we may reduce to the case $\hbar=0$ or $q=1$
as appropriate, and thus may apply the calculations of
\cite[\S5.4]{me:hitchin}.  In particular, we find that each $e_i$ on the
separating blowup which is not a component of $Q$ corresponds to an
indecomposable singularity (depending only on the separating blowup), which
appears with multiplicity $c_1(M)\cdot e_i$ in the equation corresponding
to $M$.

The nonsymmetric elliptic difference case is also straightforward; although
we cannot as easily reduce to the commutative case, we need merely observe
that $M|_Q$ is the cokernel of the restriction to $Q$ of the map $W\to
V\otimes \bar{S}_{01}$, which we may think of as a pair of maps on the two
components of $Q$.  The matrix $A$ of the equation is simply the ratio of
those two maps, and thus we may read off the singularities of
$v(z+q)=A(z)v(z)$ from the zeros and poles of $A$, just as with finite
singularities in the $q$- and additive difference cases.

It remains to consider the higher genus differential cases.  Here we note
again that the blowup is determined by the local structure of the equation,
but now the relevant point of $C$ lies under a singular point of $Q$ (since
$Q$ is nonreduced!), and thus we may base change to the local ring at that
point.  Moreover, any sufficiently good approximation to the connection
(relative to the associated valuation) will give rise to the same sequence
of blowups, and thus we may pass to the {\em complete} local ring.  But
then the same reasoning lets us approximate the equation by one over
$k(z)$, and thus reduce to the rational case.

\medskip

One application of the singularity classification is that it gives a
convenient way to compute centers of maximal orders in sufficiently large
characteristic.  That is, we can recognize a surface by looking at the
singularities of the connection (or generalized Higgs bundle) corresponding
to a sheaf with Chern class a sufficiently ample divisor on the surface, or
more precisely from those properties of the singularities that do not
depend on the specific choice of sheaf.  Thus to understand the structure
of the center, it suffices to understand how taking the direct image to the
center affects the separating blowup.  This in turn reduces to
understanding indecomposable singularities.  Finite singularities are
straightforward to control, so we focus on the singularities of equations
lying over singular points of $Q$, letting us work over the appropriate
complete local ring.

As an initial example, for the nonsymmetric $q$-difference case, an
indecomposable singularity is the restriction of scalars to $k[[z]]$ of an
equation of the form $v(qz) = z^{a/b} B(z) v(z)$ for $B(z)\in
\GL_m(k[[z^{1/b}]])$.  (Here we assume we are working over a field of
sufficiently large characteristic to avoid wild ramification.  Also, we
need to choose a $b$-th root of $q$ to extend the action of $z\mapsto qz$
to the given Puiseux series.)  If $q$ has order $n$, we may compute the
direct image of the corresponding sheaf as the restriction of scalars to
$k[[z^n]]$ of the equation
\[
v(q^n z)
=
\prod_{0\le i<n} (q^i z)^{a/b} B(q^i z) v(z)
=
(z^n)^{a/b} q^{an(n-1)/2b} \prod_{0\le i<n} B(q^i z) v(z).
\]
The associated separating blowup depends only on the leading term of this
matrix, and is thus determined by $a/b$ and the matrix
\[
q^{an(n-1)/2b} B(0)^n,
\]
while the separating blowup corresponding to the original sheaf determines
the Jordan block structure of $B(0)$.  (That $a/b$ does not change follows
from the fact that it is purely combinatorial: it corresponds to a sequence
of blowups in nodes of the anticanonical curve, which must therefore map to
nodes.  That an $n$th power is involved similarly follows from the fact
that the map from $Q$ to $Q'$ is the quotient by $\theta$ since $Q'$ is
reduced.)  The same calculation applies in the symmetric $q$-difference
case (as the classification of singularities is the same).

The additive difference case is somewhat trickier to deal with.  The direct
image can in general be computed as the restriction of scalars to
$k[[(z^p-z)^{-1}]]$ of $v(z+p) = A(z+p-1)\cdots A(z) v(z)$, but it is
nontrivial to compute a sufficiently good estimate of the latter product in
the case of an indecomposable singularity.  Such singularities correspond
to equations of the form
\[
v(z+1) = \alpha z^{a/b} (1+\sum_{1\le i<b} c_i z^{-1/b} + z^{-1}B(z)) v(z),
\]
where $c_1,\dots,c_{b-1}\in k$, $B\in \Mat_n(k[[z^{-1/b}]])$, and the
difficulty is that we need to compute the corresponding product to
precision $o(z^{-p})$.  This turns out to be doable, as long as the
characteristic is larger than $b$.  (If the characteristic is too small,
there are two issues: the computation of the requisite blowups may involve
wild ramification, and the anticanonical curve of the blowup may have
components of multiplicity $>p$, which complicates the action of taking the
center on the moduli stacks.)  We first note that we can factor this into a
scalar contribution and a matrix contribution, where the matrix is
$1+O(1/z)$.

\begin{lem}
  Let $k$ be a field of characteristic $p$, and let $b$ be a positive
  integer prime to $p$.  If $B\in \Mat_n(k[[z^{-1/b}]])$ is such that $B=1+B_0
  z^{-1}+o(z^{-1})$, then $B(z+p-1)\cdots B(z) = 1 +
  (B_0^p-B_0)z^{-p}+o(z^{-p}).$
\end{lem}

\begin{proof}
This is a polynomial equation in the coefficients of $B$ to order
$(z^{-1/b})^{pb+1}$, and thus it suffices to consider the generic case.  In
particular, we may assume that $B_0$ is diagonalizable (so WLOG diagonal)
with eigenvalues $\lambda_i$, such that $\lambda_i-\lambda_j\notin \F_p$
for $i\ne j$.  But then there is a matrix $C\in 1+o(1)$ such that
$C(z+1)B(z)C(z)^{-1}$ is diagonal, with
\[
B_{ii}(z) = 1+z^{-1} f_i((z^p-z)^{-1/b})
\]
for suitable power series $f_i$.  (Gauging by a matrix $C$ such that
$C-1$ has a single nonzero coefficient $C_{ij}-\delta_{ij}=\alpha z^{-a/b}$
adds $\alpha (\lambda_j-\lambda_i-a/b) z^{-a/b-1}$ to $B_{ij}$, and thus we
may set the $o(1/z)$ part of $B$ arbitrarily, except for those terms of
degree $-1$ modulo $p$.)

We then have
\begin{align}
\prod_{0\le j<p} B_{ii}(z+j)
&=
\prod_{0\le j<p} \frac{z+j+f_i((z^p-z)^{-1/b})}{z+j}\notag\\
&=
\frac{(z+f_i((z^p-z)^{-1/b})^p-(z+f_i((z^p-z)^{-1/b}))}
     {z^p-z}\notag\\
&=
1 + \frac{f_i((z^p-z)^{-1/b})^p-f_i((z^p-z)^{-1/b})}{z^p-z}\notag\\
&=
1+\frac{f_i(0)^p-f_i(0)}{z^p-z} + o(z^{-p})
\end{align}
as required.
\end{proof}

For the scalar factor, we assume $b<p$, and rewrite it using the polynomial
$\exp_p(z) = \sum_{0\le i<p} z^i/i!$, which is invertible under composition
and approximately a homomorphism.

\begin{lem}
  Let $k$ be a field of characteristic $p$, and let $1\le b<p$.  Then for
  $g\in z^{-1/b}k[[z^{-1/b}]]$ and $f=\exp_p(g)$, we have
  \[
  \prod_{0\le i<b} \exp_p(g(z+i)) =
  \exp_p(g(z)^p+g^{(p-1)}(z))|_{z=z-z^{1/p}} + o((z^p-z)^{-1}).
  \]
\end{lem}

\begin{proof}
  Since $\exp_p$ is a homomorphism to order $O(z^{-p/b})=o(1/z)$, and we
  have already seen that this holds when $\exp_p(g(z))=1+o(1/z)$, we see
  that both sides are homomorphisms to the desired order, and thus it
  suffices to prove the result for $\exp_p(g(z))=1-\alpha z^{-1/b}$.  In
  that case, we may take
  \[
  g(z) = -\sum_{1\le i\le b} i^{-1} \alpha^i z^{-i/b}+o(1/z)
  \]
  and thus
  \[
  g(z)^p+g^{(p-1)}(z)
  = -\Bigl(\sum_{1\le i\le b} i^{-1} \alpha^{ip} z^{-ip/b}\Bigr)
    +b^{-1} \alpha^b z^{-p} + o(z^{-p})
  \]
  so that we need to show
  \begin{align}
  \prod_{0\le i<p} (1-\alpha (z+i)^{-1/b})
  &=
  \exp_p\Bigl(-\sum_{1\le i\le b} i^{-1} \alpha^{ip} (z^p-z)^{-i/b}
  +b^{-1} \alpha^b (z^p-z)^{-1}\Bigr) + o((z^p-z)^{-1})\notag\\
  &=
  (1-\alpha^p (z^p-z)^{-1/b})(1+b^{-1} \alpha^b (z^p-z)^{-1}) + o((z^p-z)^{-1})\notag\\
  &=
  1-\alpha^p (z^p-z)^{-1/b}+b^{-1} \alpha^b (z^p-z)^{-1}+o((z^p-z)^{-1}).
  \end{align}
  Replacing $\alpha$ by $1/\lambda$ and multiplying both sides by
  $\lambda^p$ turns the left-hand side into the minimal polynomial of
  $z^{-1/b}$ over $k(((z^p-z)^{-1/b}))$, and thus to show that the given
  estimate holds, it suffices to plug in $z^{-1/b}$ and
  show that the value is sufficiently small.  Thus the claim follows from
  the fact that
  \[
  z^{-p/b} - (z^p-z)^{-1/b} + b^{-1} z^{(b-p)/b} (z^p-z)^{-1}
  =
  O(z^{2-2p-p/b}),
  \]
  which we may verify by expanding both powers of $z^p-z$ using the
  binomial series.
\end{proof}

We thus see that for an equation
\[
v(z+1) = \alpha z^{a/b} \Bigl(\exp_p\Bigl(\sum_{1\le i<b} c_i z^{-i/b} + B_0
z^{-1}\Bigr)+o(1/z)\Bigr) v(z),
\]
the direct image to the center is given by the the restriction of scalars
to $k[[(z^p-z)^{-1}]]$ of the matrix
\[
\alpha^p (z^p-z)^{a/b}
\Bigl(\exp_p\Bigl(\sum_{1\le i<b} c_i^p (z^p-z)^{-i/b} + (B_0^p-B_0) (z^p-z)^{-1}\Bigr)
+o((z^p-z)^{-1})\Bigr).
\]
Note that the same calculations also let us compute the requisite product
for a symmetric equation.

The differential case is similar; if we think of the equation as $Dv = Av$,
then the direct image is the restriction of scalars of the action of $D^p$.
An indecomposable equation over $k((z))$ is the restriction of scalars of
an equation
\[
v' = (f(z) + B(z))v(z),
\]
where $B\in z^{-1}\Mat_n(k[[z^{1/b}]])$ and $f\in k((z^{1/b}))$.  Since
$D\mapsto D-f$ is an automorphism, we may compute the action of $D^p$ by
first computing the action with $f(z)=0$ and adding $f(z)^p+f^{(p-1)}(z)$.
The contribution of $B$ may be computed in the generic case, when $B$ is
again diagonalizable by a suitable gauge transformation, and we find that
the contribution is $B(z)^p+B^{(p-1)}(z) + o(z^{-p})$, giving the action of
$D^p$ to sufficiently high precision.  Here we see quite explicitly that
the action of passing to the center on the parameters is more complicated
when $b>p$; any term of $f$ with exponent $<-1$ but congruent to $-1$
modulo $p$ contributes a term with exponent $<-p$ to $f^{(p-1)}$, and thus
the corresponding coefficient does not only appear as a $p$-th power.

\bigskip

In the discussion above, we saw that twisting by line bundles corresponds
to discrete isomonodromy transformations.  In the differential (and
ordinary difference) cases, these do not in general exhaust the full set of
monodromy- (or Stokes-, or whatever the analogue for ordinary difference
equations might be) preserving transformations: there are also typically
continuous flows between moduli spaces coming from, e.g., the fact that the
relevant fundamental group depends only on the topology of the given
punctured Riemann surface.  Although the traditional explanations of these
flows are largely analytic in nature (they are mediated by the
Riemann-Hilbert correspondence), we can hope that the {\em infinitesimal}
flow may have a more algebraic/geometric interpretation in our context.

The typical form of such a flow (we consider the differential case, but the
difference case is analogous) is (locally) given by a Lax pair, i.e., a
pair of equations
\[
v_z = Av,\qquad v_w = Bv
\]
satisfying the consistency condition ($v_{zw}=v_{wz}$)
\[
A_w = B_z + [B,A].
\]
(Here $A$ and $B$ are analytic in $w$ but algebraic in $z$.)  To first
order, this corresponds to an infinitesimal deformation of $A$, or
equivalently to the self-extension of the equation given by
\[
\begin{pmatrix}
  A & A_w\\
  0 & A
\end{pmatrix}
=
\begin{pmatrix}
  A & B_z + [B,A]\\
  0 & A
\end{pmatrix}
\]
But this is in turn the gauge transformation of the trivial self-extension
by a block unipotent matrix corresponding to $B$.  Since this is meant to
come from a flow that preserves the complexity of the singularities, we see
that the corresponding sheaf is the direct image on $X_0$ of an extension
of $M$ to a first-order deformation of the separating blowup.  Moreover, we
want this deformation to preserve the combinatorics of $Q$ (both because we
want to preserve the complexity of the singularities and because we wish to
avoid obstructed deformations), and thus in particular any singular point
of $Q^{\red}$ that must be blown up will continue to be blown up after the
deformation.  Thus we may just as well blow up such singular points before
considering any further deformations, replacing $X_0$ by some $X_{m_1}$
that still remains invariant under the deformation, but now satisfying the
condition that $M$ meets $Q^{\red}$ only in the smooth locus.

Now, the restriction to $Q$ of the corresponding sheaf on $X_{m_1}$ must
correspond to a flat deformation of $M|_Q$.  This is important because the
translation between sheaves and equations is only modulo sheaves of Chern
class $\propto f$, but any such sheaf will contribute to the restriction to
$Q$.  We thus find that an isomonodromy transformation corresponds (locally
on $C$) to a class in $\Ext^1(M,M)$.  Moreover, the block upper triangular
meromorphic gauge transformation that trivializes it corresponds to an
explicit trivialization of the induced extension of $M$ by $M^+$ where
$M^+$ is a corresponding extension of a sheaf of Chern class $\propto f$ by
$M$.  Such an explicit trivialization corresponds to a representation of
the self-extension of $M$ as the image of a class in $\Hom(M,M^+/M)$.
(Note that this is now actually {\em globally} defined on $C$, since it
corresponds to a class in the bottom-left corner of the relevant
hypercohomology spectral sequence.)  Assuming that $M$ has no quotient of
Chern class $\propto f$, the image of such a morphism will be
$0$-dimensional, and thus in order for the extension to exist, we must have
$M^+\subset M(D)$ for some divisor $D$ which is a nonnegative linear
combination of components of $Q$.

In other words, isomonodromy transformations that deform the separating
blowup without deforming $X_0$ correspond to morphisms $M\to M(D)/M$.
Moreover, since we are more precisely interested in isomonodromy
transformations that are consistent across an entire symplectic leaf, this
morphism should depend only on $M|_Q$.  Here there is a subtle point to
bear in mind: although every point in the symplectic leaf has isomorphic
restriction $M|_Q$, the universal family need not admit such an isomorphism
globally.  Thus our assignment of a morphism $M\to M(D)/M$ given $M|_Q$
must respect automorphisms of $M|_Q$.  In particular, we may apply this to
the induced morphism $M\to M|_Q(D)$, or equivalently $M|_Q\to M|_Q(D)$,
which must in particular commute with the local action of $\sO_Q^*$.  Since
we are working over a field of characteristic 0, $\_(D)$ acts nontrivially
on any component of $Q$ appearing in $D$, so that the left and right
actions of $\sO_Q^*$ only agree modulo $D\cap (Q-Q^{\red})$.  (Indeed,
since $M$ only meets $Q^{\red}$ in smooth points, the localization to that
point (in the sense of \cite{vanGastelM/VandenBerghM:1997}, which suffices
to control the category of $0$-dimensional sheaves) has the form
$k\langle\langle u,v\rangle\rangle/([u,v]-v^\mu)$ where $\mu$ is the
multiplicity of the corresponding component, see Proposition
\ref{prop:local_structure_of_order_single_component}.  This lets us compute
the action of conjugation by $v^d$ for any $d$ and verify that it is
nontrivial for $d<\mu$ which is not a multiple of the characteristic.)  We
thus find that the morphism factors through $D-(D\cap Q^{\red})$, and thus
that we may reduce to the case $D=Q^{\nr}:=Q-Q^{\red}$, and thus need to
understand the {\em natural} maps $M|_Q\to M|_{Q^{\nr}}(Q^{\nr})$.

If such a map corresponds to an unobstructed deformation of the separating
blowup, then it in particular induces deformations of the support of
$M|_{Q^{\red}}$ along $Q^{\red}$.  If the deformation of a given point is
trivial (in particular if the component it lies on has multiplicity $1$ in
$Q$), then we may again blow it up, making $M|_{Q^{\red}}$ smaller.  We
thus see that we have at most one dimension of such deformations for each
time we blow up a smooth point of $Q^{\red}$ lying on a component of
multiplicity $>1$ of $Q$.  To see that this bound is tight, it suffices to
show that the map to the tangent space is surjective at every point in the
support of $M|_{Q^{\nr}\cap Q^{\red}}$.  Fixing such a point, we may work
in the category of $0$-dimensional sheaves supported at that point, or
equivalently (as we have already discussed) in the category of
finite-length modules over $R_\mu=k\langle\langle
u,v\rangle\rangle/([u,v]-v^{\mu})$, and in particular in the subcategory of
such modules annihilated by $v^{\mu-1}$.  This is the category of modules
over a {\em commutative} ring, and in that category, the functor
$\_(Q^{\nr})$ is just tensoring with the invertible bimodule generated by
$v^{1-\mu}$, so that in particular the element $v^{1-\mu}$ induces a
natural transformation $M|_{Q^{\nr}}\to M|_{Q^{\nr}}(Q^{\nr})$.  Moreover,
this element induces a derivation on $R_\mu$ taking $v$ to 0 and $u$ to
$1-\mu$, and thus in particular does not preserve the maximal ideal; it
follows immediately that the corresponding deformation moves the base point
as required.

A major simplifying assumption made above was that the surface $X_0$ itself
was not being deformed.  (This is not an issue in the ordinary difference
case, when there are no such deformations!)  If we wish to relax this
assumption, then there are some additional issues that arise.  The first is
that, as we have seen, the relation between sheaves and equations depends
not only on the surface but on a choice of section, and there is only a
canonical such choice in the untwisted case.  Thus to have an actual
isomonodromy transformation, we either need to restrict to the untwisted
case or allow the normalizing section to deform along with the surface.
Focusing on the first case (which is in any event sufficient: we can always
separate the normalizing section along with the sheaf of interest, at which
point changing the blowdown structure reduces to the case $X_0$ untwisted,
and determining which deformations act trivially on the sheaf of interest
is then straightforward), we see that the only remaining degree of freedom
is the possibility to deform $C$ itself.  Such a deformation is given by a
class in $H^1(T_C)$, or in other words by a 1-cocycle $z_{ij}$ in the space
of first-order holomorphic differential operators on $C$.  Given a
differential equation on $C$, we can then construct an extension to the
deformation in the local form
\begin{align}
\begin{pmatrix}
  1 & z_{ij} D\\
  0 & 1
\end{pmatrix}
\begin{pmatrix}
  D-A & 0\\
  0 & D-A
\end{pmatrix}
\begin{pmatrix}
  1 & -z_{ij} D\\
  0 & 1
\end{pmatrix}
&=
\begin{pmatrix}
  D-A & [D-A,z_{ij}D]\\
  0   & D-A
\end{pmatrix}\notag\\
&=
\begin{pmatrix}
  D-A & -z'_{ij}D - z_{ij}A'\\
  0   & D-A
\end{pmatrix}.
\end{align}
This is equivalent (add $z'_{ij}$ times the second row to the first) to the
equation
\[
v' = \begin{pmatrix} A & -z_{ij}A'-z'_{ij}A\\ 0 & A\end{pmatrix}v,
\]  
which is locally the isomonodromy transformation by $-z_{ij}A$.  (Note that
if $z_{ij}$ is sufficiently close to constant at a singular point, then
this is locally just the standard isomonodromy transformation moving that
singular point, while if it vanishes at the singular point, the singular
point does not move and the combinatorics of the singularity does not
change.)  To see that this is globally monodromy preserving, we note that
we obtain a local system for each open subset in the covering, and the
$z_{ij}$ induce compatible isomorphisms between the local systems on the
overlaps, and in particular ensure that at any point of $U_j\setminus
U_{ij}$ where the equation is regular, the monodromy in $U_i$ around that
point is trivial.  Thus each $U_i$ extends to a local system on the full
complement of the singularities of the equation, and these local systems
are compatibly isomorphic.  Since Stokes data is entirely local, and the
isomorphisms also ensure isomorphisms between Stokes data, the analogous
statement continues to hold for non-Fuchsian equations.

We thus see that in addition to the isomonodromy transformations that fix
$X_0$, we also have isomonodromy transformations parametrized by
$H^1(T_C)$.  (When $C\cong \P^1$, we do not gain any new isomonodromy
transformations, but should recognize that $3$ of the existing isomonodromy
transformations correspond to global automorphisms of $\P^1$; when $C$ is
elliptic, we both gain and lose an isomonodromy transformation in this
way.)  In particular, it is straightforward to determine the number of such
deformations in general: in the differential case, we start with
$\chi(T_C)=3g-3$ and then as we pass to the separating blowup, add 1 each
time we blow up a smooth point of $Q^{\red}$ lying on a multiple component.
Given the role of $Q^{\nr}$ above, we note that an easy induction shows
that the expected number of such deformations can be encapsulated in a
single quantity, to wit $-\chi(\sO_{Q^{\nr}}(Q^{\nr}))$, this time computed
on the separating blowup, with the same number holding in the ordinary
difference case as well.  (In particular, this explains the existence of a
continuous isomonodromy transformation of the symmetric difference
equations considered in \cite{OrmerodCM/RainsEM:2017b}.)

In particular, we find that the number of isomonodromy transformations
explained by the above considerations is actually intrinsic to the surface.
This suggests that there should be a more geometrical description of these
deformations.  We have not quite been able to do this, but note the
following very suggestive facts.

First, if we consider how the continuous and discrete isomonodromy
transformations act on the given piece of the moduli stack of surfaces, we
see that they are essentially complementary: in the parametrization of
singularities coming from \cite[\S5.4]{me:hitchin}, the parameters that
move under continuous isomonodromy transformations are precisely the ones
that do not move under {\em discrete} isomonodromy transformations.
Moreover, the continuously movable parameters have a particularly nice
interpretation in finite (sufficiently large) characteristic: they are the
parameters that get taken to their $p$-th powers when passing to the
center.  In particular, we find that an infinitesimal deformation of such a
parameter induces the trivial deformation of the center!  This also applies
to the transformations that deform $C$, since the corresponding curve on
the center is the image under Frobenius.  Conversely, the parameters that
move under discrete isomonodromy transformations are acted on nontrivially
by $\_(Q)$, and thus taking the center induces a {\em separable} map on
those parameters.  We thus see that there is at least formally a
correspondence between continuous isomonodromy transformations in
characteristic 0 and center-preserving deformations in finite
characteristic\dots

Second, in finite characteristic, the associated deformations not only come
with extensions of $1$-dimensional sheaves disjoint from $Q$, but of {\em
  any} quasi-coherent sheaf disjoint from $Q$, for the simple reason that a
center-preserving deformation of an Azumaya algebra is trivial!  This
suggests more generally that a continuous isomonodromy transformation
should correspond to a deformation of $X_m$ equipped with an explicit
trivialization of the induced deformation of $X_m\setminus Q$ (i.e., of the
``quasiprojective'' category obtained by inverting the natural
transformation $\_(Q)\to \text{id}$), possibly with some additional
conditions imposed to ensure the lack of obstructions.

Even without a full geometric understanding of continuous isomonodromy
transformations, the above considerations at least suffice to tell us that
they exist when expected, and, with limited exceptions, respect changes in
blowdown structure.  (The exceptions come from the fact that we treated
deformations of $X_0$ separately from the deformations that move base
points of blowups; the latter are treated in a way that is intrinsic to the
surface!)  There is no particular issue with elementary transformations, as
those at most have the effect of a scalar gauge transformation, so do not
affect continuous isomonodromy.  Thus the only issue is the Fourier
transform.  Since this only arises when $X$ is rational, we find that the
issue in that case is not so much deformations but infinitesimal
automorphisms.  There are four cases of noncommutative $\P^1\times \P^1$s
with nonreduced $Q$, and in each case the multiple components have
multiplicity 2.  By adjunction, we find that $\sO_{Q^{\nr}}(Q^{\nr})\cong
\omega_{Q^{\nr}}^{-1}(Q^{\nr}-Q^{\red})$, and thus $\sO_{Q^{\nr}}(Q^{\nr})$
controls the deformation theory of the pair $(Q^{\nr},Q^{\nr}\cap
(Q^{\red}-Q^{\nr}))$.  In particular, there is an induced map from
infinitesimal automorphisms of the surface to global sections of
$\sO_{Q^{\nr}}(Q^{\nr})$, and in each case that map is surjective.
Moreover, in each case any automorphism of the surface acts on equations by
a change of variables, so indeed gives a trivial isomonodromy
transformation.  (Note that the automorphisms of the surface do not in
general act {\em faithfully} on equations once one takes into account the
action on the normalizing section $\Sigma$.  Moreover, the kernel of the
action on equations is not preserved by the Fourier transform!)

In particular, we find that the 2-dimensional moduli spaces arising from
Theorem \ref{thm:painleve_moduli_spaces} admit continuous isomonodromy
transformations precisely when the curve $Q$ is nonreduced.  (A fuller
geometric understanding would presumably tell us that the consistency
conditions for those isomonodromy transformations are precisely the usual
Painlev\'e equations.  This is true in the cases involving second-order
equations, since then we can reduce to the corresponding differential
case.)

We also find that the second order difference equations admitting
continuous isomonodromy transformations controlled by Painlev\'e
transcendents have ``matrix Painlev\'e'' analogues that again admit
continuous isomonodromy transformations, which again will produce the same
nonlinear consistency equations as the corresponding differential cases,
and similarly that the equations considered in
\cite{KawakamiH/NakamuraA/SakaiH:2013} associated to 4-dimensional moduli
spaces are either of this form or have second-order difference or
differential Lax pairs.

\medskip

Since the Lax pairs for Painlev\'e coming from Theorem
\ref{thm:painleve_moduli_spaces} are of particular interest, it seems worth
spelling them out.  In the simplest case, for Painlev\'e VI, the
typical form of such a Lax pair is as a connection on a vector bundle of
rank $2r$ and degree $d$ (with $r\ge 1$ and $\gcd(r,d)=1$), with four
Fuchsian singularities and residues having two $r$-dimensional eigenspaces
with fixed eigenvalues.  The main difficulty is that although the theory
guarantees the existence of a universal sheaf in this case, and thus tells
us that there {\em is} a rational parametrization of the corresponding
family of connections, it does not give us any particular idea on how to
write down such a parametrization.

In the special case $d=1$, we can write the equation down in an alternate
form which is somewhat easier to parametrize; the cost is the introduction
of an apparent singularity, but we can use the location of that singularity
as one of the parameters.  The idea is that although the matrix form of an
equation is more natural for our purposes, we can also in general write
equations in ``straight-line'' form (i.e., as a linear relation between the
higher derivatives of a single function).  This is non-unique in general,
but in our case there is a natural choice coming from the generically
unique global section of the sheaf.  Indeed, we have noted that $M$ is
generically acyclic, and generically irreducible, and an irreducible
acyclic sheaf with the given invariants is generated by its unique global
section.  The kernel of the global section is a torsion-free sheaf of rank
1, and twisting by $c_1(M)=rQ$ makes the torsion-free sheaf a point of the
$1$-point Hilbert scheme of $X$.  A generic point of that Hilbert scheme in
turn has a unique map from the line bundle $\sO_X(-f)$, the cokernel of
which is a sheaf of the form $\sO_f(-1)$.  We thus generically obtain a
complex
\[
0\to \sO_X(-rQ-f)\to \sO_X\to M\to 0
\]
which is exact except at $\sO_X$ where the cohomology has Chern class $f$.
Thus the morphism $\sO_X(-rQ-f)\to \sO_X$ represents the same equation as
$M$, apart from introducing a single apparent singularity.

In the Painlev\'e VI case, this morphism corresponds to an equation
\[
\sum_{0\le i\le 2r} c_i(z) (t(t-1)(t-\lambda))^i f^{(i)}(z) = 0
\]
such that each $c_i(z)$ is a polynomial of degree at most $4r-2i+1$.  This
equation is regular except at $0,1,\lambda,\infty$ and the unique root $v$
of $c_{2r}(z)$ (which we assume is not one of the four usual
singularities).  Each exponent at one of the true singularities induces an
arithmetic progression $e$, $e+1$,\dots,$e+r-1$ of corresponding exponents
of the straight-line equation.  The exponents at $v$ must be nonnegative
integers, and thus the global constraint on the exponents implies that they
must be $0,1,\dots,2r-2$ and $2r$; we also must have that the equation is
integrable at this point.

We start with $\sum_{0\le i\le 2r} (4r-2i+2) = (2r+1)(2r+2)$ undetermined
coefficients, and each of the eight original exponents imposes $r(r+1)/2$
linear conditions, resulting in a $2r+2$-dimensional space of equations
(that automatically satisfies the remaining condition on the exponents at
the apparent singularity; that the linear conditions are independent
follows from the fact that $\chi(\sO_X(rQ+f))=2r+2$).  Fixing the leading
term (i.e., fixing $v$ and the overall scalar) gives a $2r$-dimensional
affine space of equations, and the integrability condition at $v$ imposes
one quadratic condition for each exponent below the gap, so $2r-1$
conditions.  It follows from the general theory that the intersection of
these quadrics is a rational curve, and in fact (since the map to $v$ is
the natural ruling on the Hilbert scheme) is arithmetically $\P^1$; that
is, it admits a parametrization over the field generated by $v$ and the
exponents.  Moreover, we also find that these straight-line equations
admit a continuous isomonodromy transformation (presumably given by the
Painlev\'e VI equation), as well as a lattice of discrete isomonodromy
transformations, which are actually $r$-fold iterations of the usual
B\"acklund transformations of the space of initial conditions of PVI.  (The
full lattice can still be obtained geometrically in this setting, but now
corresponds to the translation part of the affine Weyl group action, and
thus the atomic translations are actually given by formal integral
transformations!)

\medskip

It is worth noting that not every systematic family of infinitesimal
deformations preserving the complexity of singularities corresponds to an
isomonodromy transformation.  Indeed, when $g(C)>0$, we have a
$2g$-dimensional family of such deformations that do not change the
singularities at all.  The point is that on the untwisted ruled surface,
although there is a natural choice of section, the corresponding sheaf is
no longer rigid when $g>0$, and thus we obtain a $2g$-dimensional family of
regular connections on line bundles of degree 0.  Tensoring an equation by
such a connection gives a new equation with the same singularities, and
thus the tangent space at the natural section induces an infinitesimal
deformation of the equation.  By the Riemann-Hilbert correspondence, the
connection corresponding to a nontrivial section necessarily has nontrivial
monodromy, since it is not isomorphic to the connection on $\sO_C$
corresponding to the equation $v'=0$; since the tensor product of
connections has tensor product monodromy, we see that these actions do not,
in fact, preserve monodromy.  This is likely to complicate the geometric
interpretation of continuous isomonodromy transformations, especially since
on anything other than an untwisted surface, the association between
sheaves and equations is only determined modulo precisely such tensor
products!  One way to resolve this would be to consider only the {\em
  projective} monodromy (i.e., the image in $\PGL$, say by taking the
adjoint representation, of the monodromy or Galois group), at the cost of
including additional deformations (since the above $2g$-dimensional space
of deformations indeed preserve the projective monodromy).

The above action of a (symplectic) group scheme on the moduli space
interacts nicely with the natural map $\det R\rho_*$ from the moduli space
to $\Pic(C)$.  Indeed, tensoring with a connection on a line bundle simply
tensors the determinant with the $n$-th power of that bundle, assuming the
original connection is on a rank $n$ vector bundle.  We thus see in
particular that the Lagrangian subgroup $\Gamma(\omega_C)$ fixes the map to
$\Pic(C)$, which presumably plays the role of a moment map for the action.
In particular, this suggests that the quotient by $\Gamma(\omega_C)$ of any
fiber over $\Pic(C)$ should again be symplectic.

Something similar holds in the elliptic difference case, where again the
moduli space of first order ``equations'' (isomorphisms $q^*{\cal L}\cong
{\cal L}$ for ${\cal L}$ a line bundle of degree 0) is a symplectic group
scheme with a Lagrangian subgroup (now $\G_m$) respecting $\det R\rho_*$.

Of particular interest are cases in which the quotients of fibers by the
Lagrangian subgroup are 0-dimensional (so essentially rigid), or
2-dimensional (open symplectic leaves of Poisson surfaces).  The
calculations we did above for $-1$-curves and $-2$-curves can be carried
out more generally to find all possible divisors $D$ in the fundamental
chamber such that $D\cdot Q=0$ and $D^2=2g-2$ or $D^2=2g$ (since the
overall dimension is $D^2+2$, the reduction has dimension $D^2+2-2g$).  The
only possibility that exists for $g>1$ is (with respect to an even ruling)
$D=s+(g-1)f$, which is of course just the case of first-order equations
with no singularities.  For $g=1$, we obtain three additional cases, namely
the case $D=2s+f-2e_1$ with $D^2=0$ and the cases $D=2s+f-e_1-e_2$,
$D=3s+f-2e_1$ with $D^2=2$.  The first case can be ruled out because it
would correspond via an elementary transformation to a sheaf disjoint from
$Q$ on an odd ruled surface, and such sheaves cannot exist; the third case
can also be ruled out by considering the possible decompositions of $Q$ on
$X_1$.  The remaining case $D=2s+f-e_1-e_2$ was studied in the differential
case by \cite{KawaiS:2003}, where it was found that the moduli space is
rational, and the continuous isomonodromy transformations (deforming the
elliptic curve) are controlled by a special case of the usual Painlev\'e VI
equation.

Let us consider the elliptic difference analogue more closely.  We may use
elementary transformations to ensure that both singular points lie on the
same component of $Q$.  An equation then corresponds to a holomorphic map
$V\to {\cal L}_1\otimes q^*V$ where $V$ is a rank 2 vector bundle on $E$
with fixed determinant, the determinant of the map vanishes at $x_1$,
$x_2$, and ${\cal L}_1^2(-x_1-x_2)\otimes q^{\deg(V)}\cong \sO_C$.  The
action of the Lagrangian subgroup simply multiplies the map by a scalar, so
that we are really looking at the {\em projective} space of such maps,
modulo the action of $\Aut(V)$.  The simplest version takes the determinant
of $V$ to have odd degree, say 1, as then (with limited exceptions) $V$
will be the unique indecomposable bundle of that determinant.  In
particular, we find that $\dim\Hom(V,{\cal L}_1\otimes q^*V)=4$, with the
determinant giving a quadratic map
\[
\Hom(V,{\cal L}_1\otimes q^*V)
\to
\Gamma(\det(V)^{-1}\otimes {\cal L}_1^2\otimes q^*\det(V)).
\]
Since the latter line bundle has degree 2, we obtain a pencil of quadrics
inside the $\P^3$ of maps, with an open subset of the desired moduli space
being given by the complement of the base of the pencil inside the quadric
corresponding to $x_1$.  To identify the specific quadric surface and the
embedding of the curve $C$, note that any injective morphism as described
factors through a unique bundle $M_1$ such that $M_1/V\cong \sO_{x_1}$, and
similarly for $M_2$.  These bundles have degree 2, so are generically a sum
of two line bundles of degree 1, so that the set of injective morphisms
with a given $M_i$ is a $\G_m$-torsor.  (When $M_i$ is not a sum of line
bundles, it is a nontrivial self-extension of a line bundle, and we must
replace $\G_m$ by $\G_a$.)  Since each $M_i$ is classified by a quotient of
$\Pic^1(C)$ by an involution, we see that these give rulings of the
quadric.  To see how this meets the image of $C$, note that such sheaves
correspond to morphisms which are not injective, and thus to points on the
boundary of the relevant $\G_m$-torsor (or $\G_a$-torsor).  Such sheaves
are determined by giving the degree 1 line bundle through which they
factor, and each such bundle lies on a unique $M_i$.  In particular, the
map $\Pic^1(C)\to \P^1$ corresponding to $M_i$ is the quotient by the
involution ${\cal L}\mapsto \det(V)(x_i)\otimes {\cal L}^{-1}$, since
$\det(M_i)=\det(V)(x_i)$.  As long as $x_1$ and $x_2$ are distinct, these
give distinct rulings, and thus completely determine the anticanonical
quadric surface.

This is not quite the correct moduli space, for the simple reason that $V$
need not be stable in general.  If ${\cal L}_1(-x_1)$ and ${\cal
  L}_1(-x_2)$ are nontrivial, then the only unstable bundles that admit
morphisms with the correct nonzero determinant are those of the form ${\cal
  L}\oplus \det(V)\otimes {\cal L}^{-1}$ with ${\cal L}^2\cong
\det(V)\otimes {\cal L}_1$.  Modulo the action of $\Aut(V)$, each
choice of ${\cal L}$ gives rise to an $\A^1$ worth of additional points
in the moduli space, except that one sheaf on each $\A^1$ has nontrivial
stabilizer (modulo the action of the Lagrangian subgroup, that is) of order
2, and thus corresponds to a singular point of type $A_1$.

It follows that the true moduli space is obtained as follows from the data
$\det(V)$, ${\cal L}_1$, $x_1$, $x_2$ and $q$: embed $C$ in $\P^1\times
\P^1$ by the line bundles $\det(V)(x_1)$ and $\det(V)(x_2)$, then for each
point $y_i$ such that $\det(V)\otimes {\cal L}_1(-2y_i)$ is trivial, blow
up $y_i$ twice, and finally contract the four $-2$-curves that result.  In
particular, apart from the presence of singular points, this is precisely
the space on which a special case of the elliptic Painlev\'e equation acts.
Since the isomonodromy transformations induce isomorphisms of the
quasiprojective moduli space, they necessarily correspond to isomorphisms
between these surfaces.  The correspondence is somewhat more complicated
than just a direct relation between the isomonodromy equation and elliptic
Painlev\'e, for the simple reason that the $q$ parameter of the latter is
actually twice the original $q$, and thus the elliptic Painlev\'e
translation group can only give an index two subgroup of the true group of
isomonodromy transformations.  This relates to the fact that due to the
relations between the parameters, there are additional isomorphisms between
the moduli spaces; in particular, changing $\det(V)$ has no effect, and
there is even an action of $\Pic^0(C)[2]$ corresponding to twisting $V$ by
a $2$-torsion line bundle.  Modulo this residual freedom, reflecting by the
four orthogonal roots $f-e_{2i-1}-e_{2i}$ and then by $s-f$ has the desired
effect on the parameters, and thus establishes an isomorphism of moduli
spaces agreeing with the isomonodromy transformation modulo automorphisms.
It is straightforward to verify that the only automorphisms come from
$\Pic^0(C)[2]$ (an automorphism is determined by its action on the root
lattice of $\tilde{E}_8$, and must preserve the intersection form, the map
to $\Pic^0(C)$, and the set of effective roots), and thus the isomorphisms
agree for $C$ with no rational $2$-torsion, and thus in general.

\bigskip

The relation between differential equations on an elliptic curve with a
single singularity (say at 0) and Fuchsian differential equations with four
singular points is easy enough to explain in one sense: it simply
corresponds to rewriting the equation in terms of the function $x$ on the
elliptic curve.  As long as the original equation is invariant under
$z\mapsto -z$, this change of variables can be performed, and gives a
rational equation with singular points $0$, $1$, $\infty$ and $\lambda$,
with the exponents at those singular points determined from the original
exponents at $0$.  At that point, of course, the isomonodromy
transformation deforming the elliptic curve translates directly to an
isomonodromy transformation deforming $\lambda$.

This raises the question of how the given noncommutative surfaces are
related.  More generally, given an algebraic map $\phi:C\to C'$, we should
expect some relationship between noncommutative surfaces rationally ruled
over $C$ of differential type and noncommutative surfaces rationally ruled
over $C'$ of differential type.  The most natural form such a relation
could take would be that of a morphism (i.e., an adjoint pair of functors
$\pi_*$, $\pi^*$ between the two categories).  Indeed, given an equation on
$C'$ we can certainly pull it back, while an equation on $C$ has a direct
image equation on $C'$ (of order $\deg(\phi)$ times the original order), so
that the relation we considered in the elliptic case boils down to writing
the elliptic equation as a pullback.  The analogous functors for actual
sheaves will be somewhat complicated by ramification, and indeed it seems
likely that the codomain of the morphism of noncommutative surfaces will
end up being singular in general, as it was in the elliptic difference case
considered above.  (Although we have not discussed singular noncommutative
surfaces, they are easy enough to construct: simply take the $\Z$-algebra
corresponding to a divisor on the boundary of the nef cone.  But of course
we would like to know that the result only depends on the face of the nef
cone containing that divisor, and will need to understand the categories.)

Similar sources of morphisms should arise in the difference settings; not
only should there be a direct analogue to the above coming from morphisms
between the respective curves that are equivariant with respect to $q$ or
the infinite dihedral group, one also expects morphisms taking equations
$v(qz)=A(z)v(z)$ to $v(q^k z)=A(q^{k-1}z)\cdots A(z)v(z)$, as well as
morphisms related to forgetting the symmetry of a symmetric difference
equation.  These are closely related to the notion of $G$-equivariant sheaf
discussed above in the moduli space context, with some caveats.  For a
non-symmetric difference equation, the map from
\[
v(qz)=A(q^{k-1}z)\cdots A(z)v(z)\qquad\text{to}\qquad
v(qz)=A(q^k z)\cdots A(qz)v(z)
\]
corresponds to twisting by a line bundle (assuming that $A$ has no apparent
singularities, at least!), and thus we find that translation by $q$ has the
same effect as twisting by a line bundle and permuting the blowups, so that
the sheaf is indeed isomorphic to its image under the corresponding
automorphism of order $k$ of the abelian category.  Moreover, although we
cannot recover $A$ from the new equation alone (we could, e.g., multiply
$A$ by any automorphism of order $k$ of the original equation), we {\em
  can} recover it from a specific compatible choice of isomorphism, so that
the original equation truly does represent a $\Z/k\Z$-equivariant sheaf on
the new surface.  This also applies to the case of a nonsymmetric
difference equation obtained by forgetting the symmetry of a symmetric
difference equation, where now the automorphism has order 2 and acts on $Q$
in such a way as to swap the components while preserving $q$.  The map from
$q$-difference equations to $q^k$-difference equations is more subtle in
the symmetric case, and is only the quotient by an automorphism when $k=2$;
the difficulty more generally is that the symmetry combines with the cyclic
group to form the dihedral group of order $2k$.  Of course this means that
we can still view the morphism as a map between two different quotients of
the category of nonsymmetric equations.  Similar comments apply to the case
of a $q$-difference equation equation invariant under $z\mapsto \zeta_k z$.

Also of interest in this context are the symmetries of the moduli space
coming from duality, which include cases in which the discrete connection
takes values in $\text{GO}$, $\text{GSp}$, or $\text{U}$ (or in the
corresponding Lie algebra, in the differential case).  Here it is worth
noting that second order equations {\em always} have such a symmetry, since
$\GL_2\cong \text{GSp}_2$, or more concretely since for a $2\times 2$
matrix $A$, $\det(A)A^{-t}=J A J^{-1}$, where $J$ is any nonzero
alternating matrix.  (In the differential case, this becomes
$\tr(A)-A^t=JAJ^{-1}$.)  More precisely, given a sheaf of Chern class
$2s+df-e_1-\cdots-e_m$, the symmetry is given by composing the canonical
adjoint, the longest element of $W(D_m)$, and the twist by a suitable line
bundle, with the latter in general depending on $\det(A)$.  (Thus in the
higher genus case, this is really a symmetry of the subspace of the moduli
space on which the determinant of the equation has been fixed.)  This does
not preserve $\Sigma$, so switching back to the original $\Sigma$
introduces the requisite scalar gauge.  This description only works as
stated when none of the roots of $D_m$ are effective, but can be fixed
easily enough, at least generically.  Indeed, we can typically produce a
{\em derived} equivalence corresponding to the longest element of $W(D_m)$,
and this will take our sheaf to a sheaf unless that sheaf has a subquotient
of the form $\sO_\alpha(d)$ for $\alpha$ a root of $D_m$ (which essentially
says that the equation has apparent singularities).  The only issue (apart
from nonuniqueness of the derived equivalence) would be if we were ever
trying to reflect in an effective simple root of $D_m$ that was not
irreducibly effective.  But this cannot happen: this would imply that $Q$
had a component of the form $f-e_{i_1}-\cdots-e_{i_l}$ or
$e_{i_1}-\cdots-e_{i_l}$, but these have negative intersection with the
Chern class of our sheaf.

Although this symmetry in principle survives the Fourier transform, it does
so in a particularly obscure form, and indeed it is not clear whether the
resulting symmetry has any simpler description than as a conjugate by the
Fourier transform.  However, it can still lead to interesting consequences
when combined with other symmetries.  An interesting case is related to the
Lax pair for Painlev\'e VI of \cite{NoumiM/YamadaY:2002}, a differential
equation of the form $v' = (A_1z+A_0)v$ where $A_0,A_1\in \so(8)$.  Since
$\rank(A_1)=2$, applying the Mellin transform gives a second-order
difference equation, and this equation inherits a contravariant symmetry
from the original $\so(8)$ structure.  This symmetry combines with the
contravariant symmetry coming from having rank 2 to give a nontrivial
covariant symmetry, and one then finds that a suitable scalar gauge
transformation puts that symmetry in the form $A(z)=A(-z)^{-1}$.  In other
words, this second-order equation is a symmetric difference equation, which
turns out to be precisely the linear problem arising in
\cite{OrmerodCM/RainsEM:2017b}.  (One can also verify that the same thing
happens for the other two eighth-order differential equations arising via
triality.)

It may be instructive to work the above example backwards.  If we start
with the linear problem of \cite{OrmerodCM/RainsEM:2017b} and forget the
symmetry, then the result (modulo scalar gauge) is a sheaf of Chern class
$2s+4f-e_1-\cdots-e_{12}$ on a surface on which $Q$ decomposes as
\[
 (s-e_5-e_7-e_9-e_{11})
+(s-e_6-e_8-e_{10}-e_{12})
+2(f-e_1-e_2)
+(e_1-e_3)
+(e_2-e_4),
\]
with the original symmetry being reflected via the composition of swapping
$e_{2i-1}$ with $e_{2i}$ for each $i$ and a suitable involution on $Q$.
The corresponding contravariant symmetry involves the product of
reflections in the roots $f-e_5-e_6$, $f-e_7-e_8$, $f-e_9-e_{10}$,
$f-e_{11}-e_{12}$ (once we have taken into account the fact that some roots
of $D_{12}$ are components of $Q$), and thus still does not behave well
under the inverse Mellin transform.  However, if we reflect the sheaf in
$f-e_6-e_8$ and $f-e_{10}-e_{12}$, then the contravariant symmetry only
involves a permutation of the blowups.  This still does not allow a Mellin
transform--there are now too many singularities on one of the two
horizontal components of $Q$--but we can fix this by repeatedly performing
an elementary transformation in the point where that component meets the
vertical component.  (I.e., blow up that point, permute that blowup to be
the first blowup, then take the elementary transform.)  After doing that
four times, the result is a sheaf with Chern class
$2s+8f-2e_1-2e_2-2e_3-2e_4-e_5-\cdots-e_{16}$ on a surface with
anticanonical curve decomposing as
\begin{align}
Q=&\hphantom{+}(s-e_9-e_{10}-e_{11}-e_{12}-e_{13}-e_{14}-e_{15}-e_{16})
+(s-e_1-e_2-e_3-e_4)\\
&+2(f-e_1)
+2(e_1-e_2)
+2(e_2-e_3)
+2(e_3-e_4)
+2(e_4-e_5-e_6)
+(e_5-e_7)
+(e_6-e_8),\notag
\end{align}
with a contravariant symmetry involving permuting $e_9$ through $e_{16}$
and an involution on $Q$.  It follows that the inverse Mellin transform is
an 8th order equation that still has such a contravariant symmetry, and
thus after a scalar gauge to make the trace vanish preserves a bilinear form.

The above example points out two things: First, there is a fair amount of
freedom in how we obtained a sheaf having both a transform and a
contravariant symmetry surviving that transform in a recognizable form.
Indeed, although it was natural to perform the elementary transformations
in the point ``at infinity'' on the offending component, we could instead
have performed any four such transformations that respected the involution
acting on $Q$.  Most of the time this would still produce an 8th order
equation with a contravariant symmetry (albeit rather different qualitative
behavior), but including finite singularities of the difference equation in
the set also gives rise to 6th or even 4th order equations.  Second, it is
unclear how to tell a priori that the above construction gives an equation
in $\so_8$ rather than ${\mathfrak{sp}}_8$.

Geometrically, the source of this symmetry appears to be the fact that in
the commutative case, the linear system consists entirely of hyperelliptic
curves.  Another natural instance where this happens is the matrix
Painlev\'e case (on a rational surface) $D=4s+4f-2e_1-\cdots-2e_7-e_8-e_9$,
where the linear system consists of genus 2 curves.  There is indeed a
symmetry in this case as well, though the action on equations is much more
difficult to describe, as it involves acting by the longest element of
$W(E_8A_1)$ and (in the elliptic case) reflecting in the degree 2 divisor
class $x_8+x_9$.  It turns out (since $D|_Q\sim 0$) that these operations
give a covariant equivalence to the adjoint surface, and thus induce a
contravariant autoequivalence of the original surface.

This suggests that in addition to any intrinsic interest in understanding
morphisms between noncommutative surfaces, such an understanding would also
be quite fruitful in the application to special functions, e.g., by
systematically explaining quadratic (or higher-order) transformations.  In
addition, the symmetries of the moduli space associated to dualities may
give some insight into the structure of the moduli spaces of meromorphic
$G$-connections (or discrete $G$-connections) for more general structure
groups than $\GL_n$.  Note that we cannot expect there to be any reasonable
interpretation of a $G$-structure on a {\em sheaf} per se, for the simple
reason that the order of the corresponding equation depends on a choice of
ruling on the surface, so that changing the ruling may prevent the
corresponding $V$ from being a $G$-torsor.  (Indeed, this already happens
when $G=\GL_n$ and $n$ is not the minimal order of an interpretation of $M$
as an equation.)  Thus any notion of $G$-structure on a sheaf must at the
very least be taken relative to a choice of ruling, and possibly a choice
of section $\Sigma$.

\begin{appendices}
\section{Generalized Fourier transforms}

One of the main motivations for the derived category approach we have used
above is the sheer proliferation of cases that would otherwise need to be
considered.  For instance, in the case of the Fourier transform (i.e.,
swapping the rulings of $\P^1\times \P^1$), there are 16 different cases
that naturally arise (not even including some of the issues in
characteristic 2), and one would in each case need to show that the two
representations via operators not only satisfy the same relations, but also
give rise to the same category of sheaves.  We do not attempt the latter
directly, but for applications to special functions, it is still useful to
understand the former.  The main use is that, given a new linear problem
for an integrable system of isomonodromy type, one would generally like to
know whether it is truly new, or if it can be reduced to a known linear
problem.  In particular, if we translate it into a sheaf on an appropriate
noncommutative surface, then there is a blowdown structure relative to
which its Chern class is in the fundamental domain, and thus will give a
simplest form for the equation (in particular, of lowest order).  (This may
not be unique, of course, but there will be only finitely many such forms.)
The relation to the original linear problem is via the appropriate element
of $W(E_{m+1})$, and we have already seen that $W(D_m)$ acts by scalar
gauge transformations, so that the only truly nontrivial action involves
the Fourier transform.  Thus a suitable understanding of the Fourier
transform will let us understand all of the minimal linear problems
equivalent to a given linear problem.

One tricky issue is that on the untwisted (rational) ruled surfaces for
which we have natural interpretations of equations, the divisor class $s-f$
is always effective, and thus we do not have a Fourier transform (as an
abelian equivalence, that is).  Thus there is invariably an issue with
twisting to consider.  One approach would be to choose the normalizing
section $\Sigma$ to have Chern class $s+f$; this is ample when $s-f$ is
ineffective, and thus there is always such a sheaf.  This has the
disadvantage of introducing additional parameters, and thus additional ways
for things to degenerate; we obtain a total of 48 different possibilities
for the possible structure of the anticanonical curve on the separating
blowup for such a $\Sigma$.  (To be precise, if we take into account the
order in which the points are blown up, there are 225 cases, but they fall
into 48 orbits.)  It is also somewhat cumbersome to compute the transform
in this form, as $\Sigma$ normalizes equations, but does not quite
normalize an algebra of operators: the issue is that $\Sigma$ is not
invariant under twisting by line bundles, so we can only use this
normalization to control operators acting on the trivial vector bundle.
This approach is still workable, as we can still use it to compute the
kernel of a suitable formal integral transform (or, more precisely, the
equations satisfied by the kernel), which is enough to enable the
computation of the Fourier transform on equations in straight-line form.

An alternate approach is to simply find {\em some} representation in terms
of operators in each of the 16 cases, and then check that there are
isomorphisms as required.  The main disadvantage of this approach is that
we need to represent the full category of morphisms between line bundles on
$X_0$, and there is a great deal of nonuniqueness in that representation.
In particular, if we assign to each line bundle a first-order equation,
with specified gauge equivalences between them, then we can gauge by the
resulting system to obtain a new representation.  (Note that $\Sigma$
only specifies the equation associated to the trivial bundle!)  We can also
similarly gauge by a system of automorphisms of $Q$, with the resulting
effect on matrix equations being to pull back by the automorphism
associated to the trivial bundle.  Although this nonuniqueness makes it
relatively easy to find representations, there is a significant cost when
it comes to understanding limits: to degenerate one case to another, it
may be necessary to make a suitable gauge transformation first.

If we consider applying such a representation to computing the Fourier
transform of an equation, we find that there is a considerable
simplification available.  The point is that the sheaves
$\sO_X,\sO_X(-s),\sO_X(-f),\sO_X(-s-f)$ form a strong exceptional
collection, and thus if $M$, $M(-s)$, $M(-f)$, and $M(-s-f)$ are all
acyclic, then $M$ has a resolution of the form
\[
0\to \sO_X(-s-f)^a\to \sO_X(-f)^b\oplus \sO_X(-s)^c\to \sO_X^d\to M\to 0.
\]
In other words, $M$ can be expressed as the solution of a system of $b+c$
equations in $d$ unknowns, with each equation either being a linear
equation (with coefficients in $\Hom(\sO_X(-f),\sO_X)$) or a first-order
difference/differential equation (coming from $\Hom(\sO_X(-s),\sO_X)$).  We
can recover the corresponding (discrete) connection by using the linear
equations to express the $d$ unknowns as the global sections of a vector
bundle, and observing that the remaining equations describe a (discrete)
connection on that vector bundle.  But then to understand how the Fourier
transform acts on equations, it suffices to understand how it acts on
suitable sections of $\Hom(\sO_X(-f),\sO_X)$ and $\Hom(\sO_X(-s),\sO_X)$.
Moreover, the condition for a pair of maps from those 2-dimensional spaces
to operators (with the first mapping to multiplication operators and the
second to first-order operators) to extend to a representation of the
category is basically that there be two more such maps, from
$\Hom(\sO_X(-s-f),\sO_X(-s))$ and $\Hom(\sO_X(-s-f),\sO_X(-f))$, such that
the compositions span a $4$-dimensional space.  Note that by a judicious
use of the gauge freedom, we can choose the representation so that twisting
by $s+f$ does not affect the representation: gauge by a suitable square
root of the anticanonical natural transformation.  We can furthermore
arrange in this way for all of the spaces of multiplication operators to
agree, at which point we can deduce all of the spaces of degree $s$ given
one such space and the compatibility condition.

For instance, in the (symmetric) elliptic difference case, we choose a
ramification point of $Q\to C_0$ to make it an honest elliptic curve, and
then use the remaining gauge-by-automorphisms freedom (at the cost of
choosing an element $q/2$) to ensure that all of the operators are
invariant under $z\mapsto -z$.  Then there is a representation in which the
typical operator of degree $f$ is proportional to $\vartheta(z\pm
u):=\vartheta(z+u,z-u):=\vartheta(z+u)\vartheta(z-u)$, while the typical
operator of degree $s$ is proportional to
\[
D_q(c\pm u)
:=
\frac{\vartheta(c+u+z,c-u+z)}{\vartheta(2z)} T^{1/2}
+
\frac{\vartheta(c+u-z,c-u-z)}{\vartheta(-2z)} T^{-1/2}
\]
where $(T^{\pm 1/2}f)(z)=f(z\pm q/2)$ and $c$ is a parameter depending not
only on the surface but on the domain of the morphism of line bundles.
We
in particular find (by comparing coefficients of $T^{1/2}$, say) that the
spans of
\[
D_q(c\pm v)\vartheta(z\pm u)
\qquad\text{and}\qquad
\vartheta(z\pm u)D_q(c+q/2\pm v)
\]
agree and are 4-dimensional, so that these indeed extend to give a
representation of a category of the desired form.  Moreover, we see (by
comparison to \cite{generic}, or simply by noting that each $2$-dimensional
space is a space of global sections of a line bundle, and the relations are
the same as those satisfied by the global sections) that this actually
gives the general form of a relation in the elliptic case, so every
noncommutative $\P^1\times \P^1$ with smooth anticanonical curve has a
representation in this form.  The Fourier transform can be viewed as a
formal system of operators ${\cal F}_q(c)$ such that ${\cal F}_q(-c)={\cal
  F}_q(c)^{-1}$ and
\[
\vartheta(z\pm u) {\cal F}_q(c) = {\cal F}_q(c+q/2) D_q(c+q/2\pm u).
\]
(One can in fact take ${\cal F}_q(c)$ to be a certain formal difference
operator, as the univariate case of \cite[\S8]{elldaha}, though it is
perhaps more natural to view it as a formal integral operator.)
In particular, given an equation of the form
\[
\vartheta(z\pm u_1) \lambda_1\cdot v(z) +
\vartheta(z\pm u_2) \lambda_2\cdot v(z) = 0,
\]
then $w={\cal F}_q(c)v$ formally satisfies
\[
D_q(-c+q/2\pm u_1) \lambda_1\cdot w(z) +
D_q(-c+q/2\pm u_2) \lambda_2\cdot w(z) = 0,
\]
and similarly if
\[
D_q(c+q/2\pm u_1) \lambda_1\cdot v(z) +
D_q(c+q/2\pm u_2) \lambda_2\cdot v(z) = 0,
\]
then
\[
\vartheta(z\pm u_1) \lambda_1\cdot w(z) +
\vartheta(z\pm u_2) \lambda_2\cdot w(z) = 0.
\]
The compatibility condition essentially reduces to
\[
  {\cal F}_q(c) D_q(u_0,u_1,u_2,u_3)
  =
  D_q(u_0-c,u_1-c,u_2-c,u_3-c){\cal F}_q(c)
\]
for $u_0+u_1+u_2+u_3=2c+q$, where
\[
D_q(u_0,u_1,u_2,u_3)
:=
\frac{\vartheta(u_0+z,u_1+z,u_2+z,u_3+z)}{\vartheta(2z)} T^{1/2}
+
\frac{\vartheta(u_0-z,u_1-z,u_2-z,u_3-z)}{\vartheta(-2z)} T^{-1/2}
\]
is (modulo scalars) the typical element of degree $s+f$.  (In particular,
these elements generate a graded algebra which is the quotient of the
Sklyanin algebra \cite{SklyaninEK:1982,SklyaninEK:1983} by a central
element of degree 2, see also \cite{RosengrenH:2004,sklyanin_anal}.)

The top $q$-difference case (so again with $Q$ integral) is a
straightforward limit of the elliptic case: we simply replace the function
$\vartheta$ by $\exp(-\pi\sqrt{-1}z)-\exp(\pi\sqrt{-1}z)$, and replace the
various variables and parameters by suitable logarithms.  Thus the
multiplication operators become $z+1/z-u-1/u$, the difference operators
become
\[
D_q(c u^{\pm 1})
:=
\frac{cz+1/cz-u-1/u}{1/z-z} T^{1/2}
+
\frac{c/z+z/c-u-1/u}{z-1/z} T^{-1/2}
\]
and
\begin{align}
&D_q(u_0,u_1,u_2,u_3)\notag\\
&\,:=
\frac{(1-u_0z)(1-u_1z)(1-u_2z)(1-u_3z)}{\sqrt{u_0u_1u_2u_3}z(1-z^2)}T^{1/2}
+
\frac{(1-u_0/z)(1-u_1/z)(1-u_2/z)(1-u_3/z)}{\sqrt{u_0u_1u_2u_3}z^{-1}(1-z^{-2})}T^{-1/2},
\end{align}
and the Fourier transform acts by
\begin{align}
  {\cal F}_q(q^{-1/2}c) (z+1/z-u-1/u) &= D_q(q^{1/2}u^{\pm 1}/c) {\cal F}_q(c),
\\
  {\cal F}_q(q^{1/2}c) D_q(q^{1/2}cu^{\pm 1}) &= (z+1/z-u-1/u) {\cal F}_q(c),
\\
  {\cal F}_q(c) D_q(u_0,u_1,u_2,u_3)
  &=
  D_q(u_0/c,u_1/c,u_2/c,u_3/c){\cal F}_q(c)\qquad (u_0u_1u_2u_3=qc^2).
\end{align}
The top ordinary difference case is similar: just replace $\vartheta(z)=z$
and optionally set $q$ to $1$.  Note that in these cases, it is no longer
true that every operator in the appropriate space is proportional to one of
the given form (e.g., we now have a multiplication operator $1$), but this
remains true for a dense set of operators.  Also, in these cases, we may
interpret ${\cal F}_q(c)$ as a fractional power of a suitable symmetric
lowering operator (the Askey-Wilson or Wilson operator, as appropriate).
Indeed, the limit as $u\to\infty$ of the main identity of the Fourier
transform gives
\[
{\cal F}_q(c) = {\cal F}_q(q^{1/2}c) (1/z-z)^{-1} (T^{1/2}-T^{-1/2}),
\]
and thus
\[
{\cal F}_q(q^{-n/2})
=
{\cal F}_q(1)
\bigl((1/z-z)^{-1} (T^{1/2}-T^{-1/2})\bigr)^n,
\]
with ${\cal F}_q(1)$ acting trivially.  This lowering operator takes
symmetric Laurent polynomials to symmetric Laurent polynomials, decreasing
the degree, and thus in particular the equation
\[
(1/z-z)^{-1} (T^{1/2}-T^{-1/2})v=0
\]
may be viewed as the normalizing section $\Sigma$.  The additive case
similarly corresponds to powers of $(2z)^{-1}(T^{1/2}-T^{-1/2})$.
(Something similar is true in the elliptic case, see \cite[\S8]{elldaha},
although it is no longer true that the result is actually a power of the
operator for $c=q^{-1/2}$.)

The next easiest set of degenerations to consider are those for which $Q$
has two components of class $s+f$, and thus the operators are no longer
symmetric.  Here we first encounter the issue mentioned above with limits:
in order to obtain a nonsymmetric operator from a symmetric operator, we
must conjugate by an automorphism to reintroduce the symmetry as a
parameter and then take the limit in that parameter.  Thus in the
$q$-difference case, we conjugate by $z\mapsto wz$ and take a limit
$w\to\infty$.  Doing so in a naive fashion makes the operators
of interest blow up, and rescaling the operators kills the parameters, so
we find that we must also rescale $u$.  There remains a further subtle
issue, in that we need to rescale the operators of degree $f$ but not of
degree $s$, which implicitly means that we need to include an additional
scale factor in the Fourier transforms.  In the end, we find that the
operators of degree $f$ become $z-u$, the operators of degree $s$ become
$(-cT^{1/2}+c^{-1}T^{-1/2})+(u/z)(T^{1/2}-T^{-1/2})$, and the Fourier
transform acts by
\[
  (z-u) {\cal F}'_q(c) = {\cal F}'_q(q^{1/2}c)
  ((u/z-q^{1/2}c)T^{1/2} +(-u/z+1/q^{1/2}c)T^{-1/2}).
\]
(We omit the action on elements of degree $s+f$.)  This again corresponds
to a fractional power of a lowering operator, namely
$z^{-1}(T^{-1/2}-T^{1/2})$.  Similarly, the additive case involves taking a
limit $w\to\infty$ after substituting $z\mapsto z+w$, $u\mapsto u+w$, and
gives (again up to rescaling ${\cal F}'$ accordingly)
\[
(z-u){\cal F}'_q(c) = {\cal F}'_q(c+q/2)
((u-z-c-q/2)T^{1/2}+(z-u-c-q/2)T^{-1/2}),
\]
corresponding to fractional powers of $T^{-1/2}-T^{1/2}$.  If we then let
$q,c\to 0$ at comparable rates, then a similar rescaling gives a transform
involving differential operators:
\[
(z-u) {\cal F}'(\alpha) = {\cal F}'(\alpha+1/2)
((z-u)D+(2\alpha+1)),
\]
again with ${\cal F}'(-\alpha)={\cal F}'(\alpha)^{-1}$.  This is
essentially the transform of \cite{KatzNM:1996} (often called ``middle
convolution'' in the later literature).  Moreover, by taking a limit of the
corresponding fractional-power-of-lowering-operator interpretation, we see
that it can be described via fractional differentiation: ${\cal F}'(\alpha)
= D^{-2\alpha}$; indeed, one has
\[
D^{2\alpha+1} (z-u) = ((z-u)D+(2\alpha+1)) D^{2\alpha}.
\]

We next turn to the transforms that map between symmetric and nonsymmetric
operators, or in geometric terms relate surfaces with $Q=(2s+f)+(f)$ to
surfaces with $Q=(s+2f)+(s)$.  Here the main complication is that by making
the symmetric operators truly symmetric, we obtain a transform without any
parameters, but still need mild dependence of the transform on the domain
of the morphism.  Luckily, this dependence is very mild, and indeed simply
involves a power of $T^{1/2}$.  The overall idea is to take the limit as
$c\to 0$ of the symmetric $q$-Fourier transform, except that we need to
include an overall shift of the parameter on one side to break the
symmetry.  It also turns out that we need to include an overall gauge on
the nonsymmetric side to make the limits work, which we do in such a way as
to ensure that the operators take polynomials to polynomials.  This gives a
pair of inverse transforms ${\cal F}^{<}_q$ and ${\cal F}^{>}_q$ such
that
\begin{align}
{\cal F}^{<}_{q}q^{-1/2}(z-u)T^{-1/2}
&=
\bigl(
\frac{z-u}{z^2-1}T^{-1/2}
-
\frac{z^{-1}-u}{z^{-2}-1}T^{1/2}
\bigr)
{\cal F}^{<}_{q}\\
{\cal F}^{<}_{q}
\bigl((z+1/z-u-1/u)T^{-1/2}-z^{-1}T^{1/2}\bigr)
T^{1/2}
&=
(z+1/z-u-1/u){\cal F}^{<}_{q}.
\end{align}
The additive equivalent is
\begin{align}
{\cal F}^{<}_{+}T^{-1/2}
(z-u)
&=
\bigl(
\frac{z-u-1/2}{2z}T^{-1/2}+
\frac{z+u+1/2}{2z}T^{1/2}
\bigr)
{\cal F}^{<}_{+}\\
{\cal F}^{<}_{+}
\bigl((z^2-u^2)T^{-1/2}-T^{1/2}\bigr)
T^{1/2}
&=
(z^2-u^2){\cal F}^{<}_{+}.
\end{align}

The next cases to consider are those for which $Q$ has components of Chern
class $s+f$, $s$, and $f$, which again comes in multiplicative and additive
flavors.  In these cases, the corresponding commutative Poisson surface is
uniquely determined, and thus we have the additional possibility of
arranging for the Fourier transform to act as an involution on that
surface.  There is, however, a technical issue that arises, particularly in
the multiplicative case: it turns out that the involution is not Poisson,
but rather anti-Poisson.  (It has non-isolated fixed points in the open
symplectic leaf, so negates the volume form at those points.)  As a result,
the simplest form of the corresponding Fourier transform turns out to
invert $q$.  Note also that because $Q$ has components of degrees $s$ and
$f$, we can use the corresponding natural transformations to determine the
gauge in place of a square root of the anticanonical natural
transformation.  This has the advantage of expressing the category of line
bundles as a Rees algebra relative to a filtration by $\N^2$.  The
corresponding filtered algebra may be taken to be generated by $z$ (of
degree $f$) and by the lowering operator $z^{-1}(1-T)$ (of degree $s$).
These generators satisfy the equivalent relations
\[
yx = qxy+(1-q)\qquad\text{and}\qquad
xy = q^{-1}yx+(1-q^{-1}),
\]
so that swapping $x$ and $y$ gives the algebra with $q$ inverted.  The
additive case is somewhat simpler, as in that case we can arrange for the
involution to be actually Poisson.  We obtain a bifiltered algebra with
relation
\[
[y,x]=y-x
\]
and involution that swaps $x$ and $y$, with a representation $x\mapsto z$,
$y\mapsto z+T$ in difference operators.

The remaining cases have similar interpretations in terms of bifiltered
algebras.  The multiplicative case (with $Q=(s)+(s)+(f)+(f)$) is
particularly simple: the filtered algebra is simply the $q$-Weyl algebra
$yx=qxy$, with representation $x\mapsto z$, $y\mapsto T$ and involution
that again inverts $q$.  The additive case ($Q=(s)+(s)+2(f)$) transforms
to a differential case ($Q=2(s)+(f)+(f)$), with filtered algebra
\[
[y,x] = y
\]
and representations $x\mapsto z$, $y\mapsto T$ and $x\mapsto -tD$,
$y\mapsto t$.  The corresponding transforms are quite familiar: from the
differential side to the discrete side is precisely the Mellin transform,
and the inverse is the $z$ transform.  (And, of course, the
differential-to-discrete direction underlies the standard power series
method for solving linear ODEs.)  The remaining case ($Q=2(s)+2(f)$) is the
usual Weyl algebra $[y,x]=1$, with representations $x\mapsto t$, $y\mapsto
D$ and $x\mapsto -D$, $y\mapsto t$.  This of course is just the classical
Fourier transform.

\medskip

Our purpose above was to give enough information to enable explicit
calculations of the transforms in special functions applications.  As such,
there are a number of issues we have not attempted to address.  One is that
in the elliptic case \cite{generic} and in the cases involving differential
operators (classical), the generalized Fourier transform can be expressed
as an integral operator.  Of course, in the differential cases, there are
potential issues with convergence (e.g., the usual Fourier transform is
difficult to define unless the argument decays at infinity), and in
particular it is unclear whether there is an actual operator acting on a
space containing solutions to the equations of interest.  (This was mostly
settled in the elliptic case, but even there, the actual contour integral
description requires additional constraints on the solutions, so that not
every solution appears, just a sufficiently independent set of solutions.)
The situation is worse in the discrete cases, as the solutions are only
determined up to multiplication by meromorphic functions invariant under
the shifts.  In particular, one finds that the kernel of the formal
integral transform is itself only determined up to shift-invariant
functions.  One may also be tempted to think that since the Fourier
transform takes operators of degree $s+f$ to operators of degree $s+f$,
that this means that the Fourier transform of an explicit solution of such
an equation will be easily expressible in closed form.  Although this is
true in many cases, the indeterminacy of the solution and kernel poses some
difficulties.  Indeed, all we know is that the integral is a solution of
the equation, but this only determines the integral up to a periodic
function of $z$, which need not have any nice form.

Another is that at the elliptic level, the transforms have natural
multivariate analogues related to Macdonald theory; indeed, the limit to
the most general multiplicative case is essentially a fractional power of
the Macdonald-Koornwinder lowering operator.  This suggests that all 16 of
the above transforms should have similar multivariate analogues.

An alternate approach to constructing the list of Fourier transforms
involves degenerating the canonical gauge of \cite{generic}.  In general,
any surface of type $F_1$ such that $s$ is not a component of $Q$ has such
a canonical gauge, which (up to constants) is given by gauging so that
every operator of degree $s$ annihilates $1$, while the operators of degree
$f$ remain rational functions on $Q$.  In particular, if we start with a
surface of type $F_0$ and blow up a point which is not on a component of
$Q$ of class $s$ or $f$, then either way of blowing down to $F_1$ gives a
surface with such a canonical gauge.  This makes the limits easier to take
(there is no longer a need to consider a change of gauge), at the cost of
both splitting some of the cases (e.g., at the top level, there is a
subcase in which the point being blown up is a singular point of $Q$) and
not applying to the four cases in which every component of $Q$ has class
$s$ or $f$.  There is a further variant of this approach that applies to
these cases as well: simply blow up three points in such a way that there
is no component of $Q$ having negative intersection with $s+f-e_1-e_2-e_3$,
and gauge so that the corresponding operator annihilates $1$.  Each of
these approaches is analogous to the normalization of equations via a
choice of section, except that in order to normalize morphisms of line
bundles, we must choose a normalizing sheaf for each line bundle, which in
the above approach we do by specifying the Chern class and all but one of
the points of intersection with $Q$.

\end{appendices}

\bibliographystyle{plain}

\end{document}